\documentclass[a4paper, 11pt]{amsart}

\usepackage[utf8]{inputenc}
\usepackage[english]{babel}
\usepackage[all]{xy}
\usepackage{enumitem}
\usepackage[dvipsnames, x11names]{xcolor}
\usepackage{wrapfig, graphicx}
\usepackage{amsfonts,amstext,amssymb,amsmath,amsthm,
	fancybox}
\usepackage{pifont}
\usepackage[left=2.5cm,right=2.5cm,top=2.5cm,bottom=2.5cm]{geometry}
\usepackage[colorlinks=true, pdfstartview=FitV, linkcolor=altblue, citecolor=altred, urlcolor=blue,pagebackref=false]{hyperref} %Pour un sommaire interractif.
\usepackage{fourier-orns}
\usepackage{array}
\usepackage{fancyhdr}
\usepackage{setspace}
\usepackage{tikz}
\usetikzlibrary{patterns,shapes,decorations.pathmorphing} 
\usepackage{mathrsfs}
\usepackage{dsfont}
\usepackage{mathdots}
\usepackage{cleveref}
\usepackage[colorinlistoftodos, french, bordercolor=black, linecolor=black, textsize=small]{todonotes}
\usepackage{yhmath}
\newcolumntype{M}[1]{>{\centering\arraybackslash}m{#1}}
\renewcommand{\epsilon}{\varepsilon}
\renewcommand{\phi}{\varphi}
\newcommand\R{\mathbb{R}}

\newcommand\C{\mathbb{C}}

\newcommand\Z{\mathbb{Z}}
\newcommand\T{\mathbb{T}}
\newcommand{\inv}[1]{\frac{1}{#1}}
\renewcommand{\hat}{\widehat}
\renewcommand{\tilde}{\widetilde}
\renewcommand{\bar}{\overline}
\renewcommand{\leq}{\leqslant}
\renewcommand{\geq}{\geqslant}

\renewcommand{\Im}{\operatorname{Im}}
\newcommand{\norme}[1]{\left\Vert #1\right\Vert}
\newcommand{\normetriple}[1]{{\left\vert\kern-0.25ex\left\vert\kern-0.25ex\left\vert #1 \right\vert\kern-0.25ex\right\vert\kern-0.25ex\right\vert}}
\newcommand{\prodscalbis}[1]{\left\langle #1\right\rangle}

\newcommand{\ensemble}[1]{\left\lbrace #1\right\rbrace}

\newcommand{\indicatrice}{\mathds{1}}
\newcommand{\privede}[1]{\setminus\ensemble{#1}}

\DeclareMathOperator{\Ima}{Im}

\DeclareMathOperator{\diag}{diag}

\DeclareMathOperator{\Span}{Span}
\DeclareMathOperator{\GL}{GL}

\newcommand{\CK}{Cauchy-Kovalevskaya }

\newcommand\F{\mathcal{F}}
\newcommand\E{\mathbf{E}}

\newcommand{\J}{\mathbb{J}}
\renewcommand{\P}{\mathcal{P}}

\newcommand\B{\mathcal{B}}

\newcommand{\X}{\mathbf{X}}
\newcommand{\Y}{\mathbf{Y}}
\renewcommand{\a}{\mathbf{a}}
\renewcommand{\b}{\mathbf{b}}

\newcommand{\freq}{\mathbf{n}\cdot\boldsymbol{\zeta}}

\DeclareMathOperator{\osc}{osc}
\DeclareMathOperator{\ev}{ev}
\DeclareMathOperator{\app}{app}

\DeclareMathOperator{\sign}{sign}

\DeclareMathOperator{\nc}{nc}
\DeclareMathOperator{\Lop}{Lop}
\DeclareMathOperator{\per}{per}

\definecolor{altblue}{RGB}{0, 0, 180}
\definecolor{altred}{RGB}{200, 0, 0}
\definecolor{altorange}{RGB}{243, 136, 0}
\definecolor{altorange2}{RGB}{190, 100, 0}
\definecolor{altgreen}{RGB}{0, 130, 0}
\definecolor{altpurple}{RGB}{169, 0, 243}
\definecolor{altgrey}{RGB}{40, 55, 71}
\definecolor{altpink}{RGB}{255, 48, 145}

\definecolor{rouge1}{RGB}{128, 0, 0}
\definecolor{rouge2}{RGB}{134, 45, 45}
\definecolor{rouge3}{RGB}{77, 0, 0}
\definecolor{orange1}{RGB}{153, 38, 0}
\definecolor{marron1}{RGB}{102, 51, 0}
\definecolor{jaune1}{RGB}{204, 122, 0}
\definecolor{violet1}{RGB}{153, 0, 77}
\definecolor{violet2}{RGB}{128, 0, 128}
\definecolor{rose1}{RGB}{255, 0, 128}
\definecolor{violet3}{RGB}{77, 0, 153}
\definecolor{bleu1}{RGB}{0, 0, 102}
\definecolor{bleu2}{RGB}{0, 45, 179}
\definecolor{bleu3}{RGB}{0, 122, 153}
\definecolor{vert1}{RGB}{32, 96, 64}
\definecolor{vert2}{RGB}{0, 102, 0}
\definecolor{vert3}{RGB}{77, 77, 0}

\allowdisplaybreaks

\newtheorem{theorem}{Theorem}[section]
\crefname{theorem}{Theorem}{Theorems}
\newtheorem{lemma}[theorem]{Lemma}
\crefname{lemma}{Lemma}{Lemmas}
\newtheorem{proposition}[theorem]{Proposition}
\crefname{proposition}{Proposition}{Propositions}

\crefname{corollary}{Corollary}{Corollaries}
\newtheorem{definition}[theorem]{Definition}
\crefname{definition}{Definition}{Definitions}
\newtheorem{assumption}{Assumption}
\crefname{assumption}{Assumption}{Assumptions}
\theoremstyle{remark}
\newtheorem{remark}[theorem]{Remark}
\crefname{remark}{Remark}{Remarks}

\crefname{example}{Example}{Examples}

\numberwithin{equation}{section}
\title[Transverse instability of weakly stable quasilinear boundary value problems]{Transverse instability of high frequency weakly stable quasilinear boundary value problems}

\author{Corentin Kilque}
\address{Institut de Mathématiques de Toulouse ; UMR5219 \\ Université de Toulouse ; CNRS \\
		UPS, F-31062 Toulouse Cedex 9, France}
\email{corentin.kilque@math.univ-toulouse.fr}

\begin{document}
	
	\maketitle
	
	\begin{abstract}
			This work intends to prove that strong instabilities may appear for high order geometric optics expansions of weakly stable quasilinear hyperbolic boundary value problems, when the forcing boundary term is perturbed by a small amplitude oscillating function, with a transverse frequency. Since the boundary frequencies lie in the locus where the so-called Lopatinskii determinant is zero, the amplifications on the boundary give rise to a highly coupled system of equations for the profiles. A simplified model for this system is solved in an analytical framework using the \CK theorem as well as a version of it ensuring analyticity in space and time for the solution. Then it is proven that, through resonances and amplification, a particular configuration for the phases may create an instability, in the sense that the small perturbation of the forcing term on the boundary interferes at the leading order in the asymptotic expansion of the solution. Finally we study the possibility for such a configuration of frequencies to happen for the isentropic Euler equations in space dimension three.
	\end{abstract}
	
	\tableofcontents

This work takes interest into the (in)-stability of multiphase geometric optics expansions for weakly stable quasilinear hyperbolic boundary value problems. The formal construction of such geometric optics expansions goes back to Majda, Artola, and Rosales, in \cite{MajdaRosales1983Machstem,MajdaRosales1984Machstem,ArtolaMajda1987Instabilities,MajdaArtola1988Mixed}. In this paper, we prove that, for a simplified model, infinitely accurate approximate solutions can be unstable, in the sense that a small perturbation of the boundary forcing term interferes at the leading order in the asymptotic expansion. 

For uniformly stable problems, the construction of a multiphase asymptotic expansion is performed, for Cauchy problems, notably in \cite{HunterMajdaRosales1986Resonantly} for the linear case, and in \cite{JolyMetivierRauch1995Coherent} for the quasilinear one. In the case of boundary value problems, \cite{Williams1996Boundary} studies the semilinear case with multiple frequencies on the boundary, and the quasilinear case is treated in \cite{CoulombelGuesWilliams2011Resonant} for one phase on the boundary. The case of multiple phases on the boundary is addressed in a previous work of the author, \cite{Kilque2021Weakly}. In the weakly stable case, that is, when the weak Kreiss-Lopatinksii condition is satisfied, an amplification phenomenon occurs, as shown in works of Coulombel, Guès and Williams. 
Following the pioneering works by Majda and his collaborators, the first rigorous construction of a geometric optics expan\-sion in the weakly stable case is performed in \cite{CoulombelGues2010Geometric} for linear boundary value problems. Nonlinear problems are treated in \cite{CoulombelGuesWilliams2014Semilinear,CoulombelWilliams2014Pulses} for the semilinar case, and \cite{CoulombelWilliams2017Mach} for the quasilinear one. In \cite{CoulombelGuesWilliams2014Semilinear,CoulombelWilliams2014Pulses,CoulombelWilliams2017Mach}, the authors consider one phase on the boundary, and the present w work intends to address the extension of \cite{CoulombelWilliams2017Mach} to the multiphase case. 
Here, allowing multiple phases on the boundary permits us to consider a particular configuration of frequencies on the boundary, which, thanks to the amplification phenomenon, will lead to an instability for the asymptotic expansion. We show however that, fixing a locus of breaking of the Kreiss-Lopatinskii condition, this configuration of frequencies creating an instability cannot happen for the example of gas dynamics. This leaves hope to justify the validity of the geometric optics expansions with one amplification for the gas dynamics.

This work is divided in three main parts: (i) the derivation of the equations satisfied by the profiles in the multiphase case, following \cite{CoulombelWilliams2017Mach}; (ii) the proof of existence of solutions to the obtained system in an analytical framework; and (iii) the proof of instability for this system, namely, that there exists a perturbation of the boundary forcing term interfering at the leading order in the expansion. The general system being out of our reach for the moment, both for existence and for the instability mechanism, we deal with simplify models of the general system of equations for the profiles.

For the boundary value problem considered in this paper, the boundary condition is assumed to satisfy the weak Kreiss-Lopatinskii condition, namely that the Kreiss-Lopatinskii condition breaks on a certain locus of the frequency space. More precisely, we assume here that the locus where the Kreiss-Lopatinskii condition is not satisfied lies in the hyperbolic region (see \cite[Definition 2.1]{BenzoniSerre2007Multi}). For boundary frequencies for which the Kreiss-Lopatinskii condition is not satisfied, an amplification phenomenon occurs on the boundary. The idea is to consider a particular configuration of frequencies on the boundary which will turn this amplification into a strong instability. For this purpose, we consider a boundary forcing term $ G $ oscillating at a frequency $ \phi $, belonging to the locus where the Kreiss-Lopatinskii condition is not satisfied, and we perturb this boundary forcing term $ G $ with a perturbation term $ H $, oscillating at a \emph{transverse} frequency $ \psi $ also belonging to the locus where the Kreiss-Lopatinskii condition is not satisfied, of small amplitude compared to the one of $ G $. In \cite{CoulombelWilliams2017Mach}, the boundary frequency $ \phi $ is assumed to be non-resonant, in the sense that two interior frequencies lifted from $ \phi $ cannot resonate with each other. We make the same assumption here, as well as for the boundary frequency $ \psi $. We assume however that two well chosen resonance relations between $ \phi $ and $ \psi $ exists, which will allow the perturbation $ H $ to ascend towards the leading order, through repeated amplification and resonances. We study in this article the possibility for such a configuration of frequencies to happen for the isentropic compressible Euler equations in space dimension $ 3 $. We prove that for a particular choice of locus where the Kreiss-Lopatinskii condition is not satisfied, the configuration of boundary frequencies considered here is (thankfully) impossible for the Euler system. 

The derivation of equations for the amplitudes from the BKW cascade follows the one detailed in \cite{CoulombelWilliams2017Mach}. The main difference with the iterative process in the uniformly stable case, see e.g.  \cite{JolyMetivierRauch1995Coherent,Williams1996Boundary,CoulombelGuesWilliams2011Resonant,Kilque2021Weakly}, is that, for boundary frequencies lying in the locus where the Kreiss-Lopatinski condition is not satisfied, we cannot a priori determine a boundary condition for incoming profiles. Indeed, because of the weak Kreiss-Lopatinskii condition, for such boundary frequencies, the traces of incoming profiles are expressed through an unknown scalar function. For a given order, the evolution equations satisfied by these boundary terms are derived using equations on profiles of the next order. This is where amplification occurs. The main difference with \cite{CoulombelWilliams2017Mach} is that, because of resonances, equations for each profile and for boundary terms are coupled with each other. Also, in comparison with \cite{CoulombelWilliams2017Mach}, in equations for the boundary term of a given order, there is a term involving the trace of a profile of the next order, which was proven to be zero in \cite{CoulombelWilliams2017Mach}, because resonances were absent in that work. This results into a highly coupled system. Nevertheless we discuss two points about this system, which are the existence of a solution to it and the creation of an instability. 

The existence of a solution to the system of equations for the profiles is proven here in an analytical setting. The aim is to use the abstract \CK theorem, whose proof can be found in \cite{Nirenberg1972CauchyKowalewski} and \cite{Nishida1977Nirenberg}. We use in this work the formulation of \cite{BaouendiGoulaouic1978Nishida}. The system of equations for the profiles is made of incoming and outgoing equations for interior profiles, whose traces of incoming profiles are expressed with boundary terms that in turn satisfy coupled evolution equations on the boundary. As already mentioned, the general system is difficult to treat, so we consider two simplified models of increasing difficulty for the study of existence. Both retain only a few profiles (which are the ones of interest), and remove some couplings between the equations. 
The first is only constituted by coupled equations on the boundary, and we make this first simplified model more complex into a second one by incorporating interior equations, whose traces on the boundary are given by the solutions to the equations on the boundary. For the first simplified model, containing only equations on the boundary, the formulation of \cite{BaouendiGoulaouic1978Nishida} can be applied, using a chain of spaces quantifying analyticity by means of the Fourier transform. The only difficulty is to show that a certain bilinear operator appearing in the equations is semilinear in the considered spaces of functions, and this result is obtained adapting a result of \cite{CoulombelWilliams2017Mach}. 
For the second simplified model, incorporating interior equations, the aim is to apply the \CK theorem to the interior equations, seen as propagation equations in the normal direction. Therefore, we need the boundary terms, which are solutions to boundary equations, to be analytical with respect to all their variables: both tangential space variables and \emph{time}. However, if we apply the classical \CK theorem to the boundary equations, we obtain a solution analytical only with respect to tangent space variables, and not with respect to time. We therefore need to adapt the \CK theorem to obtain analyticity with respect to all variables for solutions to the boundary equations. This is done using the method of majoring series, and the phenomenon of regularization by integration in time introduced in \cite{Ukai2001CauchyKovalevskaya}, see also \cite{Metivier2009Optics,Morisse2020Elliptic}. We define for this purpose a chain of spaces of analytic functions, with a formulation adapted from \cite{BaouendiGoulaouic1978Nishida} to the framework of majoring series, and prove the result using a fixed point theorem. We also define a chain of functional spaces suited to apply the \CK theorem for interior equations, once we have constructed the analytic boundary terms. Applying the \CK theorem to interior equations in this chain of functional spaces then presents no difficulty.

To prove that there is an instability, namely, that a small perturbation $ H $ of the boundary forcing term $ G $ interferes at a leading order, since the perturbation $ H $ is small compared to $ G $, we consider the linearized version of the general system, around the particular solution of this system when the perturbation $ H $ is zero. We obtain a linearized system with a small boundary forcing term given by $ H $, and we prove that there exists a boundary term $ H $ such that this system admits a solution whose first order profiles are not all zero. It shows that the small perturbation $ H $ interferes at the leading order for the linearized system, which constitutes an instability. The existence of $ H $ is proven by contradiction: we assume that for all boundary terms $ H $, all leading profiles are zero, and we contradict a certain condition by constructing the second order correctors. As for the part about existence, we work here with simplified models, as the coupling of the general system of equations is too difficult to handle. The first simplified model allows us to construct explicitly the solution to the linearized system, solving the considered transport equations by the method of characteristics. For the second one, the coupling is more complex, preventing us to apply the latter method, and we use a perturbation method and solve equations with a fixed point theorem.

This article is organized as follows. First we state the problem that we study here, make structural assumptions about it, and specify some assumptions and preliminary results about the oscillations at stake. Then, in a second part, the general system of equations for the profiles is derived, by detailing the iterative process for the leading profile and then the first corrector, and by writing down the general system satisfied by higher order correctors. We proceed in a third part with the proof of existence of a solution to simplified models of this general system. We start by detailing the obtaining of a first simplified model, then defining the functional framework which will be used, specifying the simplified model according to this functional framework, and applying the \CK theorem for boundary equations. Then we detail how this first simplified model is made more complex into a second one, we define additional functional spaces and specify the second simplified model accordingly, and finally we show existence and analyticity of solutions to boundary equations by proving a new version of the \CK theorem, leading to existence of solutions to interior equations, using a classical \CK theorem. The fifth part is devoted to the proof of instability, first by deriving the linearization of the general system around the particular solution where the perturbation $ H $ is zero, and then by proving, for two different simplified models, that an instability is created. Finally, in a sixth part, the example of isentropic compressible Euler equations in space dimension $ 3 $ is studied.

In all the article the letter $ C $ denotes a positive constant that may vary during the analysis, possibly without any mention being made.

\bigskip

\emph{Acknowledgments.} The author is particularly grateful to Jean-François Coulombel, whose brilliant idea is at the origin of this work, and for his numerous advice and proofreading. 

\section{Notation and assumptions}

\subsection{Position of the problem}

Given a time $ T>0 $ and an integer $ d\geq 2 $, let $ \Omega_T $ be the domain $\Omega_T:=(-\infty,T]\times \R^{d-1}\times \R_+$ and $\omega_T:=(-\infty,T]\times \R^{d-1}$ its boundary.
We denote as $ t\in(-\infty,T] $ the time variable, $ x=(y,x_d)\in\R^{d-1}\times\R_+ $ the space variable, with $ y\in\R^{d-1} $ the tangential variable and $ x_d\in\R_+ $ the normal variable, and at last $ z=(t,x)=(t,y,x_d) $. We also denote by $ z'=(t,y)\in\omega_T $ the variable of the boundary $ \ensemble{x_d=0} $. For $ i=1,\dots,d $, we denote by $ \partial_i $ the partial derivative operator with respect to $ x_i $. Finally we denote as $ \alpha\in\R^{d+1} $ and $ \zeta\in\R^d $ the dual variables of $ z\in\Omega_T $ and $ z'\in\omega_T $.  We consider the following problem
\begin{equation}\label{eq systeme 1}
	\left\lbrace \begin{array}{lr}
		L(u^{\epsilon},\partial_z)\,u^{\epsilon}:=\partial_tu^{\epsilon}+\displaystyle\sum_{i=1}^dA_i(u^{\epsilon})\,\partial_iu^{\epsilon}=0&\qquad \mbox{in } \Omega_T, \\[5pt]
		B\,u^{\epsilon}_{|x_d=0}=\epsilon^2\, g^{\epsilon}+\epsilon^M\,h^{\epsilon}&\qquad \mbox{on } \omega_T,  \\[10pt]
		u^{\epsilon}_{|t\leq 0}=0,&
	\end{array}
	\right.
\end{equation}
where the unknown $ u^{\epsilon} $ is a function from $ \Omega_T $ to an open set $ \mathcal{O} $ of $ \R^N $ containing zero, $ N\geq 1 $, the matrices $A_j$ are smooth functions of $\mathcal{O}$ with values in $\mathcal{M}_N(\R)$, the matrix $B$ belongs to $\mathcal{M}_{\tilde{p},N}(\R)$ and is of maximal rank (integer $ \tilde{p}\geq 1 $ will be made precise below). 

The boundary term is a superposition of a reference forcing oscillating term $ \epsilon^2\,g^{\epsilon} $ (of characteristic wavelength $ \epsilon $) and a smaller, transverse, oscillating term $ \epsilon^M\,h^{\epsilon} $ with $ M\geq 3 $, namely, for $ z'\in\omega_T $,
\begin{subequations}
\begin{align}\label{eq def g epsilon 1}
	g^{\epsilon}(z')&=G\left( z',\frac{z'\cdot\phi}{\epsilon}\right),\\[5pt]
	h^{\epsilon}(z')&=H\left( z',\frac{z'\cdot\psi}{\epsilon}\right),
\end{align}
\end{subequations}
where $ G,H $ are functions of the Sobolev space of infinite reguarity $ H^{\infty}(\R^d\times\T) $, are zero for negative time $ t $, and with boundary frequencies $ \phi,\psi $ given in $ \R^d\privede{0} $. Frequencies $ \phi $ and $ \psi $ are taken linearly independent over $ \R $, that is, $ \psi\notin\R\phi $. We denote by $ \boldsymbol{\zeta} $ the couple $ \boldsymbol{\zeta}:=(\phi,\psi) $. In this paper we wish to place ourselves in the framework of weakly nonlinear geometric optics. Usually to obtain this framework the amplitude of the boundary forcing term must be of order $ O(\epsilon) $. Here, because we will assume that the Kreiss-Lopatinskii condition is not satisfied for $ \phi $ (and $ \psi $), an amplification phenomenon will happen at the boundary for this frequency, so a forcing term of amplitude of order $ O(\epsilon^2) $ should be chosen on the boundary. This scaling has been studied in \cite{ArtolaMajda1987Instabilities,MajdaRosales1983Machstem,MajdaRosales1984Machstem,CoulombelWilliams2017Mach}. Note that if we set $ h^{\epsilon}=0 $ in system \eqref{eq systeme 1}, we obtain the system studied in \cite{CoulombelWilliams2017Mach}.

To simplify the equations and computations we assume that the coefficients are affine maps, that is, for $ j=1,\dots,d $,
\begin{equation*}
	A_j(u)=A_j(0)+dA_j(0)\cdot u.
\end{equation*}
We make the following structural and classical assumption on the boundary.
\begin{assumption}[noncharacteristic boundary]\label{hypothese bord non caract}
	The boundary is noncharacteristic, that is, matrix $ A_d(0) $ is invertible.
\end{assumption}

To simplify the equations and the computations we will study here the case $ M=3 $, but there is no apparent obstacle to generalize this analysis to any integer $ M\geq 4 $. For the same purpose we choose to work with the particular case of 3-dimensional vectors ($ N=3 $) since it is sufficient in this analysis to create instabilities.

In this paper we study a geometric optics asymptotic expansion for system \eqref{eq systeme 1}, namely,  we look for an approximate solution to \eqref{eq systeme 1} in the form of a formal series
\begin{equation}\label{eq ansatz provisoire}
	u^{\epsilon,\app}(z)=\sum_{n\geq 1} \epsilon^n\, U_n\Big(z,\frac{\Phi(z)}{\epsilon}\Big),
\end{equation}
where the collection of phases $ \Phi $ will be made precise later. The approximate solution is expected to be of order $ O(\epsilon) $ because of the weakly nonlinear framework.
The aim is to show that, with a well chosen configuration of frequencies, there is an instability in this asymptotic expansion, in the sense that, despite its small amplitude order $ O\big(\epsilon^3\big) $, perturbation $ \epsilon^3\,h^{\epsilon} $ interferes at the leading order, i.e. in the construction of the leading profile $ U_1 $. In addition to this instability, we will study well-posedness for a simplified model associated with the equations for the profiles, and the possibility for such a frequency configuration to occur in the case of Euler equations in space dimension 3. 

We start by making a series of structural assumptions on system \eqref{eq systeme 1} and detailing the configuration of frequencies considered here.

\bigskip

The following definition introduces the notion of characteristic frequency.

\begin{definition}
	For $\alpha=(\tau,\eta,\xi)\in\R\times\R^{d-1}\times\R$, the symbol $L(0,\alpha)$ associated with $L(0,\partial_z)$ is defined as
	\begin{equation*}
		L(0,\alpha):=\tau I+\sum_{i=1}^{d-1}\eta_iA_i(0)+\xi A_d(0).
	\end{equation*}
	Then we define its characteristic polynomial as $ p(\tau,\eta,\xi):=\det L\big(0,(\tau,\eta,\xi)\big) $. We say that $\alpha\in\R^{1+d}$ is a \emph{characteristic frequency} if it is a root of the polynomial $p$.
	%, and we denote by $ \mathcal{C} $\todocorentinlater{Utile ?} the set of characteristic frequencies.
\end{definition}

The following assumption, called \emph{strict hyperbolicity} (see  \cite[Definition 1.2]{BenzoniSerre2007Multi}), is made. Assumptions of hyperbolicity, whether strict or with constant multiplicity, are very usual, see e.g. \cite{Williams1996Boundary,CoulombelGuesWilliams2011Resonant,JolyMetivierRauch1995Coherent}, and related to the structure of the problem. Assumption of hyperbolicity with constant multiplicity, which is more general than Assumption \ref{hypothese stricte hyp} of strict hyperbolicity below, is sometimes preferred like in \cite{CoulombelGuesWilliams2011Resonant,JolyMetivierRauch1995Coherent}. We chose here to work with the latter for technical reasons. Recall that we placed ourselves in the particular case where the size of the system is $ N=3 $.

\begin{assumption}[strict hyperbolicity]\label{hypothese stricte hyp}
	There exist real functions $ \tau_1<\tau_2<\tau_3 $, analytic with respect to $ (\eta,\xi) $ in $\R^d\setminus\{0\}$, such that for all
	$(\eta,\xi)\in\R^d\setminus\ensemble{0}$ and for all 
	$\tau\in\R$, the following factorisation is verified
	\[p(\tau,\eta,\xi)=\det\Big(\tau I+\sum_{i=1}^{d-1}\eta_iA_i(0)+\xi A_d(0)\Big)=\prod_{k=1}^3\big(\tau-\tau_k(\eta,\xi)\big),\]
	where the eigenvalues $-\tau_k(\eta,\xi)$ of the matrix $A(\eta,\xi):=\sum_{i=1}^{d-1}\eta_iA_i(0)+\xi A_d(0)$ are therefore simple. 
%	Consequently, for all $(\eta,\xi)\in\R^d \setminus\ensemble{0}$, 
%	the following decomposition of $\C^3$ into dimension 1 eigenspaces hold
%	\begin{align}\label{eq decomp C^n ker L(alpha)}
%		\C^3&=\ker L\big(0,\tau_1(\eta,\xi),\eta,\xi\big)\oplus\ker L\big(0,\tau_2(\eta,\xi),\eta,\xi\big)\oplus\ker L\big(0,\tau_3(\eta,\xi),\eta,\xi\big).
%	\end{align}
%	For $k=1,2,3$ and for $(\eta,\xi)$ in $\R^d\setminus\{0\}$,
%	we define the projectors $\pi_k(\eta,\xi)$ associated with the decomposition \eqref{eq decomp C^n ker L(alpha)}.
\end{assumption}

\subsection{Weak Kreiss-Lopatinskii condition}

We define the following space of frequencies
\begin{align*}
	\Xi&:=\{\zeta=(\sigma=\tau-i\gamma,\eta)\in(\C\times\R^{d-1})\backslash\{0\} \mid \gamma\geq 0\},\\
	\Sigma&:=\ensemble{\zeta\in\Xi\mid \tau^2+\gamma^2+|\eta|^2=1},\\
	\Xi_0&:=\{\zeta\in\Xi \mid \gamma=0\},\\
	\Sigma_0&:=\Xi_0\cap\Sigma.
\end{align*}
We also define the matrix valued symbol which we get when applying the Laplace-Fourier transform to the operator $L(0,\partial_z)$. For all $\zeta=(\sigma,\eta)\in\Xi$, let
\[\mathcal{A}(\zeta):=-i\,A_d(0)^{-1}\Big(\sigma I+\sum_{i=1}^{d-1}\eta_j\, A_j(0)\Big).\]

The Hersh lemma (\cite{Hersh1963Mixed}) ensures that for $ \zeta $ in $ \Xi\backslash\Xi_0 $, the matrix $\mathcal{A}(\zeta)\in\mathcal{M}_3(\C) $ has no eigenvalue of zero real part, and that the stable subspace associated with the eigenvalues of negative real part, denoted by $E_-(\zeta)$, is of constant dimension, denoted $p$. Furthermore, the integer $ p $ is  obtained as the number of positive eigenvalues of the matrix $ A_d(0) $.
We denote by $E_+(\zeta)$ the unstable subspace $\mathcal{A}(\zeta)$ associated with eigenvalues of positive real part, which is of dimension $3-p$.

In \cite{Kreiss1970Initial} (see also \cite{ChazarainPiriou1982Introduction} and \cite{BenzoniSerre2007Multi}) it is shown that the stable and unstable subspaces $E_{\pm}$ extend continuously to the whole space $\Xi$ in the strictly hyperbolic case (Assumption \ref{hypothese stricte hyp}).
We still denote by $ E_{\pm} $ the extensions to $ \Xi $. 
The \emph{hyperbolic region}, denoted by $ \mathcal{H} $, is defined as the set of frequencies $ \zeta $ such that matrix $ \mathcal{A}(\zeta) $ has only purely imaginary eigenvalues.

The following assumption is very structural to the problem, and is the one which allows amplification on the boundary, and thus instability.

\begin{assumption}[weak Kreiss-Lopatinskii condition]\label{hypothese weak lopatinskii}
	\begin{itemize}[label=$\bullet$,leftmargin=20pt]
		\item For all $ \zeta\in\Xi\setminus\Xi_0 $, $ \ker B\cap E_-(\zeta)=\ensemble{0} $.
		\item The set $ \Upsilon:=\ensemble{\zeta\in\Sigma_0 \,\middle|\, \ker B\cap E_-(\zeta)\neq \ensemble{0}} $ is nonempty and included in the hyperbolic region $ \mathcal{H} $.
		\item There exist a neighborhood $ \mathcal{V} $ of $ \Upsilon $ in $ \Sigma $, a real valued $ \mathcal{C}^{\infty} $ function $ \kappa $ defined on $ \mathcal{V} $, a basis $ E_1(\zeta),\dots,E_p(\zeta) $ of $ E_-(\zeta) $ 
		%that is of class $ \mathcal{C}^{\infty} $ with respect to $ \zeta\in\mathcal{V} $ 
		and a matrix $ P(\zeta)\in\GL_p(\C) $ which are of class $ \mathcal{C}^{\infty} $ with respect to $ \zeta\in\mathcal{V} $ such that, for all $ \zeta $ in $ \mathcal{V} $,
		\begin{equation*}
			B\big(E_1(\zeta)\cdots E_p(\zeta)\big)=P(\zeta)\,\diag\big(\gamma+i\kappa(\zeta),1,\dots,1\big).
		\end{equation*}
 	\end{itemize}
\end{assumption}

\begin{remark}
	First point of Assumption \ref{hypothese weak lopatinskii}, requiring that  $ \ker B\cap E_-(\zeta)=\ensemble{0} $ for all $ \zeta\in\Xi\setminus\Xi_0 $, implies in particular that $ \tilde{p} $, the rank of $ B $, equals $ p $, the dimension of $ E_-(\zeta) $. These two equal integers will be denoted by $ p $ in the following. Assumption \ref{hypothese ensemble frequences} below sets furthermore the integer $ p=\tilde{p} $ to be equal to 2.
\end{remark}

The so-called \emph{Kreiss-Lopatinskii condition} is the first point of Assumption \ref{hypothese weak lopatinskii} that stands in $ \Xi\setminus \Xi_0 $, and the next two points detail how this condition breaks on the boundary $ \Xi_0 $ of $ \Xi $ (for the \emph{uniform Kreiss-Lopatinskii condition} to hold, equality $ \ker B \cap E_-(\zeta)=\ensemble{0} $ is assumed to be satisfied everywhere in $ \Xi $, see \cite{Kreiss1970Initial}). The second point asserts that the Kreiss-Lopatinskii condition breaks only in the hyperbolic region $ \mathcal{H} $, and the third one ensures that when it breaks, the space $ \ker B \cap E_-(\zeta) $ is of dimension $ 1 $, and that the default of injectivity of $ B $ on $ E_-(\zeta) $ is parameterize by the $ \mathcal{C}^{\infty} $ function $ \kappa $. In particular, $ \kappa $ must be zero on $ \Upsilon $, and nonzero on $ \Sigma_0\setminus\Upsilon $.

Together with Assumptions \ref{hypothese bord non caract} and \ref{hypothese stricte hyp}, Assumption \ref{hypothese weak lopatinskii} ensures that for all $ \epsilon>0 $, system \eqref{eq systeme 1} is weakly well-posed locally in time (which depends on $ \epsilon $). A proof of a similar result, for characteristic free boundary problems can be found in \cite{CoulombelSecchi2008Vortex}.
%\todocorentinlater{Dans \cite{CoulombelGues2010Geometric} il est expliqué en quoi cette hypothèse revient à la classe $ \mathcal{WR} $ de \cite{BenzoniSerre2007Multi}.}
Indeed, the three assumptions  \ref{hypothese bord non caract}, \ref{hypothese stricte hyp} and \ref{hypothese weak lopatinskii} are stable under small perturbations around the equilibrium, see \cite[Section 8.3]{BenzoniSerre2007Multi}.

\subsection{Oscillations}

The notion of incoming, outgoing and glancing frequencies is now introduced.

\begin{definition}\label{def sortant rentrant alpha X alpha}
	Let $\alpha=(\tau,\eta,\xi)\in\R^{d+1}\backslash\ensemble{0}$ be a characteristic frequency, and $ k $ the integer between $ 1 $ and $ 3 $ such that $\tau=\tau_k(\eta,\xi)$. The group velocity $ \mathbf{v}_{\alpha}\in\R^d $ associated with $ \alpha $ is defined as
	\begin{equation*}
		\mathbf{v}_{\alpha}:=\nabla_{\eta,\xi}\,\tau_k(\eta,\xi).
	\end{equation*}
	We shall say that $\alpha$ is glancing (resp. incoming, outgoing) if $\partial_{\xi}\tau_k(\eta,\xi)$ is zero (resp. negative, positive).  
	Then the vector field $X_{\alpha}$ associated with $\alpha$ is defined as
	\begin{equation}\label{eq champ vecteur X_alpha}
		X_{\alpha}:=\frac{-1}{\partial_{\xi}\tau_k(\eta,\xi)}\Big(\partial_t-\mathbf{v}_{\alpha}\cdot\nabla_{x}\Big)=\frac{-1}{\partial_{\xi}\tau_k(\eta,\xi)}\Big(\partial_t-\nabla_{\eta}\tau_k(\eta,\xi)\cdot\nabla_{y}-\partial_{\xi}\tau_k(\eta,\xi)\,\partial_{x_d}\Big).
	\end{equation}
\end{definition}

Lax lemma, see Lemma \ref{lemme Lax} below, ensures that these constant coefficients scalar transport operators $ X_{\alpha} $ appear naturally in the equations satisfied by the profiles arising in weakly nonlinear asymptotic expansions (see \cite{Rauch2012Hyperbolic}).

We describe now a decomposition of the stable subspace $ E_-(\zeta) $ for $ \zeta\in\Xi_0 $, that uses strict hyperbolicity (Assumption \ref{hypothese stricte hyp}).

\begin{proposition}[{\cite{Williams1996Boundary}, Proposition 3.4}]\label{prop decomp E_-}
	Consider $ \zeta=(\tau,\eta)\in\Xi_0 $. We denote by $ i\,\xi_j(\zeta) $ for $ j=1,\dots,\mathcal{M}(\zeta) $ the distinct complex eigenvalues of the matrix $ \mathcal{A}(\zeta) $, and if $ \xi_j(\zeta) $ is real, we shall denote by $ \alpha_j(\zeta):=(\tau,\eta,\xi_j(\tau,\eta)) $ the associated real characteristic frequency. If $ \xi_j(\zeta) $ is real, we also denote by $ k_j $ the integer between $ 1 $ and $ 3 $ such that $\tau=\tau_{k_j}(\eta,\xi_j(\zeta))$. Then the set $\ensemble{1,2,\dots,\mathcal{M}(\zeta)}$ decomposes as the disjoint union
	\begin{equation}\label{eq union disjointe M(zeta)}
		\ensemble{1,2,\dots,\mathcal{M}(\zeta)}=\mathcal{G}(\zeta)\cup\mathcal{I}(\zeta)\cup\P(\zeta)\cup\mathcal{O}(\zeta)\cup\mathcal{N}(\zeta),
	\end{equation}
	where the sets $\mathcal{G}(\zeta)$, $\mathcal{I}(\zeta)$, $\P(\zeta)$, $\mathcal{O}(\zeta)$ and $\mathcal{N}(\zeta)$ correspond to indexes $j$ such that respectively $\alpha_j(\zeta)$ is glancing, $\alpha_j(\zeta)$ is incoming, $\Im(\xi_j(\zeta))$ is positive, $\alpha_j(\zeta)$ is outgoing and $\Im(\xi_j(\zeta))$ is negative.
	
	Then the following decomposition of $E_-(\zeta)$ holds
	\begin{equation}\label{eq decomp E_-(zeta)}
		E_-(\zeta)=\bigoplus_{j\in\mathcal{G}(\zeta)}E^j_-(\zeta)\oplus \bigoplus_{j\in\mathcal{R}(\zeta)}E^j_-(\zeta) \oplus \bigoplus_{j\in\P(\zeta)}E^j_-(\zeta),
	\end{equation}
	where for each index $j$, the subspace $E_-^j(\zeta)$ is precisely described as follows.
	\begin{enumerate}[label=\roman*)]
		\item If $j\in \P(\zeta)$, the space $E^j_-(\zeta)$ is the generalized eigenspace $\mathcal{A}(\zeta)$ associated with the eigenvalue $i\,\xi_j(\zeta)$.
		\item If $j\in\mathcal{R}(\zeta)$, we have $E^j_-(\zeta)=\ker L\big(0,\alpha_j(\zeta)\big)$, 
		which is of dimension 1.
		\item If $j\in\mathcal{G}(\zeta)$, we denote by $n_j$ the algebraic multiplicity of the imaginary eigenvalue $i\xi_j(\zeta)$. For small positive $ \gamma $, the multiple eigenvalue $i\,\xi_j(\tau,\eta)$ splits into $n_j$ simple eigenvalues, denoted by $i\,\xi_j^k(\tau-i\gamma,\eta)$, $k=1,\dots,n_j$, all of nonzero real part. We denote by $\mu_j$ the number (independent of $\gamma>0$) of the eigenvalues $i\,\xi_j^k(\tau-i\gamma,\eta)$ of negative real part. Then $E_-^j(\zeta)$ is of dimension $ \mu_j $ and is generated by the vectors $w$ satisfying $[\mathcal{A}(\zeta)-i\xi_j(\zeta)]^{\mu_j}w=0$. Furthermore, if $n_j$ is even, $\mu_j=n_j/2$ and if $n_j$ is odd, $\mu_j$ is equal to $(n_j-1)/2$ or $(n_j+1)/2$.
	\end{enumerate}
	
	Likewise, the unstable subspace $E_+(\zeta)$ decomposes as
	\begin{equation}\label{eq decomp E_+(zeta)}
		E_+(\zeta)=\bigoplus_{j\in\mathcal{G}(\zeta)}E^j_+(\zeta)\oplus \bigoplus_{j\in\mathcal{S}(\zeta)}E^j_+(\zeta) \oplus \bigoplus_{j\in \mathcal{N}(\zeta)}E^j_+(\zeta),
	\end{equation}
	with similar description of the subspaces $E_+^j(\zeta)$. In particular, if the set $\mathcal{G}(\zeta)$ is empty, then
	\begin{equation*}
	\C^3=E_-(\zeta)\oplus E_+(\zeta).
	\end{equation*}
\end{proposition}

For $ \zeta\in\Xi_0 $, we denote by $ \mathcal{C}(\zeta) $ the set of indices such that $ \alpha_j(\zeta) $ is real characteristic, that is
\begin{equation*}
	\mathcal{C}(\zeta):=\mathcal{I}(\zeta)\cup\mathcal{O}(\zeta)\cup \mathcal{G}(\zeta).
\end{equation*}

\begin{definition}
	A frequency $ \zeta $ in $ \Xi_0 $ is said to be \emph{glancing} if there exists $ j=1,\dots,\mathcal{M}(\zeta)$ such that $ \alpha_j(\zeta) $ is glancing, i.e. if $ \mathcal{G}(\zeta) $ is nonempty, \emph{hyperbolic} if $ \mathcal{A}(\zeta) $ has only purely imaginary eigenvalues, that is if $ \mathcal{P}(\zeta)\cup\mathcal{N}(\zeta) $ is empty and \emph{mixed} if $ \mathcal{P}(\zeta)\cup\mathcal{N}(\zeta) $ is nonempty. We shall denote by $ \mathcal{G} $ (resp. $ \mathcal{H} $, $ \mathcal{EH} $) the set of glancing (resp. hyperbolic, mixed) frequencies.
\end{definition}

\begin{definition}
		For $ \zeta\in\Xi_0 $ not glancing, according to Proposition \ref{prop decomp E_-}, we have the following decomposition of $ \C^3 $:
	\begin{equation*}%\label{eq decomp C^N E + E -}
		\C^3=
		\bigoplus_{j\in\mathcal{O}(\zeta)}E^j_+(\zeta) \oplus \bigoplus_{j\in \mathcal{N}(\zeta)}E^j_+(\zeta)\oplus
		\bigoplus_{j\in\mathcal{I}(\zeta)}E^j_-(\zeta) \oplus \bigoplus_{j\in\P(\zeta)}E^j_-(\zeta).
	\end{equation*}
	In that case we denote by $ \Pi^{e}(\zeta) $ the projection from $ \C^3 $ on the stable elliptic component $ E^e_-(\zeta):=\oplus_{j\in\P(\zeta)}E^j_-(\zeta) $ according to this decomposition.
\end{definition}

The following result is adapted from \cite[Lemma 3.2]{CoulombelGues2010Geometric} to the case of mixed frequencies.

\begin{lemma}\label{lemme def Pj}
	For all $ \zeta\in\Xi_0 $ nonglancing, the following decompositions hold
\begin{subequations}
		\begin{align}
		\C^3&=\bigoplus_{j\in \mathcal{C}(\zeta)}\ker L\big(0,\alpha_j(\zeta)\big)\oplus F_{\zeta}\label{eq decomp C N }\\
		\C^3&=\bigoplus_{j\in \mathcal{C}(\zeta)}A_d(0)\,\ker L\big(0,\alpha_j(\zeta)\big)\oplus A_d(0)\,F_{\zeta},\label{eq decomp C N A d 0}
	\end{align} 
\end{subequations}
where $ F_{\zeta} $ is the generalized eigenspace of $ \mathcal{A}(\zeta) $ associated with the eigenvalues of nonzero real part. Furthermore, if we denote by $ P_j(\zeta) $ and $ P_{F_{\zeta}} $ (resp. $ Q_j(\zeta) $ and $ Q_{F_{\zeta}} $) the projectors associated with the decomposition \eqref{eq decomp C N } (resp. \eqref{eq decomp C N A d 0}), then we have
\begin{equation}\label{eq relation Qj}
	\Ima L\big(0,\alpha_j(\zeta)\big)=\ker Q_j(\zeta),
\end{equation}
for all $ j $.
\end{lemma}

In \cite{CoulombelGues2010Geometric}, the result is proven only for frequencies $ \zeta $ hyperbolic, and the proof is slightly simpler using directly the diagonalizability of matrix $ \mathcal{A}(\zeta) $. Here the matrix is only block-diagonalizable, and we have to deal with eigenvalues of nonzero real part. 

\begin{proof}
	The two decompositions come from the block-diagonalizability of matrix $ \mathcal{A}(\zeta) $, the fact that $ \zeta $ is not glancing and the invertibility of matrix $ A_d(0) $. Indeed, for any nonglancing frequency $ \zeta\in\Xi_0 $, there exists therefore an invertible matrix $ T(\zeta) $ such that $ T(\zeta)\,\mathcal{A}(\zeta)\,T(\zeta)^{-1} $ is the block diagonal matrix
	\begin{equation*}
		T(\zeta)\,\mathcal{A}(\zeta)\,T(\zeta)^{-1}=\diag\big(i\xi_1(\zeta) ,\dots,i\xi_{m_{\zeta}}(\zeta) ,\mathcal{A}_{\pm}(\zeta)\big)
	\end{equation*}
where the $ \xi_j(\zeta) $ are real scalars, and the spectrum of the block $ \mathcal{A}_{\pm}(\zeta) $ is contained in $ \C\setminus i\R $. The proof decomposes in two main steps. First we construct a sequence of diagonalizable matrix converging toward $ \mathcal{A}(\zeta) $, in order to be able to adapt the method used in \cite{CoulombelGues2010Geometric}. Then using projectors defined for this sequence of matrix, analogous to $ P_j(\zeta) $ and $ Q_j(\zeta) $, we are able to prove relation \eqref{eq relation Qj}, using diagonalizability. 

\emph{Step 1.} We consider a sequence $ (\mathcal{A}^k_{\pm}(\zeta))_{k\geq 0} $ of diagonalizable matrices converging toward $ \mathcal{A}_{\pm}(\zeta) $. For $ k\geq 0 $, we denote by $ \tilde{T}_k(\zeta) $ the invertible matrix such that
\begin{equation*}
	\tilde{T}_k(\zeta)\,\mathcal{A}^k_{\pm}(\zeta)\,\tilde{T}_k(\zeta)^{-1}=\diag (i\lambda_1,\dots,i\lambda_{3-m_{\zeta}}).
\end{equation*}
We also denote by $ T_k(\zeta) $ the block diagonalizable matrix
\begin{equation*}
	T_k(\zeta):=\diag(I_{m_{\zeta}},\tilde{T}_k(\zeta)),
\end{equation*}
and we finally define, for $ k\geq 0 $ the matrix $ \mathcal{A}^k(\zeta) $ as
\begin{equation*}
	\mathcal{A}^k(\zeta):=T(\zeta)\,T_k(\zeta)\,\diag\big(i\xi_1(\zeta) ,\dots,i\xi_{m_{\zeta}}(\zeta) ,i\lambda_1,\dots,i\lambda_{3-m_{\zeta}}\big)\,T_k(\zeta)^{-1}\,T(\zeta)^{-1}.
\end{equation*}
Note that the sequence $ \big(\mathcal{A}^k(\zeta)\big)_{k\geq 0} $ is by definition a sequence of diagonalizable matrices which converges toward $ \mathcal{A}(\zeta) $. Using this diagonalizability we get the two following decompositions of $ \C^3 $, for $ k\geq 0 $:
\begin{subequations}
	\begin{align}
	\C^3&=\bigoplus_{j=1}^{m_{\zeta}}\ker \big(\mathcal{A}^k(\zeta)-i\xi_j(\zeta) I\big)\oplus\bigoplus_{j=1}^{3-m_{\zeta}}\ker \big(\mathcal{A}^k(\zeta)-i\lambda_jI\big)\label{eq Lop decomp C 3 k}\\
	&=\bigoplus_{j=1}^{m_{\zeta}}A_d(0)\ker \big(\mathcal{A}^k(\zeta)-i\xi_j(\zeta) I\big)\oplus\bigoplus_{j=1}^{3-m_{\zeta}}A_d(0)\ker \big(\mathcal{A}^k(\zeta)-i\lambda_jI\big).\label{eq Lop decomp C 3 k Ad}
\end{align}
\end{subequations}
First we note that, by definition of the matrix $ \mathcal{A}^k(\zeta) $, the eigenspace $ \ker \big(\mathcal{A}^k(\zeta)-i\xi_j(\zeta) I\big) $ is equal to $ \ker L\big(0,\alpha_j(\zeta)\big) $ and that
\begin{equation*}
	\bigoplus_{j=1}^{3-m_{\zeta}}\ker \big(\mathcal{A}^k(\zeta)-i\lambda_jI\big)=F_{\zeta}.
\end{equation*}
Thus we define the projectors $ P_{\pm}^{k,j}(\zeta) $ (resp. $ Q_{\pm}^{k,j}(\zeta) $) on $ \ker \big(\mathcal{A}^k(\zeta)-i\lambda_jI\big) $ (resp. $ A_d(0)\allowbreak\ker \big(\mathcal{A}^k(\zeta)-i\lambda_jI\big) $) associated with the decomposition \eqref{eq Lop decomp C 3 k} (resp. \eqref{eq Lop decomp C 3 k Ad}). According to the previous remark we then have
\begin{subequations}
	\begin{align}
	I&=P_1(\zeta)+\cdots+P_{m_{\zeta}}(\zeta)+P_{\pm}^{k,1}(\zeta)+\cdots+P_{\pm}^{k,3-m_{\zeta}}(\zeta)\label{eq Lop decomp I P j}\\
	&=Q_1(\zeta)+\cdots+Q_{m_{\zeta}}(\zeta)+Q_{\pm}^{k,1}(\zeta)+\cdots+Q_{\pm}^{k,3-m_{\zeta}}(\zeta).
\end{align}
\end{subequations}

\emph{Step 2.} For $ j_0 $ between $ 1 $ and $ m_{\zeta} $, analogously to $ L\big(0,\alpha_{j_0}(\zeta)\big) $, we define
\begin{equation*}
	L_k\big(0,\alpha_{j_0}(\zeta)\big):=iA_d(0)\big(\mathcal{A}_k(\zeta)-i\xi_{j_0}(\zeta)I\big).
\end{equation*}
By definition and since the following relation is satisfied
\begin{equation*}
	L\big(0,\alpha_{j_0}(\zeta)\big)=iA_d(0)\big(\mathcal{A}(\zeta)-i\xi_{j_0}(\zeta)I\big),
\end{equation*}
the sequence $ \big(L_k\big(0,\alpha_{j_0}(\zeta)\big)\big)_{k\geq 0} $ converges to $ L\big(0,\alpha_{j_0}(\zeta)\big) $. We consider $ L_k\big(0,\alpha_{j_0}(\zeta)\big)\,X $ an element of $ \Ima L_k\big(0,\alpha_{j_0}(\zeta)\big) $ with $ X\in\C^3 $, and the aim is to prove that it belongs to $ \ker Q_{j_0}(\zeta) $. The latter is a closed space, so, since the sequence $ \big(L_k\big(0,\alpha_{j_0}(\zeta)\big)\,X\big)_{k\geq 0} $ converges to $ L\big(0,\alpha_{j_0}(\zeta)\big)\,X $, it will follow that $ \Ima  L\big(0,\alpha_{j_0}(\zeta)\big)\subset \ker Q_{j_0}(\zeta) $ and the conclusion then infers because of equality of dimension of the two spaces. 

We have, by definition of the projectors $ P_j(\zeta) $ and $ P^{k,j}_{\pm}(\zeta) $ and because of the decomposition \eqref{eq Lop decomp I P j},
\begin{align*}
	L_k\big(\alpha_{j_0}(\zeta)\big)\,X&=iA_d(0)\big(\mathcal{A}_k(\zeta)-i\xi_{j_0}(\zeta)I\big)\Big\lbrace\sum_{j=1}^{m_{\zeta}}P_j(\zeta)\,X+\sum_{j=1}^{3-m_{\zeta}}P^{k,j}_{\pm}(\zeta)\,X\Big\rbrace\\
	&=iA_d(0)\sum_{\substack{j=1\\j\neq j_0}}^{m_{\zeta}}\big(i\xi_j(\zeta)-i\xi_{j_0}(\zeta)\big)P_j(\zeta)\,X+iA_d(0)\sum_{j=1}^{3-m_{\zeta}}\big(i\lambda_j(\zeta)-i\xi_{j_0}(\zeta)\big)P^{k,j}_{\pm}(\zeta)\,X,
\end{align*}
and the last term belongs to
\begin{equation*}
		\bigoplus_{\substack{j=1\\j\neq j_0}}^{m_{\zeta}}A_d(0)\ker \big(\mathcal{A}^k(\zeta)-i\xi_j(\zeta) I\big)\oplus\bigoplus_{j=1}^{3-m_{\zeta}}A_d(0)\ker \big(\mathcal{A}^k(\zeta)-i\lambda_jI\big)=\ker Q_{j_0}(\zeta),
\end{equation*}
concluding the proof.
\end{proof}

The interest is now made on the frequencies created on the boundary and then lifted inside the domain. Recall that we considered a quasi-periodic boundary forcing term of frequencies $ \phi/\epsilon $ and $ \psi/\epsilon $, with $ \phi,\psi\in\R^{d}\privede{0} $. In the following we will make restricting assumptions on $ \phi $ and $ \psi $ in order to obtain a particular frequency configuration, eventually creating an instability. 

By nonlinear interaction, frequencies $ \phi $ and $ \psi $ on the boundary create the following lattice of frequencies on the boundary:
\begin{equation*}
	\F_b:=\phi\,\Z\oplus\psi\,\Z.
\end{equation*}
To avoid the complications induced by the glancing modes, we assume that there is no glancing frequency in $ \F_{b}\privede{0} $. This is a common assumption, see \cite{CoulombelGues2010Geometric,CoulombelGuesWilliams2011Resonant}.

\begin{assumption}\label{hypothese pas de glancing}
	We have 
	\begin{equation*}
		\big(\F_{b}\privede{0}\big)\cap \mathcal{G}=\emptyset.
	\end{equation*}
\end{assumption}

To parameterize $ \F_b $ we introduce the following subset of $ \Z^2\privede{0} $:
\begin{equation*}
	\mathcal{B}_{\Z^2}:=\ensemble{(n_1,n_2)\in\Z^2\privede{0}\,\middle|\,\begin{aligned}
			&n_2\wedge n_2=1,\\
			&n_1>0 \ \ \mbox{or}\ \  n_1=0,n_2>0
	\end{aligned}},
\end{equation*}	
of couples of coprime integers of which the first  nonzero term is positive. Then, each frequency $ \zeta:=n_1\,\phi+n_2\,\psi $ of $ \F_b\privede{0} $ is parameterized in a unique way by $ \mathbf{n}_0:=(n^0_1,n^0_2)\in\B_{\Z^2} $ and $ \lambda\in\Z^* $ such that $ (n_1,n_2)=\lambda\,(n^0_1,n^0_2) $. 

In the following, we will allow ourselves to alternate without mentioning it between the following representations of a frequency of $ \zeta\in\F_{b}\privede{0} $: $ \mathbf{n}=(n_1,n_2) $ in $ \Z^2\privede{0} $ such that $ \zeta=n_1\,\phi+n_2\,\psi $ and $ \mathbf{n}_0=(n_0^1,n_0^2) $ in $ \B_{\Z^2} $ and $ \lambda $ in $ \Z^* $ such that $ \mathbf{n}=\lambda\,\mathbf{n}_0 $.

Because of the hyperbolicity of the system, boundary frequencies $ \zeta $ of $ \F_{b} $ are lifted into frequencies $ (\zeta,\xi) $ inside the domain, which must be characteristic frequencies due to polarization conditions. Furthermore, frequencies $ (\zeta,\xi) $ with $ \Im \xi <0 $ are excluded to obtain bounded solutions, and we have already discarded glancing frequencies by Assumption \ref{hypothese pas de glancing}. Therefore, the set $ \F $ of frequencies inside the domain is given by
\begin{equation*}
	\F:=\ensemble{0}\cup\ensemble{\big(\zeta,\xi_j(\zeta)\big) \mid\zeta\in \F_b\privede{0},j\in \mathcal{C}(\zeta)\cup\mathcal{P}(\zeta)}.
\end{equation*}

The following assumption details the configuration of frequencies which is assumed to hold in order to create an instability. It is a generalization to our case of \cite[Assumptions 1.7 and 1.9]{CoulombelWilliams2017Mach}, where the only frequency of the problem, $ \phi $, was supposed to be nonresonant, hyperbolic, and in $ \Upsilon $. In \cite{CoulombelWilliams2017Mach}, the authors explain that allowing the boundary frequency $ \phi $ to be resonant could lead to an over-determination of the system. Assumption \ref{hypothese ensemble frequences} below requires in particular that frequencies $ \phi $ and $ \psi $ are nonresonant\footnote{In the sense that two frequencies lifted from $ \phi $ cannot resonate with each other, and the same for $ \psi $.}, hyperbolic, and in $ \Upsilon $. We additionally assume two resonances between frequencies lifted from $ \phi $ and $ \psi $ to hold, which will eventually allow us to create an instability. 

\begin{assumption}\label{hypothese ensemble frequences}
	There exists a frequency $ \nu $ in $ \F_{b}\privede{0} $ defined by
	\begin{equation*}
		\lambda_{\phi}\,\phi+\lambda_{\psi}\,\psi+\nu=0
	\end{equation*}
with coprime integers $ \lambda_{\phi},\lambda_{\psi} $ such that $ (-\lambda_{\phi},-\lambda_{\psi}) $ is in $ \B_{\Z^2} $, and such that the following conditions hold.
\begin{enumerate}[label=\roman*.),leftmargin=30pt,itemsep=5pt]
	\item Frequencies $ \phi $, $ \psi $ and $ \nu $ are in the hyperbolic region $ \mathcal{H} $.
	\item Frequencies lifted from $ \phi,\psi,\nu $, denoted by $ \phi_j,\psi_j,\nu_j $, $ j=1,2,3 $ are such that $ \phi_j,\psi_j,\nu_j $, $ j=1,3 $ are incoming frequencies and $ \phi_2,\psi_2,\nu_2 $ are outgoing frequencies.
	\item We have $ \F_b\cap\Upsilon=\ensemble{\phi,-\phi,\psi,-\psi} $ (so in particular we have $ \phi,\psi\in\Upsilon $ and $ \nu\in\Xi_0\setminus\Upsilon $).
	\item The following two resonances hold:
	\begin{subequations}\label{eq hyp res phi psi nu}
		\begin{align}
			\lambda_{\phi}\,\phi_1+\lambda_{\psi}\,\psi_1+\nu_2&=0\label{eq hyp res phi psi nu 1},\\
			\lambda_{\phi}\,\phi_3+\lambda_{\psi}\,\psi_2+\nu_2&=0\label{eq hyp res phi psi nu 2}.
		\end{align}
	\end{subequations}
	\item There is no other resonance between frequencies inside the domain. More precisely, if there exists a resonance relation of the form
	\begin{equation*}
		\lambda_1\,\alpha_1+\lambda_2\,\alpha_2+\lambda_3\,\alpha_3=0,
	\end{equation*}
with $ \lambda_1,\lambda_2,\lambda_3\in\Z^* $ and $ \alpha_1,\alpha_2,\alpha_3\in\F\privede{0} $ noncolinear, then, there exists $ \lambda\in\Z^* $, such that up to a renumbering, $ \lambda_1=\lambda\lambda_{\phi} $, $ \lambda_2=\lambda\lambda_{\psi} $, $ \lambda_3=\lambda $ and ($ \alpha_1=\phi_1 $, $ \alpha_2=\psi_1 $ and $ \alpha_3=\nu_2 $) or ($ \alpha_1=\phi_3 $, $ \alpha_2=\psi_2 $ and $ \alpha_3=\nu_2 $).
\end{enumerate}
\end{assumption}

Frequencies lifted inside from frequencies $ \phi $, $ \psi $ and $ \nu $ are depicted in Figure \ref{figure phi psi nu}. There is an amplification in the lifting of $ \phi $ and $ \psi $ because these frequencies are in the region $ \Upsilon $ where the Kreiss-Lopatinskii condition is not satisfied, in contrast to $ \nu $. Amplification arise since, for a frequency in $ \Upsilon $, there is an ascent of small amplitudes toward higher one, namely, a boundary source term of order $ O\big(\epsilon^{n+1}\big) $ occurs in the equations for the profile of order $ O\big(\epsilon^{n}\big) $. Therefore, when amplification occurs, inside profiles lifted from boundary terms of order $ O\big(\epsilon^{n+1}\big) $ are one order higher, namely $ O\big(\epsilon^{n}\big) $.

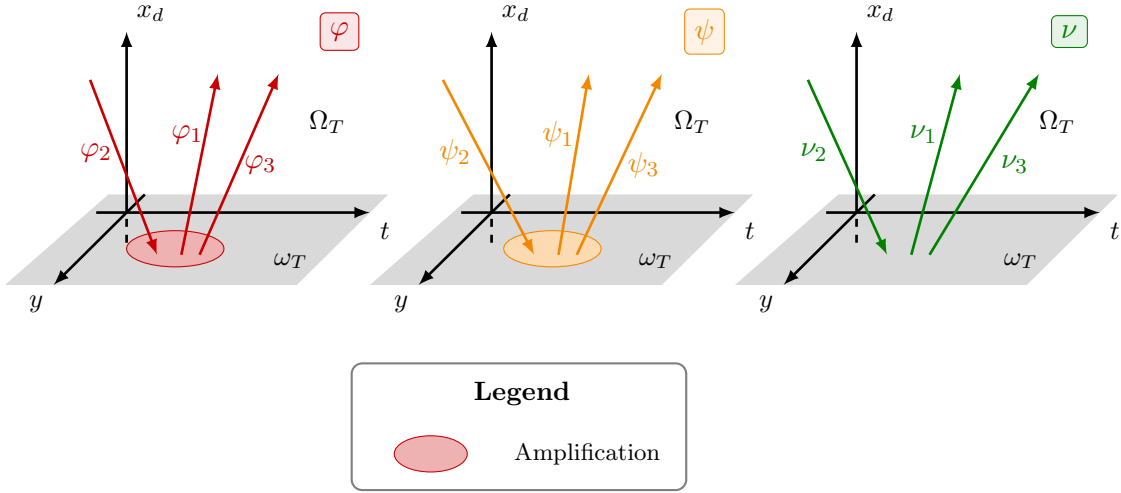
\begin{figure}[h]
	\centering
	\begin{tikzpicture}[scale=0.8]
		\begin{scope}
		\fill[black!15] (-0.3,0.3) -- (4.3,0.3) -- (2.8,-1.2) -- (-2,-1.2)--cycle;
		\draw[line width = 1pt,->,>=latex] (-0.5,0) -- (4,0);
		\draw[line width = 1pt,dashed] (0,-0.5) -- (0,0);
		\draw[line width = 1pt,->,>=latex] (0,0) -- (0,3);
		\draw[line width = 1pt,->,>=latex] (0.3,0.3) -- (-1.2,-1.2);
		\draw[altred,fill=altred!30] (0.8,-0.6) ellipse (0.8cm and 0.3cm);
		\draw[above right] (0,3) node{\small $ x_d $};
		\draw[below left] (-1.2,-1.2) node{\small $ y $};
		\draw[below right] (4,0) node{\small $ t $};
		%\draw (2.2,-1.6) node{\small $ x_d=0 $};
		\draw[altred,->,line width = 1pt,>=latex] (-0.6,2.2) -- (0.5,-0.7); 
		\draw[altred] (-0.5,1) node{$ \phi_2 $};
		\draw[altred,->,line width = 1pt,>=latex] (0.9,-0.7) -- (1.5,2.3); 
		\draw[altred] (1,1.3) node{$ \phi_1 $};
		\draw[altred,->,line width = 1pt,>=latex] (1.2,-0.7) -- (2.5,2.3); 
		\draw[altred] (2.2,0.8) node{$ \phi_3 $};
		\node[draw,altred,rounded corners=2pt,fill=altred!10] (phi) at (3.5,3) {$ \phi $};
		\draw (3.3,1.5) node{\small $ \Omega_T $};
		\draw (2.7,-0.8) node {\small $ \omega_T $};
	\end{scope}
\begin{scope}[shift={(6,0)}]
	\fill[black!15] (-0.3,0.3) -- (4.3,0.3) -- (2.8,-1.2) -- (-2,-1.2)--cycle;
	\draw[line width = 1pt,->,>=latex] (-0.5,0) -- (4,0);
	\draw[line width = 1pt,dashed] (0,-0.5) -- (0,0);
	\draw[line width = 1pt,->,>=latex] (0,0) -- (0,3);
	\draw[line width = 1pt,->,>=latex] (0.3,0.3) -- (-1.2,-1.2);
	\draw[altorange,fill=altorange!30] (1,-0.6) ellipse (0.8cm and 0.3cm);
	\draw[above right] (0,3) node{\small $ x_d $};
	\draw[below left] (-1.2,-1.2) node{\small $ y $};
	\draw[below right] (4,0) node{\small $ t $};
%	\draw (2.2,-1.6) node{\small $ x_d=0 $};
	\draw[altorange,->,line width = 1pt,>=latex] (-0.8,2.2) -- (0.7,-0.7); 
	\draw[altorange] (-0.6,1) node{$ \psi_2 $};
	\draw[altorange,->,line width = 1pt,>=latex] (1.1,-0.7) -- (1.6,2.3); 
	\draw[altorange] (1.1,1.3) node{$ \psi_1 $};
	\draw[altorange,->,line width = 1pt,>=latex] (1.4,-0.7) -- (2.8,2.3); 
	\draw[altorange] (2.5,0.8) node{$ \psi_3 $};
	\node[draw,altorange,rounded corners=2pt,fill=altorange!10] (phi) at (3.5,3) {$ \psi $};
	\draw (3.3,1.5) node{\small $ \Omega_T $};
	\draw (2.7,-0.8) node {\small $ \omega_T $};
\end{scope}
\begin{scope}[shift={(12,0)}]
\fill[black!15] (-0.3,0.3) -- (4.3,0.3) -- (2.8,-1.2) -- (-2,-1.2)--cycle;
\draw[line width = 1pt,->,>=latex] (-0.5,0) -- (4,0);
\draw[line width = 1pt,dashed] (0,-0.5) -- (0,0);
\draw[line width = 1pt,->,>=latex] (0,0) -- (0,3);
\draw[line width = 1pt,->,>=latex] (0.3,0.3) -- (-1.2,-1.2);
\draw[above right] (0,3) node{\small $ x_d $};
\draw[below left] (-1.2,-1.2) node{\small $ y $};
\draw[below right] (4,0) node{\small $ t $};
%\draw (2.2,-1.6) node{\small $ x_d=0 $};
\draw[altgreen,->,line width = 1pt,>=latex] (-0.8,2.2) -- (0.5,-0.7); 
\draw[altgreen] (-0.7,1) node{$ \nu_2 $};
\draw[altgreen,->,line width = 1pt,>=latex] (0.9,-0.7) -- (1.7,2.3); 
\draw[altgreen] (1.1,1.3) node{$ \nu_1 $};
\draw[altgreen,->,line width = 1pt,>=latex] (1.2,-0.7) -- (3,2.3); 
\draw[altgreen] (2.6,0.8) node{$ \nu_3 $};
\node[draw,altgreen,rounded corners=2pt,fill=altgreen!10] (phi) at (3.5,3) {$ \nu $};
\draw (3.3,1.5) node{\small $ \Omega_T $};
\draw (2.7,-0.8) node {\small $ \omega_T $};
\end{scope}
\begin{scope}[shift={(5,-4)}]
	% Texte légendes
	\draw[altred,fill=altred!30] (0,0) ellipse (0.6cm and 0.3cm);
	\draw[right] (1.2,0) node{\footnotesize Amplification};
	% Titre et boite -2.1
	\draw (1.5,1) node{\small \textbf{Legend}};
	\draw[rounded corners,thick,black!50] (-1.3,1.5) rectangle (4.2, -0.6) {};
\end{scope}
	\end{tikzpicture}
	\caption{Frequencies lifted from $ \phi $, $ \psi $ and $ \nu $.}
	\label{figure phi psi nu}
\end{figure}

\begin{remark}\label{remark numbering phi psi}
	\begin{itemize}[leftmargin=20pt]
		\item Point i.) of Assumption \ref{hypothese ensemble frequences} asserts that each frequency $ \phi $, $ \psi $ and $ \nu $ is lifted into three real characteristic frequencies inside the domain.
		\item Point ii.) of Assumption \ref{hypothese ensemble frequences} implies in particular that the integer $ p $, which is the rank of $ B $ and the dimension of the stable subspace $ E_-(\zeta) $ for $ \zeta $ in $ \Xi $, is equal to 2.
	\item In relations \eqref{eq hyp res phi psi nu}, the numeration of the frequencies occurring in the resonances \eqref{eq hyp res phi psi nu} is arbitrary. For the first resonance \eqref{eq hyp res phi psi nu 1}, each of the two incoming frequencies lifted from $ \phi $ and $ \psi $ can be chosen. It sets the numbering of the frequencies lifted from $ \phi $ and $ \psi $. Next, for the second resonance \eqref{eq hyp res phi psi nu 2}, there is no choice, the incoming frequency lifted from $ \phi $ which occurs in the resonance must be the one which did not occur in the first one, $ \phi_3 $ in our fixed numbering, since we already required that $ \lambda_{\phi}\,\phi_1+\nu_2=-\lambda_{\psi}\,\psi_1 $, and $ \psi_2 $ is the only outgoing frequency associated with $ \psi $. 
	\item We choose a numbering of $ \alpha_j(\zeta) $ for $ \zeta=\phi,\psi,\nu $ such that, for any $ j=1,2,3 $, we have
	\begin{equation*}
		\alpha_j(\zeta)=\zeta_j,
	\end{equation*}
where the $ \zeta_j $ are the hyperbolic frequencies defined in Assumption \ref{hypothese ensemble frequences}.
\item The condition $ (-\lambda_{\phi},-\lambda_{\psi}) \in \B_{\Z^2} $ is not restrictive and only relies on permuting the notation for $ \phi $ and $ \psi $ or $ -\phi $ and $ \phi $. It is made to simplify notation in the following.
\end{itemize}
\end{remark}

A useful notation is now introduced for the resonances. 

\begin{definition}
	For $ \zeta\in\ensemble{\phi,\psi,\nu} $, and $ j=1,2,3 $, the set $ \mathcal{R}(\zeta,j) $ is defined as the set of quadruples $ (\zeta_1,\zeta_2,j_1,j_2) $ in $ \ensemble{\phi,\psi,\nu}^2\times\ensemble{1,2,3}^2 $ such that the following resonance holds
	\begin{equation*}
		\lambda_{\zeta}\,\alpha_j(\zeta)+\lambda_{\zeta_1}\,\alpha_{j_1}(\zeta_1)+\lambda_{\zeta_2}\,\alpha_{j_2}(\zeta_2)=0,
	\end{equation*}
	where we have denoted $ \lambda_{\nu}:=1 $.
\end{definition}

For example, we have, according to Assumption \ref{hypothese ensemble frequences},
\begin{align*}
	\mathcal{R}(\phi,1)&=\ensemble{(\psi,\nu,1,2),(\nu,\psi,2,1)},\quad \mathcal{R}(\phi,2)=\emptyset,\quad\mathcal{R}(\phi,3)=\ensemble{(\psi,\nu,3,2),(\nu,\psi,2,3)},\quad \text{and}\quad\\[5pt]
	\mathcal{R}(\nu,2)&=\ensemble{(\phi,\psi,1,1),(\phi,\psi,3,2),(\psi,\phi,1,1),(\psi,\phi,2,3)}.
\end{align*}

\bigskip

We conduct now a formal discussion on how the configuration of frequencies $ \phi_j $, $ \psi_j $ and $ \nu_j $, $ j=1,2,3 $ and the two resonances \eqref{eq hyp res phi psi nu} are expected to create an instability, as represented in Figure \ref{figure configuration resonance}. First, the boundary profiles $ \epsilon^2\,g^{\epsilon} $ and $ \epsilon^3\,h^{\epsilon} $ of frequencies $ \phi $ and $ \psi $ in \eqref{eq systeme 1} create, because of the amplification due to the breaking of Kreiss-Lopatinskii condition for those frequencies, incoming interior profiles of frequencies $ \phi_1 $, $ \phi_3 $, and $ \psi_1 $, $ \psi_3 $ of orders respectively $ O\big(\epsilon\big) $ and $ O\big(\epsilon^2\big) $. Then because of the resonance relation \eqref{eq hyp res phi psi nu 1}, the profiles associated with $ \phi_1 $ and $ \psi_1 $ resonate to create a profile of outgoing frequency $ \nu_2 $ and of order $ O\big(\epsilon^2\big) $\footnote{One of the quadratic term in the equations has a factor $ 1/\epsilon $ in front of it, because the product $ A_i(u^{\epsilon})\,\partial_iu^{\epsilon} $ involves a derivative which counts as $ 1/\epsilon $ for oscillating wave packets at frequency of order $ 1/\epsilon $.}. 
This profile interacts, through resonance relation \eqref{eq hyp res phi psi nu 2}, with the one of frequency $ \phi_3 $ and of order $ O\big(\epsilon\big) $, which is lifted from the boundary forcing term $ \epsilon^2\,g^{\epsilon} $. This resonance leads to a profile of frequency $ \psi_2 $ and amplitude $ O\big(\epsilon^2\big) $, which is an outgoing profile, so a reflection and thus an amplification occur. Indeed, it creates a boundary profile of frequency $ \psi $ and order $ O\big(\epsilon^2\big) $: we obtain instability. Indeed, this boundary profile creates, through amplification on the boundary for $ \psi $, a profile of frequency $ \psi_1 $ and order $ O\big(\epsilon\big) $, which is one order higher than the profile of frequency $ \psi_1 $ we started with. Iterating this process leads to an explosion. 

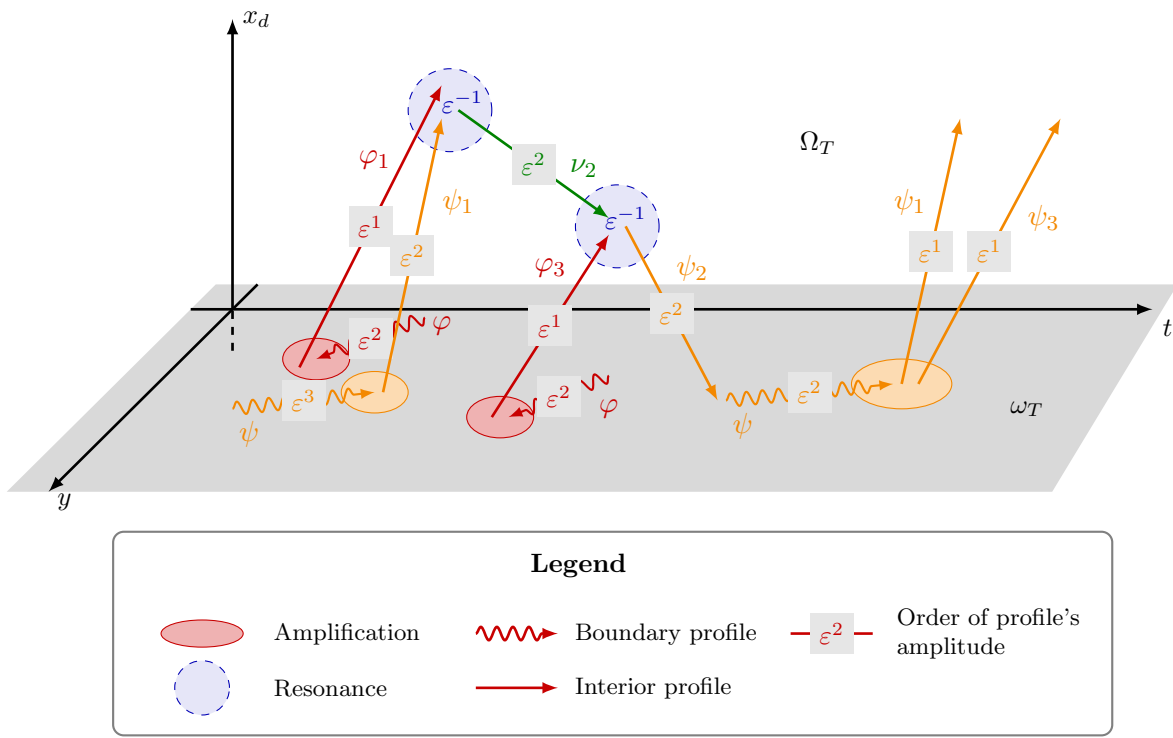
\begin{figure}[h]
	\centering
	\begin{tikzpicture}
		\tikzstyle{etiquette}=[midway,fill=black!10]
		\tikzstyle{spring}=[thick,decorate,decoration={snake,pre length=0.0cm,post
			length=0.2cm,segment length=6}]
		\begin{scope}[scale=1.1]
			% Cadre gris
			\fill[black!15] (-1.2,0.3) -- (10.3,0.3) -- (8.8,-2.2) -- (-3.7,-2.2)--cycle;
			% Repère
			\draw[line width = 1pt,->,>=latex] (-1.5,0) -- (10,0);
			\draw[line width = 1pt,dashed] (-1,-0.5) -- (-1,0);
			\draw[line width = 1pt,->,>=latex] (-1,0) -- (-1,3.5);
			\draw[line width = 1pt,->,>=latex] (-0.7,0.3) -- (-3.2,-2.2);
			% Cercles résonances
			\draw[altblue,dashed,fill=altblue!10] (1.6,2.4) circle (0.5cm);
			\node[altblue] (a) at (1.75,2.5) {\small $ \epsilon^{-1} $};
			\draw[altblue,dashed,fill=altblue!10] (3.6,1) circle (0.5cm);
			\node[altblue] (a) at (3.7,1.1) {\small $ \epsilon^{-1} $};
			% Ellipses amplification
			\draw[altorange,fill=altorange!30] (7,-0.9) ellipse (0.6cm and 0.3cm);
			\draw[altred,fill=altred!30] (0,-0.6) ellipse (0.4cm and 0.25cm);
			\draw[altorange,fill=altorange!30] (0.7,-1) ellipse (0.4cm and 0.25cm);
			\draw[altred,fill=altred!30] (2.2,-1.3) ellipse (0.4cm and 0.25cm);
			% Noms axes
			\draw[right] (-1,3.5) node{\small $ x_d $};
			\draw[below left] (-2.8,-2.1) node{\small $ y $};
			\draw[below right] (10,0) node{\small $ t $};
%			\draw (9.7,-1.7) node{\small $ x_d=0 $};
			\draw (6,2) node{\small $ \Omega_T $};
			\draw (8.5,-1.2) node {\small $ \omega_T $};
			% Flèches 
			%% Zigzags
			\draw[spring,->,>=latex,altorange] (-1,-1.2) -- (0.7,-1) node[etiquette]{\small $ \epsilon^3 $};
			\draw[altorange] (-0.8,-1.5) node{$ \psi $};
			\draw[spring,->,>=latex,altorange] (4.9,-1.1) -- (6.9,-0.9) node[etiquette]{\small $ \epsilon^2 $};
			\draw[altorange] (5.1,-1.4) node{$ \psi $};
			\draw[spring,->,>=latex,altred] (1.3,-0.1) -- (-0,-0.6) node[etiquette]{\small $ \epsilon^2 $};
			\draw[altred] (1.5,-0.2) node{$ \phi $};
			\draw[spring,->,>=latex,altred] (3.5,-0.8) -- (2.3,-1.3) node[etiquette]{\small $ \epsilon^2 $};
			\draw[altred] (3.5,-1.2) node{$ \phi $};
			%% Intérieures
			\draw[altred,->,line width = 1pt,>=latex] (-0.2,-0.7) -- (1.5,2.7) node[etiquette]{\small $ \epsilon^1 $}; 
			\draw[altred] (0.7,1.8) node{$ \phi_1 $};
			\draw[altorange,->,line width = 1pt,>=latex] (0.8,-1) -- (1.5,2.3) node[etiquette]{\small $ \epsilon^2 $}; 
			\draw[altorange] (1.7,1.3) node{$ \psi_1 $};
			\draw[altgreen,->,line width = 1pt,>=latex] (1.7,2.4) -- (3.5,1.1) node[etiquette]{\small $ \epsilon^2 $}; 
			\draw[altgreen] (3.2,1.7) node{$ \nu_2 $};
			\draw[altred,->,line width = 1pt,>=latex] (2.1,-1.3) -- (3.5,0.9) node[etiquette]{\small $ \epsilon^1 $}; 
			\draw[altred] (2.8,0.5) node{$ \phi_3 $};
			\draw[altorange,->,line width = 1pt,>=latex] (3.7,1) -- (4.8,-1.1) node[etiquette]{\small $ \epsilon^2 $}; 
			\draw[altorange] (4.5,0.5) node{$ \psi_2 $};
			\draw[altorange,->,line width = 1pt,>=latex] (7,-0.9) -- (7.7,2.3) node[etiquette]{\small $ \epsilon^1 $}; 
			\draw[altorange] (7.1,1.3) node{$ \psi_1 $};
			\draw[altorange,->,line width = 1pt,>=latex] (7.2,-0.9) -- (8.9,2.3) node[etiquette]{\small $ \epsilon^1 $}; 
			\draw[altorange] (8.7,1.1) node{$ \psi_3 $};
		\end{scope}
		\begin{scope}[shift={(-1.5,-4.3)},scale=0.9]
			% Texte légendes
			\draw[altred,fill=altred!30] (0,0) ellipse (0.6cm and 0.2cm);
			\draw[right] (0.9,0) node{\footnotesize Amplification};
			\draw[altblue,dashed,fill=altblue!10] (0,-0.8) circle (0.4cm);
			\draw[right] (0.9,-0.8) node{\footnotesize Resonance};
			\draw[spring,altred,->,line width = 1pt,>=latex] (4,0)--(5.2,0);
			\draw[right] (5.3,0) node{\footnotesize Boundary profile};
			\draw[altred,->,line width = 1pt,>=latex] (4,-0.8)--(5.2,-0.8);
			\draw[right] (5.3,-0.8) node{\footnotesize Interior profile};
			\draw[altred,line width = 1pt] (8.6,0) -- (9.8,0) node[etiquette]{\small $ \epsilon^2 $}; 
			\draw[right,align=left] (10,0) node{\footnotesize Order of profile's\\[-2pt]\footnotesize amplitude};
			% Titre et boite -2.1
			\draw (5.5,1) node{\small \textbf{Legend}};
			\draw[rounded corners,thick,black!50] (-1.3,1.5) rectangle (13.3, -1.5) {};
		\end{scope}
	\end{tikzpicture}
	\caption{Creation of instability through amplification.}
	\label{figure configuration resonance}
\end{figure}

\bigskip

We make now a small divisors assumption, which is adapted from \cite[Assumption 1.9]{CoulombelWilliams2017Mach}. This assumption is needed only for frequencies for which the uniform Kreiss-Lopatinskii condition is not satisfied, so, in our case, for $ \phi $ and $ \psi $. Analogously to \cite[Assumption 1.9]{CoulombelWilliams2017Mach}, it requires a polynomial control of the determinant of the symbol associated with combinations of incoming frequencies, using the fact that frequencies lifted from $ \phi $ do not resonate, and the same for $ \psi $. The formulation is simpler than the one of \cite[Assumption 1.9]{CoulombelWilliams2017Mach} since in our case there is only two incoming frequencies, so the only possibility for a combination of it is $ \lambda_1\phi_1+\lambda_3\phi_3 $, with $ \lambda_1,\lambda_3\in\Z^* $, and the same for $ \psi $.

\begin{assumption}\label{hyp petits diviseurs}
	There exists a constant $ C>0 $ and a real positive number $ m_0 $ such that, for $ \zeta=\phi,\psi $, for all $ \lambda_1,\lambda_3\in\Z^* $, 
	\begin{equation*}
		\big|\det L\big(0,\lambda_1\,\zeta_1+\lambda_3\,\zeta_3\big)\big|\geq C \big|(\lambda_1,\lambda_3)\big|^{-m_0}.
	\end{equation*}
\end{assumption}

Finally, we define several vectors associated with the previously introduced eigenspaces. For $ \zeta $ in $ \F_b\privede{0} $ and $ j\in \mathcal{C}(\zeta) $, we denote by $ r_{\zeta,j} $ a unit column vector of the one dimensional space $ \ker L\big(0,\alpha_j(\zeta)\big) $, and $ \ell_{\zeta,j} $ a row vector such that
\begin{equation}\label{eq relation l zeta j L alpha j}
	\ell_{\zeta,j}\,L\big(0,\alpha_j(\zeta)\big)=0
\end{equation}
with the following normalization: for all $ \zeta $ in $ \F_{b}\privede{0} $ and for all $ j,j' $ in $ \mathcal{C}(\zeta) $, we have
\begin{equation}\label{eq relation l zeta j normalisation}
	\ell_{\zeta,j'}\,A_d(0)\,r_{\zeta,j}=\delta^j_{j'}.
\end{equation}
%so that, for $ \zeta $ in $ \F_b\privede{0} $ and $ j\in \mathcal{C}(\zeta) $, projector $ \pi_{\alpha_j(\zeta)} $ is given by
%\begin{equation}\label{eq relation pi ell}
%	\pi_{\alpha_j(\zeta)}=\ell_{\zeta,j}\cdot X \,r_{\zeta,j}, \qquad \forall X\in\C^3.
%\end{equation}

The projectors $ P_j(\zeta) $, $ Q_j(\zeta) $ (defined in Lemma \ref{lemme def Pj}) and the vectors $ r_{\zeta,j} $ and $ \ell_{\zeta,j} $ are chosen to be homogeneous of degree 0 with respect to $ \zeta $. Accordingly, we define the partial inverses $ R_{\zeta,j} $, which satisfy, for $ \zeta $ in $ \F_b\privede{0} $ represented by $ \mathbf{n}_0,\lambda $ in $ \B_{\Z^2}\times\Z^* $,
\begin{equation}\label{eq relation R zeta j P}
	R_{\zeta,j}\,L\big(0,\alpha_j(\zeta)\big)=L\big(0,\alpha_j(\zeta)\big)\,R_{\zeta,j}=\lambda\,\big(I-P_{\zeta,j}\big).
\end{equation}

Consider $ \zeta\in\Upsilon $. Assumption \ref{hypothese weak lopatinskii} asserts that the space $ \ker B \cap E_-(\zeta) $ is one dimensional, so we denote by $ e_{\zeta} $ a unit vector in this space. Now, since, according to the same assumption, $ \Upsilon $ is included in the hyperbolic region $ \mathcal{H} $ and because of Proposition \ref{prop decomp E_-}, we can decompose $ e_{\zeta} $ as
\begin{equation}\label{eq decomp e zeta}
	 e_{\zeta}=\sum_{j\in\mathcal{I}(\zeta)}e_{\zeta,j},
\end{equation} 
with $ e_{\zeta,j}\in\Span r_{\zeta,j} $ for $ j\in\mathcal{I}(\zeta) $. 
We also denote by $ b_{\zeta} $ a vector of $ \C^2 $ such that
\begin{equation}\label{eq def b zeta}
	B\,E_-(\zeta)=\ensemble{X\in\C^2\,\middle|\, b_{\zeta}\cdot X=0},
\end{equation}
that is, a nonzero vector of $ \ker \,^tB_{|E_-(\zeta)} $, which is of dimension 1. Notation $ b_{\zeta}\cdot X $ refers to the complex scalar product in $ \C^2 $.

Using vectors $ r_{\zeta,j} $ and $ \ell_{\zeta,j} $, we have the following lemma, analogous to the one of \cite{Lax1957Asymptotic}. The proof of this particular result can be found in \cite{CoulombelGuesWilliams2011Resonant}, and is recalled here for the sake of clarity.

\begin{lemma}[{\cite[Lemma 2.11]{CoulombelGuesWilliams2011Resonant}}]\label{lemme Lax}
	For $ \zeta\in\F_{b}\privede{0} $ and $ j\in\mathcal{C}(\zeta) $, we have
	\begin{equation*}
		\ell_{\zeta,j}\,L(0,\partial_{z})\,r_{\zeta,j}=X_{\alpha_j(\zeta)},
	\end{equation*}
where $ X_{\alpha_j(\zeta)} $ is the vector field defined in Definition \ref{def sortant rentrant alpha X alpha}.
\end{lemma}

\begin{proof}
	Denote by $ k $ the integer between $ 1 $ and $ 3 $ such that, if $ \alpha_j(\zeta)=(\tau,\eta,\xi) $, then $ \tau=\tau_k(\eta,\xi) $. Since $ \zeta\in\F_b\privede{0} $, the frequency $ \zeta $ is not glancing, so, according to definition \ref{def sortant rentrant alpha X alpha}, we have $ \partial_{\xi}\tau_k(\eta,\xi)\neq0 $. Therefore, according to the implicit function theorem, the function $ \zeta'\mapsto\xi_j(\zeta') $ is differentiable near $ \zeta $. Indeed, $ \xi_j $ is such that, if $ \zeta'=(\tau',\eta') $,
	 \begin{equation}\label{eq inter1}
	 	\tau_k\big(\eta',\xi_j(\tau',\eta')\big)-\tau'=0.
	 \end{equation} 
Therefore\footnote{Here we extend the definition of $ r_{\zeta,j} $ to any frequency $ \zeta $ in $ \Xi\setminus\mathcal{G} $.}, seen as a function of $ \zeta $, the vector $ r_{\zeta,j} $ is also differentiable with respect to $ \zeta $. Differentiating relation \eqref{eq inter1} even proves the following relations:
 \begin{equation}\label{eq inter2}
 	\partial_{\tau}\xi_j(\tau,\eta)=\inv{\partial_{\xi}\tau_k(\eta,\xi)},\qquad \partial_{\eta_p}\xi_j(\tau,\eta)=\frac{-\partial_{\eta_p}\tau_k(\eta,\xi)}{\partial_{\xi}\tau_k(\eta,\xi)},\quad \forall p=1,\dots,d-1.
 \end{equation}

Now, differentiating the relation
\begin{equation*}
	L\big(0,(\tau,\eta,\xi_j(\tau,\eta))\big)\,r_{\zeta,j}=0
\end{equation*}
with respect to $ \tau $ and $ \eta_p $, $ p=1,\dots,d-1 $,
and multiplying on the left by $ \ell_{\zeta,j} $, gives
\begin{align*}
	\ell_{\zeta,j}\,r_{\zeta,j}+\partial_{\tau}\xi_j(\tau,\eta)\,\ell_{\zeta,j}\,A_d(0)\,r_{\zeta,j}&=0,
	\intertext{and, for $ p=1,\dots,d-1 $,} \ell_{\zeta,j}\,A_p(0)\,r_{\zeta,j}+\partial_{\eta_p}\xi_j(\tau,\eta)\,\ell_{\zeta,j}\,A_d(0)\,r_{\zeta,j}&=0.
\end{align*}
With relations \eqref{eq inter2}, the result follows.
\end{proof} 

The following result could be seen as an analogue to Lax lemma, for the boundary. Indeed, it asserts that a certain operator appearing in the boundary equations is actually a linear transport operator with constant velocity. The result is due to \cite{CoulombelGues2010Geometric}, and its technical proof is not recalled here.

\begin{lemma}[{\cite[Proposition 3.5]{CoulombelGues2010Geometric}}]\label{lemme Lax bord}
	Let $ \zeta=\phi,\psi $, and recall that $ \kappa $ is the scalar function of the weak Kreiss-Lopatinskii condition, see Assumption \ref{hypothese weak lopatinskii}. Then, there exists a nonzero real scalar $ \beta_{\zeta} $ such that
	\begin{equation*}
		b_{\zeta}\cdot B\,\Big(R_{\zeta,1}\,L(0,\partial_z)\,e_{\zeta,1}+R_{\zeta,3}\,L(0,\partial_z)\,e_{\zeta,3}\Big)=\beta_{\zeta}\left(\partial_{\tau}\kappa(\zeta)\,\partial_t+\sum_{j=1}^{d-1}\partial_{\eta_j}\kappa(\zeta)\,\partial_{x_j}\right).
	\end{equation*}
Moreover, the coefficient $ \partial_{\tau}\kappa(\zeta) $ is equal to 1.
\end{lemma}

\begin{remark}
	In particular, the previous result ensures that the operator $ b_{\zeta}\cdot B\,\big(R_{\zeta,1}\,L(0,\partial_z)\,e_{\zeta,1}\linebreak+R_{\zeta,3}\,L(0,\partial_z)\,e_{\zeta,3}\big) $ is tangent to the boundary. 
\end{remark}

\section{Derivation of the system}

This section is devoted to the derivation of the general system studied in this article. We start by detailing the ansatz we choose here, and by displaying the WKB cascade associated with system \eqref{eq systeme 1}. Then we proceed by trying to decouple this cascade for the profiles. 

\subsection{Ansatz and WKB cascade}

The ansatz for each amplitude $ U_n $ of \eqref{eq ansatz provisoire} must allow to consider both oscillating modes (associated with characteristic frequencies $ \alpha_j(\zeta) $ for $ \zeta\in\F_b\privede{0} $ and $ j\in\mathcal{C}(\zeta) $) and evanescent modes (associated with evanescent frequencies $ \alpha_j(\zeta) $ for $ \zeta\in\F_b\privede{0} $ and $ j\in\mathcal{P}(\zeta) $). We define at this purpose the following spaces of profiles. We denote by $ \T $ the one-dimensional torus.

\begin{definition}\label{def espaces profils}
	The space of evanescent profiles $ \P^{\ev}_{T} $ is defined as the set of functions $ U^{\ev} $ of $ \mathcal{C}_b\big(\R^+_{\chi_d},H^{\infty}(\Omega_T\times\T^2)\big) $ which converge to zero as $ \chi_d $ goes to infinity.
	
	The space of oscillating profiles $ \P^{\osc}_{T} $ is defined as the set of formal trigonometric functions in $ \chi_d $ with values in the Sobolev space $ H^{\infty}(\Omega_T\times\T^2) $, that is, formal series 
	\begin{equation*}
		U^{\osc}(z,\theta_1,\theta_2,\chi_d)=\sum_{\xi\in\R}U^{\osc}_{\xi}(z,\theta_1,\theta_2)\,e^{i\,\xi\,\chi_d},
	\end{equation*}
with $ U^{\osc}_{\xi}\in H^{\infty}(\Omega_T\times\T^2) $ for $ \xi\in\R $.

Finally, $ \P_{T} $ is defined as the direct sum
\begin{equation*}
	\P_{T}:=\P^{\osc}_{T}\oplus\P_{T}^{\ev}.
\end{equation*}
\end{definition}

The ansatz is the following: we look for an approximate solution 
\begin{equation*}
	u^{\epsilon,\app}(z):=v^{\epsilon}\left( z,\frac{z'\cdot\phi}{\epsilon},\frac{z'\cdot\psi}{\epsilon},\frac{x_d}{\epsilon}\right)
\end{equation*}
where the formal series $ v^{\epsilon} $ is given by
\begin{equation}\label{eq ansatz}
	v^{\epsilon}\big(z,\theta_1,\theta_2,\chi_d\big):=\sum_{n\geq 1}\epsilon^n\,U_n\big(z,\theta_1,\theta_2,\chi_d\big),
\end{equation}
where, for $ n\geq 1 $, $ U_n $ belongs to $ \P_{T} $.
% and decomposes in $ \P_{T}=\P^{\osc}_{T}\oplus\P_{T}^{\ev} $ as
%\begin{equation*}
%	U_n=U_n^{\osc}+U_n^{\ev},
%\end{equation*}
%with, for $ n\geq 1 $,
%\begin{equation*}
%	U^{\osc}_n(z,\theta_1,\theta_2,\chi_d)=\sum_{\xi\in\R}U^{n,\osc}_{\xi}(z,\theta_1,\theta_2)\,e^{i\,\xi\,\chi_d}.
%\end{equation*}

Formally plugging ansatz \eqref{eq ansatz} in system \eqref{eq systeme 1}, we obtain the following WKB cascade\footnote{We have used here the assumption that coefficients $ A_j $ are affine maps.} (see \cite{CoulombelWilliams2017Mach}) for the profiles $ (U_n)_{n\geq 1} $:
\begin{subequations}\label{eq cascade int}
\begin{align}
	\mathcal{L}(\partial_{\theta},\partial_{\chi_d})\,U_1&=0\label{eq cascade int U1}\\[5pt]
	\mathcal{L}(\partial_{\theta},\partial_{\chi_d})\,U_2+L(0,\partial_z)\,U_1+\mathcal{M}(U_1,U_1)&=0\label{eq cascade int U2}\\
	\mathcal{L}(\partial_{\theta},\partial_{\chi_d})\,U_{n+1}+L(0,\partial_z)\,U_n+\sum_{k=1}^n\mathcal{M}(U_{n-k+1},U_k)+\sum_{k=1}^{n-1}\mathcal{N}(U_{n-k},U_k)&=0,\label{eq cascade int Un}
\end{align}
\end{subequations}
where \eqref{eq cascade int Un} should hold for any $ n\geq 2 $. In \eqref{eq cascade int}, the fast operator $ \mathcal{L}(\partial_{\theta},\partial_{\chi_d}) $ and the quadratic operators $ \mathcal{M} $ and $ \mathcal{N} $ are defined by
%\begin{subequations}
\begin{align*}
	\mathcal{L}(\partial_{\theta},\partial_{\chi_d}):=&\,L(0,\phi)\,\partial_{\theta_1}+L(0,\psi)\,\partial_{\theta_2}+A_d(0)\,\partial_{\chi_d},\\[5pt]
	\mathcal{M}(u,v):=&\,L_1(u,\phi)\,\partial_{\theta_1}v+L_1(u,\psi)\,\partial_{\theta_2}v+dA_d(0)\,\cdot u\,\partial_{\chi_d}v\\
	=&\,\sum_{k=1}^{d-1}dA_k(0)\cdot u\, \big(\phi_k\,\partial_{\theta_1}+\psi_k\,\partial_{\theta_2}\big) v+dA_d(0)\cdot u\,\partial_{\chi_d}v,\\
	\mathcal{N}(u,v):=&\,L_1(u,\partial_z)\,v:=\sum_{k=1}^d dA_k(0)\cdot u\,\partial_k v,
\end{align*}
where we have denoted by, for $ X $ in $ \C^3 $ and $ \zeta=(\zeta_0,\dots,\zeta_{d-1})\in\R^{d} $, 
\begin{equation*}
		L(0,\zeta):=L\big(0,(\zeta,0)\big)=\sum_{k=1}^{d-1}\zeta_k A_k(0),\qquad L_1(X,\zeta):=\sum_{k=1}^{d-1}\zeta_k\, dA_k(0)\cdot X.
\end{equation*}
The boundary and initial conditions of \eqref{eq systeme 1} reads 
\begin{subequations}\label{eq cascade bord}
\begin{align}
	B\,\big(U_1\big)_{|x_d,\chi_d=0}(z',\theta_1,\theta_2)&=0\label{eq cascade bord U12},
	&B\,\big(U_2\big)_{|x_d,\chi_d=0}(z',\theta_1,\theta_2)&=G(z',\theta_1)\\
	B\,\big(U_3\big)_{|x_d,\chi_d=0}(z',\theta_1,\theta_2)&=H(z',\theta_2),
	&B\,\big(U_n\big)_{|x_d,\chi_d=0}(z',\theta_1,\theta_2)&=0,\quad\forall n\geq 4\label{eq cascade bord U3n},
\end{align}
\end{subequations}
and 
\begin{equation}\label{eq cascade initial}
	\big(U_n\big)_{|t\leq 0}=0,\quad \forall n\geq 1.	
\end{equation}

The aim is now to decouple cascade \eqref{eq cascade int}.
First we use polarization equation \eqref{eq cascade int U1} to obtain the form of the leading profile $ U_1 $, and proceed to show that the mean value $ U_1^* $ of $ U_1 $ is zero using evolution equation \eqref{eq cascade int U2}. Then we need to determine the oscillating part of $ U_1 $. Equation \eqref{eq cascade int U2} leads to a transport equations for each mode. When the equation is outgoing (that is, when the frequency is outgoing), the transport equation can be solved with a source term eventually depending on other leading profiles, due to resonances. When it is incoming (i.e. when the frequency is incoming), we need to determine a boundary condition from first equation of \eqref{eq cascade bord U12}. Two cases may occur. If the associated boundary frequency $ \zeta $ is not in $ \Upsilon $, that is, for $ \zeta\neq \phi,\psi $, we can write boundary condition \eqref{eq cascade bord U12} for $ \zeta $ as $ B\,X=F $, where source term $ F $ depends on the trace of the outgoing leading profile for this boundary frequency $ \zeta $, and where $ X $ (containing traces of incoming leading profiles) belongs to $ E_-(\zeta) $. Since $ B $ is invertible on $ E_-(\zeta) $ according to Assumption \ref{hypothese weak lopatinskii} and the fact that $ \zeta\notin \Upsilon $, this boundary condition $ B\,X=F $ leads to a boundary condition for traces of incoming leading profiles. The second case is more complicated. If $ \zeta\in\Upsilon $, that is, if $ \zeta=\phi,\psi $, matrix $ B $ is no longer invertible on $ E_-(\zeta) $. Therefore, boundary condition $ B\,X=F $ cannot be inverted, and leads to both a compatibility condition (which we shall see will over-determine the system), and expressions for traces of incoming leading profiles for $ \phi $ and $ \psi $ depending on unknown scalar functions $ a_{\phi}^1 $ and $ a_{\psi}^1 $. Then, to determine these functions $ a_{\phi}^1 $and $ a_{\psi}^1 $, we need investigate first corrector $ U_2 $. 

According to equation \eqref{eq cascade int U2}, the first corrector $ U_2 $ is not polarized. But this equation allows us to determine of formula for its nonpolarized part, depending on the leading profiles. We write second equation of boundary condition \eqref{eq cascade bord U12} with these expressions for the nonpolarized parts which leads to equations on the traces of the nonpolarized parts, and therefore, equations on the traces of leading profiles, namely $ a^1_{\phi} $ and $ a^1_{\psi} $. However, the system of equations obtained is still not closed, since equations on $ a^1_{\phi} $ and $ a^1_{\psi} $ involve traces of the first corrector $ U_2 $.

The next step is to obtain equations on the polarized part of the first corrector $ U_2 $, which is achieved using equation \eqref{eq cascade int Un} for $ n=3 $. Once again, depending on the frequency, we obtain incoming or outgoing transport equations with source term depending on leading profile and first corrector. For incoming equations, when the associated frequency $ \zeta $ is not in $ \Upsilon $, boundary condition \eqref{eq cascade bord U12} can be inverted to obtain a closed system. Otherwise, when $ \zeta $ belongs to $ \Upsilon $, the same arguments as for the leading profiles leads to compatibility conditions (that are these time always satisfied by previous construction) and expressions for traces of incoming first corrector profiles for $ \phi $ and $ \psi $ depending on unknown scalar functions $ a_{\phi}^2 $ and $ a_{\psi}^2 $. 

Investigating the nonpolarized part of the second corrector $ U_3 $ leads, in its turn, to equations on $ a_{\phi}^2 $ and $ a_{\psi}^2 $, depending once again on trace of the second corrector $ U_2 $, preventing to close the system. This method applies recursively to any order.

\subsection{Rewriting the equations: leading profile and first corrector}

This subsection is devoted to the almost-decoupling of the cascade \eqref{eq cascade int}, \eqref{eq cascade bord} and \eqref{eq cascade initial}. The computations are, for the most part of it, formal. Except for the leading profile, we will not detail the obtaining of formulas for the evanescent part, as it will not be interesting for the instability analysis, since all three frequencies $ \phi $, $ \psi $ and $ \nu $ are hyperbolic. 

\subsubsection{Leading profile}
We start by deriving the polarization condition for $ U_1 $ from \eqref{eq cascade int U1}, recalling the analysis of \cite{Lescarret2007Wave}. If we write $ U_1 $ in $ \P_{T} $ as
\begin{align*}
	U_1(z,\theta,\chi_d)&=U_1^{\osc}(z,\theta,\chi_d)+U_1^{\ev}(z,\theta,\chi_d)\\[5pt]
	&=\sum_{\mathbf{n}\in\Z^m}\sum_{\xi\in\R}U^{1,\osc}_{\mathbf{n},\xi}(z)\,e^{i\,\mathbf{n}\cdot\theta}\,e^{i\,\xi\,\chi_d}+\sum_{\mathbf{n}\in\Z^m}U^{1,\ev}_{\mathbf{n}}(z,\chi_d)\,e^{i\,\mathbf{n}\cdot\theta},
\end{align*}
equation \eqref{eq cascade int U1} reads
\begin{multline*}
	\sum_{\mathbf{n}\in\Z^2}\sum_{\xi\in\R}iL\big(0,(\freq,\xi)\big)\,U^{1,\osc}_{\mathbf{n},\xi}(z)\,e^{i\,\mathbf{n}\cdot\theta}\,e^{i\,\xi\,\chi_d}\\
	+\sum_{\mathbf{n}\in\Z^2}\big\{iL(0,\freq)+A_d(0)\,\partial_{\chi_d}\big\}\,U^{1,\ev}_{\mathbf{n}}(z,\chi_d)\,e^{i\,\mathbf{n}\cdot\theta}=0.
\end{multline*}
Therefore, on one hand, for the oscillating part, we get $ L\big(0,(\freq,\xi)\big)\,U^{1,\osc}_{\mathbf{n},\xi}=0 $ for every $ \mathbf{n}\in\Z^2\privede{0} $ and $ \xi\in\R $, so, if $ (\freq,\xi) $ is noncharacteristic, $ U^{1,\osc}_{\mathbf{n},\xi}=0 $ and if $ \xi=\xi_j(\freq) $ for some $ j\in\mathcal{C}(\freq) $, we find that $ U^{1,\osc}_{\mathbf{n},\xi} $ belongs to $ \ker L\big(0,\alpha_j(\freq)\big)=\Span r_{\freq,j} $. Thus we write 
\begin{equation*}
	U^{1,\osc}_{\mathbf{n},\xi}=\sigma^1_{\mathbf{n}_0,j,\lambda} \,r_{\mathbf{n}_0\cdot\boldsymbol{\zeta},j},
\end{equation*}
if $ \mathbf{n}=\lambda\,\mathbf{n}_0 $ with $ \mathbf{n}_0\in\B_{\Z^2} $ and $ \lambda\in\Z^* $, and where $ \sigma^1_{\mathbf{n}_0,j,\lambda} $ is a scalar function of $ \Omega_T $. On the other hand, for the evanescent part, we get $ U^{1,\ev}_{0}=0 $, and, for $ \mathbf{n}\in\Z^2\privede{0} $, multiplying by $ A_d(0)^{-1} $,
\begin{equation*}
	\partial_{\chi_d}\,U^{1,\ev}_{\mathbf{n}}-\mathcal{A}(\freq)\,U^{1,\ev}_{\mathbf{n}}=0.
\end{equation*}
Solving this differential equation in $ \mathcal{P}^{\ev}_{T} $ leads to
\begin{equation*}
	U^{1,\ev}_{\mathbf{n}}(z,\chi_d)=e^{\chi_d\mathcal{A}(\mathbf{n}\cdot\boldsymbol{\zeta})}\,\Pi^e(\mathbf{n}\cdot\boldsymbol{\zeta})\,U^{1,\ev}_{\mathbf{n}}(z,0).
\end{equation*}
In short, polarization equation \eqref{eq cascade int U1} asserts that $ U_1 $ reads
\begin{multline}\label{eq ecriture U 1}
	U_1(z,\theta,\chi_d)=U^*_1(z)+\sum_{\mathbf{n}\in\B_{\Z^2}}\sum_{j\in\mathcal{C}(\freq)}\sum_{\lambda\in\Z^*}\sigma^1_{\mathbf{n},j,\lambda}(z)\,e^{i\,\lambda\,\mathbf{n}\cdot\theta}\,e^{i\,\lambda\,\xi_j(\freq)\,\chi_d}\,r_{\freq,j} \\
	+\sum_{\mathbf{n}\in\Z^2\privede{0}}e^{\chi_d\mathcal{A}(\mathbf{n}\cdot\boldsymbol{\zeta})}\,\Pi^e(\mathbf{n}\cdot\boldsymbol{\zeta})\,U^{1,\ev}_{\mathbf{n}}(z,0)\,e^{i\,\mathbf{n}\cdot\theta}.
\end{multline}

We start by showing that the mean value $ U^*_1 $ is zero, using equation \eqref{eq cascade int U2}. The oscillating part of $ L(0,\partial_z)\,U_1+\mathcal{M}(U_1,U_1) $ is given by $ L(0,\partial_z)\,U^{\osc}_1+\mathcal{M}(U^{\osc}_1,U^{\osc}_1) $ and according to \eqref{eq ecriture U 1} and the expression of the quadratic operator $ \mathcal{M} $, the latter reads
\begin{subequations}\label{eq derivation L U1 + M U1}
	\begin{align}\label{eq derivation L U1 + M U1 a}
		&L(0,\partial_z)\,U^*_1+\sum_{\mathbf{n}\in\B_{\Z^2}}\sum_{j\in\mathcal{C}(\freq)}\sum_{\lambda\in\Z^*}L(0,\partial_z)\, \sigma^1_{\mathbf{n},j,\lambda}\, e^{i\lambda\,\mathbf{n}\cdot\theta}\, e^{i\lambda\xi_j(\mathbf{n}\cdot\boldsymbol{\zeta})\chi_d}\,r_{\mathbf{n}\cdot\boldsymbol{\zeta},j}\\\label{eq derivation L U1 + M U1 b}
		&+\sum_{\mathbf{n}\in\B_{\Z^2}}\sum_{j\in\mathcal{C}(\freq)}\sum_{\lambda\in\Z^*}L_1\big(U^*_1,i\,\lambda\,\alpha_j(\mathbf{n}\cdot\boldsymbol{\zeta})\big)\,r_{\mathbf{n}\cdot\boldsymbol{\zeta},j}\,\sigma^1_{\mathbf{n},j,\lambda}\, e^{i\lambda\,\mathbf{n}\cdot\theta}\, e^{i\lambda\xi_j(\mathbf{n}\cdot\boldsymbol{\zeta})\chi_d} \\\label{eq derivation L U1 + M U1 c}
		&+\sum_{\mathbf{n}_1,\mathbf{n}_2\in\B_{\Z^2}}\sum_{\substack{j_1\in\mathcal{C}(\mathbf{n}_1\cdot\boldsymbol{\zeta})\\j_2\in\mathcal{C}(\mathbf{n}_2\cdot\boldsymbol{\zeta})}}\sum_{\lambda_1,\lambda_2\in\Z^*}L_1\big(r_{\mathbf{n}_1\cdot\boldsymbol{\zeta},j_1},i\,\lambda_2\,\alpha_{j_2}(\mathbf{n}_2\cdot\boldsymbol{\zeta})\big)\,r_{\mathbf{n}_2\cdot\boldsymbol{\zeta},j_2}\,\\\nonumber
		&\qquad\qquad\sigma^1_{\mathbf{n}_1,j_1,\lambda_1}\,\sigma^1_{\mathbf{n}_2,j_2,\lambda_2}\,e^{i\,(\lambda_1\mathbf{n}_1+\lambda_2\mathbf{n}_2)\cdot\theta}\,e^{i\,(\lambda_1\xi_{j_1}(\mathbf{n}_1\cdot\boldsymbol{\zeta})+\lambda_2\xi_{j_2}(\mathbf{n}_2\cdot\boldsymbol{\zeta}))\,\chi_d},
	\end{align}
\end{subequations}
where, for $ \alpha=(\alpha_0,\alpha_1,\dots,\alpha_d)\in\R^{d+1} $ and $ X\in\C^3 $, we have denoted
\begin{equation*}
	L_1(X,\alpha):=\sum_{k=1}^d\alpha_k\,dA_k(0)\cdot X.
\end{equation*}
We now isolate the nonoscillating terms in \eqref{eq derivation L U1 + M U1}, to obtain a system satisfied by the mean value $ U^*_1 $. In equation \eqref{eq derivation L U1 + M U1}, the terms in the sums in \eqref{eq derivation L U1 + M U1 a} and \eqref{eq derivation L U1 + M U1 b} are always oscillating since $ \mathbf{n}\in\B_{\Z^2} $. As for them, the terms in the sum in \eqref{eq derivation L U1 + M U1 c}are not oscillating if and only if $ \lambda_1\mathbf{n}_1+\lambda_2\mathbf{n}_2=0 $ and $ \lambda_1\,\xi_{j_1}(\mathbf{n}_1\cdot\boldsymbol{\zeta})+\lambda_2\,\xi_{j_2}(\mathbf{n}_2\cdot\boldsymbol{\zeta})=0 $, that is, if $ \mathbf{n}_1=\mathbf{n}_2 $, $ \lambda_1=-\lambda_2 $ and $ j_1=j_2 $. Therefore, we deduce from \eqref{eq derivation L U1 + M U1} that we have
\begin{equation}\label{eq derivation inter1}
	L(0,\partial_z)\,U^*_1+\sum_{\mathbf{n}\in\B_{\Z^2}}\sum_{j\in\mathcal{C}(\freq)}\sum_{\lambda\in\Z^*}L_1\big(r_{\mathbf{n}\cdot\boldsymbol{\zeta},j},-i\,\lambda\,\alpha_j(\mathbf{n}\cdot\boldsymbol{\zeta})\big)\,r_{\mathbf{n}\cdot\boldsymbol{\zeta},j}\,\sigma^1_{\mathbf{n},j,\lambda}\,\sigma^1_{\mathbf{n},j,-\lambda}=0.
\end{equation}
By a change of variable $ \lambda=-\lambda $ we prove that the second term in the left-hand side of \eqref{eq derivation inter1} is actually zero, so we have the following linear constant coefficient equation
\color{altpink}
\begin{equation*}
	L(0,\partial_z)\,U^*_1=0.
\end{equation*}\color{black}
With the following boundary and initial conditions obtained from \eqref{eq cascade bord U12} and \eqref{eq cascade initial},
\color{altpink}
\begin{equation*}
	B\big(U^*_1\big)_{|x_d,\chi_d=0}=0,\qquad \big(U^*_1\big)_{|t\leq0}=0,
\end{equation*}\color{black}
we get that the mean value $ U^*_1 $ satisfies a system which is weakly well-posed, see \cite{Coulombel2005Wellposedness}, with zero source term, boundary forcing term and initial term, so the mean value $ U^*_1 $ is zero. 

Since $ U^*_1  $ is zero, equation \eqref{eq cascade int U2}  now reads, for each nonzero characteristic mode $ \lambda\,\alpha_j(\freq) $, with $ \mathbf{n}\in\B_{\Z^2} $, $ j\in\mathcal{C}(\freq) $ and $ \lambda\in\Z^* $,
\begin{multline}\label{eq derivation L U1 + M U1 2}
	i\,L\big(0,\alpha_j(\freq)\big)\,U^{2,\osc}_{\mathbf{n},\xi_j(\freq)}+L(0,\partial_z)\, \sigma^1_{\mathbf{n},j,\lambda}\,r_{\mathbf{n}\cdot\boldsymbol{\zeta},j}\\
	+\sum_{(\mathbf{n}_1,\mathbf{n}_2,j_1,j_2,\lambda_1,\lambda_2)}L_1\big(r_{\mathbf{n}_1\cdot\boldsymbol{\zeta},j_1},i\,\lambda_2\,\alpha_{j_2}(\mathbf{n}_2\cdot\boldsymbol{\zeta})\big)\,r_{\mathbf{n}_2\cdot\boldsymbol{\zeta},j_2}\sigma^1_{\mathbf{n}_1,j_1,\lambda_1}\,\sigma^1_{\mathbf{n}_2,j_2,\lambda_2}=0,
\end{multline}
where the sum is over the set of 6-tuples $ (\mathbf{n}_1,\mathbf{n}_2,j_1,j_2,\lambda_1,\lambda_2) $ in $ \big(\B_{\Z^2}\big)^2\times\mathcal{C}(\mathbf{n}_1\cdot\boldsymbol{\zeta})\times\mathcal{C}(\mathbf{n}_2\cdot\boldsymbol{\zeta})\times(\Z^*)^2 $  such that $ \lambda_1\alpha_{j_1}(\mathbf{n}_1\cdot\boldsymbol{\zeta})+\lambda_2\alpha_{j_2}(\mathbf{n}_2\cdot\boldsymbol{\zeta})=\lambda\alpha_{j}(\freq) $.
There are two possibilities for that to happen.
\begin{itemize}[leftmargin=20pt]
	\item Either frequencies $ \lambda_1\,\alpha_{j_1}(\mathbf{n}_1\cdot\boldsymbol{\zeta}) $ and $ \lambda_2\,\alpha_{j_2}(\mathbf{n}_2\cdot\boldsymbol{\zeta}) $ are colinear (therefore colinear to $ \alpha_j(\freq) $), that is to say $ \mathbf{n}_1=\mathbf{n}_2=\mathbf{n} $ and $ j_1=j_2=j $. This is called \emph{self-interaction} of frequency $ \alpha_{j}(\freq) $ with itself. Note that the obtained frequency $ \lambda_1\alpha_{j_1}(\mathbf{n}_1\cdot\boldsymbol{\zeta})+\lambda_2\alpha_{j_2}(\mathbf{n}_2\cdot\boldsymbol{\zeta}) $ is then always real characteristic.
	\item Or frequencies $ \lambda_1\,\alpha_{j_1}(\mathbf{n}_1\cdot\boldsymbol{\zeta}) $ and $ \lambda_2\,\alpha_{j_2}(\mathbf{n}_2\cdot\boldsymbol{\zeta}) $  are noncolinear, in which case a true resonance in the sense of Assump\-tion \ref{hypothese ensemble frequences} occurs, namely \eqref{eq hyp res phi psi nu 1} or \eqref{eq hyp res phi psi nu 2}. For example, if $ \alpha_j(\freq)=\psi_2 $ (i.e. if $ \mathbf{n}=(0,1) $ and $ j=2 $), then according to Assumption \ref{hypothese ensemble frequences}, it implies that $ \lambda=k\lambda_{\psi} $ for some $ k\in\Z^* $ and, up to a permutation, $ \mathbf{n}_1=(1,0) $, $ j_1=3 $, $ \lambda_1=-k\lambda_{\psi} $ and $ \mathbf{n}_2=(-\lambda_{\phi},-\lambda_{\psi}) $, $ j_2=2 $, $ \lambda_2=-k $.
\end{itemize}
Recall that, for a frequency $ \lambda\,\zeta=\lambda\,\mathbf{n}\cdot\boldsymbol{\zeta}\in\F_{b}\privede{0} $ with $ \mathbf{n}\in\B_{\Z^2} $ and $ \lambda\in\Z^* $, we alternate from the representations $ \lambda\,\zeta $ and $ (\lambda,\mathbf{n}) $, so we shall denote
\begin{equation*}
	\sigma^1_{\zeta,j,\lambda}:=\sigma^1_{\mathbf{n},j,\lambda},\quad \forall j=1,2,3,\forall \lambda\in\Z^*.
\end{equation*}

According to the previous analysis, we can now write the system satisfied by the leading profiles. For example, for $ \psi_2 $ which is involved in resonance \eqref{eq hyp res phi psi nu 2}, multiplying equation \eqref{eq derivation L U1 + M U1 2} for $ \mathbf{n}=(0,1) $ and $ j=2 $ by the vector $ \ell_{\psi,2} $ cancels the first term of \eqref{eq derivation L U1 + M U1 2}, according to \eqref{eq relation l zeta j L alpha j}. Thus we obtain, for $ \lambda\in\Z^* $,
\begin{multline*}
	\ell_{\psi,2}\,L(0,\partial_z)\,r_{\psi,2}\,\sigma^1_{\psi,2,\lambda}+\ell_{\psi,2}\,L_1\big(r_{\psi,2},\psi_2\big)\,r_{\psi,2}\sum_{\lambda_1+\lambda_2=\lambda}i\lambda_2\,\sigma^1_{\psi,2,\lambda_1}\,\sigma^1_{\psi,2,\lambda_2}\\
	+\indicatrice_{\lambda=k\lambda_{\psi}}	\ell_{\psi,2}\Big\{L_1\big(r_{\phi,3},-\nu_2\big)\,r_{\nu,2}+L_1\big(r_{\nu,2},-\lambda_{\phi}\phi_3\big)\,r_{\phi,3}\Big\}\,ik\,\sigma^1_{\phi,3,-k\lambda_{\phi}}\,\sigma^1_{\nu,2,-k}=0.
\end{multline*}
In the left-hand side of the previous equation, the first term, the transport one, corresponds to the second term of the left-hand side of \eqref{eq derivation L U1 + M U1 2}, the second one, the Burgers type term corresponds to the self-interaction part of the third term of the left-hand side of  \eqref{eq derivation L U1 + M U1 2}, while the last one, the resonant term, corresponds to the resonance part of the third term of the left-hand side of  \eqref{eq derivation L U1 + M U1 2}. This splitting between transport, self-interaction and resonance terms can be generalized to any frequency. For $ \zeta=\phi,\psi,\nu $, $ j=1,2,3 $ and $ \lambda\in\Z^* $, we have
\begin{subequations}\label{eq evolution sigma 1}
\begin{multline}\label{eq evolution sigma 1 phi psi nu}
	\color{altblue}	X_{\alpha_j(\zeta)}\,\sigma^1_{\zeta,j,\lambda}+D_{\zeta,j}\sum_{\lambda_1+\lambda_2=\lambda}i\lambda_2\,\sigma^1_{\zeta,j,\lambda_1}\,\sigma^1_{\zeta,j,\lambda_2}\\	\color{altblue}
	+\indicatrice_{\lambda=k\lambda_{\zeta}}\sum_{\substack{(\zeta_1,\zeta_2,j_1,j_2)\\\in\mathcal{R}(\zeta,j)}}J^{\zeta_1,j_1}_{\zeta_2,j_2}\,ik\,\sigma^1_{\zeta_1,j_1,-k\lambda_{\zeta_1}}\,\sigma^1_{\zeta_2,j_2,-k\lambda_{\zeta_2}}=0,
\end{multline}
\color{black}
and for other frequencies $ \zeta\in\F_{b}\privede{0,\phi,\psi,\nu} $, $ j\in\mathcal{C}(\zeta) $ and $ \lambda\in\Z^* $,
\begin{equation}\label{eq evolution sigma 1 autre}
		X_{\alpha_j(\zeta)}\,\sigma^1_{\zeta,j,\lambda}+D_{\zeta,j}\sum_{\lambda_1+\lambda_2=\lambda}i\lambda_2\,\sigma^1_{\zeta,j,\lambda_1}\,\sigma^1_{\zeta,j,\lambda_2}=0.
\end{equation}
\end{subequations}
We have denoted, for $ \zeta\in\F_{b} $ and $ j\in\mathcal{C}(\zeta) $, the vector field $ X_{\alpha_j(\zeta)} $ and the self-interaction coefficient $ D_{\zeta,j} $ as %(recall that we have $ \alpha_j(\zeta)=\zeta_j $ for $ \zeta=\phi,\psi,\nu $ and $ j=1,2,3 $),
\begin{subequations}\label{eq prof princ def X D J}
	\begin{align}
	X_{\alpha_j(\zeta)}&:=\ell_{\zeta,j}\,L(0,\partial_z)\,r_{\zeta,j},\qquad D_{\zeta,j}:=\ell_{\zeta,j}\,L_1\big(r_{\zeta,j},\alpha_j(\zeta)\big)\,r_{\zeta,j} ,\label{eq def X D}
	\intertext{and, for $ (\zeta_1,\zeta_2,j_1,j_2)\in\mathcal{R}(\zeta,j) $ (in the case $ \zeta=\phi,\psi,\nu $), the resonance coefficient $ J_{\zeta_2,j_2}^{\zeta_1,j_1} $ as}
	 J^{\zeta_1,j_1}_{\zeta_2,j_2}&:=\ell_{\zeta,j}\, L_1\big(r_{\zeta_1,j_1},\lambda_{\zeta_2}\,\alpha_{j_2}(\zeta_2)\big)\,r_{\zeta_2,j_2}
	 %+L_1\big(r_{\zeta_2,j_2},\lambda_{\zeta_1}\,\alpha_{j_1}(\zeta_1)\big)\,r_{\zeta_1,j_1}\Big\rbrace
	 .\label{eq def J}
\end{align}
\end{subequations}
According to the Lax Lemma \ref{lemme Lax}, the operator $ \ell_{\zeta,j}\,L(0,\partial_z)\,r_{\zeta,j} $ of \eqref{eq def X D} is equal to the vector field \eqref{eq champ vecteur X_alpha} which has already been denoted by $ X_{\alpha_j(\zeta)} $, so the notation is coherent. In the following, for $ \zeta\in\F_b\privede{0} $ and $ j\in\mathcal{C}(\zeta) $, we denote by $ X_{\zeta,j} $ the vector field $ X_{\alpha_j(\zeta)} $ and $ \mathbf{v}_{\zeta,j}:=\mathbf{v}_{\alpha_{j}(\zeta)} $ the velocity vector associated with it. 

\bigskip

For all frequencies $ \zeta=\mathbf{n}\cdot\boldsymbol{\zeta} $ with $ \mathbf{n}\in\B_{\Z^2} $ except for $ \zeta=\phi,\psi,\nu $, and for all $ j\in\mathcal{C}(\zeta) $, since the frequency $ \alpha_j(\zeta) $ does not occur in any resonance, if we denote by $ \sigma^1_{\zeta,j} $ the series
\begin{equation*}
	\sigma^1_{\zeta,j}(z,\Theta):=\sum_{\lambda\in\Z^*}\sigma^1_{\zeta,j,\lambda}(z)\,e^{i\lambda\Theta},
\end{equation*}
then, according to \eqref{eq evolution sigma 1 autre}, we have the Burgers type equation
\begin{equation}\label{eq evolution S hors cas part}
	\color{altpurple}X_{\zeta,j}\,\sigma^1_{\zeta,j}+D_{\zeta,j}\,\sigma^1_{\zeta,j}\partial_{\Theta}\sigma^1_{\zeta,j}=0,
\end{equation}
along with the initial condition $ \color{altpurple}(\sigma^1_{\zeta,j})_{t=0}=0 $\color{black}, which is a nonlinear scalar transport equation in the half space $ \Omega_T $. If the frequency $ \alpha_j(\zeta) $ is outgoing, i.e. if the last component of $ \mathbf{v}_{\zeta,j} $ is positive, there is no need for a boundary condition, so we deduce from \eqref{eq evolution S hors cas part} that $ \sigma^1_{\zeta,j} $ is zero. Therefore, for $ \lambda\in\Z^* $, we have $ \sigma^1_{\zeta,j,\lambda}=0 $. The same arguments can be applied for the outgoing frequency $ \phi_2 $ since there is no resonance for it.  

If $ \alpha_j(\zeta) $ is incoming we need in this case a boundary condition to determine the trace of $ \sigma^1_{\zeta,k}(z,\Theta) $ at $ x_d=0 $. From boundary condition \eqref{eq cascade bord U12} for the frequency $ \lambda\zeta $, $ \lambda\in\Z^* $, writing of $ U_1 $ \eqref{eq ecriture U 1} and the fact that all outgoing frequencies for $ \zeta $ have been proved to be zero, we deduce that 
\begin{equation}\label{eq cond bord sigma 1 hors cas part}
	B\sum_{j\in\mathcal{I}(\zeta)}\big(\sigma^1_{\zeta,j,\lambda}\big)_{|x_d=0}r_{\zeta,j}+B\,\Pi^e(\zeta)\,\big(U^{1,\ev}_{\lambda\zeta}\big)_{|x_d,\chi_d=0}=0.
\end{equation}
Since $ \lambda\zeta\in\F_b\setminus\Upsilon $ (because for now we consider $ \zeta\neq \phi,\psi $) and, according to decomposition \eqref{eq decomp E_-(zeta)} of the stable subspace $ E_-(\zeta) $, the vector in \eqref{eq cond bord sigma 1 hors cas part} to which matrix $ B $ applies lies in $ E_-(\lambda\zeta) $, on which $ B $ is invertible, according to Assumption \ref{hypothese weak lopatinskii}. Therefore we deduce from the previous equation \eqref{eq cond bord sigma 1 hors cas part}, using vectors $ \ell_{\zeta,j} $, that
\color{altpurple}
\begin{equation}\label{eq prof princ bord autre}
	\big(\sigma^1_{\zeta,j}\big)_{|x_d=0}=0,\quad \forall j\in\mathcal{I}(\zeta).
\end{equation} 
\color{black}
Along with \eqref{eq evolution S hors cas part} we get $ \sigma^1_{\zeta,j,\lambda}=0 $ for all $ \lambda\in\Z^* $. We have therefore proven that for all frequencies $ \zeta=\mathbf{n}\cdot\boldsymbol{\zeta} $ with $ \mathbf{n}\in\B_{\Z^2} $ except for $ \zeta=\phi,\psi,\nu $, and for all $ j\in\mathcal{C}(\zeta) $ and $ \lambda\in\Z^* $, we have
\begin{equation*}
	\sigma^1_{\zeta,j,\lambda}=0.
\end{equation*}
In the same way we deduce from \eqref{eq cond bord sigma 1 hors cas part}, using decomposition \eqref{eq decomp E_-(zeta)},
\begin{equation*}
	\big(U^{1,\ev}_{\lambda\zeta}\big)_{|x_d,\chi_d=0}=0,
\end{equation*}
so, according to \eqref{eq ecriture U 1}, we can set the evanescent part $ U^{1,\ev} $ of $ U^1 $ to be zero.

We now need to determine boundary conditions for the incoming frequencies $ \phi_j $, $ \psi_j $ and $ \nu_j $ for $ j=1,3 $ as well.
For the frequency $ \nu $ we obtain, in the same fashion as before, since $ \nu $ is in the hyperbolic region $ \mathcal{H} $, for $ \lambda\in\Z^* $,
\begin{equation*}
	B\,\Big(\big(\sigma_{\nu,1,\lambda}^1\big)_{|x_d=0}\,r_{\nu,1}+\big(\sigma_{\nu,2,\lambda}^1\big)_{|x_d=0}\,r_{\nu,2}+\big(\sigma_{\nu,3,\lambda}^1\big)_{|x_d=0}\,r_{\nu,3}\Big)=0,
\end{equation*}
so, for $ j=1,3 $, according to the normalization \eqref{eq relation l zeta j normalisation},
\begin{equation*}
	\big(\sigma_{\nu,j,\lambda}^1\big)_{|x_d=0}=-\ell_{\nu_j}\,A_d(0)\, \big(B_{|E_-(\nu)}\big)^{-1}\,B\big(\sigma_{\nu,2,\lambda}^1\big)_{|x_d=0}\,r_{\nu,2}.
\end{equation*}
We denote by $ \mu_{\nu,j} $, for $ j=1,3 $, the coefficient
\begin{equation}\label{eq prof princ def mu nu}
	\mu_{\nu,j}:=-\ell_{\nu,j}\,A_d(0)\,\big(B_{|E_-(\nu)}\big)^{-1}\,B\,r_{\nu,2},
\end{equation}
%and $ \mu_{\nu,2}:=1 $,
so that, for $ j=1,3 $, we have
\begin{equation}\label{eq prof princ bord nu}
	\color{altgreen}\big(\sigma_{\nu,j,\lambda}^1\big)_{|x_d=0}\,r_{\nu,j}=\mu_{\nu,j}\,\big(\sigma_{\nu,2,\lambda}^1\big)_{|x_d=0}\,r_{\nu,j}.
\end{equation}
\color{black}

For $ \phi $ we have, in a similar manner, since $ \sigma^1_{\phi,2,\lambda} $ is zero for $ \lambda\in\Z^* $, 
\begin{equation*}
	B\,\Big(\big(\sigma_{\phi,1,\lambda}^1\big)_{|x_d=0}\,r_{\phi,1}+\big(\sigma_{\phi,3,\lambda}^1\big)_{|x_d=0}\,r_{\phi,3}\Big)=0,
\end{equation*}
for every $ \lambda $ in $ \Z^* $, so, according to \eqref{eq decomp E_-(zeta)}, the vector in factor of $ B $ in the left-hand side belongs to $ \ker B \cap E_-(\phi) $. But since $ \phi $ is in $ \Upsilon $, the latter space is of dimension $ 1 $ and reads $ \Span e_{\phi} $, so there exists a scalar function $ a_{\phi,\lambda} $ of $ \omega_T $ such that, 
\begin{equation*}
	\big(\sigma_{\phi,1,\lambda}^1\big)_{|x_d=0}\,r_{\phi,1}+\big(\sigma_{\phi,3,\lambda}^1\big)_{|x_d=0}\,r_{\phi,3}=a_{\phi,\lambda}^1\,e_{\phi}.
\end{equation*}
Therefore, according to decomposition \eqref{eq decomp e zeta} of $ e_{\zeta} $, for $ j=1,3 $,
\begin{equation}\label{eq prof princ bord phi}
\color{altred}	\big(\sigma_{\phi,j,\lambda}^1\big)_{|x_d=0}\,r_{\phi,j}=a_{\phi,\lambda}^1\,e_{\phi,j}.
\end{equation}
\color{black}

Finally the case of  the phase $ \psi $ gather the two previous ones. We have, for $ \lambda\in\Z^* $,
\begin{equation*}
	B\,\Big(\big(\sigma_{\psi,1,\lambda}^1\big)_{|x_d=0}\,r_{\psi,1}+\big(\sigma_{\psi,2,\lambda}^1\big)_{|x_d=0}\,r_{\psi,2}+\big(\sigma_{\psi,3,\lambda}^1\big)_{|x_d=0}\,r_{\psi,3}\Big)=0.
\end{equation*}
In particular, because of \eqref{eq decomp E_-(zeta)}, it implies 
\begin{equation*}
	B\big(\sigma_{\psi,2,\lambda}^1\big)_{|x_d=0}\,r_{\psi,2}\in \Ima B_{|E_-(\psi)}=\Big(\ker \,^tB_{|E_-(\psi)}\Big)^{\perp}.
\end{equation*}
Therefore, according to the definition of $ b_{\psi} $, the following necessary condition follows:
\begin{equation*}
	b_{\psi}\cdot B\big(\sigma_{\psi,2,\lambda}^1\big)_{|x_d=0}\,r_{\psi,2}=0.
\end{equation*}
But since the scalar $ b_{\psi}\cdot B\,r_{\psi,2} $ is not zero\footnote{the linear form $ b_{\psi}\cdot B $ is not uniformly zero and is already zero on two of the three vectors $ r_{\psi,1} $, $ r_{\psi,2} $ and $ r_{\psi,3} $ constituting a basis of $ \C^3 $, so cannot be on the third one}, we necessarily have
\begin{equation}\label{eq prof princ cond psi 2 zero}
	\big(\sigma_{\psi,2,\lambda}^1\big)_{|x_d=0}=0.
\end{equation}
When it is satisfied we can write, in the same way as for $ \phi $, for $ j=1,3 $, 
\begin{equation}\label{eq prof princ bord psi}
	\color{altorange2}\big(\sigma_{\psi,j,\lambda}^1\big)_{|x_d=0}\,r_{\psi,j}=a_{\psi,\lambda}^1\,e_{\psi,j},
\end{equation}
\color{black}
with $ a_{\psi,\lambda} $ a scalar function of $ \omega_T $.

At this point we have obtained an constant coefficient equation for $ U_1^* $, and transport equations \eqref{eq evolution sigma 1 phi psi nu} and \eqref{eq evolution S hors cas part}, associated with (when incoming) boundary conditions \eqref{eq prof princ bord autre}, \eqref{eq prof princ bord nu}, \eqref{eq prof princ bord phi} and \eqref{eq prof princ bord psi}, but the last two ones are expressed through scalar functions $ a^1_{\zeta,\lambda} $ which are still to be determined, so the system is not closed at this stage. Also note that condition \eqref{eq prof princ cond psi 2 zero} might raise an issue of over-determination of the system.

To determine the equations satisfied by coefficients $ a^1_{\phi,\lambda} $ and $ a^1_{\psi,\lambda} $, we need to study the nonpolarized part of the first corrector $ U_2 $.

\subsubsection{Nonpolarized part of the first corrector}

For the first corrector we no longer have a polarization condition such as \eqref{eq cascade int U1}, so noncharacteristic modes may appear through quadratic interaction of characteristic modes. Thus the first corrector $ U_2 $ reads
\begin{align*}%\label{eq ecriture U 2}
	U_2(z,\theta,\chi_d)&=U^*_2(z)+\sum_{\mathbf{n}\in\B_{\Z^2}}\sum_{j\in\mathcal{C}(\freq)}\sum_{\lambda\in\Z^*}U^{2,\osc}_{\mathbf{n},j,\lambda}(z)\,e^{i\,\lambda\,\mathbf{n}\cdot\theta}\,e^{i\,\lambda\,\xi_j(\freq)\,\chi_d} \\\nonumber
	&\quad+\sum_{\mathbf{n}\in\B_{\Z^2}}\sum_{\lambda\in\Z^*}U^{2,\ev}_{\mathbf{n},\lambda}(z,\chi_d)\,e^{i\,\lambda\,\mathbf{n}\cdot\theta}+U^{2,\nc}(z,\theta,\chi_d),
\end{align*}
where $ U^*_2 $ is the mean value of $ U_2 $ and $ U^{2,\nc} $ corresponds to the noncharacteristic modes. According to \eqref{eq cascade int U2}, the noncharacteristic part $ U^{2,\nc} $ satisfies (since there are only characteristic modes in $ L(0,\partial_z)\,U_1 $), 
\begin{multline}\label{eq 1cor expr U 2 nc}
	\mathcal{L}(\partial_{\theta},\partial_{\chi_d})\,U^{2,\nc}=-\sum_{\substack{(\zeta_1,\zeta_2,j_1,j_2,\\\lambda_1,\lambda_2)\in\mathcal{NR}}}L_1(r_{\zeta_1,j_1},\alpha_{j_2}(\zeta_2))\,r_{\zeta_2,j_2}\,i\lambda_2\,\sigma_{\zeta_1,j_1,\lambda_1}^1\,\sigma_{\zeta_2,j_2,\lambda_2}^1\,\\
	e^{i(\lambda_1\mathbf{n}_1+\lambda_2\mathbf{n}_2)\cdot\theta}\,e^{i(\lambda_1\xi_{j_1}(\zeta_1)+\lambda_2\xi_{j_2}(\zeta_2))\chi_d},
\end{multline}
where $ \mathcal{NR} $ denotes the set of 6-tuples $ (\zeta_1,\zeta_2,j_1,j_2,\lambda_1,\lambda_2) $ such that the frequency
\begin{equation*}
	\lambda_1\,\alpha_{j_1}(\zeta_1)+\lambda_2\,\alpha_{j_2}(\zeta_2)
\end{equation*}
is noncharacteristic (which is such that there is no resonance). Note that in \eqref{eq 1cor expr U 2 nc}, only occur the boundary frequencies $ \phi,\psi $ and $ \nu $, since for all the others, the first profile is zero. Since all frequencies in $ U^{2,\nc} $ are noncharacteristic, equation \eqref{eq 1cor expr U 2 nc} determine $ U^{2,\nc} $ totally. Indeed, for each mode of noncharacteristic frequency $ \lambda_1\,\alpha_{j_1}(\zeta_1)+\lambda_2\,\alpha_{j_2}(\zeta_2) $ with  $ (\zeta_1,\zeta_2,j_1,j_2,\lambda_1,\lambda_2) \in\mathcal{NR} $, the operator $ \mathcal{L}(\partial_{\theta},\partial_{\chi_d}) $ reads $ i\,L\big(0,\lambda_1\alpha_{j_1}(\zeta_1)+\lambda_2\alpha_{j_2}(\zeta_2)\big) $, which is an invertible matrix. 

Since for every boundary frequency $ \zeta=\mathbf{n}\cdot\boldsymbol{\zeta}\in\F_b\privede{0,\phi,\psi,\nu} $, the oscillating profile $ \sigma^1_{\zeta,j,\lambda} $ is zero for $ j\in\mathcal{C}(\zeta) $ and $ \lambda\in\Z^* $ and since there is no resonances generating theses frequencies, according to \eqref{eq cascade int U2}, the profile $ U^{2,\osc}_{\mathbf{n},j,\lambda} $ satisfies
\begin{equation*}
	L\big(0,\alpha_j(\zeta)\big)\,U^{2,\osc}_{\mathbf{n},j,\lambda}=0,
\end{equation*}
so it is polarized, and we denote by $ \sigma^2_{\zeta,j,\lambda} $ the scalar function of $ \Omega_T $ such that 
\begin{equation*}
	U^{2,\osc}_{\mathbf{n},j,\lambda}=\sigma^2_{\zeta,j,\lambda}\,r_{\zeta,j}.
\end{equation*}

With the same arguments as for the leading profile, we get the following polarization condition for the evanescent part: $ U^{2,\ev}_{0}=0 $, and, for $ \mathbf{n}\in\Z^2\privede{0} $, 
\begin{equation*}
	U^{2,\ev}_{\mathbf{n}}(z,\chi_d)=e^{\chi_d\mathcal{A}(\mathbf{n}\cdot\boldsymbol{\zeta})}\,\Pi^e(\mathbf{n}\cdot\boldsymbol{\zeta})\,U^{2,\ev}_{\mathbf{n}}(z,0).
\end{equation*}
Therefore, $ U_2 $ reads as the more precise following way,
\begin{align}\label{eq ecriture U 2}
	U_2(z,\theta,\chi_d)&=U^*_2(z)+\sum_{\mathbf{n}\in\B_{\Z^2}}\sum_{j\in\mathcal{C}(\freq)}\sum_{\lambda\in\Z^*}U^{2,\osc}_{\mathbf{n},j,\lambda}(z)\,e^{i\,\lambda\,\mathbf{n}\cdot\theta}\,e^{i\,\lambda\,\xi_j(\freq)\,\chi_d} \\\nonumber
	&\quad+\sum_{\mathbf{n}\in\B_{\Z^2}}\sum_{\lambda\in\Z^*}e^{\chi_d\mathcal{A}(\mathbf{n}\cdot\boldsymbol{\zeta})}\,\Pi^e(\mathbf{n}\cdot\boldsymbol{\zeta})\,U^{2,\ev}_{\mathbf{n}}(z,0)\,e^{i\,\lambda\,\mathbf{n}\cdot\theta}+U^{2,\nc}(z,\theta,\chi_d).
\end{align}

\bigskip

Writing down boundary equation \eqref{eq cascade bord U12} for $ U_2 $ will lead to equations on boundary terms $ a_{\phi}^1 $ and $ a_{\psi}^1 $. Thus we need to determine the nonpolarized part of the amplitudes associated with frequencies lifted from $ \phi,\psi,\nu $. For $ \zeta=\mathbf{n}\cdot\boldsymbol{\zeta}\in\ensemble{\phi,\psi,\nu} $, $ j=1,2,3 $, and $ \lambda\in\Z^* $, from polarization equation \eqref{eq cascade int U1}, we get, (using the notation $ U^{2,\osc}_{\zeta,j,\lambda}:=U^{2,\osc}_{\mathbf{n},j,\lambda} $),
\begin{align*}
	i\,L\big(0,\alpha_j(\zeta)\big)\,U^{2,\osc}_{\zeta,j,\lambda}=
	&-L(0,\partial_z)\,\sigma_{\zeta,j,\lambda}^1\,r_{\zeta,j}-L_1\big(r_{\zeta,j},\alpha_j(\zeta)\big)\,r_{\zeta,j}\sum_{\lambda_1+\lambda_2=\lambda}i\lambda_2\,\sigma_{\zeta,j,\lambda_1}^1\,\sigma_{\zeta,j,\lambda_2}^1\\
	&-\indicatrice_{\lambda=k\lambda_{\zeta}}\sum_{\substack{(\zeta_1,\zeta_2,j_1,j_2)\\\in\mathcal{R}(\zeta,j)}}ik\,\Big\lbrace L_1\big(r_{\zeta_1,j_1},\lambda_{\zeta_2}\,\alpha_{j_2}(\zeta_2)\big)\,r_{\zeta_2,j_2}\\
	&\qquad\qquad+L_1\big(r_{\zeta_2,j_2},\lambda_{\zeta_1}\,\alpha_{j_1}(\zeta_1)\big)\,r_{\zeta_1,j_1}\Big\rbrace\,
	%&\qquad\qquad
	\sigma^1_{\zeta_1,j_1,-k\lambda_{\zeta_1}}\,\sigma^1_{\zeta_2,j_2,-k\lambda_{\zeta_2}}.
\end{align*}
Then we multiply this equation on the left by the partial inverse $ R_{\zeta,j} $ to obtain, according to relation \eqref{eq relation R zeta j P},
\begin{align}\label{eq 1cor nonpola part U 2 zeta}
	i\lambda\,\big(I-&P_{\zeta,j}\big)\,U^{2,\osc}_{\zeta,j,\lambda}=\\[5pt]\nonumber
	&-R_{\zeta,j}\,L(0,\partial_z)\,\sigma_{\zeta,j,\lambda}^1\,r_{\zeta,j}-R_{\zeta,j}\,L_1\big(r_{\zeta,j},\alpha_j(\zeta)\big)\,r_{\zeta,j}\sum_{\lambda_1+\lambda_2=\lambda}i\lambda_2\,\sigma_{\zeta,j,\lambda_1}^1\,\sigma_{\zeta,j,\lambda_2}^1\\\nonumber
	&-R_{\zeta,j}\,\indicatrice_{\lambda=k\lambda_{\zeta}}\sum_{\substack{(\zeta_1,\zeta_2,j_1,j_2)\\\in\mathcal{R}(\zeta,j)}}ik\,\Big\lbrace L_1\big(r_{\zeta_1,j_1},\lambda_{\zeta_2}\,\alpha_{j_2}(\zeta_2)\big)\,r_{\zeta_2,j_2}\\\nonumber
	&\qquad\qquad+L_1\big(r_{\zeta_2,j_2},\lambda_{\zeta_1}\,\alpha_{j_1}(\zeta_1)\big)\,r_{\zeta_1,j_1}\Big\rbrace\,
	%&\qquad\qquad
	\sigma^1_{\zeta_1,j_1,-k\lambda_{\zeta_1}}\,\sigma^1_{\zeta_2,j_2,-k\lambda_{\zeta_2}}.
\end{align}
We now write the boundary conditions for the first corrector $ U_2 $, for the frequencies $ \phi $ and $ \psi $.

\bigskip

Note that since $ \phi $ is in the hyperbolic region, the stable elliptic component $ E^e_-(\phi) $ is zero, so $ \Pi^e(\phi) $ is also zero. Therefore, boundary condition \eqref{eq cascade bord U12} for mode $ \phi $ reads, according to equation \eqref{eq ecriture U 2},
\begin{align}\label{eq 1cor boundary cond phi}
	B\,P_{\phi,1}\,\big(U^2_{\phi,1,\lambda}\big)_{|x_d,\chi_d=0}+B\,P_{\phi,3}\,\big(U^2_{\phi,3,\lambda}\big)_{|x_d,\chi_d=0}\\\nonumber
	+B\,(I-P_{\phi,1})\,\big(U^2_{\phi,1,\lambda}\big)_{|x_d,\chi_d=0}+B\,(I-P_{\phi,3})\,\big(U^2_{\phi,3,\lambda}\big)_{|x_d,\chi_d=0}\\\nonumber
	+B\,\big(U^2_{\phi,2,\lambda}\big)_{|x_d,\chi_d=0}+B\,\big(U^{2,\nc}_{\phi,\lambda}\big)_{|x_d,\chi_d=0}&=G_{\lambda},
\end{align}
where we have expanded the source term $ G $ in Fourier series as
\begin{equation*}
	G(z',\Theta)=\sum_{\lambda\in\Z}G_{\lambda}(z')\,e^{i\lambda\Theta},
\end{equation*}
and where we have denoted by $ U^{2,\nc}_{\phi,\lambda} $ the sum of all the terms of $ U^{2,\nc} $ of which the trace on the boundary of the associated frequency is equal to $ \lambda\phi $, namely,
\begin{multline*}
	U^{2,\nc}_{\phi,\lambda}=-\sum_{\substack{(\zeta_1,\zeta_2,j_1,j_2,\\\lambda_1,\lambda_2)\in\mathcal{NR}\\\lambda_1\zeta_1+\lambda_2\zeta_2=\lambda\phi}}L\big(0,\lambda_1\alpha_{j_1}(\zeta_1)+\lambda_2\alpha_{j_2}(\zeta_2)\big)^{-1}\,L_1(r_{\zeta_1,j_1},\alpha_{j_2}(\zeta_2))\,r_{\zeta_2,j_2}\,\lambda_2\,\sigma^1_{\zeta_1,j_1,\lambda_1}\,\sigma^1_{\zeta_2,j_2,\lambda_2}\,\\
	e^{i(\lambda_1\mathbf{n}_1+\lambda_2\mathbf{n}_2)\cdot\theta}\,e^{i(\lambda_1\xi_{j_1}(\zeta_1)+\lambda_2\xi_{j_2}(\zeta_2))\chi_d}.
\end{multline*}
We investigate now which frequencies occur in this sum. If we denote by, for $ i=1,2 $, $ \zeta_i=\mathfrak{m}_i\phi+\mathfrak{n}_i\psi $ with $ (\mathfrak{m}_i,\mathfrak{n}_i)\in\B_{\Z^2} $, since there are only frequencies lifted from $ \phi,\psi,\nu $ in $ \mathcal{NR} $, we necessarily have ($ \mathfrak{m}_i=1 $ and $ \mathfrak{n}_i=0 $) or ($ \mathfrak{m}_i=0 $ and $ \mathfrak{n}_i=1 $) or ($ \mathfrak{m}_i=\lambda_{\phi} $ and $ \mathfrak{n}_i=\lambda_{\psi} $). In this notation, the condition $ \lambda_1\zeta_1+\lambda_2\zeta_2=\lambda\phi $ is equivalent to 
\begin{equation*}
	\left\lbrace\begin{array}{l}
		\lambda_1\,\mathfrak{m}_1+\lambda_2\,\mathfrak{m}_2=\lambda\\
		\lambda_1\,\mathfrak{n}_1+\lambda_2\,\mathfrak{n}_2=0,
	\end{array}\right.
\end{equation*}
and using that $ \lambda_{\phi},\lambda_{\psi} $ are coprime integers, we find that this system admits the following solutions
\begin{equation*}
	\left\lbrace\begin{array}{l}
		(\mathfrak{m}_1,\mathfrak{n}_1)=(1,0)\\
		(\mathfrak{m}_2,\mathfrak{n}_2)=(1,0)\\
		\lambda_1+\lambda_2=\lambda
	\end{array}\right.,
	\quad 
	\left\lbrace\begin{array}{l}
		(\mathfrak{m}_1,\mathfrak{n}_1)=(0,1)\\
		(\mathfrak{m}_2,\mathfrak{n}_2)=(\lambda_{\phi},\lambda_{\psi})\\
		(\lambda,\lambda_1,\lambda_2)=k\,(\lambda_{\phi},-\lambda_{\psi},1)
	\end{array}\right.,
\mbox{ and }
\left\lbrace\begin{array}{l}
	(\mathfrak{m}_1,\mathfrak{n}_1)=(\lambda_{\phi},\lambda_{\psi})\\
	(\mathfrak{m}_2,\mathfrak{n}_2)=(0,1)\\
	(\lambda,\lambda_1,\lambda_2)=k\,(\lambda_{\phi}1,-\lambda_{\psi})
\end{array}\right..
\end{equation*}
Selecting only 6-tuples of $ \mathcal{NR} $, we obtain that $ U^{2,\nc}_{\phi,\lambda} $ is equal to
\begin{align}\label{eq 1cor U 2 nc phi}
	U^{2,\nc}_{\phi,\lambda}&=\indicatrice_{\lambda=k\lambda_{\phi}}\sum_{\substack{j_1,j_2=1,2,3 \\ (j_1,j_2)\neq(2,1), (2,2)}} L\big(0,k\nu_{j_1}-\lambda_{\psi}k\psi_{j_2}\big)^{-1}\,\big\lbrace \lambda_{\psi}\,L_1(r_{\nu,j_1},\psi_{j_2})\,r_{\psi,j_2}\\\nonumber
	&\qquad\qquad-L_1(r_{\psi,j_2},\nu_{j_1})\,r_{\nu,j_1}\big\rbrace
	\,k\,\sigma^1_{\nu,j_1,k}\,\sigma^1_{\psi,j_2,-\lambda_{\psi}k}\,
	e^{ik\lambda_{\phi}\theta_1}\,e^{i(k\xi_{j_1}(\nu)-\lambda_{\psi}k\xi_{j_2}(\psi))\chi_d}\\\nonumber
	&\quad-\sum_{\lambda_1+\lambda_2=\lambda} L\big(0,\lambda_1\phi_1+\lambda_2\phi_3\big)^{-1}\big\lbrace \lambda_1 \,L_1(r_{\phi,3},\phi_{1})\,r_{\phi,1}\\\nonumber
	&\qquad\qquad+ \lambda_2\, L_1(r_{\phi,1},\phi_{3})\,r_{\phi,3}\big\rbrace\,\sigma^1_{\phi,1,\lambda_1}\,\sigma^1_{\phi,3,\lambda_2}\,
	e^{i\lambda\,\theta_1}\,e^{i(\lambda_1\xi_1(\phi)+\lambda_2\xi_3(\phi))\chi_d}.
	\end{align}
Since the vectors $ P_{\phi,1}\,\big(U^2_{\phi,1,\lambda}\big)_{|x_d,\chi_d=0} $ and $ P_{\phi,3}\,\big(U^2_{\phi,3,\lambda}\big)_{|x_d,\chi_d=0} $ in \eqref{eq 1cor boundary cond phi} are respectively in $ E_-^1(\phi) $ and $ E_-^3(\phi) $, by definition \eqref{eq def b zeta} of $ b_{\phi} $, we have
\begin{equation*}
	b_{\phi}\cdot B\,P_{\phi,1}\,\big(U^2_{\phi,1,\lambda}\big)_{|x_d,\chi_d=0} = b_{\phi}\cdot B\, P_{\phi,3}\,\big(U^2_{\phi,3,\lambda}\big)_{|x_d,\chi_d=0} =0.
\end{equation*}
So if we take the scalar product of $ b_{\phi} $ with equality \eqref{eq 1cor boundary cond phi} multiplied by $ i\lambda $, using \eqref{eq 1cor nonpola part U 2 zeta}, \eqref{eq 1cor U 2 nc phi} and the boundary conditions \eqref{eq prof princ bord phi} for the leading profile associated with $ \lambda\phi $, we get the amplitude equation
\begin{multline}\label{eq 1cor eq evol a phi}
	\color{altred}X^{\Lop}_{\phi}\,a^1_{\phi,\lambda}
	+D^{\Lop}_{\phi}\sum_{\lambda_1+\lambda_2=\lambda}i\lambda_2\,a^1_{\phi,\lambda_1}\,a^1_{\phi,\lambda_2}+i\lambda\sum_{\lambda_1+\lambda_3=\lambda}\gamma_{\phi}(\lambda_1,\lambda_3)\,a^1_{\phi,\lambda_1}\,a^1_{\phi,\lambda_3}\\[5pt]\color{altred}
	%+\indicatrice_{\lambda=k\lambda_{\phi}}\,\Gamma^{\phi,k}_{1}\,ik\,\big(\sigma^1_{\psi,2,-k\lambda_{\psi}}\big)_{|x_d=0}\,\big(\sigma^1_{\nu,2,-\lambda}\big)_{|x_d=0}
	+\indicatrice_{\lambda=k\lambda_{\phi}}\,\Gamma^{\phi,k}_{1}\,ik\,a^1_{\psi,-k\lambda_{\psi}}\,\big(\sigma^1_{\nu,2,-k}\big)_{|x_d=0}=i\lambda\,b_{\phi}\cdot B\,\big(U^2_{\phi,2,\lambda}\big)_{|x_d,\chi_d=0}-i\lambda\,b_{\phi}\cdot G_{\lambda},
\end{multline}
\color{black}
with
\begin{subequations}\label{eq 1cor eq X lop v phi etc}
\begin{align}
	X^{\Lop}_{\phi}&:=b_{\phi}\cdot B\,\Big(R_{\phi,1}\,L(0,\partial_z)\,e_{\phi,1}+R_{\phi,3}\,L(0,\partial_z)\,e_{\phi,3}\Big),\\
	%=&\beta_{\phi}\Big(\partial_t+\sum_{j=1}^{d-1}\partial_{\eta_j}\sigma(\phi)\,\partial_j\Big),\\
	% Opérateur transport
	D^{\Lop}_{\phi}&:=b_{\phi}\cdot B\,\Big(R_{\phi,1}\,L_1\big(e_{\phi,1},\phi_1\big)\,e_{\phi,1}+R_{\phi,3}\,L_1\big(e_{\phi,3},\phi_3\big)\,e_{\phi,3}\Big)\label{eq 1cor eq v phi},\\[5pt]\label{eq 1cor eq gamma phi lambda 1 lambda 3}
	% gamma($\lambda_1,\lamda_2)
	\gamma_{\phi}(\lambda_1,\lambda_3)&:=b_{\phi}\cdot B\,L\big(\lambda_1\phi_1+\lambda_3\phi_3\big)^{-1}\big\lbrace \lambda_1 \,L_1(e_{\phi,3},\phi_{1})\,e_{\phi,1}+ \lambda_3\, L_1(e_{\phi,1},\phi_{3})\,e_{\phi,3}\big\rbrace,\\[5pt]
	% Gamma 1
%	\Gamma^{\phi,k}_1&:=b_{\phi}\cdot B\,R_{\phi,1}\,\Big(
%	 \mu_{\psi,1}\,L_1(r_{\psi,1},-\nu_2)\,r_{\nu,2}+\mu_{\psi,1}\,L_1(r_{\nu,2},-\lambda_{\psi}\psi_1)\,r_{\psi,1}\\
%	&\nonumber\qquad+\mu_{\nu,1}\, L_1(r_{\psi,2},-\nu_1)\,r_{\nu,1}+\mu_{\nu,1}\,L_1(r_{\nu,1},-\lambda_{\psi}\psi_2)\,r_{\psi,2}\Big)\\
%	&\nonumber+b_{\phi}\cdot B\sum_{\substack{j_1,j_2=1,2,3 \\ (j_1,j_2)\neq(1,2), (2,1)}}\mu_{\nu,j_2}\,\mu_{\psi,j_1}\,L\big(k\nu_{j_2}-\lambda_{\psi}k\psi_{j_1}\big)^{-1}\,\Big\lbrace L_1(r_{\psi,j_1},\nu_{j_2})\,r_{\nu,j_2}\\
%	&\nonumber\qquad-L_1(r_{\nu,j_2},\psi_{j_1})\,r_{\psi,j_1}\,\lambda_{\psi}\Big\rbrace,\\
	%Gamma 2
	\Gamma^{\phi}&:=b_{\phi}\cdot B\,R_{\phi,1}\,\Big( L_1(e_{\psi,1},-\nu_2)\,r_{\nu,2}+L_1(r_{\nu,2},-\lambda_{\psi}\psi_1)\,e_{\psi,1}\Big)\label{eq 1cor eq gamma phi 1}\\
	&\nonumber\quad+\lambda_{\phi}\,b_{\phi}\cdot B\sum_{\substack{j_1=1,3,j_2=1,2,3 \\ (j_1,j_2)\neq(1,2)}}\mu_{\nu,j_2}\, L\big(\nu_{j_2}-\lambda_{\psi}\psi_{j_1}\big)^{-1}\,\Big\lbrace L_1(r_{\psi,j_1},\nu_{j_2})\,r_{\nu,j_2}\\
	&\nonumber\qquad-L_1(r_{\nu,j_2},\psi_{j_1})\,r_{\psi,j_1}\,\lambda_{\psi}\Big\rbrace.
\end{align}
\end{subequations}
Equation \eqref{eq 1cor eq evol a phi} differs from the analogous one of \cite[equation (2.19)]{CoulombelWilliams2017Mach} by the two terms $ \indicatrice_{\lambda=k\lambda_{\phi}}\,\Gamma^{\phi}\,ik\,a^1_{\psi,-k\lambda_{\psi}}\,\big(\sigma^1_{\nu,2,-k}\big)_{|x_d=0}$ and $ i\lambda\,b_{\phi}\cdot B\,\big(U^2_{\phi,2,\lambda}\big)_{|x_d,\chi_d=0} $. The first one appears here because of the resonances, and the second one wasn't in \cite{CoulombelWilliams2017Mach} because profiles  $ U^2_{\phi,2,\lambda} $ were zero\footnote{Note that here, with a more precise analysis, we could also show that every profile $ U^n_{\phi,2,\lambda} $ is zero, since frequency $ \phi_2 $ does not occur in any resonance. We choose however to not detail it, in order to simplify the whole analysis, and since $ U^n_{\psi,2,\lambda} $ will not be zero, and proving that $ U^n_{\phi,2,\lambda} $ would not simplify the solving of the equations}.
Computing $ L\big(\lambda_1\phi_1+\lambda_2\phi_3\big)^{-1} $ on $ r_{\phi,2} $ leads to the following alternative expression for $ \gamma_{\phi}(\lambda_1,\lambda_3) $: 
\begin{equation*}
	\gamma_{\phi}(\lambda_1,\lambda_3)=b_{\phi}\cdot B\,r_{\phi,2}\,\frac{i\lambda_1\,\ell_{\phi,2} \,E^{\phi}_{3,1}+ i\lambda_3\,\ell_{\phi,2}\, E^{\phi}_{1,3}}{\lambda_1\,\big(\xi_1(\phi)-\xi_2(\phi)\big)+\lambda_3\,\big(\xi_3(\phi)-\xi_2(\phi)\big)},
\end{equation*}
where we have denoted
\begin{equation*}
	E^{\phi}_{3,1}:=L_1(e_{\phi,3},\phi_{1})\,e_{\phi,1},\qquad E^{\phi}_{1,3}:=L_1(e_{\phi,1},\phi_{3})\,e_{\phi,3}.
\end{equation*}
This rewriting, which can be found in \cite{CoulombelWilliams2017Mach}, will be useful in the following to study the bilinear operator associated with the symbol $ \gamma_{\phi}(\lambda_1,\lambda_2) $ of \eqref{eq 1cor eq gamma phi lambda 1 lambda 3}. 
%Note that the index $ 1 $ in $ \Gamma_1^{\phi} $ does not refer to the fact that we are dealing with the leading profile $ U_1 $, but is indicated since a second coefficient $ \Gamma_2^{\phi} $ will appear in higher order equations. 
Finally, according to Lemma \ref{lemme Lax bord}, the operator $ X_{\phi}^{\Lop} $ is actually equal to the tangential vector field
\begin{equation*}
	\beta_{\phi}\Big(\partial_t+\nabla_{\eta}\kappa(\phi)\cdot\nabla_y\Big),
\end{equation*}
which we still denote by $ X_{\phi}^{\Lop} $.

\bigskip

Similarly, for $ \psi $ we have
\begin{align}\label{eq 1cor boundary cond psi}
	B\,P_{\psi,1}\,\big(U^2_{\psi,1,\lambda}\big)_{|x_d,\chi_d=0}+B\,P_{\psi,3}\,\big(U^2_{\psi,3,\lambda}\big)_{|x_d,\chi_d=0}\\\nonumber
	+B\,(I-P_{\psi,1})\,\big(U^2_{\psi,1,\lambda}\big)_{|x_d,\chi_d=0}+B\,(I-P_{\psi,3})\,\big(U^2_{\psi,3,\lambda}\big)_{|x_d,\chi_d=0}\\\nonumber
	+B\,\big(U^2_{\psi,2,\lambda}\big)_{|x_d,\chi_d=0}+B\,\big(U^{2,\nc}_{\psi,\lambda}\big)_{|x_d,\chi_d=0}&=0,
\end{align}
where the vector $ U^{2,\nc}_{\psi,\lambda} $ is given by
%\begin{multline*}
%	U^{2,\nc}_{\psi,\lambda}=-\sum_{\substack{(\zeta_1,\zeta_2,j_1,j_2,\\\lambda_1,\lambda_2)\in\mathcal{NR}\\\lambda_1\zeta_1+\lambda_2\zeta_2=\lambda\psi}}L\big(0,\lambda_1\alpha_{j_1}(\zeta_1)+\lambda_2\alpha_{j_2}(\zeta_2)\big)^{-1}\,L_1(r_{\zeta_1,j_1},\alpha_{j_2}(\zeta_2))\,r_{\zeta_2,j_2}\,i\lambda_2\,\sigma_{\zeta_1,j_1,\lambda_1}^1\,\sigma_{\zeta_2,j_2,\lambda_2}^1\,\\
%	e^{i(\lambda_1\mathbf{n}_1+\lambda_2\mathbf{n}_2)\cdot\theta}\,e^{i(\lambda_1\xi_{j_1}(\zeta_1)+\lambda_2\xi_{j_2}(\zeta_2))\chi_d}.
%\end{multline*}
%In the same way as for $ \phi $, we determine that $ U^{2,\nc}_{\psi,\lambda} $ writes
\begin{align}\label{eq 1cor U 2 nc psi}
	U^{2,\nc}_{\psi,\lambda}&=\indicatrice_{\lambda=k\lambda_{\psi}}\sum_{\substack{j_1=1,3,j_2=1,2,3 \\ (j_1,j_2)\neq(1,2), (3,2)}} L\big(0,k\nu_{j_2}-\lambda_{\phi}k\phi_{j_1}\big)^{-1}\,\big\lbrace \lambda_{\phi}\,L_1(r_{\nu,j_2},\phi_{j_1})\,r_{\phi,j_1}\\\nonumber
	&\qquad\qquad-L_1(r_{\phi,j_1},\nu_{j_2})\,r_{\nu,j_2}\big\rbrace\,k\,\sigma_{\nu,j_2,k}^1\,\sigma_{\phi,j_1,-\lambda_{\phi}k}^1\,
	e^{ik\lambda_{\phi}\theta_1}\,e^{i(k\xi_{j_2}(\nu)-\lambda_{\phi}k\xi_{j_1}(\phi))\chi_d}\\[5pt]\nonumber
	&\quad-\sum_{\substack{j_1,j_2=1,2,3\\ j_1\neq j_2}}\sum_{\lambda_1+\lambda_2=\lambda}L\big(0,\lambda_1\psi_{j_1}+\lambda_2\psi_{j_2}\big)^{-1}\,L_1(r_{\psi,j_1},\psi_{j_2})\,r_{\psi,j_2}\,\lambda_2\,\\\nonumber
	&\qquad\qquad\sigma^1_{\psi,j_1,\lambda_1}\,\sigma^1_{\psi,j_2,\lambda_2}\, e^{i\lambda\,\theta_2}\,e^{i(\lambda_1\xi_{j_1}(\psi)+\lambda_2\xi_{j_2}(\psi))\chi_d}.
\end{align}
If we take the scalar product of $ b_{\psi} $ with equality \eqref{eq 1cor boundary cond psi} multiplied by $ i\lambda $, using \eqref{eq 1cor nonpola part U 2 zeta}, \eqref{eq 1cor U 2 nc phi} and the boundary conditions \eqref{eq 1cor boundary cond psi} for the leading profile associated with $ \lambda\psi $, we get the second amplitude equation
\begin{multline}\label{eq 1cor eq evol a psi}
\color{altorange2}X^{\Lop}_{\psi}\,a^1_{\psi,\lambda}
+v_{\psi}\,\sum_{\lambda_1+\lambda_2=\lambda}i\lambda_2\,a^1_{\psi,\lambda_1}\,a^1_{\psi,\lambda_2}
+i\lambda\sum_{\lambda_1+\lambda_3=\lambda}\gamma_{\psi}(\lambda_1,\lambda_3)\,a^1_{\psi,\lambda_1}\,a^1_{\psi,\lambda_3}\\[5pt]\color{altorange2}
+\indicatrice_{\lambda=k\lambda_{\psi}}\,\Gamma^{\psi}\,ik\,\big(\sigma_{\nu,2,k}\big)_{|x_d=0}\,a^1_{\phi,-\lambda_{\phi}k}=i\lambda\,b_{\psi}\cdot B\,\big(U^2_{\psi,2,\lambda}\big)_{|x_d,\chi_d=0}.
\end{multline}
\color{black}
with, using once again Lemma \ref{lemme Lax bord},
\begin{subequations}\label{eq 1cor eq X lop v psi etc}
	\begin{align}
		% Opérateur transport
		X^{\Lop}_{\psi}&:=b_{\psi}\cdot B\Big(R_{\psi,1}\,L(0,\partial_z)\,e_{\psi,1}+R_{\psi,3}\,L(0,\partial_z)\,e_{\psi,3}\Big)\nonumber\\
		&\ =\beta_{\psi}\Big(\partial_t+\nabla_{\eta}\kappa(\psi)\cdot\nabla_y\Big),\\
		% Coefficient Burgers
		v_{\psi}&:=b_{\psi}\cdot B\Big(R_{\psi,1}\,L_1\big(e_{\psi,1},\psi_1\big)\,e_{\psi,1}+R_{\psi,3}\,L_1\big(e_{\psi,3},\psi_3\big)\,e_{\psi,3}\Big),\label{eq 1cor eq v psi}\\[5pt]
		%gamma lamnda_1,lambda_2
		\gamma_{\psi}(\lambda_1,\lambda_3)&:=
		b_{\psi}\cdot B\,L\big(\lambda_1\psi_1+\lambda_3\psi_3\big)^{-1}\big\lbrace \lambda_1 \,L_1(e_{\psi,3},\psi_{1})\,e_{\psi,1}+ \lambda_3\, L_1(e_{\psi,1},\psi_{3})\,e_{\psi,3}\big\rbrace\\[5pt]
		% Gamma
		\Gamma^{\psi}&:=b_{\psi}\cdot B\,R_{\psi,1}\,\Big( L_1(e_{\phi,1},-\nu_2)\,r_{\nu,2}+L_1(r_{\nu,2},-\lambda_{\phi}\phi_1)\,e_{\phi,1}\Big)\\
		&\quad+\lambda_{\psi}\,b_{\psi}\cdot B\sum_{\substack{j_1=1,3,j_2=1,2,3 \\ (j_1,j_2)\neq(1,2), (3,2)}}\mu_{\nu,j_2}\,L\big(\nu_{j_2}-\lambda_{\phi}\phi_{j_1}\big)^{-1}\,\big\lbrace L_1(e_{\phi,j_1},\nu_{j_2})\,r_{\nu,j_2} \nonumber\\
		&\qquad\qquad-\lambda_{\phi}\,L_1(r_{\nu,j_2},\phi_{j_1})\,e_{\phi,j_1}\big\rbrace.\nonumber
	\end{align}
\end{subequations}
Again, comparing to \cite[equation (2.19)]{CoulombelWilliams2017Mach}, equation \eqref{eq 1cor eq evol a psi} features two additional terms $ \indicatrice_{\lambda=k\lambda_{\psi}}\,\Gamma^{\psi}\,ik\,\big(\sigma_{\nu,2,k}\big)_{|x_d=0}\,a^1_{\phi,-\lambda_{\phi}k} $ and $ i\lambda\,b_{\psi}\cdot B\,\big(U^2_{\psi,2,\lambda}\big)_{|x_d,\chi_d=0} $ for the same reason. The is no boundary forcing term here because the one for $ \psi $ is of order $ O(\epsilon^3) $. In the same way as for $ \phi $ we have
\begin{equation*}
	\gamma_{\psi}(\lambda_1,\lambda_3)=b_{\psi}\cdot B\,r_{\psi,2}\,\frac{i\lambda_1\,\ell_{\psi,2} \,E^{\psi}_{3,1}+ i\lambda_3\,\ell_{\psi,2}\, E^{\psi}_{1,3}}{\lambda_1\,\big(\xi_1(\psi)-\xi_2(\psi)\big)+\lambda_3\,\big(\xi_3(\psi)-\xi_2(\psi)\big)},
\end{equation*}
where we have denoted
\begin{equation*}
E^{\psi}_{3,1}:=L_1(e_{\psi,3},\psi_{1})\,e_{\psi,1},\qquad E^{\psi}_{1,3}:=L_1(e_{\psi,1},\psi_{3})\,e_{\psi,3}.
\end{equation*}

In conclusion, in this paragraph, we have determined the nonpolarized part of the first corrector, with \eqref{eq 1cor nonpola part U 2 zeta}, and the evolution equations \eqref{eq 1cor eq evol a phi} and \eqref{eq 1cor eq evol a psi} satisfied by $ a^1_{\phi,\lambda} $ and $ a^1_{\psi,\lambda} $. However, the obtained system of equation is still not closed since equations \eqref{eq 1cor eq evol a phi} and \eqref{eq 1cor eq evol a psi} involve traces of the first corrector $ U_2 $, of which the polarized part is still undetermined.

Therefore we proceed inductively, by deriving equations for the polarized part of the first corrector $ U_2 $, and then studying the nonpolarized part of the second corrector $ U_3 $, in order to obtain evolution equations on the boundary terms for the polarized part of $ U_2 $.

\subsubsection{Polarized part of the first corrector}
For $ \zeta=\mathbf{n}\cdot\boldsymbol{\zeta}\in\F_b\privede{0} $ with $ \mathbf{n}\in\B_{\Z^2} $, for $ j\in\mathcal{C}(\zeta) $ and $ \lambda\in\Z^* $, we decompose the profile $ U^{2,\osc}_{\mathbf{n},j,\lambda} $ as $ U^{2,\osc}_{\mathbf{n},j,\lambda}=P_{\zeta,j}\,U^{2,\osc}_{\mathbf{n},j,\lambda}+(I-P_{\zeta,j})\,U^{2,\osc}_{\mathbf{n},j,\lambda} $. Recall that the nonpolarized part $ (I-P_{\zeta,j})\,U^{2,\osc}_{\mathbf{n},j,\lambda} $ is given by \eqref{eq 1cor nonpola part U 2 zeta}, so it remains to determine the polarized part, which is written as
\begin{equation*}
	P_{\zeta,j}\,U^{2,\osc}_{\mathbf{n},j,\lambda}=\sigma_{\zeta,j,\lambda}^2\,r_{\zeta,j},
\end{equation*}
with $ \sigma^2_{\zeta,j,\lambda} $ a scalar function of $ \Omega_T $. We start by determining the mean value $ U_2^* $, as in the general case of a corrector $ U_n $, the mean value $ U_n^* $ appears in equations for the polarized components. According to equation \eqref{eq cascade int U2}, $ U_2^* $ satisfies the equation
\color{altpink}
\begin{align*}
	L(0,\partial_z)\,U^*_2+\sum_{\mathbf{n}\in\B_{\Z^2}}\sum_{j\in\mathcal{C}(\freq)}\sum_{\lambda\in\Z^*}L_1\big(r_{\mathbf{n}\cdot\boldsymbol{\zeta},j},-i\,\lambda\,\alpha_j(\mathbf{n}\cdot\boldsymbol{\zeta})\big)\,r_{\mathbf{n}\cdot\boldsymbol{\zeta},j}\,\Big(\sigma^1_{\mathbf{n},j,\lambda}\,\sigma^2_{\mathbf{n},j,-\lambda}+\sigma^2_{\mathbf{n},j,\lambda}\,\sigma^1_{\mathbf{n},j,-\lambda}\Big)\\
	+\sum_{\mathbf{n}\in\B_{\Z^2}}\sum_{j\in\mathcal{C}(\freq)}\sum_{\lambda\in\Z^*}L_1\big((I-P_{\freq,j})\,U^{2,\osc}_{\mathbf{n},j,\lambda},-i\,\lambda\,\alpha_j(\mathbf{n}\cdot\boldsymbol{\zeta})\big)\,\sigma^1_{\mathbf{n},j,-\lambda}\,r_{\freq,j}\\
	+\sum_{\mathbf{n}\in\B_{\Z^2}}\sum_{j\in\mathcal{C}(\freq)}\sum_{\lambda\in\Z^*}L_1\big(\sigma^1_{\mathbf{n},j,\lambda},-i\,\lambda\,\alpha_j(\mathbf{n}\cdot\boldsymbol{\zeta})\big)\,(I-P_{\freq,j})\,U^{2,\osc}_{\mathbf{n},j,-\lambda}\,r_{\freq,j}\\
	+\sum_{\mathbf{n}\in\B_{\Z^2}}\sum_{j\in\mathcal{C}(\freq)}\sum_{\lambda\in\Z^*}L_1\big(U^{1,\osc}_{\mathbf{n},j,\lambda},\partial_z\big)\,U^{1,\osc}_{\mathbf{n},j,\lambda}&=0.
\end{align*}\color{black}
The change of variable $ \lambda=-\lambda $ shows that the second term in first line of previous equation is zero. Boundary and initial condition \eqref{eq cascade bord U12} and \eqref{eq cascade initial} as well as writing \eqref{eq ecriture U 2} for $ U_2 $ leads to the following boundary and initial condition for $ U_2^* $:
\color{altpink}
\begin{equation*}
	B\,\big(U_2^*\big)_{|x_d=0}=G_0-B\,\big(U^{2,\nc}_0\big)_{|x_d,\chi_d=0},\qquad B\,\big(U_2^*\big)_{|t\leq 0}=0,
\end{equation*}\color{black}
where $ G_0 $ is the mean value of the forcing term $ G $ and $ U^{2,\nc}_0 $ is the sum of all terms of $ U^{2,\nc} $  of which the trace on the boundary has zero frequency. According to expressions \eqref{eq 1cor nonpola part U 2 zeta} of the nonpolarized parts and \eqref{eq 1cor expr U 2 nc} of $ U^{2,\nc} $, the mean value $ U_2^* $ satisfies an initial boundary value problem which is weakly well posed and of which the source and boundary terms depend only on the leading profile $ U_1 $ (with possibly a first order derivative applied to it).

We consider now modes $ \lambda\,\alpha_j(\zeta) $ with $ \zeta=\mathbf{n}\cdot\boldsymbol{\zeta}\in\F_b\privede{0,\phi,\psi,\nu} $, $ \mathbf{n}\in\B_{\Z^2} $, $ j\in\mathcal{C}(\zeta) $ and $ \lambda\in\Z^* $. Recall that it has been proven in the previous part that for these modes, the profile $ U^{2,\osc}_{\mathbf{n},j} $ is polarized, that is
\begin{equation*}
	U^{2,\osc}_{\mathbf{n},j} =\sigma^2_{\zeta,j,\lambda}\,r_{\zeta,j}.
\end{equation*}
Since there is no resonance generating frequency $ \alpha_j(\zeta) $, since the mean value $ U^*_1 $ is zero and since profiles $ \sigma_{\zeta,j,\lambda}^1 $ are also zero, the terms $ \mathcal{M}(U_1,U_2) $,  $ \mathcal{M}(U_2,U_1) $, and $ \mathcal{N}(U_1,U_1) $ contain no term of frequency $ \alpha_j(\zeta) $.
Therefore, analogously as for the leading profile, multiplying equation \eqref{eq cascade int Un} for $ n=3 $ on the left by $ \ell_{\zeta,j} $ leads to the following system of transport equations for the scalar functions $ \sigma_{\zeta,j,\lambda}^2 $:
\begin{subequations}\label{eq 1cor systeme sigma zeta hors cas part}
\begin{equation}\label{eq 1cor systeme sigma zeta hors cas part1}
			\color{altpurple}X_{\alpha_j(\zeta)}\,\sigma_{\zeta,j,\lambda}^2=0,\qquad
		\big(\sigma_{\zeta,j,\lambda}^2\big)_{|t\leq 0}=0.
\end{equation}
\color{black}
When frequency $ \alpha_j(\zeta) $ is outgoing, transport equation \eqref{eq 1cor systeme sigma zeta hors cas part1} leads to $ \sigma_{\zeta,j,\lambda}^2=0 $. When frequency $ \alpha_j(\zeta) $ is incoming, according to boundary condition \eqref{eq cascade bord U12} and decomposition \eqref{eq ecriture U 2} of $ U_2 $, since $ \zeta $ is not in $ \Upsilon $, since $ G $ does  not contain any oscillation in $ \zeta $ and since the outgoing profile $ \sigma^2_{\zeta,j,\lambda} $, $ j\in\mathcal{O}(\zeta) $, is zero, we have
\begin{equation}
	\color{altpurple}\big(\sigma_{\zeta,j,\lambda}^2\big)_{|x_d=0}=-\ell_{\zeta,j}\,A_d(0)\,\big(B_{|E_-(\zeta)}\big)^{-1}B\,\big(U^{2,\nc}_{\zeta,\lambda}\big)_{|x_d,\chi_d=0},
\end{equation}
\color{black}
\end{subequations}
where $ U^{2,\nc}_{\zeta,\lambda} $ is the sum of all terms of $ U^{2,\nc} $ of which the trace on the boundary is $ \lambda\,\zeta $. It is fully determined by \eqref{eq 1cor expr U 2 nc}, and thus depends only on $ \big(\sigma^1_{\zeta,j,\lambda}\big)_{|x_d=0} $ for $ \zeta=\phi,\psi,\nu $. Therefore, if the traces $ \big(\sigma^1_{\zeta,j,\lambda}\big)_{|x_d=0} $ for $ \zeta=\phi,\psi,\nu $ are determined, system \eqref{eq 1cor systeme sigma zeta hors cas part} allows us to construct the profiles $ \sigma^2_{\zeta,j,\lambda} $, for $ \zeta\in\F_b\privede{0,\phi,\psi,\nu} $ and $ j\in\mathcal{I}(\zeta) $.

\bigskip

We now take interest into modes $ \phi_j $, $ \psi_j $ and $ \nu_j $ for $ j=1,2,3 $. Applying $ \ell_{\zeta,j} $ to equation \eqref{eq cascade int Un} for $ n=3 $ leads to the following equation for $ \sigma^2_{\zeta,j,\lambda} $,
\begin{align}\label{eq 1cor evol sigma phi psi nu}
	&\color{altblue}X_{\zeta,j}\,\sigma^2_{\zeta,j,\lambda}+D_{\zeta,j}\sum_{\lambda_1+\lambda_2=\lambda}i\lambda\,\sigma_{\zeta,j,\lambda_1}^2\,\sigma_{\zeta,j,\lambda_2}^1+\indicatrice_{\lambda=k\lambda_{\zeta}}\sum_{\substack{(\zeta_1,\zeta_2,j_1,j_2)\\\in\mathcal{R}(\zeta,j)}}J^{\zeta_2,j_2}_{\zeta_1,j_1}\,ik\,\sigma^1_{\zeta_1,j_1,-k\lambda_{\zeta_1}}\,\sigma^2_{\zeta_2,j_2,-k\lambda_{\zeta_2}}\\\nonumber\color{altblue}
	%transport 
	=&\color{altblue}-\ell_{\zeta,j}\,L(0,\partial_z)\,(I-P_{\zeta,j})\,U^{2}_{\zeta,j,\lambda}
	%auto-interaction 1
	-\ell_{\zeta,j}\sum_{\lambda_1+\lambda_2=\lambda}L_1\big((I-P_{\zeta,j})\,U^2_{\zeta,j,\lambda_1},\zeta_j\big)\,r_{\zeta,j}\,i\lambda_2\,\sigma_{\zeta,j,\lambda_2}^1\\\nonumber
	%auto-interaction 2
	&\color{altblue}-\ell_{\zeta,j}\sum_{\lambda_1+\lambda_2=\lambda}L_1\big(r_{\zeta,j},\zeta_j\big)\,(I-P_{\zeta,j})\,U^2_{\zeta,j,\lambda_2}\,i\lambda_2\,\sigma_{\zeta,j,\lambda_1}^1\\\nonumber
	% Résonance 1
	&\color{altblue}-\indicatrice_{\lambda=k\lambda_{\zeta}}\sum_{\substack{(\zeta_1,\zeta_2,j_1,j_2)\\\in\mathcal{R}(\zeta,j)}}\ell_{\zeta,j}\,ik\,\Big\lbrace L_1\big(r_{\zeta_1,j_1},-\lambda_{\zeta_2}\alpha_{j_2}(\zeta_2)\big)\,(I-P_{\zeta_2,j_2})\,U^2_{\zeta_2,j_2,-k\lambda_{\zeta_2}}\,\sigma_{\zeta_1,j_1,-k\lambda_{\zeta_1}}^1\\
	&\color{altblue}\nonumber\qquad+L_1\big((I-P_{\zeta_2,j_2})\,U^2_{\zeta_2,j_2,-k\lambda_{\zeta_2}},-\lambda_{\zeta_1}\alpha_{j_1}(\zeta_1)\big)\,r_{\zeta_1,j_1}\,\sigma_{\zeta_1,j_1,-k\lambda_{\zeta_1}}^1\Big\rbrace\\\nonumber
	%auto-interaction N_1
	&\color{altblue}-\ell_{\zeta,j}\sum_{\lambda_1+\lambda_2=\lambda}L_1\big(\sigma_{\zeta,j,\lambda_1}^1\,r_{\zeta,j},\partial_z\big)\,\sigma_{\zeta,j,\lambda_2}^1\,r_{\zeta,j}\\\nonumber
	%résonance N_1
	&\color{altblue}-\indicatrice_{\lambda=k\lambda_{\zeta}}\sum_{\substack{(\zeta_1,\zeta_2,j_1,j_2)\\\in\mathcal{R}(\zeta,j)}}\ell_{\zeta,j}\, L_1\big(\sigma_{\zeta_1,j_1,-k\lambda_{\zeta_1}}^1\,r_{\zeta_1,j_1},\partial_z\big)\,\sigma_{\zeta_2,j_2,-k\lambda_{\zeta_2}}^1\,r_{\zeta_2,j_2}.
	\end{align}
\color{black}
Note that the source term on the right-hand side of equation \eqref{eq 1cor evol sigma phi psi nu} only depends on the leading profile $ U_1 $, according to formula \eqref{eq 1cor nonpola part U 2 zeta} for nonpolarized parts, with possibly second order derivatives applied to it (since in the expression of the nonpolarized part, first order derivatives are applied to $ U_1 $).

For the incoming frequencies $ \phi_1 $, $ \phi_3 $ and so on, boundary conditions must be determined to solve the above transport equations \eqref{eq 1cor evol sigma phi psi nu}. We have already seen that boundary equation \eqref{eq cascade bord U12} for $ U_2 $ reads, for mode $ \lambda\phi $, as \eqref{eq 1cor boundary cond phi}. From this boundary condition, according to decomposition \eqref{eq decomp E_-(zeta)} of $ E_-(\zeta) $ and relation \eqref{eq def b zeta} defining the vector $ b_{\phi} $, we get the following necessary solvability condition
\begin{align*}
b_{\phi}\cdot\Big(B\,(I-P_{\phi,1})\,\big(U^2_{\phi,1,\lambda}\big)_{|x_d,\chi_d=0}+B\,(I-P_{\phi,3})\,\big(U^2_{\phi,3,\lambda}\big)_{|x_d,\chi_d=0}\\\nonumber
+B\,\big(U^2_{\phi,2,\lambda}\big)_{|x_d,\chi_d=0}+B\,\big(U^{2,\nc}_{\phi,\lambda}\big)_{|x_d,\chi_d=0}\Big)&=b_{\phi}\cdot G_{\lambda},
\end{align*}
which is satisfied as soon as the scalar functions $ a_{\phi,\lambda} $ satisfy the evolution equation \eqref{eq 1cor eq evol a phi}, since these two equations are different writings of the same one. Thus we obtain, for $ j=1,3 $, in a similar manner than for the leading profile,
\begin{equation}\label{eq 1cor trace phi}
	\color{altred}\big(\sigma_{\phi,j,\lambda}^2\big)_{|x_d=0}\,r_{\phi,j}=a_{\phi,\lambda}^2\,e_{\phi,j}
	%+\mu_{\phi,j}\,\big(\sigma_{\phi,2,\lambda}^2\big)_{|x_d=0}\,r_{\phi,j}
	+\tilde{F}^2_{\phi,j,\lambda},
\end{equation}
\color{black} 
with $ a^2_{\phi,\lambda} $ a scalar function defined on $ \omega_T $ and where, for $ j=1,3 $, we have denoted by $ \tilde{F}^2_{\phi,j,\lambda} $ the function
\begin{align*}
	\tilde{F}^2_{\phi,j,\lambda}&:=\ell_{\phi,j}\cdot A_d(0)\,\big(B_{|E_-(\phi)}\big)^{-1}\Big(G_{\lambda}-B\,(I-P_{\phi,1})\,\big(U^2_{\phi,1,\lambda}\big)_{|x_d,\chi_d=0}\\
	&\quad-B\,(I-P_{\phi,3})\,\big(U^2_{\phi,3,\lambda}\big)_{|x_d,\chi_d=0}-B\,(I-P_{\phi,2})\,\big(U^2_{\phi,2,\lambda}\big)_{|x_d,\chi_d=0}\\
	&\quad-B\,\big(\sigma_{\phi,2,\lambda}^2\big)_{|x_d=0}\,r_{\phi,2}-B\,\big(U^{2,\nc}_{\phi,\lambda}\big)_{|x_d,\chi_d=0}\Big)\,r_{\phi,j}.
\end{align*}
%and, analogously as for $ \nu $,
%\begin{equation}\label{eq 1cor def mu phi}
%	\mu_{\phi,j}:=-\ell_{\phi,j}\,A_d(0)\,\big(B_{|E_-(\phi)}\big)^{-1}\,B\,r_{\phi,2}.
%\end{equation}

In the same way, for $ \psi $, from \eqref{eq 1cor boundary cond psi}, the following condition must be satisfied: 
\begin{align*}
	b_{\psi}\cdot\Big(B\,(I-P_{\psi,1})\,\big(U^2_{\psi,1,\lambda}\big)_{|x_d,\chi_d=0}+B\,(I-P_{\psi,3})\,\big(U^2_{\psi,3,\lambda}\big)_{|x_d,\chi_d=0}\\\nonumber
	+B\,\big(U^2_{\psi,2,\lambda}\big)_{|x_d,\chi_d=0}+B\,\big(U^{2,\nc}_{\psi,\lambda}\big)_{|x_d,\chi_d=0}\Big)&=0,
\end{align*}
and it is the case when $ a_{\psi,\lambda} $ verify equation \eqref{eq 1cor eq evol a psi}.
Therefore, for $ j=1,3 $,
\begin{align}\label{eq 1cor trace psi}
	\color{altorange2} \big(\sigma_{\psi,j,\lambda}^2\big)_{|x_d=0}\,r_{\psi,j}=a_{\psi,\lambda}^2\,e_{\psi,j}
	%+\mu_{\psi,j}\,\big(\sigma_{\psi,2,\lambda}^2\big)_{|x_d=0}\,r_{\psi,j}
	+\tilde{F}^2_{\psi,j,\lambda},
\end{align}
\color{black}
with $ a^2_{\psi,\lambda} $ a scalar function of $ \omega_T $,
%where, 
%\begin{equation}\label{eq 1cor def mu psi}
%	\mu_{\psi,j}=-\ell_{\psi,j}\cdot A_d(0)\, \big(B_{|E_-(\psi)}\big)^{-1}\,B\,r_{\psi,2},
%\end{equation}
%%and $ \mu_{\psi,2}:=1 $, 
and where we have denoted by $ \tilde{F}^2_{\psi,j,\lambda} $ the function
\begin{align*}
\tilde{F}^2_{\psi,j,\lambda}&:=-\ell_{\psi,j}\cdot A_d(0)\,\big(B_{|E_-(\psi)}\big)^{-1}\Big(B\,(I-P_{\psi,1})\,\big(U^2_{\psi,1,\lambda}\big)_{|x_d,\chi_d=0}\\
&\quad+B\,(I-P_{\psi,3})\,\big(U^2_{\psi,3,\lambda}\big)_{|x_d,\chi_d=0}+B\,(I-P_{\psi,2})\,\big(U^2_{\psi,2,\lambda}\big)_{|x_d,\chi_d=0}\\
&\quad -B\,\big(\sigma_{\psi,2,\lambda}^2\big)_{|x_d=0}\,r_{\psi,2}+B\,\big(U^{2,\nc}_{\psi,\lambda}\big)_{|x_d,\chi_d=0}\Big).
\end{align*}
Note that expressions \eqref{eq 1cor trace phi} and \eqref{eq 1cor trace psi} of incoming traces $ \big(\sigma_{\phi,j,\lambda}^2\big)_{|x_d=0} $ and $ \big(\sigma_{\psi,j,\lambda}^2\big)_{|x_d=0} $ for $ j=1,3 $ is respectively coupled to the outgoing trace $ \big(\sigma_{\phi,2,\lambda}^2\big)_{|x_d=0} $ and $ \big(\sigma_{\psi,2,\lambda}^2\big)_{|x_d=0} $ through terms $ \tilde{F}^2_{\phi,j,\lambda} $ and $ \tilde{F}^2_{\psi,j,\lambda} $.

Finally, for amplitudes associated with the boundary phase $ \nu $, we need to write a boundary condition for the first corrector. Boundary condition \eqref{eq cascade bord U12} for $ U_2 $ reads \begin{align*}%\label{eq 1cor boundary cond nu}
	B\,P_{\nu,1}\,\big(U^2_{\nu,1,\lambda}\big)_{|x_d,\chi_d=0}+B\,P_{\nu,3}\,\big(U^2_{\nu,3,\lambda}\big)_{|x_d,\chi_d=0}\\\nonumber
	+B\,(I-P_{\nu,1})\,\big(U^2_{\nu,1,\lambda}\big)_{|x_d,\chi_d=0}+B\,(I-P_{\nu,3})\,\big(U^2_{\nu,3,\lambda}\big)_{|x_d,\chi_d=0}\\\nonumber
	+B\,\big(U^2_{\nu,2,\lambda}\big)_{|x_d,\chi_d=0}+B\,\big(U^{2,\nc}_{\nu,\lambda}\big)_{|x_d,\chi_d=0}&=0,
\end{align*}
where $ U^{2,\nc}_{\nu,\lambda} $ is the sum of all the terms of $ U^{2,\nc} $ of which the trace on the boundary of the associated frequency is equal to $ \lambda\,\nu $, which is fully determined by \eqref{eq 1cor expr U 2 nc}.
%We determine that there is no such term, so that $ U^{2,\nc}_{\nu,\lambda} $ is zero. 
Therefore, since $ \nu\in\F_{b}\setminus \Upsilon $, we get, for $ j=1,3 $, 
\begin{align}\label{eq 1cor trace nu}
	\color{altgreen}\big(\sigma_{\nu,j,\lambda}^2\big)_{|x_d=0}\,r_{\nu,j}=\mu_{\nu,j}\,\big(\sigma_{\nu,2,\lambda}^2\big)_{|x_d=0}\,r_{\nu,j}+\tilde{F}^2_{\nu,j,\lambda},
\end{align}
\color{black}
where $ \mu_{\nu,j} $ has been defined in equation \eqref{eq prof princ def mu nu} and where we have denoted by $ \tilde{F}^2_{\nu,j,\lambda} $ the function
\begin{multline*}
\tilde{F}^2_{\nu,j,\lambda}:=-\ell_{\nu,j}\cdot A_d(0)\,\big(B_{|E_-(\nu)}\big)^{-1}\Big(B\,(I-P_{\nu,1})\,\big(U^2_{\nu,1,\lambda}\big)_{|x_d,\chi_d=0}\\
+B\,(I-P_{\nu,3})\,\big(U^2_{\nu,3,\lambda}\big)_{|x_d,\chi_d=0}+B\,(I-P_{\nu,2})\,\big(U^2_{\nu,2,\lambda}\big)_{|x_d,\chi_d=0}+B\,\big(U^{2,\nc}_{\nu,\lambda}\big)_{|x_d,\chi_d=0}\Big).
\end{multline*}

In the same way as for the leading profile, we need to investigate the nonpolarized part of the second corrector to find equations on $ a_{\phi,\lambda}^2 $ and $ a_{\psi,\lambda}^2 $.

\subsubsection{Nonpolarized part of the second corrector} We follow the same analysis as for the first corrector. With similar arguments we get that the second corrector $ U_3 $ reads
\begin{align}\label{eq ecriture U 3}
	U_3(z,\theta,\chi_d)&=U^*_3(z)+\sum_{\mathbf{n}\in\B_{\Z^2}}\sum_{j\in\mathcal{C}(\freq)}\sum_{\lambda\in\Z^*}U^{3,\osc}_{\mathbf{n},j,\lambda}(z)\,e^{i\,\lambda\,\mathbf{n}\cdot\theta}\,e^{i\,\lambda\,\xi_j(\freq)\,\chi_d} \\\nonumber
	&\quad+\sum_{\mathbf{n}\in\B_{\Z^2}}\sum_{\lambda\in\Z^*}e^{\chi_d\mathcal{A}(\mathbf{n}\cdot\boldsymbol{\zeta})}\,\Pi^e(\mathbf{n}\cdot\boldsymbol{\zeta})\,U^{3,\ev}_{\mathbf{n}}(z,0)\,e^{i\,\lambda\,\mathbf{n}\cdot\theta}+U^{3,\nc}(z,\theta,\chi_d),
\end{align}
with $ U_3^* $ the mean value of $ U_3 $, $ U^{3,\nc} $ the noncharacteristic terms, and where, for $ \zeta=\freq\in\F_{b}\privede{0} $, $ \mathbf{n}\in\B_{\Z^2} $, $ j\in\mathcal{C}(\zeta) $ and $ \lambda\in\Z^* $, profile $ U^{3,\osc}_{\mathbf{n},j,\lambda} $ decomposes as 
\begin{equation*}
	 U^{3,\osc}_{\mathbf{n},j,\lambda}=\sigma^3_{\zeta\,j,\lambda}\,r_{\zeta,j}+\big(I-P_{\zeta,j}\big)\,U^{3,\osc}_{\mathbf{n},j,\lambda},
\end{equation*}
with $ \sigma_{\zeta,j,\lambda}^3 $ a scalar function of $ \Omega_T $. Furthermore, according to \eqref{eq cascade int Un} for $ n=2 $, the noncharacteristic part $ U^{3,\nc} $ is given by
\begin{align}\label{eq 2cor expr U 3 nc}
	\mathcal{L}(&\partial_{\theta},\partial_{\chi_d})\,U^{3,\nc}=\\\nonumber
	&-\sum_{\substack{(\zeta_1,\zeta_2,j_1,j_2,\\\lambda_1,\lambda_2)\in\mathcal{NR}}}\Big(L_1\big(U^{1,\osc}_{\zeta_1,j_1,\lambda_1},i\,\lambda_2\alpha_{j_2}(\zeta_2)\big)\,U^{2,\osc}_{\zeta_2,j_2,\lambda_2}
	+L_1\big(U^{2,\osc}_{\zeta_1,j_1,\lambda_1},i\,\lambda_2\alpha_{j_2}(\zeta_2)\big)\,U^{1,\osc}_{\zeta_2,j_2,\lambda_2}\\\nonumber
	&+L_1\big(U^{1,\osc}_{\zeta_1,j_1,\lambda_1},\partial_z\big)\,U^{1,\osc}_{\zeta_2,j_2,\lambda_2}\Big)\, e^{i(\lambda_1\mathbf{n}_1+\lambda_2\mathbf{n}_2)\cdot\theta}\,e^{i(\lambda_1\xi_{j_1}(\zeta_1)+\lambda_2\xi_{j_2}(\zeta_2))\chi_d},
\end{align}
where the set $ \mathcal{NR} $ of nonresonant frequencies has already been defined. Since all frequencies in $ U^{3,\nc} $ are noncharacteristic, equation \eqref{eq 2cor expr U 3 nc} totally determines $ U^{3,\nc} $. Note that opposite to what was done for the first corrector, this is no longer true that in $ U^{3,\nc} $ there are only profiles of modes $ \phi_j $, $ \psi_j $ and $ \nu_j $, since now second order profiles $ \sigma^2_{\zeta,j,\lambda} $ for $ \zeta\in\F_{b}\privede{0,\phi,\psi,\nu} $ are possibly nonzero.

For the same reason, profiles $ U^{3,\osc}_{\mathbf{n},j,\lambda} $ for $ \mathbf{n}\cdot\boldsymbol{\zeta}\in\F_{b}\privede{0,\phi,\psi,\nu} $ are not necessarily polarized. Therefore we derive now the nonpolarized part for each frequency $ \zeta=\freq\in\F_{b}\privede{0} $, $ \mathbf{n}\in\B_{\Z^2} $ and $ j\in\mathcal{C}(\zeta) $, $ \lambda\in\Z^* $. Multiplying equation \eqref{eq cascade int Un} for $ n=2 $ and frequency $ \lambda\,\alpha_j(\zeta) $ on the left by the partial inverse $ R_{\zeta,j} $ leads to the relation
\begin{align}\label{eq 2cor nonpola part U 3 zeta}
	i\lambda\,\big(I-&P_{\zeta,j}\big)\,U^{3,\osc}_{\mathbf{n},j,\lambda}=\\\nonumber
	&-R_{\zeta,j}\,L(0,\partial_z)\,\sigma_{\zeta,j,\lambda}^2\,r_{\zeta,j}-R_{\zeta,j}\,L_1\big(r_{\zeta,j},\alpha_j(\zeta)\big)\,r_{\zeta,j}\sum_{\lambda_1+\lambda_2=\lambda}i\lambda\,\sigma_{\zeta,j,\lambda_1}^1\,\sigma_{\zeta,j,\lambda_2}^2\\\nonumber
	&-R_{\zeta,j}\,\indicatrice_{\lambda=k\lambda_{\zeta}}\sum_{\substack{(\zeta_1,\zeta_2,j_1,j_2)\\\in\mathcal{R}(\zeta,j)}}L_1\big(r_{\zeta_1,j_1},-\lambda_{\zeta_2}\alpha_{j_2}(\zeta_2)\big)\,r_{\zeta_2,j_2}\,ik\,\sigma_{\zeta_1,j_1,-k\lambda_{\zeta_1}}^1\,\sigma_{\zeta_2,j_2,-k\lambda_{\zeta_2}}^2\\\nonumber
	&-R_{\zeta,j}\,\partial_z\,\mbox{\emph{terms in }} \big(U_1,(I-P)\,U_2,U_2^*\big),
\end{align}
where notation $ \partial_z\,\mbox{\emph{terms in }} \big(U_1,(I-P)\,U_2\big) $ refers to quadratic terms involving the leading profile $ U_1 $, the nonpolarized parts of the first corrector $ U_2 $ and the mean value $ U_2^* $, with possibly first order derivative in front of it. The key point is that since the leading profile is polarized and since the mean value $ U_1^* $ is zero, all the terms involving a profile $ \sigma^2_{\zeta',j',\lambda'} $ depend only on leading order polarized profiles $ \sigma^1_{\zeta'',j'',\lambda''} $, and not on $ (I-P)\,U_2 $ or $ U_2^* $.
We now write the boundary conditions for the second corrector, for the frequencies $ \phi $ and $ \psi $, which will lead to equations on the amplitudes $ a_{\phi}^2 $ and $ a^{2}_{\psi} $. 

For $ \phi $ we have, according to boundary condition \eqref{eq cascade bord U3n} and writing \eqref{eq ecriture U 3} of $ U_3 $, since the elliptic component $ E^e_-(\phi) $ of the stable subspace $ E_-(\zeta) $ is zero,
\begin{align}\label{eq 2 cor boundary cond phi}
	B\,P_{\phi,1}\,\big(U^3_{\phi,1,\lambda}\big)_{|x_d,\chi_d=0}+B\,P_{\phi,3}\,\big(U^3_{\phi,3,\lambda}\big)_{|x_d,\chi_d=0}\\\nonumber
	+B\,(I-P_{\phi,1})\,\big(U^3_{\phi,1,\lambda}\big)_{|x_d,\chi_d=0}+B\,(I-P_{\phi,3})\,\big(U^3_{\phi,3,\lambda}\big)_{|x_d,\chi_d=0}\\\nonumber
	+B\,\big(U^{3,\osc}_{\phi,2,\lambda}\big)_{|x_d,\chi_d=0}+B\,\big(U^{3,\nc}_{\phi,\lambda}\big)_{|x_d,\chi_d=0}&=0,
\end{align}
where $ U^{3,\nc}_{\phi,\lambda} $ is the sum of all the terms of $ U^{3,\nc} $ of which the trace on the boundary of the associated frequency is equal to $ \lambda\,\phi $, namely, according to expression \eqref{eq 2cor expr U 3 nc} of $ U^{3,\nc} $,
\begin{align}\label{eq 2cor U3 nc phi}
	U^{3,\nc}_{\phi,\lambda}&=\indicatrice_{\lambda=k\lambda_{\phi}}\sum_{\substack{j_1,j_2=1,2,3 \\ (j_1,j_2)\neq\\(2,1), (2,2)}} L\big(0,k\nu_{j_1}-\lambda_{\psi}k\psi_{j_2}\big)^{-1}\,\big\lbrace\lambda_{\psi} L_1(r_{\nu,j_1},\psi_{j_2})\,r_{\psi,j_2}-L_1(r_{\psi,j_2},\nu_{j_1})\,r_{\nu,j_1}\big\rbrace\\\nonumber
	&\qquad\qquad\,k\,\big\lbrace\sigma^1_{\nu,j_1,k}\,\sigma^2_{\psi,j_2,-\lambda_{\psi}k}+\sigma^2_{\nu,j_1,k}\,\sigma^1_{\psi,j_2,-\lambda_{\psi}k}\big\rbrace\,
	e^{ik\lambda_{\phi}\theta_1}\,e^{i(k\xi_{j_1}(\nu)-\lambda_{\psi}k\xi_{j_2}(\psi))\chi_d}\\[5pt]\nonumber
	&-\sum_{\lambda_1+\lambda_2=\lambda}L\big(0,\lambda_1\phi_1+\lambda_2\phi_3\big)^{-1}\big\lbrace \lambda_1 \,L_1(r_{\phi,3},\phi_{1})\,r_{\phi,1}+ \lambda_2\, L_1(r_{\phi,1},\phi_{3})\,r_{\phi,3}\big\rbrace\\\nonumber
	&\qquad\qquad\big\lbrace\sigma^1_{\phi,1,\lambda_1}\,\sigma^2_{\phi,3,\lambda_2}+\sigma^2_{\phi,1,\lambda_1}\,\sigma^1_{\phi,3,\lambda_2}\big\rbrace
\,	e^{i\lambda\,\theta_1}\,e^{i(\lambda_1\xi_1(\phi)+\lambda_2\xi_3(\phi))\chi_d}\\[5pt]\nonumber
	&+\partial_{z,\theta}\,\mbox{\emph{terms in}}\,\big(U_1,(I-P)\,U_2,(P\,U_2)_{\zeta\neq \phi,\psi,\nu},U_2^*\big),
\end{align}
where the notation $ \partial_z\,\mbox{\emph{terms in}}\,\big(U_1,(I-P)\,U_2,(P\,U_2)_{\zeta\neq \phi,\psi,\nu},U_2^*\big) $ refers to quadratic terms in $ U_1 $, the nonpolarized part of $ U_2 $, the polarized part of $ U_2 $ of which the associated modes are different from $ \lambda\phi_j $, $ \lambda\psi_j $ and $ \lambda\nu_j $, and the mean value $ U_2^* $, with possibly one derivative in front of it. Once again, the key point here is that since $ U_1 $ is polarized and of zero mean value, and since only the profiles $ \sigma^1_{\zeta,j,\lambda} $ for $ \zeta=\phi,\psi,\nu $ are nonzero, in $ U^{3,\nc}_{\phi,\lambda} $, every term involving $ \sigma^2_{\zeta,j,\lambda'} $ for $ \zeta=\phi,\psi,\nu $, is a quadratic term with a profile $ \sigma^1_{\zeta',j',\lambda''} $ for $ \zeta'=\phi,\psi,\nu $.

Similarly as for the leading profile, taking the scalar product of $ b_{\phi} $ with equation \eqref{eq 2 cor boundary cond phi} multiplied by $ i\lambda $, using \eqref{eq 2cor nonpola part U 3 zeta}, \eqref{eq 2cor U3 nc phi} and expression of traces \eqref{eq 1cor trace phi}, \eqref{eq 1cor trace psi} and \eqref{eq 1cor trace nu}, we get
\begin{align}\label{eq 2cor eq evol a phi}
	&\color{altred} X^{\Lop}_{\phi}\,a_{\phi,\lambda}^2+D^{\Lop}_{\phi}\,i\lambda\sum_{\lambda_1+\lambda_2=\lambda}a_{\phi,\lambda_1}^1\,a_{\phi,\lambda_2}^2+i\lambda\sum_{\lambda_1+\lambda_2=\lambda}\gamma_{\phi}(\lambda_1,\lambda_2)\big(a_{\phi,\lambda_1}^1\,a_{\phi,\lambda_2}^2+a_{\phi,\lambda_1}^2\,a_{\phi,\lambda_2}^1\big)\\\nonumber
	&\color{altred}+\indicatrice_{\lambda=k\lambda_{\phi}}\,\Gamma^{\phi}\,ik\,\Big\lbrace\big(\sigma_{\nu,2,-k}^1\big)_{|x_d=0}\,a_{\psi,-k\lambda_{\psi}}^2+\big(\sigma_{\nu,2,-k}^2\big)_{|x_d=0}\,a_{\psi,-k\lambda_{\psi}}^1\Big\rbrace\\[5pt]\nonumber
%	\\\nonumber
%	&\color{altred}+\indicatrice_{\lambda=k\lambda_{\phi}}\,\Gamma^{\phi}_2\,ik\,
%	\big(\sigma_{\nu,2,-k}^1\big)_{|x_d=0}\,\big(\sigma_{\psi,2,-k\lambda_{\psi}}^2\big)_{|x_d=0}
	&\color{altred}+\partial_{z,\theta}\,\mbox{\emph{terms in}}\,\Big[\tilde{F}^2_{\phi},\big(\tilde{F}^2_{\phi},a^1_{\phi}\big),\big(\tilde{F}^2_{\psi},(\sigma^1_{\nu,2})_{|x_d=0}\big),\big((\sigma^2_{\phi})_{|x_d=0},a^1_{\phi}\big),\big((\sigma^2_{\psi})_{|x_d=0},(\sigma^1_{\nu,2})_{|x_d=0}\big)\Big]
	\\[5pt]\nonumber
	&\color{altred}=i\lambda\,b_{\psi}\cdot B\,\big(U^{3,\osc}_{\phi,2,\lambda}\big)_{|x_d,\chi_d=0}+\partial_{z,\theta}\,\mbox{\emph{terms in}}\,\big(U_1,(I-P)\,U_2,(P\,U_2)_{\zeta\neq \phi,\psi,\nu},U_2^*\big)_{|x_d,\chi_d=0},
\end{align}
\color{black}
where $ X^{\Lop}_{\phi} $, $ D^{\Lop}_{\phi} $, $ \gamma_{\phi}(\lambda_1,\lambda_2) $ and $ \Gamma^{\phi} $ have already been defined in \eqref{eq 1cor eq X lop v phi etc}, and where 
\begin{equation*}
	\partial_{z,\theta}\,\mbox{\emph{terms in}}\,\Big[\tilde{F}^2_{\phi},\big(\tilde{F}^2_{\phi},a^1_{\phi}\big),\big(\tilde{F}^2_{\psi},(\sigma^1_{\nu,2})_{|x_d=0}\big),\big((\sigma^2_{\phi})_{|x_d=0},a^1_{\phi}\big),\big((\sigma^2_{\psi})_{|x_d=0},(\sigma^1_{\nu,2})_{|x_d=0}\big)\Big]
\end{equation*}
refers to linear terms in $ \tilde{F}^2_{\phi,j,\lambda} $ for $ j=1,2 $ and $ \lambda\in\Z^* $, and quadratic terms in $ \big(\tilde{F}^2_{\phi,j,\lambda},a^1_{\phi,\lambda}\big) $,  $ \big(\tilde{F}^2_{\psi,j,\lambda},(\sigma^1_{\nu,2,\lambda})_{|x_d=0}\big) $, $ \big((\sigma^2_{\phi,j,\lambda})_{|x_d=0},a^1_{\phi,\lambda}\big) $ or $ \big((\sigma^2_{\psi,j,\lambda})_{|x_d=0},(\sigma^1_{\nu,2,\lambda})_{|x_d=0}\big) $ for $ j=1,2 $ and $ \lambda\in\Z^* $, with, for all these terms, possibly one derivative in front of them. Recall that terms $ \tilde{F}^2_{\phi} $ and  $ \tilde{F}^2_{\psi} $ depend on the traces $ \big(\sigma^2_{\phi,2,\lambda})_{|x_d=0} $ and $ \big(\sigma^2_{\psi,2,\lambda})_{|x_d=0} $.
%we have denoted
%\begin{align}
%	\Gamma^{\phi}_2:=&b_{\phi}\cdot B\,R_{\phi,1}\,\Big(
%		 \mu_{\psi,1}\,L_1(r_{\psi,1},-\nu_2)\,r_{\nu,2}+\mu_{\psi,1}\,L_1(r_{\nu,2},-\lambda_{\psi}\psi_1)\,r_{\psi,1}\\
%		&\nonumber\qquad+\mu_{\nu,1}\, L_1(r_{\psi,2},-\nu_1)\,r_{\nu,1}+\mu_{\nu,1}\,L_1(r_{\nu,1},-\lambda_{\psi}\psi_2)\,r_{\psi,2}\Big)\\
%		&\nonumber+\lambda_{\phi}\,b_{\phi}\cdot B\sum_{\substack{j_1,j_2=1,2,3 \\ (j_1,j_2)\neq(2,1), (2,2)}}\mu_{\nu,j_2}\,\mu_{\psi,j_1}\,L\big(\nu_{j_2}-\lambda_{\psi}\psi_{j_1}\big)^{-1}\,\Big\lbrace L_1(r_{\psi,j_1},\nu_{j_2})\,r_{\nu,j_2}\\
%		&\nonumber\qquad-L_1(r_{\nu,j_2},\psi_{j_1})\,r_{\psi,j_1}\,\lambda_{\psi}\Big\rbrace.
%\end{align}

For phase $ \psi $ we have the following boundary condition
\begin{align}\label{eq 2 cor boundary cond psi}
	B\,P_{\psi,1}\,\big(U^3_{\psi,1,\lambda}\big)_{|x_d,\chi_d=0}+B\,P_{\psi,3}\,\big(U^3_{\psi,3,\lambda}\big)_{|x_d,\chi_d=0}\\\nonumber
	+B\,(I-P_{\psi,1})\,\big(U^3_{\psi,1,\lambda}\big)_{|x_d,\chi_d=0}+B\,(I-P_{\psi,3})\,\big(U^3_{\psi,3,\lambda}\big)_{|x_d,\chi_d=0}\\\nonumber
	+B\,\big(U^3_{\psi,2,\lambda}\big)_{|x_d,\chi_d=0}+B\,\big(U^{3,\nc}_{\psi,\lambda}\big)_{|x_d,\chi_d=0}&=H_{\lambda},
\end{align}
where we have expanded $ H $ in Fourier series as 
 \begin{equation*}
 	H(z',\Theta)=\sum_{\lambda\in\Z}H_{\lambda}(z')\,e^{i\lambda\Theta},
 \end{equation*}
 and $ U^{3,\nc}_{\psi,\lambda} $ is the sum of all the terms of $ U^{3,\nc} $ of which the trace on the boundary of the associated frequency is equal to $ \lambda\psi $, that is,
\begin{align}\label{eq 2cor U3 nc psi}
	U^{3,\nc}_{\psi,\lambda}&=\indicatrice_{\lambda=k\lambda_{\psi}}\sum_{\substack{j_1=1,3,j_2=1,2,3 \\ (j_1,j_2)\\\neq(1,2), (3,2)}} L\big(k\nu_{j_2}-\lambda_{\phi}k\phi_{j_1}\big)^{-1}\,\big\lbrace \lambda_{\phi}\,L_1(r_{\nu,j_2},\phi_{j_1})\,r_{\phi,j_1}-L_1(r_{\phi,j_1},\nu_{j_2})\,r_{\nu,j_2}\big\rbrace\\\nonumber
	&\qquad\qquad\,k\,\big\lbrace\sigma_{\nu,j_2,k}^1\,\sigma_{\phi,j_1,-\lambda_{\phi}k}^2+\sigma_{\nu,j_2,k}^2\,\sigma_{\phi,j_1,-\lambda_{\phi}k}^1\big\rbrace\,
	e^{ik\lambda_{\phi}\theta_1}\,e^{i(k\xi_{j_2}(\nu)-\lambda_{\phi}k\xi_{j_1}(\phi))\chi_d}\\\nonumber
	&-\sum_{\substack{j_1,j_2=1,2,3\\ j_1\neq j_2}}\sum_{\lambda_1+\lambda_2=\lambda}L\big(\lambda_1\psi_{j_1}+\lambda_2\psi_{j_2}\big)^{-1}\,L_1(r_{\psi,j_1},\psi_{j_2})\,r_{\psi,j_2}\,\lambda_2\,\big\lbrace\sigma^1_{\psi,j_1,\lambda_1}\,\sigma^2_{\psi,j_2,\lambda_2}\,\\\nonumber
	&\qquad\qquad+\sigma^2_{\psi,j_1,\lambda_1}\,\sigma^1_{\psi,j_2,\lambda_2}\big\rbrace\,e^{i\lambda\,\theta_2}\,e^{i(\lambda_1\xi_{j_1}(\psi)+\lambda_2\xi_{j_2}(\psi))\chi_d}\\[5pt]\nonumber
	&+\partial_{z,\theta}\,\mbox{\emph{terms in}}\,\big(U_1,(I-P)\,U_2,(P\,U_2)_{\zeta\neq \phi,\psi,\nu},U_2^*\big).
\end{align}
If we take the scalar product of $ b_{\psi} $ with the equality \eqref{eq 2 cor boundary cond psi} multiplied by $ \lambda $, using \eqref{eq 2cor nonpola part U 3 zeta}, \eqref{eq 2cor U3 nc phi} and expression of traces \eqref{eq 1cor trace phi}, \eqref{eq 1cor trace psi} and \eqref{eq 1cor trace nu}, we get the amplitude equation
\begin{align}\label{eq 2cor eq evol a psi}
	&\color{altorange2} X^{\Lop}_{\psi}\,a^2_{\psi,\lambda}
	+D^{\Lop}_{\psi}\,i\lambda\!\sum_{\lambda_1+\lambda_2=\lambda}a^1_{\psi,\lambda_1}\,a^2_{\psi,\lambda_2}
	+i\lambda\!\sum_{\lambda_1+\lambda_2=\lambda}\gamma_{\psi}(\lambda_1,\lambda_2)\big(a^1_{\psi,\lambda_1}\,a^2_{\psi,\lambda_2}+a^2_{\psi,\lambda_1}\,a^1_{\psi,\lambda_2}\big)
%	+\tilde{X}^{\Lop}_{\psi}\,\big(\sigma_{\psi,2,\lambda}^2\big)_{|x_d=0}\\\nonumber
%	%
%	&\color{altorange2}
%	+\tilde{D}^{\Lop}_{\psi}\sum_{\lambda_1+\lambda_2=\lambda}i\lambda\,
%	\big(\sigma_{\psi,2,\lambda_1}^2\big)_{|x_d=0}\,a^1_{\psi,\lambda_2}
%	+i\lambda\sum_{\lambda_1+\lambda_2=\lambda}\tilde{\gamma}_{\psi}(\lambda_1,\lambda_2)\,
%	\big(\sigma_{\psi,2,\lambda_1}^2\big)_{|x_d=0}\,a^1_{\psi,\lambda_2}
	\\\nonumber
	&\color{altorange2}+\indicatrice_{\lambda=k\lambda_{\psi}}\,\Gamma^{\psi}\,ik\,\big\lbrace\big(\sigma_{\nu,2,k}^1\big)_{|x_d=0}\,a^2_{\phi,-\lambda_{\phi}k}+\big(\sigma_{\nu,2,k}^2\big)_{|x_d=0}\,a^1_{\phi,-\lambda_{\phi}k}\big\rbrace\\\nonumber
	&\color{altorange2}+\partial_{z,\theta}\,\mbox{\emph{terms in}}\,\Big[\tilde{F}^2_{\psi},\big(\tilde{F}^2_{\psi},a^1_{\psi}\big),\big(\tilde{F}^2_{\phi},(\sigma^1_{\nu,2})_{|x_d=0}\big),\big((\sigma^2_{\psi})_{|x_d=0},a^1_{\psi}\big),\big((\sigma^2_{\phi})_{|x_d=0},(\sigma^1_{\nu,2})_{|x_d=0}\big)\Big]
	\\[5pt]\nonumber
	&\color{altorange2}=i\lambda\,b_{\psi}\cdot B\,\big(U^{3,\osc}_{\psi,2,\lambda}\big)_{|x_d,\chi_d=0}-i\lambda\,b_{\psi}\cdot H_{\lambda}\\\nonumber
	&\color{altorange2}+\partial_{z,\theta}\,\mbox{\emph{terms in}}\,\big(U_1,(I-P)\,U_2,(P\,U_2)_{\zeta\neq \phi,\psi,\nu},U_2^*\big)_{|x_d,\chi_d=0},
\end{align}
\color{black}
where $ X^{\Lop}_{\psi} $, $ D^{\Lop}_{\psi} $, $ \gamma_{\psi}(\lambda_1,\lambda_2) $, and $ \Gamma^{\psi} $ have already been defined in \eqref{eq 1cor eq X lop v psi etc}, and where 
\begin{equation*}
	\partial_{z,\theta}\,\mbox{\emph{terms in}}\,\Big[\tilde{F}^2_{\psi},\big(\tilde{F}^2_{\psi},a^1_{\psi}\big),\big(\tilde{F}^2_{\phi},(\sigma^1_{\nu,2})_{|x_d=0}\big),\big((\sigma^2_{\psi})_{|x_d=0},a^1_{\psi}\big),\big((\sigma^2_{\phi})_{|x_d=0},(\sigma^1_{\nu,2})_{|x_d=0}\big)\Big]
\end{equation*}
refers to linear terms in $ \tilde{F}^2_{\phi,j,\lambda} $ for $ j=1,2 $ and $ \lambda\in\Z^* $, and quadratic terms in $ \big(\tilde{F}^2_{\psi,j,\lambda},a^1_{\psi,\lambda}\big) $,  $ \big(\tilde{F}^2_{\phi,j,\lambda},(\sigma^1_{\nu,2,\lambda})_{|x_d=0}\big) $, $ \big((\sigma^2_{\psi,j,\lambda})_{|x_d=0},a^1_{\psi,\lambda}\big) $ or $ \big((\sigma^2_{\phi,j,\lambda})_{|x_d=0},(\sigma^1_{\nu,2,\lambda})_{|x_d=0}\big) $ for $ j=1,2 $ and $ \lambda\in\Z^* $, with, for all these terms, possibly one derivative in front of them. 
%we have denoted
%\begin{subequations}\label{eq 2cor eq X lop v psi etc}
%	\begin{align}
%		% Opérateur transport bis
%		\tilde{X}^{\Lop}_{\psi}:=&b_{\psi}\cdot B\Big(\mu_{\psi,1}\,R_{\psi,1}\,L(0,\partial_z)\,r_{\psi,1}+\mu_{\psi,3}\,R_{\psi,3}\,L(0,\partial_z)\,r_{\psi,3}\Big),\\
%				% Coefficient Burgers tilde
%				\tilde{D}^{\Lop}_{\psi}:=&b_{\psi}\cdot B\Big(\mu_{\psi,1}\,R_{\psi,1}\,L_1\big(e_{\psi,1},\psi_1\big)\,r_{\psi,1}+\mu_{\psi,3}\,R_{\psi,3}\,L_1\big(e_{\psi,3},\psi_3\big)\,r_{\psi,3}\Big),\\
%				%gamma lamnda_1,lambda_2 tilde
%				\tilde{\gamma}_{\psi}(\lambda_1,\lambda_2):=&\sum_{\substack{j_1=1,2,3,j_2=1,3\\ j_1\neq j_2}}\mu_{\psi,j_1}\,L\big(\lambda_1\psi_{j_1}+\lambda_2\psi_{j_2}\big)^{-1}\, \big\lbrace \lambda_2\,L_1(r_{\psi,j_1},\psi_{j_2})\,e_{\psi,j_2}\\\nonumber
%				&\qquad+\lambda_1\,L_1(e_{\psi,j_2},\psi_{j_1})\,r_{\psi,j_1}\big\rbrace.
%	\end{align}
%\end{subequations}
%We have denoted $ \mu_{\psi,2}:=1 $.

\bigskip

Note that equations \eqref{eq 2cor eq evol a phi} and \eqref{eq 2cor eq evol a psi} can be seen as linearizations around the trace of the leading profile $ U_1 $ of equations  \eqref{eq 1cor eq evol a phi} and \eqref{eq 1cor eq evol a psi}. This is usual in weakly nonlinear geometric optics, where equations for the leading profile are nonlinear, and equations for higher order are linearizations of the former equations around the leading profile $ U_1 $. Again, the obtained system of equations is still not closed since traces of the second corrector $ U_3 $ appear in amplitude equations \eqref{eq 2cor eq evol a phi} and \eqref{eq 2cor eq evol a psi}.

With the obtained equations, can have the intuition on how lower terms ascent toward higher order terms, eventually leading to an instability. In equation \eqref{eq 2cor eq evol a psi} for amplitudes $ a^2_{\psi,\lambda} $, the boundary forcing term $ H $ occurs, and therefore this forcing term ascents to first corrector profiles $ \sigma^2_{\psi,j,\lambda} $ for $ j=1,3 $ and $ \lambda\in\Z^* $ through boundary conditions \eqref{eq 1cor trace psi}. Eventually, because of the resonances \eqref{eq hyp res phi psi nu} leading to resonances terms in transport equations \eqref{eq 1cor evol sigma phi psi nu} for first corrector profiles, the boundary term $ H $ arises in profiles $ \sigma^2_{\psi,2,\lambda} $ for $ \lambda\in\Z^* $. In their turn, these profiles $ \sigma^2_{\psi,2,\lambda} $ for $ \lambda\in\Z^* $ interfere in amplitude equation \eqref{eq 1cor eq evol a psi} for $ a^1_{\psi,\lambda} $, for $ \lambda\in\Z^* $, because of the trace $ \big(U^2_{\psi,2,\lambda}\big)_{|x_d,\xi_d=0} $. Then this reasoning can be applied recursively to obtain that the boundary forcing term $ H $ interferes in leading profiles $ \sigma^1_{\zeta,j,\lambda} $, for $ \zeta=\phi,\psi,\nu $, $ j=1,2,3 $ and $ \lambda\in\Z^* $.

\subsection{General system}

The above arguments can be extended recursively to any corrector $ U_n $, $ n\geq 3 $. Doing so we get that the $ n $-th profile $ U_n $ reads
\begin{multline}\label{eq ecriture U n}
	U_n(z,\theta,\chi_d)=U^*_n(z)+\sum_{\mathbf{n}\in\B_{\Z^2}}\sum_{j\in\mathcal{C}(\mathbf{n})}\sum_{\lambda\in\Z^*}U^{n,\osc}_{\mathbf{n},j,\lambda}\,e^{i\lambda\,\mathbf{n}\cdot\theta}\,e^{i\lambda\xi_j(\mathbf{n}\cdot\boldsymbol{\zeta})\chi_d}\\
	+\sum_{\mathbf{n}\in\B_{\Z^2}}\sum_{\lambda\in\Z^*}e^{\chi_d\mathcal{A}(\mathbf{n}\cdot\boldsymbol{\zeta})}\,\Pi^e(\mathbf{n}\cdot\boldsymbol{\zeta})\,U^{n,\ev}_{\mathbf{n}}(z,0)\,e^{i\,\lambda\,\mathbf{n}\cdot\theta}+U^{n,\nc}(z,\theta,\chi_d),
\end{multline}
with $ U_n^* $ the mean value of $ U_n $, $ U^{n,\nc} $ the noncharacteristic terms, and where, for $ \zeta=\freq\in\F_{b}\privede{0} $, $ \mathbf{n}\in\B_{\Z^2} $, $ j\in\mathcal{C}(\zeta) $ and $ \lambda\in\Z^* $, the oscillating profile $ U^{n,\osc}_{\mathbf{n},j,\lambda} $ decomposes as 
\begin{equation*}
U^{n,\osc}_{\mathbf{n},j,\lambda}=\sigma^n_{\zeta\,j,\lambda}\,r_{\zeta,j}+\big(I-P_{\zeta,j}\big)\,U^{n,\osc}_{\mathbf{n},j,\lambda},
\end{equation*}
with $ \sigma_{\zeta,j,\lambda}^n $ a scalar function defined on $ \Omega_T $. According to equation \eqref{eq cascade int Un} for $ n-1 $ and since $ U_1 $ is polarized and of zero mean value, $ U^{n,\nc} $ is determined by the formula
\begin{align}\label{eq ncor expr U n nc}
	\mathcal{L}(\partial_{\theta},\partial_{\chi_d})\,U^{n,\nc}&=
	-\sum_{\substack{(\zeta_1,\zeta_2,j_1,j_2,\\\lambda_1,\lambda_2)\in\mathcal{NR}}}L_1(r_{\zeta_1,j_1},\alpha_{j_2}(\zeta_2))\,r_{\zeta_2,j_2}\,i\lambda\,\sigma_{\zeta_1,j_1,\lambda_1}^1\,\sigma_{\zeta_2,j_2,\lambda_2}^{n-1}\,\\\nonumber
	& \qquad\qquad e^{i(\lambda_1\mathbf{n}_1+\lambda_2\mathbf{n}_2)\cdot\theta}\,e^{i(\lambda_1\xi_{j_1}(\zeta_1)+\lambda_2\xi_{j_2}(\zeta_2))\chi_d}\\\nonumber
	&\quad+\partial_{z,\theta}\,\mbox{\emph{terms in}}\,\big(U_1,\dots,U_{n-2},(I-P)\,U_{n-1},U_{n-1}^*\big),
\end{align}
where notation $ \partial_{z,\theta}\,\mbox{\emph{terms in}}\,\big(U_1,\dots,U_{n-2},(I-P)\,U_{n-1},U_{n-1}^*\big) $ refers to quadratic terms involving the profiles $ U_1,\dots,U_{n-2} $, the nonpolarized parts of the corrector $ U_{n-1} $ and the mean value $ U_{n-1}^* $, with possibly first order derivatives in front of it. 
As for it, for $ \zeta=\freq\in\F_{b}\privede{0} $, $ \mathbf{n}\in\B_{\Z^2} $ and $ j\in\mathcal{C}(\zeta) $, $ \lambda\in\Z^* $, the nonpolarized part $ \big(I-P_{\zeta,j}\big)\,U^{n,\osc}_{\zeta,j,\lambda} $ of $ U^{n,\osc}_{\zeta,j,\lambda} $ is given by
\begin{align}\label{eq ncor nonpola part U n zeta}
	i\lambda\,\big(I-&P_{\zeta,j}\big)\,U^{n,\osc}_{\mathbf{n},j,\lambda}=\\\nonumber
	&-R_{\zeta,j}\,L(0,\partial_z)\,\sigma_{\zeta,j,\lambda}^{n-1}\,r_{\zeta,j}-R_{\zeta,j}\,L_1\big(r_{\zeta,j},\alpha_j(\zeta)\big)\,r_{\zeta,j}\sum_{\lambda_1+\lambda_2=\lambda}i\lambda\,\sigma_{\zeta,j,\lambda_1}^1\,\sigma_{\zeta,j,\lambda_2}^{n-1}\\\nonumber
	&-R_{\zeta,j}\,\indicatrice_{\lambda=k\lambda_{\zeta}}\sum_{\substack{(\zeta_1,\zeta_2,j_1,j_2)\\\in\mathcal{R}(\zeta,j)}}L_1\big(r_{\zeta_1,j_1},-\lambda_{\zeta_2}\alpha_{j_2}(\zeta_2)\big)\,r_{\zeta_2,j_2}\,ik\,\sigma_{\zeta_1,j_1,-k\lambda_{\zeta_1}}^1\,\sigma_{\zeta_2,j_2,-k\lambda_{\zeta_2}}^{n-1}\\\nonumber
	&-R_{\zeta,j}\,\partial_{z,\theta}\,\mbox{\emph{terms in}}\,\big(U_1,\dots,U_{n-2},(I-P)\,U_{n-1},U_{n-1}^*\big).
\end{align}
This formula is obtained by multiplying equation \eqref{eq cascade int Un} for $ n-1 $ by the partial inverse $ R_{\zeta,j} $, using that $ U_1 $ is polarized and of zero mean value. 

We specify now equations satisfied by the mean value $ U_n^* $ and the polarized components $ \sigma^n_{\zeta,j,\lambda} $. Since $ U_1 $ is polarized and of zero mean value, equations \eqref{eq cascade int Un}, \eqref{eq cascade bord U3n} for $ n $ and \eqref{eq cascade initial} lead to the following system for the mean value $ U_n^* $:
\begin{equation}\label{eq systeme moyenne}
	\color{altpink}\begin{cases}
		L(0,\partial_z)\,U_n^*=\partial_{z,\theta}\,\mbox{\emph{terms in}}\,\big(U_1,\dots,U_{n-1},(I-P)\,U_n\big)\\[5pt]
		B\,\big(U_n^*\big)_{|x_d=0}=\indicatrice_{n=3}\,H_0-B\,\big(U^{n,\nc}_{0}\big)_{|x_d,\chi_d=0}\\[5pt]
		B\,\big(U_n^*\big)_{|t\leq0}=0,
	\end{cases}
\end{equation}
where $ U^{n,\nc}_0 $ refers to the sum of all terms of $ U^{n,\nc} $ of which the trace on the boundary is of zero frequency.

For a frequency $ \zeta=\freq\in\F_{b}\privede{0,\phi,\psi,\nu} $, $ \mathbf{n}\in\B_{\Z^2} $, $ j\in\mathcal{C}(\zeta) $ and $ \lambda\in\Z^* $, multiplying equation \eqref{eq cascade int Un} by $ \ell_{\zeta,j} $ leads to, since all harmonics $ \sigma^1_{\zeta,j,\lambda'} $, $ \lambda\in\Z^* $ are zero and since $ U_1 $ is polarized and of zero mean value, 
\begin{subequations}\label{eq systeme sigma hors cas part}
\begin{align}
	\color{altpurple}X_{\alpha_j(\zeta)}\,\sigma^n_{\zeta,j,\lambda}&\color{altpurple}=\partial_{z,\theta}\,\mbox{\emph{terms in}}\,\big(U_1,\dots,U_{n-1},(I-P)\,U_n,U_n^*\big)\\[5pt]\color{altpurple}
	\big(\sigma^n_{\zeta,j,\lambda}\big)_{|t\leq 0}&\color{altpurple}= 0,
\intertext{\color{black}with initial condition \eqref{eq cascade initial}, and, if $ j\in\mathcal{I}(\zeta) $, boundary condition \eqref{eq cascade bord U3n} gives, since $ \zeta\in\F_{b}\setminus\Upsilon $,}
\color{altpurple}
	\big(\sigma^n_{\zeta,j,\lambda}\big)_{|x_d=0}&\color{altpurple}=-\ell_{\zeta,j}\,A_d(0)\,\big(E_-(\zeta)\big)^{-1}\,B\,\big(U^{n,\nc}_{\zeta,\lambda}\big)_{|x_d,\chi_d=0}.
\end{align}
\end{subequations}
\color{black}

Finally, for $ \zeta=\freq\in\ensemble{\phi,\psi,\nu} $, $ j\in\mathcal{C}(\zeta) $ and $ \lambda\in\Z^* $, multiplying equation \eqref{eq cascade int Un} by $ \ell_{\zeta,j} $ gives, with the same arguments,
\begin{subequations}\label{eq systeme sigma}
\begin{align}\label{eq systeme sigma evol}
	&\color{altblue}\quad X_{\zeta,j}\,\sigma^n_{\zeta,j,\lambda}+D_{\zeta,j}\sum_{\lambda_1+\lambda_2=\lambda}i\lambda\,\sigma_{\zeta,j,\lambda_1}^n\,\sigma_{\zeta,j,\lambda_2}^1+\indicatrice_{\lambda=k\lambda_{\zeta}}\sum_{\substack{(\zeta_1,\zeta_2,j_1,j_2)\\\in\mathcal{R}(\zeta,j)}}J^{\zeta_2,j_2}_{\zeta_1,j_1}\,ik\,\sigma^1_{\zeta_1,j_1,-k\lambda_{\zeta_1}}\,\sigma^n_{\zeta_2,j_2,-k\lambda_{\zeta_2}}\\\nonumber
	&\color{altblue}=\partial_{z,\theta}\,\mbox{\emph{terms in}}\,\big(U_1,\dots,U_{n-1},(I-P)\,U_n,U_n^*\big),
\end{align}
\color{black}
where the notation are defined by \eqref{eq prof princ def X D J}. 
Equation \eqref{eq systeme sigma evol} is coupled with the following initial condition,
\begin{equation}\label{eq systeme sigma init}
	\color{altblue}\big(\sigma^n_{\zeta,j,\lambda}\big)_{|t\leq 0}=0.
\end{equation}
\end{subequations}
It remains to determine the traces on the boundary of the corresponding incoming frequencies. The trace of the amplitudes associated with the boundary phase $ \phi $ on the boundary is given by, for $ j=1,3 $,
\begin{subequations}\label{eq systeme sigma bord}
\begin{equation}\label{eq systeme sigma bord phi}
	\color{altred}\big(\sigma_{\phi,j,\lambda}^n\big)_{|x_d=0}\,r_{\phi,j}=a_{\phi,\lambda}^n\,e_{\phi,j}
	%+\mu_{\phi,j}\,\big(\sigma_{\phi,2,\lambda}^n\big)_{|x_d=0}\,r_{\phi,j}
	+\tilde{F}^n_{\phi,j,\lambda},
\end{equation}
\color{black}
with $ a^n_{\phi,\lambda} $ a scalar function defined on $ \omega_T $ and
\begin{align*}
	\tilde{F}^n_{\phi,j,\lambda}&:=-\ell_{\phi,j}\cdot A_d(0)\,\big(B_{|E_-(\phi)}\big)^{-1}\Big(B\,(I-P_{\phi,1})\,\big(U^{n,\osc}_{\phi,1,\lambda}\big)_{|x_d,\chi_d=0}\\
	&\quad+B\,(I-P_{\phi,3})\,\big(U^{n,\osc}_{\phi,3,\lambda}\big)_{|x_d,\chi_d=0}+B\,(I-P_{\phi,2})\,\big(U^{n,\osc}_{\phi,2,\lambda}\big)_{|x_d,\chi_d=0}\\
	&\quad+B\,\big(\sigma_{\phi,2,\lambda}^n\big)_{|x_d=0}\,r_{\phi,2}+B\,\big(U^{n,\nc}_{\phi,\lambda}\big)_{|x_d,\chi_d=0}\Big)\,r_{\phi,j}.
\end{align*}
For $ \psi $ we have, for $ j=1,3 $, 
\begin{align}\label{eq systeme sigma bord psi}
\color{altorange2}	\big(\sigma_{\psi,j,\lambda}^n\big)_{|x_d=0}\,r_{\psi,j}=a_{\psi,\lambda}^n\,e_{\psi,j}%+\mu_{\psi,j}\,\big(\sigma_{\psi,2,\lambda}^n\big)_{|x_d=0}\,r_{\psi,j}
+\tilde{F}^n_{\psi,j,\lambda},
\end{align}
\color{black}
with $ a^n_{\psi,\lambda} $ a scalar function of $ \omega_T $ and
\begin{align*}
	\tilde{F}^n_{\psi,j,\lambda}&:=-\ell_{\psi,j}\cdot A_d(0)\,\big(B_{|E_-(\psi)}\big)^{-1}\Big(-\indicatrice_{n=3}\,H_{\lambda}+B\,(I-P_{\psi,1})\,\big(U^{n,\osc}_{\psi,1,\lambda}\big)_{|x_d,\chi_d=0}\\
	&\quad+B\,(I-P_{\psi,3})\,\big(U^{n,\osc}_{\psi,3,\lambda}\big)_{|x_d,\chi_d=0}+B\,(I-P_{\psi,2})\,\big(U^{n,\osc}_{\psi,2,\lambda}\big)_{|x_d,\chi_d=0}\\
	&\quad+B\,\big(\sigma_{\psi,2,\lambda}^n\big)_{|x_d=0}\,r_{\psi,2}+B\,\big(U^{n,\nc}_{\psi,\lambda}\big)_{|x_d,\chi_d=0}\Big)\,r_{\psi,j}.
\end{align*}
Note that $ \tilde{F}^n_{\phi,j,\lambda} $ and $ \tilde{F}^n_{\psi,j,\lambda} $, for $ j=1,3 $ and $ \lambda\in\Z^* $ depends respectively on the traces $ \big(\sigma_{\phi,2,\lambda}^n\big)_{|x_d=0} $ and $ \big(\sigma_{\psi,2,\lambda}^n\big)_{|x_d=0} $.
Finally, for $ \nu $,we have for $ j=1,3 $, 
\begin{align}\label{eq systeme sigma bord nu}
	\color{altgreen}\big(\sigma_{\nu,j,\lambda}^n\big)_{|x_d=0}\,r_{\nu,j}=\mu_{\nu,j}\,\big(\sigma_{\nu,2,\lambda}^n\big)_{|x_d=0}\,r_{\nu,j}+\tilde{F}^n_{\nu,j,\lambda},
\end{align}
\color{black}
with
\begin{multline*}
	\tilde{F}^n_{\nu,j,\lambda}:=-\ell_{\nu,j}\cdot A_d(0)\,\big(B_{|E_-(\nu)}\big)^{-1}\Big(B\,(I-P_{\nu,1})\,\big(U^n_{\nu,1,\lambda}\big)_{|x_d,\chi_d=0}\\
	+B\,(I-P_{\nu,3})\,\big(U^n_{\nu,3,\lambda}\big)_{|x_d,\chi_d=0}+B\,(I-P_{\nu,2})\,\big(U^n_{\nu,2,\lambda}\big)_{|x_d,\chi_d=0}+B\,\big(U^{n,\nc}_{\nu,\lambda}\big)_{|x_d,\chi_d=0}\Big)\,r_{\nu,j}.
\end{multline*}
\end{subequations}
Coefficients $ \mu_{\nu,j} $ have been introduced in \eqref{eq prof princ def mu nu}.

Scalar functions $ a^{n}_{\phi,\lambda} $ and $ a^n_{\psi,\lambda} $ satisfy  the following equations, which are derived using boundary condition \eqref{eq cascade bord U3n} and formulas \eqref{eq ncor expr U n nc} and \eqref{eq ncor nonpola part U n zeta},
\begin{subequations}\label{eq systeme a}
\begin{align}\label{eq systeme a phi}
	&\color{altred} X^{\Lop}_{\phi}\,a_{\phi,\lambda}^n+D^{\Lop}_{\phi}\,i\lambda\!\sum_{\lambda_1+\lambda_2=\lambda}\!a_{\phi,\lambda_1}^1\,a_{\phi,\lambda_2}^n+i\lambda\!\sum_{\lambda_1+\lambda_2=\lambda}\!\gamma_{\phi}(\lambda_1,\lambda_2)\big(a_{\phi,\lambda_1}^1\,a_{\phi,\lambda_2}^n+a_{\phi,\lambda_1}^n\,a_{\phi,\lambda_2}^1\big)\\\nonumber
	&\color{altred}+\indicatrice_{\lambda=k\lambda_{\phi}}\,\Gamma^{\phi}\,ik\,\Big\lbrace\big(\sigma_{\nu,2,-k}^1\big)_{|x_d=0}\,a_{\psi,-k\lambda_{\psi}}^n+\big(\sigma_{\nu,2,-k}^n\big)_{|x_d=0}\,a_{\psi,-k\lambda_{\psi}}^1\Big\rbrace\\[5pt]\nonumber
%	&\color{altred}+\indicatrice_{\lambda=k\lambda_{\phi}}\,\Gamma^{\phi}_2\,ik\,
%	\big(\sigma_{\nu,2,-k}^1\big)_{|x_d=0}\,\big(\sigma_{\psi,2,-k\lambda_{\psi}}^n\big)_{|x_d=0}
%\\\nonumber
	&\color{altred}+\partial_{z,\theta}\,\mbox{\emph{terms in}}\,\Big[\tilde{F}^n_{\phi},\big(\tilde{F}^n_{\phi},a^1_{\phi}\big),\big(\tilde{F}^n_{\psi},(\sigma^1_{\nu,2})_{|x_d=0}\big),\big((\sigma^n_{\phi})_{|x_d=0},a^1_{\phi}\big),\big((\sigma^n_{\psi})_{|x_d=0},(\sigma^1_{\nu,2})_{|x_d=0}\big)\Big]
\\[5pt]\nonumber
	&\color{altred}=i\lambda\,b_{\phi}\cdot B\,\big(U^{n+1,\osc}_{\phi,2,\lambda}\big)_{|x_d,\chi_d=0}\\\nonumber
	&\color{altred}+\partial_{z,\theta}\,\mbox{\emph{terms in}}\,\big(U_1,\dots,U_{n-1},(I-P)\,U_n,(P\,U_n)_{\zeta\neq \phi,\psi,\nu},U_n^*\big)_{|x_d,\chi_d=0},
\end{align}
\color{black}
and
\begin{align}\label{eq systeme a psi}
	&\color{altorange2}X^{\Lop}_{\psi}\,a^n_{\psi,\lambda}
	+D^{\Lop}_{\psi}\,i\lambda\!\sum_{\lambda_1+\lambda_2=\lambda}\!a^1_{\psi,\lambda_1}\,a^n_{\psi,\lambda_2}
	+i\lambda\!\sum_{\lambda_1+\lambda_2=\lambda}\!\gamma_{\psi}(\lambda_1,\lambda_2)\big(a_{\psi,\lambda_1}^1\,a_{\psi,\lambda_2}^n+a_{\psi,\lambda_1}^n\,a_{\psi,\lambda_2}^1\big)\\\nonumber
	%
%	&\color{altorange2}+\tilde{X}^{\Lop}_{\psi}\,\big(\sigma_{\psi,2,\lambda}^n\big)_{|x_d=0}
%	+\tilde{D}^{\Lop}_{\psi}\sum_{\lambda_1+\lambda_2=\lambda}i\lambda\,
%	\big(\sigma_{\psi,2,\lambda_1}^n\big)_{|x_d=0}\,a^1_{\psi,\lambda_2}
%	\\\nonumber
	%
%	&\color{altorange2}+i\lambda\sum_{\lambda_1+\lambda_2=\lambda}\tilde{\gamma}_{\psi}(\lambda_1,\lambda_2)\,
%	\big(\sigma_{\psi,2,\lambda_1}^n\big)_{|x_d=0}\,a^1_{\psi,\lambda_2}
%	\\\nonumber
	&\color{altorange2}+\indicatrice_{\lambda=k\lambda_{\psi}}\,\Gamma^{\psi}\,ik\,\big\lbrace\big(\sigma_{\nu,2,k}^1\big)_{|x_d=0}\,a^n_{\phi,-\lambda_{\phi}k}+\big(\sigma_{\nu,2,k}^n\big)_{|x_d=0}\,a^1_{\phi,-\lambda_{\phi}k}\big\rbrace\\\nonumber
	&\color{altorange2}+\partial_{z,\theta}\,\mbox{\emph{terms in}}\,\Big[\tilde{F}^n_{\psi},\big(\tilde{F}^n_{\psi},a^1_{\psi}\big),\big(\tilde{F}^n_{\phi},(\sigma^1_{\nu,2})_{|x_d=0}\big),\big((\sigma^n_{\psi})_{|x_d=0},a^1_{\psi}\big),\big((\sigma^n_{\phi})_{|x_d=0},(\sigma^1_{\nu,2})_{|x_d=0}\big)\Big]
	\\[5pt]\nonumber
	&\color{altorange2}=i\lambda\,b_{\psi}\cdot B\,\big(U^{n+1,\osc}_{\psi,2,\lambda}\big)_{|x_d,\chi_d=0}-\indicatrice_{n=3}\,b_{\psi}\cdot H_{\lambda}\\\nonumber
	&\color{altorange2}+\partial_{z,\theta}\,\mbox{\emph{terms in}}\,\big(U_1,\dots,U_{n-1},(I-P)\,U_n,(P\,U_n)_{\zeta\neq \phi,\psi,\nu},U_n^*\big)_{|x_d,\chi_d=0},
\end{align}
\color{black}
where the notation have been defined in \eqref{eq 1cor eq X lop v phi etc} and \eqref{eq 1cor eq X lop v psi etc}.
These two equations come with the following initial conditions
\begin{equation}\label{eq systeme init a phi psi}
	\color{altred}\big(a^n_{\phi,\lambda}\big)_{|t\leq 0}=0,\qquad \color{altorange2} \big(a^n_{\psi,\lambda}\big)_{|t\leq 0}=0.
\end{equation}
\end{subequations}

The system of equations \eqref{eq systeme moyenne}, \eqref{eq systeme sigma hors cas part}, \eqref{eq systeme sigma}, \eqref{eq systeme sigma bord} and \eqref{eq systeme a} is highly coupled. In all equations for the corrector of order $ n $, there are terms depending on $ U_1,\dots,U_{n-1},(I-P)\,U_n,(P\,U_n)_{\zeta\neq \phi,\psi,\nu},U_n^* $, but this is not a big issue, since, if the lower order correctors $ U_1,\dots,\linebreak U_{n-1} $ are constructed, $ (I-P)\,U_n,(P\,U_n)_{\zeta\neq \phi,\psi,\nu},U_n^* $ can be determined with \eqref{eq ncor nonpola part U n zeta}, \eqref{eq systeme moyenne} and \eqref{eq systeme sigma hors cas part}. The terms inducing coupling which seem the most problematic are the terms $ \lambda\,b_{\psi}\cdot B\,\big(U^{n+1,\osc}_{\phi,2,\lambda}\big)_{|x_d,\chi_d=0} $ and $ \lambda\,b_{\psi}\cdot B\,\big(U^{n+1,\osc}_{\psi,2,\lambda}\big)_{|x_d,\chi_d=0} $ in \eqref{eq systeme a phi} and \eqref{eq systeme a psi} which couple evolution equations for $ a^n_{\phi,\lambda} $ and $ a^n_{\psi,\lambda} $ (and therefore evolution equations for the corrector $ U^n $ of order $ n $), with the corrector of one order higher, $ U^{n+1} $. In equations \eqref{eq systeme a} there are also traces of profiles of order $ n $, which prevents to solve this equations (having determined lower order correctors) before solving the evolution equations for $ U^n $.

In addition to being highly coupled, system of equations \eqref{eq systeme moyenne}, \eqref{eq systeme sigma hors cas part}, \eqref{eq systeme sigma}, \eqref{eq systeme sigma bord} and \eqref{eq systeme a} seems also over-determined. Indeed, condition \eqref{eq prof princ cond psi 2 zero} imposing that the \emph{outgoing} leading profile $ \sigma^1_{\psi,2,\lambda} $ is of zero trace on the boundary gives one more boundary condition than the two boundary conditions prescribed by the structure of the problem. Therefore, this is not clear at all that the system of equations \eqref{eq systeme moyenne}, \eqref{eq systeme sigma hors cas part}, \eqref{eq systeme sigma}, \eqref{eq systeme sigma bord} and \eqref{eq systeme a} admits a solution satisfying the additional condition \eqref{eq prof princ cond psi 2 zero}.

\section{Existence of an analytic solution}\label{section existence}

In this section we focus on the well-posedness of \eqref{eq systeme moyenne}, \eqref{eq systeme sigma hors cas part}, \eqref{eq systeme sigma}, \eqref{eq systeme sigma bord} and \eqref{eq systeme a}. Both because of the high coupling of the system, and the over-determination of it, we choose to concentrate on a simplified version of the general system and try to prove well-posedness for it. This simplified model should focus on the profiles associated with frequencies $ \phi_j $, $ \psi_j $ and $ \nu_j $, because on one hand it greatly reduces the number of equations, and therefore complexity of the system and of the functional framework, and on the other hand because it seems that, due to amplification and resonances, equations on the profiles associated with $ \phi_j $, $ \psi_j $ and $ \nu_j $ carry the main difficulties of system of equations \eqref{eq systeme moyenne}, \eqref{eq systeme sigma hors cas part}, \eqref{eq systeme sigma} and \eqref{eq systeme a}. Indeed, we already pointed out that if profiles $ \sigma^n_{\zeta,j,\lambda} $ for $ n\geq 1 $, $ \zeta=\phi,\psi,\nu $, $ j=1,2,3 $ and $ \lambda\in\Z^* $ are determined, system \eqref{eq systeme moyenne}-\eqref{eq systeme sigma hors cas part} becomes upper triangular, and could be studied in a rather classical way, see for example \cite{Kilque2021Weakly}. Since we wish to study simplified versions of the system \eqref{eq systeme moyenne}, \eqref{eq systeme sigma hors cas part}, \eqref{eq systeme sigma}, \eqref{eq systeme sigma bord} and \eqref{eq systeme a} in an analytical setting, the initial conditions in \eqref{eq systeme moyenne}, \eqref{eq systeme sigma hors cas part}, \eqref{eq systeme sigma} and \eqref{eq systeme a}, requiring that the profiles $ \sigma_{\zeta,j,\lambda} $ and their boundary terms $ a_{\zeta,\lambda} $ are zero for negative times $ t $ are not suited for analytic functions, since it would imply that these profiles and boundary terms are zero everywhere. Therefore, in the simplified models, we modify, in a non equivalent way, these boundary conditions into conditions requiring that the solutions are zero at $ t=0 $, which are now adapted for analytic functions. 

We start by describing a first simplified model, very simple, which concentrate on boundary equations, and detail the functional framework which will be used to prove well-posedness of it, and proceed with the proof. Then we describe a second simplified model, more elaborate, which incorporates interior (incoming) equations, introduce additional functional framework , make some specifications on the simplified model with regard to the functional framework, and state the main result, before proceeding by proving it.

\subsection{First simplified model}

For the first simplified model, we focus only on boundary equations \eqref{eq systeme a}. In these equations, terms 
\begin{equation*}
	X^{\Lop}_{\phi}\,a_{\phi,\lambda}^n,\quad D^{\Lop}_{\phi}\,i\lambda\!\sum_{\lambda_1+\lambda_2=\lambda}\!a_{\phi,\lambda_1}^1\,a_{\phi,\lambda_2}^n\quad \text{and}\quad  i\lambda\!\sum_{\lambda_1+\lambda_2=\lambda}\!\gamma_{\phi}(\lambda_1,\lambda_2)\big(a_{\phi,\lambda_1}^1\,a_{\phi,\lambda_2}^n+a_{\phi,\lambda_1}^n\,a_{\phi,\lambda_2}^1\big)
\end{equation*}
as well as the analogous ones for $ \psi $ appear in the chosen first simplified model, and the last one is rewritten as a semilinear term. On the contrary, terms
\begin{align*}
	&\indicatrice_{\lambda=k\lambda_{\phi}}\,\Gamma^{\phi}\,ik\,\Big\lbrace\big(\sigma_{\nu,2,-k}^1\big)_{|x_d=0}\,a_{\psi,-k\lambda_{\psi}}^n+\big(\sigma_{\nu,2,-k}^n\big)_{|x_d=0}\,a_{\psi,-k\lambda_{\psi}}^1\Big\rbrace,\\[5pt]
	&\partial_{z,\theta}\,\mbox{\emph{terms in}}\,\Big[\tilde{F}^n_{\phi},\big(\tilde{F}^n_{\phi},a^1_{\phi}\big),\big(\tilde{F}^n_{\psi},(\sigma^1_{\nu,2})_{|x_d=0}\big),\big((\sigma^n_{\phi})_{|x_d=0},a^1_{\phi}\big),\big((\sigma^n_{\psi})_{|x_d=0},(\sigma^1_{\nu,2})_{|x_d=0}\big)\Big]
\end{align*}
and the analogous ones for $ \psi $ are removed in the simplified model, since they involved traces of outgoing interior profiles (recall that $ \tilde{F}^n_{\phi,j,\lambda} $ and $ \tilde{F}^n_{\psi,j,\lambda} $ depend respectively on $ \big(\sigma^n_{\phi,j,\lambda}\big)_{|x_d=0} $ and  $ \big(\sigma^n_{\psi,j,\lambda}\big)_{|x_d=0} $). For the same reasons, terms 
\begin{equation*}
	i\lambda\,b_{\phi}\cdot B\,\big(U^{n+1,\osc}_{\phi,2,\lambda}\big)_{|x_d,\chi_d=0}\quad \text{and}\quad i\lambda\,b_{\psi}\cdot B\,\big(U^{n+1,\osc}_{\psi,2,\lambda}\big)_{|x_d,\chi_d=0}
\end{equation*}
are not kept. Finally, source terms 
\begin{equation*}
	\partial_{z,\theta}\,\mbox{\emph{terms in}}\,\big(U_1,\dots,U_{n-1},(I-P)\,U_n,(P\,U_n)_{\zeta\neq \phi,\psi,\nu},U_n^*\big)_{|x_d,\chi_d=0}
\end{equation*}
involves quadratic terms in the traces of profiles $ \sigma^k_{\zeta,j,\lambda} $ for $ 1\leq k\leq n $, $ \zeta\in\F_{b}\privede{0} $, $ j\in\mathcal{C}(\zeta) $ and $ \lambda\in\Z^* $, with possibly derivatives of order up to $ n $ in front of it. They are simplified in three ways: we keep only traces of profiles associated with boundary frequencies $ \phi $, $ \psi $ and $ \nu $, we express traces only through functions $ a^k_{\zeta,\lambda} $ for $ \zeta=\phi,\psi $, and we choose only first order derivatives in $ \Theta $ (but we shall see in the following that considering derivatives in $ y $ would present no additional difficulty). Finally, boundary terms $ G $ and $ H $ are represented by functions $ H^n_{\zeta} $, belonging to a space specified later on. Multiplying equations \eqref{eq systeme a} by $ e^{i\lambda\Theta} $ for $ \Theta\in\T $ a periodic variable therefore leads to the following simplified model amplitude equations
\begin{subequations}\label{eq ana1 a}
	\begin{align}\label{eq ana1 a phi} 
		\color{altred}X^{\Lop}_{\phi}\,a_{\phi}^n+D^{\Lop}_{\phi}\,\partial_{\Theta}\big(a^1_{\phi}\,a^n_{\phi}\big)
		+\mathbf{w}_{\phi}\,\mathbb{F}_{\phi}^{\per}\big[\partial_{\Theta}\,a^1_{\phi},\partial_{\Theta}\,a^n_{\phi}\big]\color{altred}&\color{altred}=H^n_{\phi}+K^{\Lop}_{\phi}\sum_{k=1}^{n-1}\partial_{\Theta}\big(a^k_{\phi}\,a^{n-k}_{\psi}\big),\\\label{eq ana1 a psi}
		\color{altorange2}X^{\Lop}_{\psi}\,a_{\psi}^n+D^{\Lop}_{\psi}\,\partial_{\Theta}\big(a^1_{\psi}\,a^n_{\psi}\big)
		+\mathbf{w}_{\psi}\,\mathbb{F}_{\psi}^{\per}\big[\partial_{\Theta}\,a^1_{\psi},\partial_{\Theta}\,a^n_{\psi}\big]&\color{altorange2}=H^n_{\psi}+K^{\Lop}_{\psi}\sum_{k=1}^{n-1}\partial_{\Theta}\big(a^k_{\phi}\,a^{n-k}_{\psi}\big),
	\end{align}
\end{subequations}
where, for $ \zeta=\phi,\psi $, we have denoted by $ a^n_{\zeta} $ the function of $ \omega_T\times\T $ defined as
\begin{equation*}
	a^n_{\zeta}(z,\Theta):=\sum_{\lambda\in\Z^*}a^n_{\zeta,\lambda}(z)\,e^{i\lambda\,\Theta},
\end{equation*}
where, for $ \zeta=\phi,\psi $, the bilinear operator $ \mathbb{F}^{\per}_{\zeta} $ is defined as 
\begin{equation}\label{eq ana1 def F per}
	\mathbb{F}^{\per}_{\zeta}\big[a,b\big]:= \sum_{\lambda\in\Z^*}\sum_{\substack{\lambda_1+\lambda_3=\lambda\\\lambda_1\lambda_3\neq0}}\frac{i\,a_{\lambda_1}\,b_{\lambda_3}}{\lambda_1\,\delta^1_{\zeta}+\lambda_3\,\delta^3_{\zeta}}\,e^{i\lambda\,\Theta},
\end{equation}
with $ \delta^1_{\zeta} $ and $ \delta^3_{\zeta} $ scalars defined as
\begin{equation*}
	\delta^1_{\zeta}:=\frac{\xi_1(\zeta)-\xi_2(\zeta)}{\xi_3(\zeta)-\xi_1(\zeta)},\qquad \delta^3_{\zeta}:=\frac{\xi_3(\zeta)-\xi_2(\zeta)}{\xi_3(\zeta)-\xi_1(\zeta)},
\end{equation*}
and where $ \mathbf{w}_{\phi},\mathbf{w}_{\psi},K^{\Lop}_{\phi},K^{\Lop}_{\psi}\in\R $.
Here we have used analysis of \cite[Section 3.1]{CoulombelWilliams2017Mach} to rewrite terms involving the $ \gamma_{\zeta}(\lambda_1,\lambda_3) $ coefficients in \eqref{eq systeme a} as $ \mathbf{w}_{\phi}\,\mathbb{F}_{\phi}^{\per}\big[\partial_{\Theta}\,a^1_{\phi},\partial_{\Theta}\,a^n_{\phi}\big] $, up to changing definition of the coefficients $ D^{\Lop}_{\zeta} $. Note that since $ \phi $ and $ \psi $ are nonresonant, the denominators in equation \eqref{eq ana1 def F per} defining $ \mathbb{F}^{\per}_{\zeta} $ are nonzero. Up to changing all notation by a harmless multiplicative constant, we can assume that, for $ \zeta=\phi,\psi $, vector fields $ X_{\zeta}^{\Lop} $ read
\begin{equation}\label{eq ana1 def X Lop}
	X_{\zeta}^{\Lop}=\partial_t-\mathbf{v}^{\Lop}_{\zeta}\cdot \nabla_y,
\end{equation}
with $ \mathbf{v}^{\Lop}_{\zeta}\in\R^{d-1} $.
Equations \eqref{eq ana1 a} are coupled with the initial condi\-tions
\begin{equation}\label{eq ana1 a init}
	\color{altred}\big(a_{\phi}^n\big)_{|t=0}=0,\qquad \color{altorange2}\big(a_{\psi}^n\big)_{|t=0}=0.
\end{equation}
Again, these initial conditions \eqref{eq ana1 a init} are not the same as \eqref{eq systeme init a phi psi}, and are written in this (non equivalent) form to be suited for the analytical framework. Note that equations \eqref{eq ana1 a} are quasilinear for $ a^1_{\phi},a^1_{\psi} $ when $ n=1 $, and linear for $ a^n_{\phi},a^n_{\psi} $ when $ n\geq 2 $. As we will prove later that terms $ \mathbb{F}_{\phi}^{\per}\big[\partial_{\Theta}\,a^1_{\phi},\partial_{\Theta}\,a^n_{\phi}\big] $ and $ \mathbb{F}_{\psi}^{\per}\big[\partial_{\Theta}\,a^1_{\psi},\partial_{\Theta}\,a^n_{\psi}\big] $ are semilinear, equations \eqref{eq ana1 a} are transport equations, with a Burgers type term (when $ n=1 $), a semilinear term, a source term and a convolution type term. System of equations \eqref{eq ana1 a}, \eqref{eq ana1 a init} is a simplification of system \eqref{eq systeme a} for which we propose to set up the analytical tools to solve it.

\bigskip

The aim is to solve system \eqref{eq ana1 a}-\eqref{eq ana1 a init} with the following Cauchy-Kovalevskaya theorem. First proofs of this kind of result are due to \cite{Nirenberg1972CauchyKowalewski} and then \cite{Nishida1977Nirenberg}, and the proof of the following formulation goes back to \cite{BaouendiGoulaouic1978Nishida}.

\begin{theorem}[{\cite{BaouendiGoulaouic1978Nishida}}]\label{theorem Cauchy-Kovalevskaya}
	Let $ (B_r)_{r_0\leq r\leq r_1} $ be a decreasing sequence of Banach spaces (with $ 0\leq r_0\leq r_1\leq 1 $), i.e. such that, for $ r_0\leq r'<r\leq r_1 $, \begin{equation*}
		B_r\subset B_{r'},\quad \norme{.}_{r'}\leq \norme{.}_{r}.
	\end{equation*}
	Let $ T,R,C $ and $ M $ be positive real numbers, and consider a continuous function $ F $ from $ [-T,T]\times\ensemble{u\in B_r\,\middle|\,\norme{u}_r<R} $ to $ B_{r'} $ for every $ r_0\leq r'<r\leq r_1 $ which satisfies
	\begin{equation}\label{eq ana hyp F 1}
		\sup_{|t|\leq T}\norme{F(t,u)-F(t,v)}_{r'}\leq \frac{C}{r-r'}\,\norme{u-v}_r
	\end{equation}
	for all $ r_0\leq r'<r\leq r_1 $, $ |t|<T $, and for all $ u,v $ in $ B_r $ such that $ \norme{u}_r\leq R $, $ \norme{v}_r\leq R $, and 
	\begin{equation}\label{eq ana hyp F 2}
		\sup_{|t|\leq T}\norme{F(t,0)}_r\leq \frac{M}{r_1-r},
	\end{equation}
	for every $ r_0\leq r<r_1 $.
	
	Then there exists a real number $ \delta $ in $ (0,T) $ and a unique function $ u $, belonging to $ \mathcal{C}^1\big((-\delta(r_1-r),\delta(r_1-r)), B_r\big) $ for every $ r_0\leq r\leq r_1 $, satisfying
	\begin{equation*}
		\sup_{|t|<\delta(r_1-r)}\norme{u(t)}_r<R,
	\end{equation*}
	and the system
	\begin{equation*}
		\begin{cases}
			u'(t)=F\big(t,u(t)\big) & \mbox{for } |t|<\rho(r_1-r)\\
			u(0)=0.
		\end{cases}
	\end{equation*}
\end{theorem}

We therefore need to define a chain of Banach spaces of analytic functions adapted to our problem \eqref{eq ana1 a}.

\subsection{Functional framework}

\subsubsection{Functional spaces}

For a function $ u $ of $ L^2(\R^{d-1}) $, the symbol $ \hat{u} $ refers to the Fourier transform of $ u $, with the following convention
\begin{equation*}
	\hat{u}(\xi):=\int_{\R^{d-1}}u(y)\,e^{-i\,\xi\cdot y}\,dy, \quad \forall \xi\in\R^{d-1}.
\end{equation*}
For a complex vector $ X $, notation $ |X| $ refers to the norm $ \sqrt{X\cdot X^*} $, and we denote by $ \prodscalbis{.} $ the Japanese bracket, that is, for a complex vector $ X $, 
\begin{equation*}
	\prodscalbis{X}:=\big(1+|X|^2\big)^{1/2}.
\end{equation*}
We set $ d^* $ to be an integer such that $ d^*>\tilde{m}_0+2+(d+1)/2 $, where $ \tilde{m}_0 $ is the nonnegative real number of Lemma \ref{lemma hyp petit diviseurs}. The following definition quantifies analyticity by means of an exponential decay of the Fourier transform.

\begin{definition}
	For $ s\in(0,1) $, the space $ Y_{s} $ is defined as the space of all functions $ u $ of $ L^2(\R^{d-1}\times\T) $ such that, if their Fourier series expansion in $ \Theta $ reads
	\begin{equation*}
		u(y,\Theta)=\sum_{\lambda\in\Z}u_{\lambda}(y)\,e^{i\lambda\Theta},
	\end{equation*}
	then
	\begin{equation*}
		\norme{u}^2_{Y_s}:=\int_{\R^{d-1}}\sum_{\lambda\in\Z}e^{2s|(\lambda,\xi)|}\prodscalbis{(\lambda,\xi)}^{2d^*}\big|\hat{u_{\lambda}}(\xi)\big|^2\,d\xi<+\infty.
	\end{equation*}
\end{definition}

The following results make precise how $ y $ and $ \Theta $ derivatives act on $ Y_s $, and assert that $ Y_s $ is a Banach algebra.

\begin{lemma}\label{lemma ana dy dtheta Y s}
	There exists $ C>0 $ such that, for $ 0\leq s'<s\leq 1 $, for $ u $ in $ Y_s $, functions $ \nabla_y\, u $ and $ \partial_{\Theta}u $ belong to $ Y_{s'} $, and we have
	\begin{equation}\label{eq ana est dt dtheta Ys}
		\norme{\nabla_y\, u }_{Y_{s'}'}\leq \frac{C}{s-s'}\norme{u}_{Y_s}\quad\text{and}\quad \norme{\partial_{\Theta}u}_{Y_{s'}}\leq \frac{C}{s-s'}\norme{u}_{Y_s}.
	\end{equation} 
\end{lemma}

\begin{proof}
	We prove the estimate for $ \nabla_y\, u $, the one for $ \partial_{\Theta}u $ being similar. We have, by definition of the $ Y_{s'} $-norm,
	\begin{multline*}
		\norme{\nabla_y\, u}_{Y_{s'}}^2=\int_{\R^{d-1}}\sum_{\lambda\in\Z}e^{2s'|(\lambda,\xi)|}\prodscalbis{(\lambda,\xi)}^{2d^*}|\xi|^2\big|\hat{u_{\lambda}}(\xi)\big|^2\,d\xi\\
		\leq \frac{C^2}{(s-s')^2}\int_{\R^{d-1}}\sum_{\lambda\in\Z}e^{2s|(\lambda,\xi)|}\prodscalbis{(\lambda,\xi)}^{2d^*}\big|\hat{u_{\lambda}}(\xi)\big|^2\,d\xi= \frac{C^2}{(s-s')^2}\,\norme{u}^2_{s},
	\end{multline*}
	since $ |\xi|^2\exp(2s'|\xi|)\leq C^2\exp(2s|\xi|)/(s-s')^2 $ for $ \xi $ in $ \R^{d} $, with $ C>0 $ independent of $ s,s' $ and $ \xi $, which reads precisely $ C=2e^{-1} $.
\end{proof}

\begin{lemma}\label{lemma Ys Banach algebra}
	For $ s\in(0,1) $, the space $ Y_s $ is a Banach algebra, up to a positive constant, that is, there exists $ C>0 $ (independent of $ s $), such that for $ u,v$ in $ Y_s $, the function $ uv $ belongs to $ Y_s $ and we have
	\begin{equation*}
		\norme{uv}_{Y_s}\leq C\norme{u}_{Y_s}\norme{v}_{Y_s}.
	\end{equation*}
\end{lemma}

\begin{proof}
	Let $ s $ be in $ (0,1) $, and consider $ u,v $ in $ Y_s $. We have, 
	\begin{align*}
		&\quad\int_{\R^{d-1}}\sum_{\lambda\in\Z}e^{2s|(\lambda,\xi)|}\prodscalbis{(\lambda,\xi)}^{2d^*}\big|\hat{(uv)_{\lambda}}(\xi)\big|^2\,d\xi\\
		&=\int_{\R^{d-1}}\sum_{\lambda\in\Z}e^{2s|(\lambda,\xi)|}\prodscalbis{(\lambda,\xi)}^{2d^*}\left|\int_{\R^{d-1}}\sum_{\mu\in\Z}\hat{u_{\mu}}(\eta)\hat{v_{\lambda-\mu}}(\xi-\eta)\,d\eta\right|^2\,d\xi\\
		&\leq\int_{\R^{d-1}}\sum_{\lambda\in\Z}\left(\int_{\R^{d-1}}\sum_{\mu\in\Z}\frac{\prodscalbis{(\lambda,\xi)}^{2d^*}}{\prodscalbis{(\mu,\eta)}^{2d^*}\prodscalbis{(\lambda-\mu,\xi-\eta)}^{2d^*}}\,d\eta\right)\\
		&\quad\int_{\R^{d-1}}\sum_{\mu\in\Z}e^{2s|(\mu,\eta)|}\prodscalbis{(\mu,\eta)}^{2d^*}\big|\hat{u_{\mu}}(\eta)\big|^2 e^{2s|(\lambda-\mu,\xi-\eta)|}\prodscalbis{(\lambda-\mu,\xi-\eta)}^{2d^*}\big|\hat{v_{\lambda-\mu}}(\xi-\eta)\big|^2\,d\eta\,d\xi\\
		&\leq C \norme{u}^2_{Y_s}\norme{v}^2_{Y_s},
	\end{align*}
	by Cauchy-Schwarz inequality, if $ \int_{\R^{d-1}}\sum_{\mu\in\Z}\frac{\prodscalbis{(\lambda,\xi)}^{2d^*}}{\prodscalbis{(\mu,\eta)}^{2d^*}\prodscalbis{(\lambda-\mu,\xi-\eta)}^{2d^*}}\,d\eta $ is bounded uniformly with respect to $ (\lambda,\xi) $. We conclude by making the proof of this latter result.
	
	For $ \mu,\lambda\in\Z $ and $ \xi,\eta\in\R^{d-1} $, we have
	\begin{equation*}
		\big|(\lambda,\xi)\big|^2\leq 2 \big|(\mu,\eta)\big|^2+2 \big|(\lambda-\mu,\xi-\eta)\big|^2
	\end{equation*}
	so
	\begin{equation*}
		\prodscalbis{(\lambda,\xi)}^2\leq 2 \prodscalbis{(\mu,\eta)}^2+2 \prodscalbis{(\lambda-\mu,\xi-\eta)}^2,
	\end{equation*} 
and 
\begin{equation*}
	\prodscalbis{(\lambda,\xi)}^{2d^*}\leq 2^{d^*+1} \prodscalbis{(\mu,\eta)}^{2d^*}+2^{d^*+1} \prodscalbis{(\lambda-\mu,\xi-\eta)}^{2d^*}
\end{equation*} 
Therefore, by a change of variables $ (\eta,\mu)=(\xi-\eta,\lambda-\mu) $,
\begin{equation*}
	\int_{\R^{d-1}}\sum_{\mu\in\Z}\frac{\prodscalbis{(\lambda,\xi)}^{2d^*}}{\prodscalbis{(\mu,\eta)}^{2d^*}\prodscalbis{(\lambda-\mu,\xi-\eta)}^{2d^*}}\,d\eta \leq  2^{d^*+2}\int_{\R^{d-1}}\sum_{\mu\in\Z}\inv{\prodscalbis{(\mu,\eta)}^{2d^*}}\,d\eta \leq C\,2^{d^*+2}
\end{equation*}
with $ C $ depending only on $ d^* $, since $ d^* $ is such that $ d^*>d/2 $.
\end{proof}

As we work with sequences of functions $ (a_{\phi}^n)_{n\geq 1} $ and $ (a_{\psi}^n)_{n\geq 1} $, we define a functional space accordingly. We also specify the norm chosen on the product space, since we will work with couples of sequences.

\begin{definition}\label{def Y s suite}
	For $ s\in(0,1) $, the space $ \Y_{s} $ is defined as the set of sequences $ \mathbf{a}=\big(a_n\big)_{n\in\mathbb{N}} $ of $ Y_{s} $ such that
	\begin{equation*}
		\normetriple{\mathbf{a}}^2_{\Y_{s}}:=\sum_{n\geq 1}e^{2sn}\prodscalbis{n}^{2d^*}\normetriple{a_n}^2_{Y_s}<+\infty.
	\end{equation*}
For $ s\in(0,1) $, the norm on the product space $ \Y_s^2 $ is defined, for $ (\a,\b)\in\Y_s^2 $, as
\begin{equation*}
	\normetriple{\big(\a,\b\big)}_{\Y_s^2}^2:=\normetriple{\a}_{\Y_s}^2+\normetriple{\b}_{\Y_s}^2.
\end{equation*}
\end{definition}

The space $ \Y_s $ satisfies analogous properties as the ones of Lemmas \ref{lemma ana dy dtheta Y s} and \ref{lemma Ys Banach algebra} (with the convolution on sequences for product), but they will not be used directly. 

\subsubsection{Specifications on the simplified model}

We are now able to precise some properties of the study system \eqref{eq ana1 a}, \eqref{eq ana1 a init}.
%, and, when needed, explain how these assumptions are reasonable in view of the general system \eqref{eq systeme a}.

Boundary source terms $ H^n_{\phi} $, $ H^n_{\psi} $ for $ n\geq 1 $ are taken such that, defining $ \mathbf{H}:=\big(\mathbf{H}_{\phi},\mathbf{H}_{\psi}\big):=\big(H^n_{\phi},H^n_{\psi}\big)_{n\geq 1} $, function $ \mathbf{H} $ is in $ \mathcal{C}\big([-T,T],Y^2_1\big) $. In the statement \ref{prop WP bord CK} below of existence and uniqueness for system \eqref{eq ana1 a}, \eqref{eq ana1 a init}, there will be an additional assumption on $ \mathbf{H} $, requiring that there exists $ M>0 $ such that, for $ 0\leq s<1 $,
\begin{equation*}
	\sup_{|t|<T}\normetriple{\mathbf{H}}_{\Y^2_s}\leq \frac{M}{1-s}.
\end{equation*}
This assumption on $ \mathbf{H} $ is stronger than requiring  $ H $ and $ G $ of \eqref{eq systeme 1} to be in $ H^{\infty}(\R^d\times\T) $, as it imposes analyticity with respect to space variables, but with bound on the norm increasing with regularity. 
We denote by $ \gamma_0>0 $ a positive constant such that,	for $ \zeta=\phi,\psi $,
\begin{subequations}\label{eq ana1 hyp toy model}
	\begin{equation}\label{eq ana1 hyp toy model v D bord}
		\big|\mathbf{v}^{\Lop}_{\zeta}\big|\leq \gamma_0,\qquad \big|D^{\Lop}_{\zeta}\big|\leq \gamma_0, \quad |\mathbf{w}_{\zeta}|\leq\gamma_0^{1/2}\quad \text{and}\quad \big|K^{\Lop}_{\zeta}\big|\leq \gamma_0,
	\end{equation}
	and, for $ s\in(0,1) $, for $ u,v $ in $ Y_s $, and for $ \zeta=\phi,\psi $,
	\begin{equation}\label{eq ana1 hyp toy model F per}
		\norme{\mathbb{F}^{\per}_{\zeta}\big[\partial_{\Theta}\,u,\partial_{\Theta}\,v\big]}_{Y_s}\leq \gamma_0^{1/2}\,\normetriple{u}_{Y_s}\normetriple{v}_{Y_s}.
	\end{equation}
\end{subequations}
All estimates relies on the fact that scalars $ D^{\Lop}_{\zeta} $, $ \mathbf{w}_{\zeta} $ and $ K^{\Lop}_{\zeta} $, vectors $ \mathbf{v}^{\Lop}_{\zeta} $ and operators $ \mathbb{F}^{\per}_{\zeta} $ are indexed by finite sets. As for it, estimate \eqref{eq ana1 hyp toy model F per} asserting  that the operator $ \mathbb{F}^{\per}_{\zeta} $, composed with derivation in $ \Theta $, acts as a semilinear operator, is a result of \cite[Theorem 3.1]{CoulombelWilliams2017Mach}. The proof in our case is a straightforward adaptation to our functional framework of the one of \cite{CoulombelWilliams2017Mach}, which we detail here. 

\begin{proposition}[{\cite[Theorem 3.1]{CoulombelWilliams2017Mach}}]\label{prop F per semilin}
	There exists a constant $ C>0 $ such that for $ s\in[0,1] $, for $ u,v $ in $ Y_s $ and $ \zeta=\phi,\psi $, we have 
	\begin{equation}\label{eq ana1 est prop F per}
		\norme{\mathbb{F}^{\per}_{\zeta}\big[\partial_{\Theta}\,u,\partial_{\Theta}\,v\big]}_{Y_s}\leq C\,\normetriple{u}_{Y_s}\normetriple{v}_{Y_s}.
	\end{equation}
\end{proposition}

\begin{proof}
	The result relies on the following lemma, which constitutes a reformulation of the small divisors Assumption \ref{hyp petits diviseurs}. Its proof is the same as the one in \cite{CoulombelWilliams2017Mach}, and is recalled here for the sake of completeness.  
	
	\begin{lemma}[{\cite[Lemma 3.2]{CoulombelWilliams2017Mach}}]\label{lemma hyp petit diviseurs}
		There exists a constant $ C>0 $ and a real number $ \tilde{m}_0 $ such that, for $ \lambda_1,\lambda_3\in\Z^* $, and for $ \zeta=\phi,\psi $, we have
		\begin{equation}\label{eq ana1 est lemme hyp petits div}
			\inv{\big|\lambda_1\,\delta^1_{\zeta}+\lambda_3\,\delta^3_{\zeta}\big|}\leq C \min\big(|\lambda_1|^{\tilde{m}_0},|\lambda_3|^{\tilde{m}_0}\big).
		\end{equation}
	\end{lemma}

\begin{proof}
	Without loss of generality, we consider $ \zeta=\phi $. The aim is to use the bound of Assumption \ref{hyp petits diviseurs}. Using equality $ L(0,\phi_2)\,r_{\phi,2}=0 $, we get,
	\begin{equation*}
		L\big(0,\lambda_1\,\phi_1+\lambda_3\,\phi_3\big)\,r_{\phi,2}=\Big[\lambda_1\big(\xi_1(\phi)-\xi_2(\phi)\big)+\lambda_3\big(\xi_3(\phi)-\xi_2(\phi)\big)\Big]A_d(0)\,r_{\phi,2},
	\end{equation*}
so the quantity $ \lambda_1\big(\xi_1(\phi)-\xi_2(\phi)\big)+\lambda_3\big(\xi_3(\phi)-\xi_2(\phi)\big) $ is nonzero since otherwise $ r_{\phi,2} $ would be a nonzero vector in the kernel of $ L\big(0,\lambda_1\,\phi_1+\lambda_3\,\phi_3\big) $, contradicting Assumption \ref{hypothese ensemble frequences} asserting that $ \lambda_1\,\phi_1+\lambda_3\,\phi_3 $ is never characteristic for $ \lambda_1,\lambda_3\in\Z^* $. Therefore we have
\begin{equation*}
	\inv{\Big|\lambda_1\big(\xi_1(\phi)-\xi_2(\phi)\big)+\lambda_3\big(\xi_3(\phi)-\xi_2(\phi)\big)\Big|}\leq C \norme{L\big(0,\lambda_1\,\phi_1+\lambda_3\,\phi_3\big)^{-1}},
\end{equation*}
with a constant $ C>0 $ independent on $ \lambda_1,\lambda_3 $. Using Assumption \ref{hyp petits diviseurs} and a polynomial bound on the transpose of the comatrix, we get that there exists a nonnegative real number $ \tilde{m}_0 $ such that
\begin{equation*}
	\inv{\Big|\lambda_1\big(\xi_1(\phi)-\xi_2(\phi)\big)+\lambda_3\big(\xi_3(\phi)-\xi_2(\phi)\big)\Big|}\leq C \big|(\lambda_1,\lambda_3)\big|^{\tilde{m}_0},
\end{equation*}
with a new constant $ C>0 $ independent on $ \lambda_1,\lambda_3 $. Up to changing constant $ C>0 $, we obtain,
\begin{equation}\label{eq ana inter 1}
	\inv{\big|\lambda_1\,\delta^1_{\phi}+\lambda_3\,\delta^3_{\phi}\big|}\leq C \big|(\lambda_1,\lambda_3)\big|^{\tilde{m}_0}.
\end{equation}
To get the formulation of \eqref{eq ana inter 1} with a minimum, we see that two cases may occur. Either $ \big|\lambda_1\,\delta^1_{\phi}+\lambda_3\,\delta^3_{\phi}\big|>|\delta^1_{\phi}| $, and in this case, with $ C\geq 1/|\delta^1_{\phi}| $, we have
	\begin{equation*}
		\inv{\big|\lambda_1\,\delta^1_{\phi}+\lambda_3\,\delta^3_{\phi}\big|}\leq \inv{\big|\delta^1_{\phi}\big|}\leq C \leq C \,|\lambda_1|^{\tilde{m}_0}
	\end{equation*}
since $ \tilde{m}_0\geq 0 $. In the other case, if $ \big|\lambda_1\,\delta^1_{\phi}+\lambda_3\,\delta^3_{\phi}\big|\leq|\delta^1_{\phi}| $, we have
\begin{equation*}
	|\lambda_3|\leq \inv{\big|\delta^3_{\phi}\big|}\big|\lambda_1\,\delta^1_{\phi}+\lambda_3\,\delta^3_{\phi}\big|+\inv{\big|\delta^3_{\phi}\big|}\big|\lambda_1\,\delta^1_{\phi}\big|\leq 2\frac{\big|\delta^1_{\phi}\big|}{\big|\delta^3_{\phi}\big|}|\lambda_1|,
\end{equation*}
so, up to changing constant $ C $, estimate \eqref{eq ana inter 1} rewrites 
\begin{equation*}
	\inv{\big|\lambda_1\,\delta^1_{\phi}+\lambda_3\,\delta^3_{\phi}\big|}\leq C \,|\lambda_1|^{\tilde{m}_0}.
\end{equation*}
Applying the same arguments for $ \lambda_3 $ leads to the aimed estimate \eqref{eq ana1 est lemme hyp petits div}.
\end{proof}

The proof of Proposition \ref{prop F per semilin} also relies on the following technical result, whose formulation is the one of \cite{CoulombelWilliams2017Mach}. Its proof is an immediate adaptation of a result of \cite{RauchReed1982Microlocal}, and is not recalled here.

\begin{lemma}[{\cite[Lemma 1.2.2]{RauchReed1982Microlocal}},{\cite[Lemma 3.3]{CoulombelWilliams2017Mach}}]\label{lemma technique Cou}
	Let $ \mathbb{K} : \R^{d-1}\times\Z\times \R^{d-1}\times\Z \rightarrow \C $  be a locally integrable measurable function such that, either
	\begin{equation*}
		\sup_{(\xi,\lambda)\in \R^{d-1}\times\Z}\int_{\R^{d-1}}\sum_{\mu\in\Z}\big|\mathbb{K}(\xi,\lambda,\eta,\mu)\big|^2\,d\eta<+\infty,
	\end{equation*}
or
\begin{equation*}
\sup_{(\eta,\mu)\in \R^{d-1}\times\Z}\int_{\R^{d-1}}\sum_{\lambda\in\Z}\big|\mathbb{K}(\xi,\lambda,\eta,\mu)\big|^2\,d\xi<+\infty.
\end{equation*}
Then the map
\begin{equation*}
	(f,g)\mapsto \int_{\R^{d-1}}\sum_{\mu\in\Z}\mathbb{K}(\xi,\lambda,\eta,\mu)\,f(\xi-\eta,\lambda-\mu)\,g(\eta,\mu)\,d\eta
\end{equation*}
is bounded on $ L^2(\R^{d-1}\times\Z)\times L^2(\R^{d-1}\times\Z) $ with values in $ L^2(\R^{d-1}\times\Z) $.
\end{lemma}
	
	We now proceed with the proof of  Proposition \ref{prop F per semilin}, and we consider without loss of generality $ \zeta=\phi $. For $ \xi\in\R^{d-1} $, $ \lambda\in\Z $ and for $ u,v $ in $ Y_s $, the Fourier transform of the $ \lambda $-th term of the Fourier series expansion of $ \mathbb{F}^{\per}_{\zeta}\big[\partial_{\Theta}\,u,\partial_{\Theta}\,v\big] $ is given by 
	\begin{equation*}
		\widehat{\mathbb{F}^{\per}_{\zeta}\big[\partial_{\Theta}\,u,\partial_{\Theta}\,v\big]_{\lambda}}(\xi)=-i\int_{\R^{d-1}}\sum_{\substack{\mu\in\Z \\ \mu\neq 0, \lambda}}\frac{\mu(\lambda-\mu)}{(\lambda-\mu)\delta^1_{\phi}+\mu\delta^3_{\phi}}\,\hat{u_{\lambda-\mu}}(\xi-\eta)\,\hat{v_{\mu}}(\eta)\,d\eta.
	\end{equation*}
	Therefore, to obtain inequality \eqref{eq ana1 est prop F per}, we have to estimate the quantity
	\begin{equation*}
		\int_{\R^{d-1}}\sum_{\lambda\in\Z}\left|\int_{\R^{d-1}}\sum_{\mu\in\Z}e^{s|(\lambda,\xi)|}\prodscalbis{(\lambda,\xi)}^{d^*}\mathbf{F}(\lambda,\mu)\,\hat{u_{\lambda-\mu}}(\xi-\eta)\,\hat{v_{\mu}}(\eta)\,d\eta\right|^2d\xi,
	\end{equation*}
where we have denoted, for $ \lambda,\mu\in\Z $,
\begin{equation*}
	\mathbf{F}(\lambda,\mu):=\begin{cases}
		\frac{\mu(\lambda-\mu)}{(\lambda-\mu)\delta^1_{\phi}+\mu\delta^3_{\phi}} & \text{if } \mu \neq 0,\lambda\\
		0 & \text{otherwise,}
	\end{cases}
\end{equation*}
and we will do it using Lemma \ref{lemma technique Cou}.
We consider two nonnegative functions $ \chi_1,\chi_2 $ on $ \R^d\times\R^d $ such that $ \chi_1+\chi_2\equiv 1 $ and 
\begin{align*}
	\chi_1(\xi,\lambda,\eta,\mu)=0 &\quad \text{if } \prodscalbis{(\eta,\mu)}\geq (2/3)\prodscalbis{(\xi,\lambda)}\\
	\chi_2(\xi,\lambda,\eta,\mu)=0 &\quad \text{if } \prodscalbis{(\eta,\mu)}\leq (1/3)\prodscalbis{(\xi,\lambda)}.
\end{align*}
%We wish to apply Lemma \ref{lemma technique Cou}, and w
We first consider the quantity
\begin{subequations}\label{eq ana inter 2}
\begin{align}
	&\int_{\R^{d-1}}\sum_{\mu\in\Z}\chi_1(\xi,\lambda,\eta,\mu)\,e^{s|(\lambda,\xi)|}\prodscalbis{(\lambda,\xi)}^{d^*}\mathbf{F}(\lambda,\mu)\,\hat{u_{\lambda-\mu}}(\xi-\eta)\,\hat{v_{\mu}}(\eta)\,d\eta,
	\intertext{rewritten as}
	&\int_{\R^{d-1}}\sum_{\mu\in\Z}\frac{\chi_1(\xi,\lambda,\eta,\mu)\,e^{s|(\lambda,\xi)|}\prodscalbis{(\lambda,\xi)}^{d^*}\mathbf{F}(\lambda,\mu)}{e^{s|(\lambda-\mu,\xi-\eta)|}\prodscalbis{(\lambda-\mu,\xi-\eta)}^{d^*}
		e^{s|(\mu,\eta)|}\prodscalbis{(\mu,\eta)}^{d^*}}\,\\\nonumber
	&\qquad\times\Big(e^{s|(\lambda-\mu,\xi-\eta)|}\prodscalbis{(\lambda-\mu,\xi-\eta)}^{d^*}\hat{u_{\lambda-\mu}}(\xi-\eta)\Big)\Big(e^{s|(\mu,\eta)|}\prodscalbis{(\mu,\eta)}^{d^*}\hat{v_{\mu}}(\eta)\Big)d\eta.
\end{align}
\end{subequations}
We have
\begin{equation*}
	e^{2s|(\lambda,\xi)|} \leq e^{2s|(\mu,\eta)|} e^{2s|(\lambda-\mu,\xi-\eta)|},
\end{equation*}
and, on the support of $ \chi_1 $,
\begin{equation*}
	\prodscalbis{(\lambda-\mu,\xi-\eta)}\geq \prodscalbis{(\lambda,\xi)}-\prodscalbis{(\mu,\eta)}\geq \inv{3}\prodscalbis{(\lambda,\xi)},
\end{equation*}
so
\begin{equation*}
	\int_{\R^{d-1}}\sum_{\mu\in\Z}\left|\frac{\chi_1(\xi,\lambda,\eta,\mu)\,e^{s|(\lambda,\xi)|}\prodscalbis{(\lambda,\xi)}^{d^*}\mathbf{F}(\lambda,\mu)}{e^{s|(\lambda-\mu,\xi-\eta)|}\prodscalbis{(\lambda-\mu,\xi-\eta)}^{d^*}
		e^{s|(\mu,\eta)|}\prodscalbis{(\mu,\eta)}^{d^*}}\right|^2d\eta\leq C\int_{\R^{d-1}}\sum_{\mu\in\Z}\left|\frac{\mathbf{F}(\lambda,\mu)}{\prodscalbis{(\mu,\eta)}^{d^*}}\right|^2d\eta.
\end{equation*}
Using Lemma \ref{lemma hyp petit diviseurs} we get
\begin{equation*}
	\left|\frac{\mu(\lambda-\mu)}{(\lambda-\mu)\delta^1_{\phi}+\mu\delta^3_{\phi}}\right|=\inv{\big|\delta^1_{\phi}\big|}\left|\mu-\frac{\delta^3_{\phi}\,\mu^2}{(\lambda-\mu)\delta^1_{\phi}+\mu\delta^3_{\phi}}\right|\leq C |\mu|^{\tilde{m}_0+2},
\end{equation*}
so 
\begin{multline*}
	\sup_{(\xi,\lambda)\in \R^{d-1}\times\Z}\int_{\R^{d-1}}\sum_{\mu\in\Z}\left|\frac{\chi_1(\xi,\lambda,\eta,\mu)\,e^{s|(\lambda,\xi)|}\prodscalbis{(\lambda,\xi)}^{d^*}\mathbf{F}(\lambda,\mu)}{e^{s|(\lambda-\mu,\xi-\eta)|}\prodscalbis{(\lambda-\mu,\xi-\eta)}^{d^*}
		e^{s|(\mu,\eta)|}\prodscalbis{(\mu,\eta)}^{d^*}}\right|^2d\eta\\
	\leq C\int_{\R^{d-1}}\sum_{\mu\in\Z}\frac{|\mu|^{2(\tilde{m}_0+2)}}{\prodscalbis{(\mu,\eta)}^{2d^*}}\,d\eta<+\infty,
\end{multline*}
since we chose $ d^*>\tilde{m}_0+2+(d+1)/2 $. Applying Lemma \ref{lemma technique Cou} to the quantity in \eqref{eq ana inter 2} we obtain
\begin{align*}
	&\int_{\R^{d-1}}\sum_{\lambda\in\Z}\left|\int_{\R^{d-1}}\sum_{\mu\in\Z}\chi_1(\xi,\lambda,\eta,\mu)\,e^{s|(\lambda,\xi)|}\prodscalbis{(\lambda,\xi)}^{d^*}\mathbf{F}(\lambda,\mu)\,\hat{u_{\lambda-\mu}}(\xi-\eta)\,\hat{v_{\mu}}(\eta)\,d\eta\right|^2d\xi\\[5pt]
	&\leq C \norme{u}^2_{Y_s}\norme{v}^2_{Y_s}.
\end{align*}
Applying similar arguments for $ \chi_2 $ leads to the analogous estimate for $ \chi_2 $, and combining estimates for $ \chi_1 $ and $ \chi_2 $ gives the sought one \eqref{eq ana1 est prop F per}, concluding the proof.
\end{proof}

We are now able to prove well-posedness for system \eqref{eq ana1 a}, \eqref{eq ana1 a init}, using Theorem \ref{theorem Cauchy-Kovalevskaya}, with the above properties. 

\subsection{A Cauchy-Kovalevskaya theorem for boundary equations}

System \eqref{eq ana1 a}, \eqref{eq ana1 a init} reads 
\begin{equation}\label{eq ana1 dt a}
	\begin{cases}
		\partial_t \,\mathbf{a} = \mathbf{F}\big(t,\a):= L \,\mathbf{a} -\partial_{\Theta}\,\mathbf{D}^{\Lop}\big(\a,\a\big) - \mathbb{F}^{\per}\big(\partial_{\Theta}\,\a,\partial_{\Theta}\,\a\big)+\mathbf{H}+\mathbf{K}^{\Lop}\big(\a,\a\big),\\[5pt]
		\a(0)=0,
	\end{cases}
\end{equation}
where $ \a:=(\a_{\phi},\a_{\psi}):=\big(a^n_{\phi},a^n_{\psi}\big)_{n\geq 1} $, and, if $ \mathbf{c}:=\big(c^n_{\phi},c^n_{\psi}\big)_{n\geq 1} $, 
\begin{align*}
	L\,\a&:=\big(\mathbf{v}_{\phi}^{\Lop}\cdot\nabla_y\,a^n_{\phi},\mathbf{v}_{\phi}^{\Lop}\cdot\nabla_y\,a^n_{\psi}\big)_{n\geq 1}, \\[5pt]
	\mathbf{D}^{\Lop}(\a,\mathbf{c})&:=\big(D^{\Lop}_{\phi}\,a^1_{\phi}\,c^n_{\phi}, D^{\Lop}_{\psi}\,a^1_{\psi}\,c^n_{\psi}\big), \\[5pt]
	\mathbb{F}^{\per}(\a,\mathbf{c})&:=\big(\mathbf{w}_{\phi}\,\mathbb{F}^{\per}_{\phi}[a^1_{\phi},c^{n}_{\phi}], \mathbf{w}_{\psi}\linebreak\mathbb{F}^{\per}_{\psi}[a^1_{\psi},c^{n}_{\psi}]\big)_{n\geq 1}, \\[5pt]
	\mathbf{K}^{\Lop}(\a,\mathbf{c})&:=\Big(K^{\Lop}_{\phi}\sum_{k=1}^{n-1}\partial_{\Theta}\big(a^k_{\phi}\,c^{n-k}_{\psi}\big),K^{\Lop}_{\psi}\sum_{k=1}^{n-1}\partial_{\Theta}\big(a^k_{\phi}\,c^{n-k}_{\psi}\big)\Big)_{n\geq 1}.
\end{align*}

System \eqref{eq ana1 dt a} is now in the right shape to apply Theorem \ref{theorem Cauchy-Kovalevskaya}, and we prove the following result. 

\begin{proposition}\label{prop WP bord CK}
	For $ M_0>0 $, there exists $ \delta\in(0,T) $ such that for all $ \mathbf{H} $ in $ \mathcal{C}\big([-T,T],\Y_1^2) $ satisfying, for $ 0<s<1 $,
	\begin{equation}\label{eq ana1 hyp H WP}
		\sup_{|t|<T}\normetriple{\mathbf{H}(t)}_{Y^2_s}\leq \frac{M_0}{1-s},
	\end{equation}
	system \eqref{eq ana1 dt a} admits a unique solution $ \a $ in $ \mathcal{C}^1\big((-\delta(1-s),\delta(1-s)), \Y^2_s\big) $ for every $ 0<s\leq 1 $.
\end{proposition}

The key estimates to prove this result are the following. These are classical, and their proof is recalled here for the sake of completeness. 

\begin{lemma}\label{lemma ana1 est LBK Y s}
	There exists a constant $ C>0 $ such that for $ 0<s'<s\leq 1 $, for $ \b,\mathbf{c} $ in $ \Y_s^2 $, the following estimates hold
	\begin{subequations}\label{eq ana1 est LBK Y s}
		\begin{align}
			\label{eq ana1 est LBK Y s:L}
			\normetriple{L\,\b}_{\Y_{s'}^2}&\leq  \frac{C\,\gamma_0\, \normetriple{\b}_{\Y_s^2}}{s-s'},\\[5pt]
			\label{eq ana1 est LBK Y s:B}
			\normetriple{\partial_{\Theta}\,\mathbf{D}^{\Lop}\big(\b,\mathbf{c}\big)}_{\Y_{s'}^2}&\leq\frac{C\,\gamma_0\, \normetriple{\b}_{\Y_s^2}\normetriple{\mathbf{c}}_{\Y_s^2}}{s-s'},\\[5pt]
			\label{eq ana1 est LBK Y s:F}
			\normetriple{\mathbb{F}^{\per}\big(\partial_{\Theta}\b,\partial_{\Theta}\mathbf{c}\big)}_{\Y_{s}^2}&\leq C\,\gamma_0\, \normetriple{\b}_{\Y_s^2}\normetriple{\mathbf{c}}_{\Y_s^2},\\[5pt]
			\label{eq ana1 est LBK Y s:K}
			\normetriple{\mathbf{K}^{\Lop}\big(\b,\mathbf{c}\big)}_{\Y_{s'}^2}&\leq\frac{C\,\gamma_0\, \normetriple{\b}_{\Y_s^2}\normetriple{\mathbf{c}}_{\Y_s^2}}{s-s'}.
		\end{align}
	\end{subequations}
\end{lemma}

\begin{proof}
	First, estimate \eqref{eq ana1 est LBK Y s:L} follows directly from estimate \eqref{eq ana1 hyp toy model v D bord} and Lemma \ref{lemma ana dy dtheta Y s} since $ L $ is a linear combination of a bounded vector and a first order derivative. For the second one \eqref{eq ana1 est LBK Y s:B}, we have, according to Lemma \ref{lemma Ys Banach algebra}, for $ \zeta=\phi,\psi $ and $ n\geq 1 $, 
	\begin{equation*}
		\norme{D^{\Lop}_{\zeta}\,b^1_{\zeta}\,c^n_{\zeta}}_{Y_s}\leq C\,\gamma_0\norme{b^1_{\zeta}}_{Y_s}\norme{c^n_{\zeta}}_{Y_s}\leq C\,\gamma_0\norme{\b}_{\Y_s^2}\norme{c^n_{\zeta}}_{Y_s}.
	\end{equation*}
Therefore, according to Lemma \ref{lemma ana dy dtheta Y s},
\begin{equation*}
	\norme{\partial_{\Theta}\,D^{\Lop}_{\zeta}\,b^1_{\zeta}\,c^n_{\zeta}}_{Y_{s'}^2}\leq \frac{C}{s-s'}\norme{D^{\Lop}_{\zeta}\,b^1_{\zeta}\,c^n_{\zeta}}_{Y_s}\leq \frac{C\,\gamma_0}{s-s'}\norme{\b}_{\Y_s^2}\norme{c^n_{\zeta}}_{Y_s}.
\end{equation*}
Multiplying by $ e^{2sn}\prodscalbis{n}^{2d^*} $ and summing over $ n\geq 1 $ and $ \zeta=\phi,\psi $ gives the estimate \eqref{eq ana1 est LBK Y s:B} for the $ \Y_s^2 $-norm. With \eqref{eq ana1 hyp toy model F per}, the proof of \eqref{eq ana1 est LBK Y s:F} is analogous but simpler since the operator is semilinear. Finally, for \eqref{eq ana1 est LBK Y s:K}, according to Lemmas \ref{lemma ana dy dtheta Y s} and \ref{lemma Ys Banach algebra}, we have, for $ n\geq 1 $,
\begin{equation*}
	\norme{\sum_{k=1}^{n-1}\partial_{\Theta}\big(b^k_{\phi}\,c^{n-k}_{\psi}\big)}_{Y_{s'}}\leq \frac{C}{s-s'}
	\sum_{k=1}^{n-1}\norme{b^k_{\phi}}_{Y_s}\norme{c^{n-k}_{\psi}}_{Y_s}.
\end{equation*}
Thus, by Cauchy-Schwarz inequality,
\begin{align*}
	2\sum_{n\geq 1}&e^{2s'n}\prodscalbis{n}^{2d^*}\norme{\sum_{k=1}^{n-1}\partial_{\Theta}\big(b^k_{\phi}\,c^{n-k}_{\psi}\big)}_{Y_{s'}}^2\\
	&\leq \frac{C^2}{(s-s')^2}\sum_{n\geq 1}e^{2sn}\prodscalbis{n}^{2d^*}\left(\sum_{k=1}^{n-1}\norme{b^k_{\phi}}_{Y_s}\norme{c^{n-k}_{\psi}}_{Y_s}\right)^2\\
	&\leq \frac{C^2}{(s-s')^2}\sum_{n\geq 1}\left(\sum_{k=1}^{n-1}\frac{\prodscalbis{n}^{2d^*}}{\prodscalbis{k}^{2d^*}\prodscalbis{n-k}^{2d^*}}\right)\\
	&\quad\times\sum_{k=1}^{n-1}e^{2sk}\prodscalbis{k}^{2d^*}\norme{b^k_{\phi}}_{Y_s}^2e^{2s(n-k)}\prodscalbis{n-k}^{2d^*}\norme{c^{n-k}_{\psi}}_{Y_s}^2\\
	&\leq \frac{C^2}{(s-s')^2}\normetriple{\b}_{\Y_s^2}^2\normetriple{\mathbf{c}}_{\Y_s^2}^2,
\end{align*}
since $ \sum_{k=1}^{n-1}\frac{\prodscalbis{n}^{2d^*}}{\prodscalbis{k}^{2d^*}\prodscalbis{n-k}^{2d^*}} $ is bounded uniformly with respect to $ n\geq 1 $, and the result follows.
\end{proof}

We proceed with proof of Proposition \ref{prop WP bord CK}, which essentially amounts to verify assumptions of Theorem \ref{theorem Cauchy-Kovalevskaya}.

\begin{proof}[Proof\emph{ (Proposition \ref{prop WP bord CK})}]
	We apply Theorem \ref{theorem Cauchy-Kovalevskaya} to system \eqref{eq ana dt a} with the scale of Banach spaces $ \big(\Y_s^2\big)_{0< s\leq 1} $. First note that assumption \eqref{eq ana hyp F 2} is satisfied as soon as assumption \eqref{eq ana1 hyp H WP} for $ \mathbf{H} $ is verified. Next we take interest into continuity assumption for $ \mathbf{F} $ and assumption \eqref{eq ana hyp F 1}. For $ 0<s'<s\leq1 $, and for $ t,t'\in(-T,T) $ and $ \b,\mathbf{c} $ in $ \Y_s^2 $, we have,
	\begin{align*}
		\mathbf{F}\big(t,\b\big)-\mathbf{F}\big(t',\mathbf{c}\big)&=\mathbf{L}\big(\b-\mathbf{c}\big)-\partial_{\Theta}\,\mathbf{D}^{\Lop}\big(\b,\b-\mathbf{c}\big) -\partial_{\Theta}\,\mathbf{D}^{\Lop}\big(\b-\mathbf{c},\mathbf{c}\big)
		\\&\quad -\mathbb{F}^{\per}\big(\partial_{\Theta}\,\b,\partial_{\Theta}\,(\b-\mathbf{c})\big)- \mathbb{F}^{\per}\big(\partial_{\Theta}\,(\b-\mathbf{c}),\partial_{\Theta}\,\mathbf{c}\big)+\mathbf{H}(t)-\mathbf{H}(t')\\&\quad+\mathbf{K}^{\Lop}\big(\b,\b-\mathbf{c}\big)+\mathbf{K}^{\Lop}\big(\b-\mathbf{c},\mathbf{c}\big)
	\end{align*}
	so, according to estimates of Lemma \ref{lemma ana1 est LBK Y s},
	\begin{multline}\label{eq ana1 inter1}
		\normetriple{\mathbf{F}\big(t,\b\big)-\mathbf{F}\big(t',\mathbf{c}\big)}_{\Y_{s'}^2}\\
		\leq \normetriple{\mathbf{H}(t)-\mathbf{H}(t')}_{\Y_s^2}+C\,\gamma_0\Big(1+3\normetriple{\b}_{\Y_s^2}+3\normetriple{\mathbf{c}}_{\Y_s^2}\Big)\frac{\normetriple{\b-\mathbf{c}}_{\Y_s^2}}{s-s'}.
	\end{multline}
Therefore, since $ \mathbf{H} $ is continuous from $ [-T,T] $ to $ \Y_1^2 $, if we set $ R>0 $ (which is therefore arbitrary), we both get, from \eqref{eq ana1 inter1}, continuity of $ \mathbf{F} $ from $ [-T,T]\times\ensemble{\b\in\Y_s^2\mid \normetriple{\b}_{\Y_s^2}<R} $ to $ \Y_{s'}^2 $, and (setting $ t'=t $) estimate \eqref{eq ana hyp F 1}, with constant $ C $ given by $ C\,\gamma_0\big(1+6R\big) $. Theorem \ref{theorem Cauchy-Kovalevskaya} therefore applies here and gives the sought result. 
\end{proof}

Here we used that system \eqref{eq ana1 a}, \eqref{eq ana1 a init} presents quadratic nonlinearities, but, using the same arguments, other types of nonlinearities could also be treated.

\subsection{Second simplified model}

We now refine the previous simplified model by incorporating interior equations in it. According to remarks from the introduction of this section, in the new chosen simplified model, we remove the coupling with profiles of frequencies different from $ \phi_j $, $ \psi_j $ and $ \nu_j $, which were appearing in \eqref{eq systeme sigma evol} in terms  $ \partial_{z,\theta}\,\mbox{\emph{terms in}}\,\big(U_1,\dots,U_{n-1},(I-P)\,U_n,U_n^*\big) $. The latter terms are also simplified since they carry derivatives of order higher than one\footnote{In expression \eqref{eq ncor nonpola part U n zeta} of nonpolarized parts, there are already derivatives.}, and we keep only first order derivatives in $ \Theta $ (once again, considering derivatives in $ y $ presents no additional difficulty). They are therefore represented through terms of the form
\begin{equation*}
	\sum_{\zeta_1,\zeta_2=\phi,\psi,\nu}\,\sum_{j_1,j_2=1,3}\,\sum_{k=1}^{n-1}\partial_{y,\Theta}\big(\sigma^k_{\zeta_1,j_1}\,\sigma^{n-k}_{\zeta_2,j_2}\big).
\end{equation*}
We also remove couplings with outgoing frequencies $ \phi_2 $, $ \psi_2 $ and $ \nu_2 $, as incoming equations will be solved seen as propagation in the normal variable equations, a form which is not suited to solve outgoing equations. Finally, we multiply equations \eqref{eq systeme sigma} by $ e^{i\lambda\Theta} $ for $ \Theta\in\T $ a periodic variable. It leads to the following study interior evolution equations, for $ n\geq 1 $, $ \zeta=\phi,\psi,\nu $ and $ j=1,3 $,
\begin{multline}\label{eq ana sigma}
	\color{altblue}X_{\zeta,j}\,\sigma^n_{\zeta,j}+D_{\zeta,j}\,\partial_{\Theta}\big(\sigma_{\zeta,j}^n\,\sigma_{\zeta,j}^1\big)+\sum_{\substack{\zeta_1,\zeta_2\ensemble{\phi,\psi,\nu}\\j_1,j_2\in\ensemble{1,3}}}\partial_{\Theta}\,\J^{\zeta_2,j_2}_{\zeta_1,j_1}\big[\sigma^1_{\zeta_1,j_1},\sigma^n_{\zeta_2,j_2}\big]\\[-10pt]
	\color{altblue}=K_{\zeta,j}\sum_{\substack{\zeta_1,\zeta_2\ensemble{\phi,\psi,\nu}\\j_1,j_2\in\ensemble{1,3}}}\,\sum_{k=1}^{n-1}\partial_{\Theta}\big(\sigma^k_{\zeta_1,j_1}\,\sigma^{n-k}_{\zeta_2,j_2}\big),
\end{multline}
\color{black}
where we have defined $ \sigma^n_{\zeta,j} $, a function of $ \Omega_T\times\T $,  as
\begin{equation*}
	\sigma^n_{\zeta,j}(z,\Theta):=\sum_{\lambda\in\Z^*}\sigma^n_{\zeta,j,\lambda}(z)\,e^{i\lambda\,\Theta},
\end{equation*}
and where $ K_{\zeta,j}\in\R $ for $ \zeta=\phi,\psi,\nu $ and $ j=1,3 $.
For $ (\zeta_1,\zeta_2,j_1,j_2)\in\ensemble{\phi,\psi,\nu}^2\times\ensemble{1,3}^2 $, the bilinear operator $ \J_{\zeta_1,j_1}^{\zeta_2,j_2} $ is defined as
\begin{equation}\label{eq ana def J}
	\J_{\zeta_1,j_1}^{\zeta_2,j_2}\big[\sigma,\tau\big]=J_{\zeta_1,j_1}^{\zeta_2,j_2}\sum_{\lambda\in\Z^*}\sigma_{\lambda}\tau_{\lambda}\,e^{i\lambda\Theta},
\end{equation}
with some coefficients $ J_{\zeta_1,j_1}^{\zeta_2,j_2} $. Similarly as for the boundary equations, up to changing all notation by a harmless multiplicative constant, according to expression \eqref{eq champ vecteur X_alpha} of vector field $ X_{\zeta,j} $, it can be assumed to read
\begin{equation*}
	X_{\zeta,j}=\partial_{t}-\mathbf{v}_{\zeta,j}\cdot\nabla_x,
\end{equation*}
where vector $ \mathbf{v}_{\zeta,j} $ has been defined in Definition \ref{def sortant rentrant alpha X alpha}. Recall last component of each vector $ \mathbf{v}_{\zeta,j} $ is positive.
%The $ O $-symbol will be made precise later, and the term associated with it stands for $ \partial_{z,\theta}\,\mbox{\emph{terms in}}\,\big(U_1,\dots,U_{n-1},(I-P)\,U_n,U_n^*\big) $. 
%It is a simplification because it carries out only first order derivatives, in only $ (y,\Theta) $ variables, and depends only on profiles $ \sigma^n_{\zeta,j} $ for $ n\geq 1 $, $ \zeta=\phi,\psi,\nu $ and $ j=1,3 $.
Equation \eqref{eq ana sigma} is not provided with an initial condition, as we will see it as a propagation in the normal variable equation. 
%However, if boundary terms for \eqref{eq ana sigma} are zero for negative time $ t $, due to the finite speed of propagation, the obtained solution will also be zero for negative time $ t $.
%\begin{equation}\label{eq ana sigma init}
%	\color{altblue}\big(\sigma^n_{\zeta,j}\big)_{|t\leq 0}=0.
%\end{equation}

For boundary conditions for profiles $ \sigma^n_{\zeta,j} $, $ j=1,3 $, the coupling terms in $ \big(\sigma^n_{\zeta,2}\big)_{|x_d=0} $ (appearing in terms in $ \tilde{F}^n_{\zeta,j,\lambda} $ for boundary conditions for profiles associated with $ \phi $ and $ \psi $) are not kept, since it would require trace estimates to solve interior equations, and we do not have such estimates in our possession. Terms in $ \tilde{F}^n_{\zeta,j,\lambda} $ also convey first order derivatives of lower order terms $ a^1_{\zeta},\dots,a^{n-1}_{\zeta} $. For the functional framework chosen later, these derivatives are an issue, and since coupling with lower order terms $ a^1_{\zeta},\dots,a^{n-1}_{\zeta} $ will be expressed in evolution equations for $ a^n_{\zeta} $, terms $ \tilde{F}^n_{\zeta,j,\lambda} $ are only represented in the study equations by boundary terms $ g^n_{\zeta,j} $, belonging to one of the spaces defined later on. This lead to the following study boundary conditions, for $ j=1,3 $,
\begin{subequations}\label{eq ana bord sigma}
\begin{align}\label{eq ana bord sigma phi psu}
	\color{altred}\big(\sigma_{\phi,j}^n\big)_{|x_d=0}&\color{altred}=(e_{\phi,j}\cdot r_{\phi,j})\,a^n_{\phi}+g^n_{\phi,j},\qquad \color{altorange2}\big(\sigma_{\psi,j}^n\big)_{|x_d=0}=(e_{\psi,j}\cdot r_{\psi,j})\,a^n_{\psi}+g^n_{\psi,j},\\[5pt]\label{eq ana bord sigma nu}
	\color{altgreen}\big(\sigma_{\nu,j}^n\big)_{|x_d=0}&\color{altgreen}=g^n_{\nu,j},
\end{align}
\end{subequations}
where, for $ \zeta=\phi,\psi $, we have denoted by $ a^n_{\zeta} $ the function of $ \Omega_T\times\T $ defined as
\begin{equation*}
	a^n_{\zeta}(z,\Theta):=\sum_{\lambda\in\Z^*}a^n_{\zeta,\lambda}(z)\,e^{i\lambda\,\Theta}.
\end{equation*}

Finally, equations for boundary terms $ a^n_{\phi} $ and $ a^n_{\psi} $ are the same as for the first simplified model, namely,
\begin{subequations}\label{eq ana a}
\begin{align}\label{eq ana a phi} 
\color{altred}X^{\Lop}_{\phi}\,a_{\phi}^n+D^{\Lop}_{\phi}\,\partial_{\Theta}\big(a^1_{\phi}\,a^n_{\phi}\big)
+\mathbf{w}_{\phi}\,\mathbb{F}_{\phi}^{\per}\big[\partial_{\Theta}\,a^1_{\phi},\partial_{\Theta}\,a^n_{\phi}\big]\color{altred}&\color{altred}=H^n_{\phi}+K^{\Lop}_{\phi}\sum_{k=1}^{n-1}\partial_{\Theta}\big(a^k_{\phi}\,a^{n-k}_{\psi}\big),\\\label{eq ana a psi}
\color{altorange2}X^{\Lop}_{\psi}\,a_{\psi}^n+D^{\Lop}_{\psi}\,\partial_{\Theta}\big(a^1_{\psi}\,a^n_{\psi}\big)
+\mathbf{w}_{\psi}\,\mathbb{F}_{\psi}^{\per}\big[\partial_{\Theta}\,a^1_{\psi},\partial_{\Theta}\,a^n_{\psi}\big]&\color{altorange2}=H^n_{\psi}+K^{\Lop}_{\psi}\sum_{k=1}^{n-1}\partial_{\Theta}\big(a^k_{\phi}\,a^{n-k}_{\psi}\big),
\end{align}
\end{subequations}
 coupled with the initial condi\-tions
\begin{equation}\label{eq ana a init}
	\color{altred}\big(a_{\phi}^n\big)_{|t= 0}=0,\qquad \color{altorange2}\big(a_{\psi}^n\big)_{|t= 0}=0.
\end{equation}

The strategy to solve the above system of equations \eqref{eq ana sigma}, \eqref{eq ana bord sigma}, \eqref{eq ana a} and \eqref{eq ana a init} is to apply a Cauchy-Kovalevskaya theorem such as Theorem \ref{theorem Cauchy-Kovalevskaya} to interior system \eqref{eq ana sigma},  \eqref{eq ana bord sigma}, seen as a propagation equation in the normal variable. In order to do that, we need the boundary terms in \eqref{eq ana bord sigma} to be analytical with respect to all variables (even with respect to time). For the first simplified model, in Proposition \ref{prop WP bord CK}, we obtained only continuity with respect to time. Therefore we need to refine this result to obtain analyticity with respect to all variables. In the next part we define functional spaces which will be used for this purpose.

\subsection{Additional functional framework}

\subsubsection{Functional spaces}

We define two different types of spaces, which all are spaces of functions defined on $ \omega_T\times\T $, analytical with respect to all variables $ (t,y,\Theta) $. The first ones, which will be denoted by $ E_{\rho} $ and $ \E_{\rho} $, will be used to solve boundary equations \eqref{eq ana a}-\eqref{eq ana a init}, which will be viewed as a fixed point problem in $ \E_{\rho} $. The second one, denoted by $ X_{r} $, $ \X_{r} $, are the one fitted for interior system \eqref{eq ana sigma}-\eqref{eq ana bord sigma}, where equation \eqref{eq ana sigma} will be seen as a differential equation with values in $ \X_r $. Features and relations of this spaces are summarized in Figure \ref{figure fct spaces}.
In addition to defining the functional spaces, we have to describe action of differentiation on it, and to prove that every function of $ \E_{\rho} $ is in $ \X_r $.

Previously introduced spaces $ Y_s $, $ s\in(0,1) $, are used to defined spaces $ E_{\rho} $, $ \rho\in(0,1) $. If $ I\subset \R $ is an interval and $ E $ a Banach space, we denote by $ \mathcal{C}^{\omega}(I,E) $ the space of analytic functions from $ I $ to $ E $.

\begin{definition}
	For $ \rho\in(0,1) $, the space $ \tilde{E}_{\rho} $ is defined as
	\begin{equation*}
		\tilde{E}_{\rho}:=\bigcap_{s\in(0,1)}\mathcal{C}^{\omega}\Big(\big(-\rho(1-s),\rho(1-s)\big),Y_s\Big).
	\end{equation*}
\end{definition}

In the next definition we use the \emph{Catalan numbers} (see \cite{Comtet1974Combinatorics}), defined by, for $ n\geq 0 $, 
\begin{equation*}
\mathfrak{C}_n:=\inv{n+1}\binom{2n}{n}.
\end{equation*}
They satisfy, for $ n\geq 0 $,
\begin{equation}\label{eq ana relation catalan}
\sum_{i=0}^n\mathfrak{C}_i\,\mathfrak{C}_{n-i}=\mathfrak{C}_{n+1}.
\end{equation}
The Catalan numbers appear in the power series expansion of $ x\mapsto (1-x)^{-1/2} $:
\begin{equation}\label{eq ana power series 1/sqrt(1-x)}
\inv{\sqrt{1-x}}=\sum_{n\geq 0}\inv{n!}\frac{(n+1)!\,\mathfrak{C}_n}{4^n}\,x^n,\quad \forall |x|<1.
\end{equation}
Next definition takes inspiration from the method of majoring series, see for example \cite[Chapter II]{John1991PDE}, since, in this formalism, it requires for a function to admit a dilatation of $ x\mapsto (1-x)^{-1/2} $ as a majoring series.

\begin{definition}\label{def E rho}
For $ \rho\in(0,1) $, the space $ E_{\rho} $ is defined as the set of functions $ a $ of $ \tilde{E}_{\rho} $ such that there exists $ M>0 $ such that for all $ s\in(0,1) $ and $ \nu\in\mathbb{N} $, 
\begin{equation}\label{eq ana est norme a}
	\norme{\partial_t^{\nu} a(0)}_{Y_s}\leq \frac{M}{(1-s)^{\nu+1}}\frac{(\nu+1)!\,\mathfrak{C}_{\nu}}{(4\rho)^{\nu}}.
\end{equation}
The infimum of all $ M $ satisfying condition \eqref{eq ana est norme a} is denoted by $ \normetriple{a}_{E_{\rho}} $.
\end{definition}

If $ a $ is in $ E_{\rho} $ for some $ \rho\in(0,1) $, then, for $ s\in(0,1) $, for $ |t|<\rho(1-s) $, by expanding $ a $ in power series with respect to $ t $ at $ 0 $, using estimate \eqref{eq ana est norme a} and the power series expansion of $ x\mapsto(1-x)^{-1/2} $, we get
\begin{equation*}
	\norme{a(t)}_{Y_s}\leq \frac{\normetriple{a}_{E_{\rho}}}{1-s}\left(1-\frac{|t|}{\rho(1-s)}\right)^{-1/2}.
\end{equation*}
We find here the formulation of \cite{BaouendiGoulaouic1978Nishida}.

Since we work with couples of sequences of functions, we define a space accordingly, and we specify the norm used on the product space.

\begin{definition}\label{def E rho suite}
		For $ \rho\in(0,1) $, the space $ \E_{\rho} $ is defined as the set of sequences $ \mathbf{a}=\big(a_n\big)_{n\in\mathbb{N}} $ of $ E_{\rho} $ such that
	\begin{equation*}
		\normetriple{\mathbf{a}}^2_{\E_{\rho}}:=\sum_{n\geq 1}e^{2\,\rho\, n}\prodscalbis{n}^{2d^*}\normetriple{a_n}^2_{E_{\rho}}<+\infty.
	\end{equation*}
For $ \rho\in(0,1) $, the norm on the product space $ \E_{\rho}^2 $ is defined, for $ (\a,\b)\in\E_{\rho}^2 $, as
\begin{equation*}
\normetriple{\big(\a,\b\big)}_{\E_{\rho}^2}^2:=\normetriple{\a}_{\E_{\rho}}^2+\normetriple{\b}_{\E_{\rho}}^2.
\end{equation*}
\end{definition}

%In the following we will consider couples of elements of $ \E_{\rho} $, and the norm on the product space $ \E_{\rho}^2 $  is chosen as the $ \ell^2 $-norm. 
Spaces $ E_{\rho} $ and $ \E_{\rho} $ are not normed algebras, and neither do they satisfy a derivation property such as \eqref{eq ana est dt dtheta Ys}. Indeed, for a function $ a $ of $ E_{\rho} $ with $ \rho\in(0,1) $, we have, by Lemma \ref{lemma ana dy dtheta Y s}, for $ 0<s'<s<1 $,
\begin{equation*}
	\norme{\partial_t^{\nu} \partial_{\Theta}\,a(0)}_{Y_{s'}}\leq \inv{s-s'} \norme{\partial_t^{\nu} a(0)}_{Y_s} \leq   \inv{s-s'}\frac{\normetriple{a}_{E_{\rho}}}{(1-s)^{\nu+1}}\frac{(\nu+1)!\,\mathfrak{C}_{\nu}}{(4\rho)^{\nu}}.
\end{equation*}
To obtain an estimate for \eqref{eq ana est norme a}, it seems that we should have the existence of $ C>0 $ such that for all $ 0<s'<s<1 $ and $ \nu\geq 0 $, 
\begin{equation*}
	\inv{(s-s')(1-s)^{\nu+1}}\leq \frac{C}{(1-s')^{\nu+1}},
\end{equation*}
which is false.
However, as we shall see later, estimating $ t\mapsto \int_0^t \partial_{\Theta}\,a(s)\,ds $ instead of $ \partial_{\Theta}\,a $ could solve the problem. This is what is referred to as \emph{regularization by integration in time}, see \cite{Ukai2001CauchyKovalevskaya, Morisse2020Elliptic}. These spaces seem well suited to prove existence of solutions to boundary system \eqref{eq ana a}, analytical with respect to all variables, but the absence of above mentioned properties prevents to apply a Cauchy-Kovalevskaya theorem with these spaces for interior system. This is why we need to define other, more appropriate spaces.

Spaces for interior equations are spaces in $ (t,y,\Theta) $ variables since interior equations will be seen as propagation equation in $ x_d $, valued in these spaces. In the following, $ H^{d^*} $ denotes the Sobolev space $ H^{d^*}(\R^{d-1}_y\times\T_{\Theta})  $ of regularity $ d^* $. Recall that $ d^* $ has been chosen such that $ d^*>\tilde{m}_0+2+(d+1)/2 $, where $ \tilde{m}_0 $ is the real nonnegative number of Lemma \ref{lemma hyp petit diviseurs}. The next definition is based on the classical way to characterize analytic functions. For a $ (d+1) $-tuple $ \alpha=(\alpha_0,\dots,\alpha_d)\in\mathbb{N}^{d+1} $, notation $ \alpha! $ refers to $ \alpha!:=\alpha_0!\cdots \alpha_d! $.

\begin{definition}
	Consider $ \rho\in(0,1) $. For $ r\in(0,1) $, the space $ X_{r} $ is defined as the set of smooth functions $ a $ of $ (t,y,\Theta)\in[-\rho/2,\rho/2]\times\R^{d-1}\times\T $ with values in $ \C $ such that there exists $ M>0 $ such that for every $ \alpha $ in $ \mathbb{N}^{d+1} $,
	\begin{equation*}
		\norme{\partial_{t,y,\Theta}^{\alpha}a(0,.,.)}_{H^{d^*}}\leq \frac{M\alpha!}{r^{|\alpha|}\,(|\alpha|^{2d+1}+1)}.
	\end{equation*}
The infimum of such $ M>0 $ is denoted by $ \normetriple{a}_{X_r} $.
\end{definition}

Note that in the previous definition, the space $ X_{r} $ depends on the fixed constant $ \rho\in(0,1) $, but we chose not to include this dependence in the notation since in the following $ \rho $ will be fixed. The time interval of the form $ [-\rho/2,\rho/2] $ is required because, in the following, functions of $ X_r $ will come from functions of $ E_{\rho} $, which are defined on time intervals $ \big(-\rho(1-s),\rho(1-s)\big) $ for $ s\in(0,1) $, so we choose arbitrarily $ s=1/2 $. Analogously as for $ E_{\rho} $, we define a space for sequences of $ X_r $. 

\begin{definition}
	For $ r\in(0,1) $, the space $ \X_{r} $ is defined as the set of sequences $ \mathbf{a}=\big(a_n\big)_{n\in\mathbb{N}} $ of $ X_{r} $ such that
	\begin{equation*}
		\normetriple{\mathbf{a}}^2_{\X_{r}}:=\sum_{n\geq 1}e^{2\,r\, n}\prodscalbis{n}^{2d^*}\normetriple{a_n}^2_{X_{r}}<+\infty.
	\end{equation*}
For $ r\in(0,1) $, the norm on the product space $ \X_r^6 $ is defined, for $ \a=(\a_1,\dots,\a_6)\in\X_r^6 $, as
\begin{equation*}
\normetriple{\a}_{\X_r^6}^2:=\normetriple{\a_1}_{\X_r}^2+\cdots+\normetriple{\a_6}_{\X_r}^2.
\end{equation*}
\end{definition}

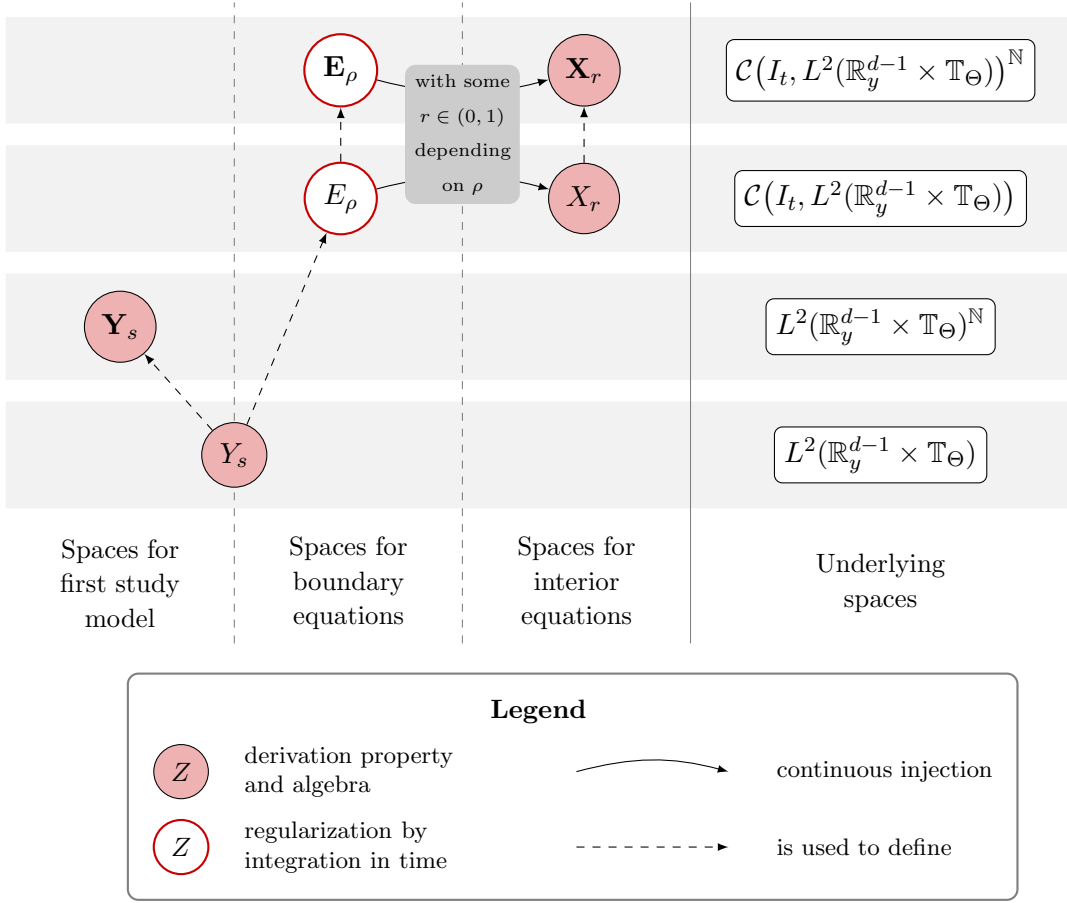
\begin{figure}
	\begin{tikzpicture}
		% Bandes espaces sous-jacents
		\draw[fill=black!5,black!5] (-4.5,2) rectangle (9.5,3.4);
		\draw[fill=black!5,black!5] (-4.5,0.3) rectangle (9.5,1.7);
		\draw[fill=black!5,black!5] (-4.5,-1.4) rectangle (9.5,0);
		\draw[fill=black!5,black!5] (-4.5,-3.1) rectangle (9.5,-1.7);
		% Espaces 
		\node[draw,circle,fill=altred!30] (YY) at (-3,-0.7) {$ \Y_s $};
		\node[draw,circle,altred,thick,text=black,fill=white] (E) at (-0.1,1) {$ E_{\rho} $};
		\node[draw,circle,altred,thick,text=black,fill=white] (EE) at (-0.1,2.7) {$ \E_{\rho} $};
		\node [draw,circle,fill=altred!30](X) at (3.1,1) {$ X_r $};
		\node[draw,circle,fill=altred!30] (XX) at (3.1,2.7)  {$ \X_r $};
		\node[draw,rounded corners=3pt,fill=white]  (LL2) at (7,-0.7) {$ L^2(\R_{y}^{d-1}\times\T_{\Theta})^{\mathbb{N}} $};
		\node[draw,rounded corners=3pt,fill=white]  (L2) at (7,-2.4) {$ L^2(\R_{y}^{d-1}\times\T_{\Theta}) $};
		\node[draw,rounded corners=3pt,fill=white] (CL2) at (7,1) {$ \mathcal{C}\big(I_t,L^2(\R_{y}^{d-1}\times\T_{\Theta})\big) $};
		\node[draw,rounded corners=3pt,fill=white] (CCL2) at (7,2.7) {$ \mathcal{C}\big(I_t,L^2(\R_{y}^{d-1}\times\T_{\Theta})\big)^{\mathbb{N}} $};
		% Noms colonnes
		\draw[align=center](-3,-4.1) node{\small Spaces for\\\small first study\\\small model};
		\draw[align=center](3,-4.1) node{\small Spaces for\\\small interior\\\small equations};
		\draw[align=center](0,-4.1) node{\small Spaces for\\\small boundary\\\small equations};
		\draw[align=center](7,-4.1) node{\small Underlying\\\small spaces};
		% Séparation colonnes
		\draw[black!50] (4.5,-4.9) -- (4.5,3.6);
		\draw[dashed,black!50] (1.5,-4.9) -- (1.5,3.6);
		\draw[dashed,black!50] (-1.5,-4.9) -- (-1.5,3.6);
		% Espace sur séparation
		\node[draw,circle,fill=altred!30] (Y) at (-1.5,-2.4) {$ Y_s $};
		% Flèches pleines
		\draw[->,>=latex] (EE) to[bend right=15] (XX);
		\draw[->,>=latex] (E) to[bend left=15] (X);
		% Flèches pointillées verticales
		\draw[->,>=latex,dashed] (Y) to (E);
		\draw[->,>=latex,dashed] (E) to (EE);
		\draw[->,>=latex,dashed] (X) to (XX);
		\draw[->,>=latex,dashed] (Y) to (YY);
		% Pour r
		\node[rounded corners,fill=black!20,align=center] (A) at (1.5,1.85) {\tiny with some \\\tiny $ r\in(0,1) $\\\tiny depending \\\tiny on $ \rho $};
\begin{scope}[shift={(-2,-1.9)}]
		% Items légendes
		\node[draw,circle,altred,thick,text=black] (Z) at (-0.2,-5.7) {\small $ Z $};
		\node [draw,circle,fill=altred!30](Z) at (-0.2,-4.7) {\small $ Z $};
		\draw[->,>=latex] (5,-4.7) to[bend left=15] (7,-4.7);
		\draw[->,>=latex,dashed] (5,-5.7) to (7,-5.7);
		% Texte légendes
		\draw[right,align=left] (0.5,-5.7) node{\footnotesize regularization by\\[-2pt] \footnotesize integration in time};
		\draw[right,align=left] (0.5,-4.7) node{\footnotesize derivation property\\[-2pt] \footnotesize and algebra};
		\draw[right,align=left] (7.5,-4.7) node{\footnotesize continuous injection};
		\draw[right,align=left] (7.5,-5.7) node{\footnotesize is used to define};
		% Titre et boite -2.1
		\draw (4.5,-3.9) node{\small \textbf{Legend}};
		\draw[rounded corners,thick,black!50] (-0.9, -6.4) rectangle (10.8, -3.4) {};
	\end{scope}
	\end{tikzpicture}
	\caption{Features of functional spaces and links between them}
	\label{figure fct spaces}
\end{figure}

The following result asserts that, for every $ \rho\in(0,1) $, there exists $ r\in(0,1) $ such that space $ \E_{\rho} $ is continuously injected in $ \X_r $, with a constant independent of $ \rho $. The proof is recalled here for the sake of clarity. 

\begin{lemma}\label{lemma Erho dans Xr}
	There exists $ C>0 $, such that, for $ \rho\in(0,1) $, if $ \a $ is in $ \E_{\rho} $ then there exists $ r\in(0,1) $ such that $ \a $ belongs to $ \X_r $ (for the same $ \rho $) and, furthermore, 
	\begin{equation*}
		\normetriple{\a}_{\X_r}\leq C\normetriple{\a}_{\E_{\rho}}.
	\end{equation*}
\end{lemma}

\begin{proof}
	Let $ \a=(a^n)_{n\geq 1} $ be in $ \E_{\rho} $ for some $ \rho\in(0,1) $. For $ \alpha=(\alpha_0,\alpha',\beta)=(\alpha_0,\alpha_1,\dots,\alpha_{d-1},\beta) $ in $ \mathbb{N}\times\mathbb{N}^{d-1}\times\mathbb{N} $, we have, for $ s\in(0,1) $ and $ n\geq 1 $,
	\begin{align*}
		\norme{\partial_{t,y,\Theta}^{\alpha}a^n(0,.,.)}^2_{H^{d^*}}&=\int_{\R^{d-1}}\sum_{\lambda\in\Z}\xi_1^{2\alpha_1}\cdots\xi_{d-1}^{2\alpha_{d-1}}\lambda^{2\beta}\prodscalbis{(\xi,\lambda)}^{2d^*}\left|\partial_t^{\alpha_0}\hat{a^n_{\lambda}}(0,\xi)\right|^2\,d\xi\\
		&\leq\int_{\R^{d-1}}\sum_{\lambda\in\Z}\frac{\alpha'!^2\beta!^2}{s^{2(|\alpha'|+\beta)}}\,e^{2s|(\xi,\lambda)|}\prodscalbis{(\xi,\lambda)}^{2d^*}\left|\partial_t^{\alpha_0}\hat{a^n_{\lambda}}(0,\xi)\right|^2\,d\xi\\
		&=\frac{\alpha'!^2\beta!^2}{s^{2(|\alpha'|+\beta)}}\norme{\partial_t^{\alpha_0}a^n(0)}^2_{Y_s},
	\end{align*}
using the inequality
\begin{equation*}
	\frac{(s\xi)^{\alpha'}(s\lambda)^{\beta}}{\alpha'!\beta!}\leq e^{s|(\xi,\lambda)|}.
\end{equation*} 
Since $ a^n $ is in $ E_{\rho} $ (because $ \a $ is in $ \E_{\rho} $), we therefore have 
\begin{equation*}
	\norme{\partial_{t,y,\Theta}^{\alpha}a^n(0,.,.)}_{H^{d^*}}\leq \frac{\alpha'!\beta!}{s^{|\alpha'|+\beta}} \frac{\normetriple{a^n}_{E_{\rho}}\,(\alpha_0+1)!\,\mathfrak{C}_{\alpha_0}}{(1-s)^{\alpha_0+1}\,(4\rho)^{\alpha_0}}\leq \frac{\alpha'!\beta!}{s^{|\alpha'|+\beta}} \frac{C\,\normetriple{a^n}_{E_{\rho}}\,\alpha_0!}{(1-s)^{\alpha_0+1}\,(3\rho)^{\alpha_0}}
\end{equation*}
using $ (\alpha_0+1)\,\mathfrak{C}_{|\alpha_0|}/(4\rho)^{\alpha_0}\leq C/(3\rho)^{\alpha_0} $. Finally, we have, if $ s\leq \min(\rho,2/3) $,
\begin{equation*}
	\norme{\partial_{t,y,\Theta}^{\alpha}a^n(0,.,.)}_{H^{d^*}}\leq \frac{\alpha!}{s^{|\alpha|}} \frac{C\,\normetriple{a^n}_{\E_{\rho}}}{3}\leq \frac{C\,\normetriple{a^n}_{\E_{\rho}}\alpha!}{s'^{|\alpha|}\,(|\alpha|^{2d+1}+1)},
\end{equation*}
with $ s'<s $, because $ s'^{|\alpha|}\,(|\alpha|^{2d+1}+1)\leq Cs^{|\alpha|} $. %Therefore $ a^n $ belongs to $ X_r $ with $ r<\min(\rho,2/3) $ that depends only on $ \rho $. 
Therefore, 
\begin{equation}\label{eq ana inter 3}
	\normetriple{a^n}_{X_{s'}}\leq C \normetriple{a^n}_{E_{\rho}},
\end{equation}
with a constant $ C $ which does not depend on $ n\geq 1 $. Therefore, multiplying inequality \eqref{eq ana inter 3} by $ e^{2s'n}\prodscalbis{n}^{2d^*} $ and summing over $ n\geq 1 $ leads to 
\begin{equation*}
	\normetriple{\a}_{\X_{s'}}\leq C \normetriple{\a}_{\E_{\rho}},
\end{equation*}
so $ \a $ belongs to $ \X_r $ with $ r<\min(\rho,2/3) $ which depends only on $ \rho $, concluding the proof.
\end{proof}

Following results state that partial derivatives with respect to $ t,y,\Theta $ act on $ \X_r $ in the same way as partial derivatives with respect to $ y,\Theta $ act on $ Y_s $, and that spaces $ \X_r $ satisfy an algebra property. For the sake of completeness, we recall the proof of these classical results. 

\begin{lemma}\label{lemma ana dy dtheta Xr}
	There exists $ C>0 $ such that, for $ 0\leq r'<r\leq 1 $, for $ \a $ in $ \X_r $ and for $ e_j $ in $ \mathbb{N}^{d+1} $ with $ |e_j|=1 $, function $ \partial_{t,y,\Theta}^{e_j}\a $ belongs to $ \X_{r'} $ and satisfies
	\begin{equation*}
		\normetriple{\partial_{t,y,\Theta}^{e_j}\a}_{\X_{r'}}\leq \frac{C}{r-r'}\,\normetriple{\a}_{\X_r}.
	\end{equation*}
\end{lemma}

\begin{proof}
	In the same way as for the previous Lemma \ref{lemma Erho dans Xr}, proving the estimate for the space $ X_r $ leads to the one for $ \X_r $ and the associate result, by multiplying by $ e^{2s'n}\prodscalbis{n}^{2d^*} $ and summing over $ n\geq 1 $. Let $ a $ be in $ X_r $. Without loss of generality, we make the proof for $ e_j=e_0=(1,0,\dots,0) $. For $ \alpha=(\alpha_0,\alpha',\beta) $ in $ \mathbb{N}\times\mathbb{N}^{d-1}\times\mathbb{N} $, we have, by definition of $ X_r $-norm,
	\begin{align*}
		\norme{\partial_{t,y,\Theta}^{\alpha}\partial_{t,y,\Theta}^{e_0}a(0,.,.)}_{H^{d^*}}&\leq\normetriple{a}_{X_r}\frac{(\alpha+e_0)!}{r^{|\alpha|+1}\,(|\alpha+e|^{2d+1}+1)}\\
		&=\normetriple{a}_{X_r}\frac{\alpha!}{(r')^{|\alpha|}\,(|\alpha|^{2d+1}+1)}\frac{(|\alpha|^{2d+1}+1)}{(|\alpha+e|^{2d+1}+1)}(\alpha_0+1)\frac{(r')^{|\alpha|}}{r^{|\alpha|+1}}.
	\end{align*}
Since $ (|\alpha|^{2d+1}+1)/(|\alpha+e|^{2d+1}+1) $ is bounded uniformly with respect to $ \alpha\in\mathbb{N}^{d+1} $ and since
\begin{equation*}
	(\alpha_0+1)\frac{(r')^{|\alpha|}}{r^{|\alpha|+1}}\leq \big(|\alpha|+1\big)\frac{(r')^{|\alpha|}}{r^{|\alpha|+1}}\leq \inv{r-r'},
\end{equation*}
the result follows.
\end{proof}

\begin{lemma}\label{lemma Xr Banach algebra}
	For $ r\in(0,1) $, spaces $ X_r $ and $ \X_r $ are Banach algebras (the latter for the convolution on sequences), up to a positive constant, that is, there exists $ C>0 $ (independent of $ r\in(0,1) $), such that for $ \a,\b $ in $ \X_r $, the function $ \a\b $ belongs to $ \X_r $ and we have
	\begin{equation*}
		\normetriple{\a\b}_{\X_r}\leq C\normetriple{\a}_{\X_r}\normetriple{\b}_{\X_r},
	\end{equation*}
and the analogous estimate for $ X_r $.
\end{lemma}

\begin{proof}
	We make the proof for spaces $ X_r $, and the result for $ \X_r $ follows using the same arguments as in the proof of Lemma \ref{lemma Ys Banach algebra}. Let $ r $ be in $ (0,1) $ and consider $ a,b $ in $ X_r $. We need to show that there exists $ C>0 $ such that for all $ \alpha\in\mathbb{N}^{d+1} $, we have
	\begin{equation*}
		\norme{\partial_{t,y,\Theta}^{\alpha}\big(ab\big)(0,.,.)}_{H^{d^*}}\leq \frac{C\normetriple{a}_{X_r}\normetriple{b}_{X_r}\,\alpha!}{r^{|\alpha|}\,(|\alpha|^{2d+1}+1)} 
	\end{equation*}
So consider $ \alpha\in\mathbb{N}^{d+1} $. We have, since $ d^*>d/2  $, so $ H^{d^*} $ is an algebra,
\begin{align*}
	\norme{\partial_{t,y,\Theta}^{\alpha}\big(ab\big)(0,.,.)}_{H^{d^*}}&\leq C\sum_{\beta\leq\alpha}\binom{\alpha}{\beta}\norme{\partial_{t,y,\Theta}^{\beta}a(0,.,.)}_{H^{d^*}}\norme{\partial_{t,y,\Theta}^{\alpha-\beta}b(0,.,.)}_{H^{d^*}}\\
	&\leq C\sum_{\beta\leq\alpha}\binom{\alpha}{\beta}\frac{\normetriple{a}_{X_r}\,\beta!}{r^{|\beta|}\,(|\beta|^{2d+1}+1)}\frac{\normetriple{b}_{X_r}\,(\alpha-\beta)!}{r^{|\alpha|-|\beta|}\,(|\alpha-\beta|^{2d+1}+1)} \\
	&=\frac{C\,\normetriple{a}_{X_r}\normetriple{b}_{X_r}\alpha!}{r^{|\alpha|}\,(|\alpha|^{2d+1}+1)}\sum_{\beta\leq\alpha}\frac{(|\alpha|^{2d+1}+1)}{(|\beta|^{2d+1}+1)\,(|\alpha-\beta|^{2d+1}+1)},
\end{align*}
and the result follows since $ \sum_{\beta\leq\alpha}\frac{(|\alpha|^{2d+1}+1)}{(|\beta|^{2d+1}+1)\,(|\alpha-\beta|^{2d+1}+1)} $ is bounded uniformly with respect to $ \alpha\in\mathbb{N}^{d+1} $.
\end{proof}

We summarize the main features of the functional spaces introduced in this section in Figure \ref{figure fct spaces}. The concept of \emph{regularization by integration in time} is detailed below, in Lemma \ref{lemma ana est LBK E rho}.

\subsubsection{Specifications on the simplified model and main result}
In view of the functional spaces defined above, we are able to make precise the study system \eqref{eq ana sigma}-\eqref{eq ana bord sigma}-\eqref{eq ana a}-\eqref{eq ana a init}. 
%More precisely, we detail the assumptions on the toy system, and, when needed, justify why this assumptions are reasonable regarding the general system \eqref{eq systeme moyenne}, \eqref{eq systeme sigma hors cas part}, \eqref{eq systeme sigma}, \eqref{eq systeme sigma bord} and \eqref{eq systeme a}.

Boundary terms $ g^n_{\zeta,j} $ appearing in \eqref{eq ana bord sigma} are taken such that, if we define $ \mathbf{g}_{\zeta,j}:=\big(g^n_{\zeta,j}\big)_{n\in\mathbb{N}} $, then function $ \mathbf{g}_{\zeta,j} $ is in $ \X_1 $ for $ \zeta=\phi,\psi,\nu $ and $ j=1,3 $. Analogously, source terms $ H^n_{\zeta} $ of boundary equations \eqref{eq ana a} are taken such that, defining $ \mathbf{H}_{\zeta}:=\big(H^n_{\zeta}\big)_{n\geq 1} $, sequence $ \mathbf{H}_{\zeta} $ is in $ \E_1 $ for $ \zeta=\phi,\psi $. Once again, assumption on $ \mathbf{H}_{\zeta} $ imposing it to be analytical with respect to all its variables is stronger than requiring $ H $ and $ G $ to be in $ H^{\infty}(\R^d\times\T) $. We also denote by $ \gamma_0>0 $ a positive constant such that,
for $ \zeta=\phi,\psi,\nu $ and $ j=1,3 $,
\begin{subequations}\label{eq ana hyp toy model}
\begin{equation}\label{eq ana hyp toy model v D int}
	|\mathbf{v}_{\zeta,j}|\leq \gamma_0,\quad  \big|D_{\zeta,j}\big|\leq \gamma_0 \quad \text{and}\quad \big|K_{\zeta,j}\big|\leq \gamma_0,
\end{equation}
for $ \zeta=\phi,\psi $,
	\begin{equation}\label{eq ana hyp toy model v D bord}
		\big|\mathbf{v}^{\Lop}_{\zeta}\big|\leq \gamma_0,\qquad \big|D^{\Lop}_{\zeta}\big|\leq \gamma_0, \quad |\mathbf{w}_{\zeta}|\leq\gamma_0^{1/2}, \quad \text{and}\quad  \big|K^{\Lop}_{\zeta}\big|\leq \gamma_0,
	\end{equation}
for $ r\in(0,1) $ and for $ \sigma,\tau $ in $ X_r $, for $ \zeta_1,\zeta_2=\phi,\psi,\nu $ and $ j_1,j_2=1,3 $,
\begin{equation}\label{eq ana hyp toy model J}
	\normetriple{\J_{\zeta_1,j_1}^{\zeta_2,j_2}\big[\sigma,\tau\big]}_{X_r}\leq \gamma_0\,\normetriple{\sigma}_{X_r}\normetriple{\tau}_{X_r},
\end{equation}
and, for $ s\in(0,1) $, for $ u,v $ in $ Y_s $, and for $ \zeta=\phi,\psi $,
\begin{equation}\label{eq ana hyp toy model F per}
	\norme{\mathbb{F}^{\per}_{\zeta}\big[\partial_{\Theta}\,u,\partial_{\Theta}\,v\big]}_{Y_s}\leq \gamma_0^{1/2}\,\norme{u}_{Y_s}\norme{v}_{Y_s}.
\end{equation}
\end{subequations}
All estimates relies on the fact that scalars $ D_{\zeta,j} $, $ K_{\zeta,j} $, $ D^{\Lop}_{\zeta} $, $ \mathbf{w}_{\zeta} $ and $ K^{\Lop}_{\zeta} $, vectors $ \mathbf{v}_{\zeta,j} $ and $ \mathbf{v}^{\Lop}_{\zeta} $ and operators $ \mathbb{J}^{\zeta_1,j_1}_{\zeta_2,j_2} $ and $ \mathbb{F}^{\per}_{\zeta} $ are indexed by finite sets. Estimate \eqref{eq ana hyp toy model F per} is the result of Proposition \ref{prop F per semilin}, and \eqref{eq ana hyp toy model J} is the result of the following lemma. 

\begin{lemma}\label{lemma J est}
	There exists a constant $ C>0 $ such that, for $ r\in(0,1) $, for $ \sigma,\tau $ in $ X_r $, $ \zeta_1,\zeta_2 $ in$ \ensemble{\phi,\psi} $ and $ j_1,j_2 $ in $ \ensemble{1,3} $, we have 
	\begin{equation}\label{eq ana1 est lemma J}
		\normetriple{\J_{\zeta_1,j_1}^{\zeta_2,j_2}\big[\sigma,\tau\big]}_{X_r}\leq C\,\normetriple{\sigma}_{X_r}\normetriple{\tau}_{X_r}.
	\end{equation}
\end{lemma}

\begin{proof}
	Without loss of generality, we make the proof for $ \zeta_1=\zeta_2=\phi $ and $ j_1=j_2=1 $. Recall that $ \mathbb{J}^{\phi,1}_{\phi,1} $ is defined by \eqref{eq ana def J} as
	\begin{equation*}
		\J^{\phi,1}_{\phi,1}\big[\sigma,\tau\big]=J^{\phi,1}_{\phi,1}\sum_{\lambda\in\Z^*}\sigma_{\lambda}\tau_{\lambda}\,e^{i\lambda\Theta}.
	\end{equation*}
	For $ \alpha=(\alpha',\alpha_d) $ in $ \mathbb{N}^{d}\times\mathbb{N} $, we want to estimate in $ H^{d^*} $ the following function
	\begin{equation}\label{eq ana inter 4}
		\partial_{t,y,\Theta}^{\alpha}\J^{\phi,1}_{\phi,1}\big[\sigma,\tau\big](0,.,.)
		=J^{\phi,1}_{\phi,1}\sum_{\lambda\in\Z^*}\sum_{\beta'\leq \alpha'}\binom{\alpha'}{\beta'}\big(i\lambda\big)^{\alpha_d}\partial_{t,y}^{\beta'}\sigma_{\lambda}(0,.)\,\partial_{t,y}^{\alpha'-\beta'}\tau_{\lambda}(0,.)\,e^{i\lambda\Theta}.
	\end{equation}
We prove now an intermediate result. For every function $ u,v $ defined on $ \R^{d-1}\times\T $, whose Fourier series expansions read
\begin{equation*}
	u(y,\Theta):=\sum_{\lambda\in\Z}u_{\lambda}(y)\,e^{i\lambda\Theta},\qquad v(y,\Theta):=\sum_{\lambda\in\Z}v_{\lambda}(y)\,e^{i\lambda\Theta},
\end{equation*}
we have,
\begin{align}\nonumber
	&\norme{(y,\Theta)\mapsto \sum_{\lambda\in\Z} u_{\lambda}(y)v_{\lambda}(y)\,e^{i\lambda\Theta}}^2_{H^{d^*}}
	\\\nonumber
%	&\qquad=\sum_{k\in\Z^*}\int_{\R^{d-1}}\prodscalbis{(\lambda,\xi)}^{2d^*}\left|\hat{u_{\lambda}\,v_{\lambda}}(\xi)\right|^2d\xi\\
	&\qquad=\sum_{k\in\Z^*}\int_{\R^{d-1}}\prodscalbis{(\lambda,\xi)}^{2d^*}\left|\int_{\R^{d-1}}\hat{u_{\lambda}}(\eta)\,\hat{v_{\lambda}}(\xi-\eta)\,d\eta\right|^2d\xi\\\nonumber
	&\qquad\leq\sum_{k\in\Z^*}\int_{\R^{d-1}}\left(\int_{\R^{d-1}}\frac{\prodscalbis{(\lambda,\xi)}^{2d^*}}{\prodscalbis{(\lambda,\eta)}^{2d^*}\prodscalbis{(\lambda,\xi-\eta)}^{2d^*}}\,d\eta\right)\\\nonumber
	&\qquad\quad\times\int_{\R^{d-1}}\prodscalbis{(\lambda,\eta)}^{2d^*}\big|\hat{u_{\lambda}}(\eta)\big|^2\,\prodscalbis{(\lambda,\xi-\eta)}^{2d^*}\big|\hat{v_{\lambda}}(\xi-\eta)\big|^2\,d\eta\,d\xi
	\\\label{eq ana inter 5}
	&\qquad\leq C\norme{(y,\Theta)\mapsto \sum_{\lambda\in\Z} u_{\lambda}(y)\,e^{i\lambda\Theta}}^2_{H^{d^*}}\norme{(y,\Theta)\mapsto \sum_{\lambda\in\Z} v_{\lambda}(y)\,e^{i\lambda\Theta}}^2_{H^{d^*}},
\end{align}
with a constant $ C>0 $ independent on $ u $ and $ v $, since $ \int_{\R^{d-1}}\frac{\prodscalbis{(\lambda,\xi)}^{2d^*}}{\prodscalbis{(\lambda,\eta)}^{2d^*}\prodscalbis{(\lambda,\xi-\eta)}^{2d^*}}\,d\eta $ is bounded uniformly with respect to $ (\lambda,\xi) $. We have, according to \eqref{eq ana inter 4},
\begin{multline}\label{eq ana inter 6}
	\norme{\partial_{t,y,\Theta}^{\alpha}\J^{\phi,1}_{\phi,1}\big[\sigma,\tau\big](0,.,.)}_{H^{d^*}}^2\\
	\qquad\leq \big(J^{\phi,1}_{\phi,1}\big)^2\sum_{\beta'\leq \alpha'}\binom{\alpha'}{\beta'} \norme{(y,\Theta)\mapsto \sum_{\lambda\in\Z^*}\big(i\lambda\big)^{\alpha_d}\partial_{t,y}^{\beta'}\sigma_{\lambda}(0,.)\,\partial_{t,y}^{\alpha'-\beta'}\tau_{\lambda}(0,.)\,e^{i\lambda\Theta}},
\end{multline}
so, applying inequality \eqref{eq ana inter 5} to quantity \eqref{eq ana inter 6} we get
\begin{align*}
	&\norme{\partial_{t,y,\Theta}^{\alpha}\J^{\phi,1}_{\phi,1}\big[\sigma,\tau\big](0,.,.)}_{H^{d^*}}^2\\
	&\qquad\leq C\big(J^{\phi,1}_{\phi,1}\big)^2\sum_{\beta'\leq \alpha'}\binom{\alpha'}{\beta'}\norme{(y,\Theta)\mapsto \sum_{\lambda\in\Z} \big(i\lambda\big)^{\alpha_d}\partial_{t,y}^{\beta'}\sigma_{\lambda}(0,y)\,e^{i\lambda\Theta}}^2_{H^{d^*}}\\
	&\qquad\quad\times\norme{(y,\Theta)\mapsto \sum_{\lambda\in\Z} \partial_{t,y}^{\alpha'-\beta'}\tau_{\lambda}(0,y)\,e^{i\lambda\Theta}}^2_{H^{d^*}}\\
	&\qquad= C\big(J^{\phi,1}_{\phi,1}\big)^2\sum_{\beta'\leq \alpha'}\binom{\alpha'}{\beta'}\norme{\partial_{t,y,\Theta}^{(\beta',\alpha_d)}\sigma(0,.,.)}^2_{H^{d^*}}\norme{ \partial_{t,y,\Theta}^{(\alpha'-\beta',0)}\tau(0,.,.)}^2_{H^{d^*}}.
\end{align*}
Therefore, by definition of the $ X_r $-norm,
\begin{align*}
	&\norme{\partial_{t,y,\Theta}^{\alpha}\J^{\phi,1}_{\phi,1}\big[\sigma,\tau\big](0,.,.)}_{H^{d^*}}^2\\
	&\qquad\leq C\sum_{\beta'\leq\alpha'}\binom{\alpha'}{\beta'}\frac{\normetriple{\sigma}_{X_r}\,\beta'!\alpha_d!}{r^{|\beta'|+\alpha_d}\,(|(\beta,\alpha_d)|^{2d+1}+1)}\frac{\normetriple{b}_{X_r}\,(\alpha'-\beta')!}{r^{|\alpha'|-|\beta'|}\,(|\alpha'-\beta'|^{2d+1}+1)} \\[5pt]
	&\qquad=\frac{C\,\normetriple{a}_{X_r}\normetriple{b}_{X_r}\alpha!}{r^{|\alpha|}\,(|\alpha|^{2d+1}+1)}\sum_{\beta'\leq\alpha'}\frac{(|\alpha|^{2d+1}+1)}{(|(\beta',\alpha_d)|^{2d+1}+1)\,(|\alpha'-\beta'|^{2d+1}+1)},
\end{align*}
and the result follows since $ \sum_{\beta'\leq\alpha'}\frac{(|\alpha|^{2d+1}+1)}{(|(\beta',\alpha_d)|^{2d+1}+1)\,(|\alpha'-\beta'|^{2d+1}+1)} $ is bounded uniformly with respect to $ \alpha\in\mathbb{N}^{d+1} $.
\end{proof}

The rest of the section is devoted to the proof of the following existence and uniqueness result.

\begin{theorem}\label{theorem existence ana}
	For every $ M_0,M_1>0 $, there exist $ 0<r_1<1 $ and $ \delta>0 $ such that for every $ \mathbf{g} $ in $ \X^6_{r_1} $ and $ \mathbf{H} $ in $ \E^2_1 $ satisfying respectively  $ \normetriple{\mathbf{H}}_{\E_1^2}<M_0 $ and $ \normetriple{\mathbf{g}}_{\X^6_{r_1}}<M_1 $, system of equations \eqref{eq ana sigma},  \eqref{eq ana bord sigma}, \eqref{eq ana a} and \eqref{eq ana a init} admits a unique solution given by $ \boldsymbol{\sigma} $ in $ \mathcal{C}^1\big((-\delta(r_1-r),\delta(r_1-r)),\X_r^6\big) $ for each $ 0<r<r_1 $ and $ \a $ in $ \X^2_{r_1} $, where we have denoted $ \boldsymbol{\sigma}:=\big(\sigma^n_{\phi,1},\sigma_{\phi,3}^n,\sigma^n_{\psi,1},\sigma^n_{\psi,3},\sigma_{\nu,1}^n,\sigma_{\nu,3}^n\big)_{n\geq1} $, $ \mathbf{g}:=\big(\mathbf{g}_{\phi,1},\mathbf{g}_{\phi,3},\mathbf{g}_{\psi,1},\mathbf{g}_{\psi,3},\mathbf{g}_{\nu,1},\mathbf{g}_{\nu,3}\big) $ and $ \mathbf{H}:=\big(\mathbf{H}_{\phi},\mathbf{H}_{\psi}\big) $.
\end{theorem}

To prove this result, we will start by proving existence for boundary system \eqref{eq ana a}-\eqref{eq ana a init}, with the Banach fixed point theorem applied in a closed ball of $ \E^2_{\rho} $ for some $ \rho\in(0,1) $. We will therefore have a solution $ \a $ of \eqref{eq ana a}-\eqref{eq ana a init} analytical both with respect to space and time. The strategy is to write equations \eqref{eq ana a} as a fixed point, by the change of variables $ \b:=\partial_t\,\a $, to obtain a problem like the one of \cite{BaouendiGoulaouic1978Nishida}. This will allow us to prove that the operator at stake is a contraction, using the phenomenon of regularization by integration in time. Then we proceed with the existence of solution for interior system \eqref{eq ana sigma}-\eqref{eq ana bord sigma}, by applying a classical Cauchy-Kovalevskaya result, in $ \X_r^6 $ for some $ r\in(0,1) $. For this purpose, equations \eqref{eq ana sigma} will be seen as propagation equations in the normal variable. Verifying assumptions of Theorem \ref{theorem Cauchy-Kovalevskaya} for interior equations presents no difficulty. Finally Lemma \ref{lemma Erho dans Xr} will be used to assert that the obtained solution $ \a $ of \eqref{eq ana a}-\eqref{eq ana a init} in $ \E_{\rho}^2 $ is actually in $ \X_{r}^2 $ for some $ r\in(0,1) $. 

\subsection{Time analyticity on the boundary and Cauchy-Kovalevskaya theorem for incoming equations}

\subsubsection{Existence and time analyticity for boundary equations}

This part is devoted to solving  boundary  system \eqref{eq ana a}-\eqref{eq ana a init}. The goal is to obtain solutions which are analytical not only with respect to $ (y,\Theta) $  but also with respect to time. In the same way as for the first simplified model, system \eqref{eq ana a}-\eqref{eq ana a init} can be displayed in the form
\begin{equation}\label{eq ana dt a}
	\begin{cases}
	\partial_t \,\mathbf{a} = L \,\mathbf{a} -\partial_{\Theta}\,\mathbf{D}^{\Lop}\big(\a,\a\big) - \mathbb{F}^{\per}\big(\partial_{\Theta}\,\a,\partial_{\Theta}\,\a\big)+\mathbf{H}+\mathbf{K}^{\Lop}\big(\a,\a\big),\\[5pt]
	\a(0)=0,
	\end{cases}
\end{equation}
where $ \a:=(\a_{\phi},\a_{\psi}):=\big(a^n_{\phi},a^n_{\psi}\big)_{n\geq 1} $, $ \mathbf{H}:=\big(\mathbf{H}_{\phi},\mathbf{H}_{\psi}\big) $, and, if $ \mathbf{c}:=\big(c^n_{\phi},c^n_{\psi}\big)_{n\geq 1} $, 
\begin{align*}
	L\,\a&:=\big(\mathbf{v}_{\phi}^{\Lop}\cdot\nabla_y\,a^n_{\phi},\mathbf{v}_{\phi}^{\Lop}\cdot\nabla_y\,a^n_{\psi}\big)_{n\geq 1}, \\[5pt]
	\mathbf{D}^{\Lop}(\a,\mathbf{c})&:=\big(D^{\Lop}_{\phi}\,a^1_{\phi}\,a^n_{\phi}, D^{\Lop}_{\psi}\,a^1_{\psi}\,a^n_{\psi}\big)_{n\geq 1}, \\[5pt]
	\mathbb{F}^{\per}(\a,\mathbf{c})&:=\big(\mathbf{w}_{\phi}\,\mathbb{F}^{\per}_{\phi}[a^1_{\phi},c^{n}_{\phi}], \mathbf{w}_{\psi}\linebreak\mathbb{F}^{\per}_{\psi}[a^1_{\psi},c^{n}_{\psi}]\big)_{n\geq 1}, \\[5pt]
	\mathbf{K}^{\Lop}(\a,\mathbf{c})&:=\Big(K^{\Lop}_{\phi}\sum_{k=1}^{n-1}\partial_{\Theta}\big(a^k_{\phi}\,c^{n-k}_{\psi}\big),K^{\Lop}_{\psi}\sum_{k=1}^{n-1}\partial_{\Theta}\big(a^k_{\phi}\,c^{n-k}_{\psi}\big)\Big)_{n\geq 1}.
\end{align*}
Setting $ \b:=\partial_t\a $, system \eqref{eq ana dt a} is equivalent to
\begin{multline}\label{eq ana b = F(b)}
	\b(t)=L\int_0^t\b(\sigma)\,d\sigma-\partial_{\Theta}\,\mathbf{D}^{\Lop}\Big(\int_0^t\b(\sigma)\,d\sigma,\int_0^t\b(\sigma)\,d\sigma\Big)\\
	\quad-\mathbb{F}^{\per}\Big(\partial_{\Theta}\int_0^t\b(\sigma)\,d\sigma,\partial_{\Theta}\int_0^t\b(\sigma)\,d\sigma\Big)
	+\mathbf{K}^{\Lop}\Big(\int_0^t\b(\sigma)\,d\sigma,\int_0^t\b(\sigma)\,d\sigma\Big)+\mathbf{H}(t),
\end{multline}
with $ \a(t)=\int_0^t\b(\sigma)\,d\sigma $.
The aim is to solve equation \eqref{eq ana b = F(b)} with a fixed point theorem in $ \E_{\rho}^2 $, so we start by proving the following key estimates, which will allow us to prove contraction for the operator at stake.

Estimates \eqref{eq ana est LBK E rho:L}, \eqref{eq ana est LBK E rho:B} and \eqref{eq ana est LBK E rho:K} below constitute what we call \emph{regularization by integration in time}, where composing derivation in $ (y,\Theta) $ with integration in time leads to no loss of regularity. This phenomenon was introduced by \cite{Ukai2001CauchyKovalevskaya}, and can also be found in \cite{Metivier2009Optics,Morisse2020Elliptic}. 

\begin{lemma}\label{lemma ana est LBK E rho}
	There exists $ C>0 $ such that for $ \rho\in(0,1) $, for $ \b,\mathbf{c} $ in $ \E_{\rho}^2 $, the following estimates hold
	\begin{subequations}\label{eq ana est LBK E rho}
		\begin{align}
			\label{eq ana est LBK E rho:L}
			\normetriple{t\mapsto L\int_0^t\b(\sigma)\,d\sigma}_{\E_{\rho}^2}&\leq  C\,\rho\,\gamma_0\, \normetriple{\b}_{\E_{\rho}^2},\\
			\label{eq ana est LBK E rho:B}
			\normetriple{t\mapsto \partial_{\Theta}\,\mathbf{D}^{\Lop}\Big(\int_0^t\b(\sigma)\,d\sigma,\int_0^t\mathbf{c}(\sigma)\,d\sigma\Big)}_{\E_{\rho}^2}&\leq C\,\rho^2\,\gamma_0\, \normetriple{\b}_{\E_{\rho}^2}\normetriple{\mathbf{c}}_{\E_{\rho}^2},\\
			\label{eq ana est LBK E rho:F}
			\normetriple{t\mapsto \mathbb{F}^{\per}\Big(\partial_{\Theta}\int_0^t\b(\sigma)\,d\sigma,\partial_{\Theta}\int_0^t\mathbf{c}(\sigma)\,d\sigma\Big)}_{\E_{\rho}^2}&\leq C\,\rho^2\,\gamma_0\, \normetriple{\b}_{\E_{\rho}^2}\normetriple{\mathbf{c}}_{\E_{\rho}^2},\\
			\label{eq ana est LBK E rho:K}
			\normetriple{t\mapsto \mathbf{K}^{\Lop}\Big(\int_0^t\b(\sigma)\,d\sigma,\int_0^t\mathbf{c}(\sigma)\,d\sigma\Big)}_{\E_{\rho}^2}&\leq C\,\rho^2\,\gamma_0\, \normetriple{\b}_{\E_{\rho}^2}\normetriple{\mathbf{c}}_{\E_{\rho}^2}.
		\end{align}
	\end{subequations}
\end{lemma}

\begin{proof}
	First  note that since $ \b,\mathbf{c} $ are in $ \E_{\rho}^2 $, functions which we wish to estimate are in $ \big(\tilde{E}_{\rho}^{\mathbb{N}^*}\big)^2 $. In all this proof, we denote
	\begin{equation*}
		\b:=(\b_{\phi},\b_{\psi}):=\big(b_{\phi}^n,b_{\psi}^n\big)_{n\geq 1},\qquad \mathbf{c}:=(\mathbf{c}_{\phi},\mathbf{c}_{\psi}):=\big(c_{\phi}^n,c_{\psi}^n\big)_{n\geq 1}.
	\end{equation*}
	For $ \nu\geq 1 $ and $ s\in(0,1) $, we define $ s_{\nu}:=s+\inv{\nu+1} (1-s) $ which is such that $ s_{\nu}>s $ and satisfy, for $ \nu\geq 1 $,
	\begin{equation}\label{eq ana relations sn}
		\frac{1-s}{s_{\nu}-s}=\nu+1\qquad \text{and}\qquad \frac{1-s}{1-s_{\nu}}=1+\inv{\nu}.
	\end{equation}
	
	We proceed with the proof of estimate \eqref{eq ana est LBK E rho:L} dealing with function $ \mathbf{V}:=L\int_0^t\b(\sigma)\,d\sigma $, and we denote $ \mathbf{V}:=(\mathbf{V}_{\phi},\mathbf{V}_{\psi}):=\big(V_{\phi}^n,V_{\psi}^n\big)_{n\geq 1} $. According to Definition \ref{def E rho}, the aim is to estimate, for $ s\in(0,1) $, $ \nu\geq 0 $, $ \zeta=\phi,\psi $ and $ n\geq 1 $, the $ Y_s $-norm of $ \partial_t^{\nu} \,V_{\zeta}^n(0) $. Fix $ s\in(0,1) $, $ \zeta=\phi,\psi $ and $ n\geq 1 $, and recall that 
	\begin{equation*}
		V_{\zeta}^n(t)=\mathbf{v}^{\Lop}_{\zeta}\cdot\nabla_y\int_0^tb_{\zeta}^n(\sigma)\,d\sigma.
	\end{equation*}
	Therefore we have $ V_{\zeta}^n(0)=0 $ and, for $ \nu\geq 1 $,
	\begin{equation*}
		\partial_t^{\nu}V_{\zeta}^n(0)=\mathbf{v}^{\Lop}_{\zeta}\cdot\nabla_y\,\partial_t^{\nu-1}b_{\zeta}^n(0),
	\end{equation*}
	so, for $ s\in(0,1) $, using \eqref{eq ana est dt dtheta Ys} with $ 0<s<s_{\nu}\leq 1 $, estimate \eqref{eq ana hyp toy model v D int} and definition \eqref{eq ana est norme a} of $ E_{\rho} $-norm,
	\begin{align*}
		\norme{\partial_t^{\nu}\,V_{\zeta}^n(0)}_{Y_s}&\leq \gamma_0\norme{\nabla_y\,\partial_t^{\nu-1}b_{\zeta}^n(0)}_{Y_s}\leq \frac{C\,\gamma_0}{s_{\nu}-s}\norme{\partial_t^{\nu-1}b_{\zeta}^n(0)}_{Y_{s_{\nu}}}\\
		&\leq 	\frac{C\,\gamma_0}{s_{\nu}-s}\frac{1}{(1-s_{\nu})^{\nu}}\frac{\nu!\,\mathfrak{C}_{\nu-1}}{(4\rho)^{\nu-1}}\normetriple{b_{\zeta}^n}_{E_{\rho}}.
	\end{align*}
	Therefore, using relations \eqref{eq ana relations sn},
	\begin{align*}
		\norme{\partial_t^{\nu}\,V_{\zeta}^n(0)}_{Y_s}&\leq 	C\,\gamma_0\,\normetriple{b_{\zeta}^n}_{E_{\rho}}\, \inv{(1-s)^{\nu+1}}\frac{(\nu+1)\,\nu!\,\mathfrak{C}_{\nu-1}}{(4\rho)^{\nu-1}}\left(1+\inv{\nu}\right)^{\nu}\\
		&\leq C\,\rho\,\gamma_0\,\normetriple{b_{\zeta}^n}_{E_{\rho}}\,\frac{(\nu+1)!\,\mathfrak{C}_{\nu}}{(1-s)^{\nu+1}\,(4\rho)^{\nu}},
	\end{align*}
	so that $ \normetriple{V^n_{\zeta}}_{E_{\rho}}\leq C\,\rho\,\gamma_0\,\normetriple{b_{\zeta}^n}_{E_{\rho}} $. Since this estimate is independent of $ n,\zeta $, multiplying it by $ e^{2\rho n}\prodscalbis{n}^{2d^*} $ and summing over $ n\geq 1 $ leads to the analogous one for $ \E_{\rho}^2 $, which reads as \eqref{eq ana est LBK E rho:L}.
	
	For $ \mathbf{V}:=\partial_{\Theta}\,\mathbf{D}^{\Lop}\Big(\int_0^t\b(\sigma)\,d\sigma,\int_0^t\mathbf{c}(\sigma)\,d\sigma\Big) $ which we decompose in $ \big(\tilde{E}_{\rho}^{\mathbb{N}^*}\big)^2 $ as  $ \mathbf{V}=:\linebreak(\mathbf{V}_{\phi},\mathbf{V}_{\psi})=:\big(V_{\phi}^n,V_{\psi}^n\big)_{n\geq 1} $, with, for $ n\geq 1 $ and $ \zeta=\phi,\psi $,
	\begin{equation*}
		V_{\zeta}^n:=\partial_{\Theta}\,D^{\Lop}_{\zeta}\int_0^tb_{\zeta}^1(\sigma)\,d\sigma\cdot\int_0^tc_{\zeta}^n(\sigma)\,d\sigma,
	\end{equation*}
	for $ n\geq 1 $ and $ \zeta=\phi,\psi $, we compute $ V_{\zeta}^n(0)=\partial_tV_{\zeta}^n(0)=0 $, and for $ \nu\geq 2 $,
	\begin{equation*}
		\partial_t^{\nu}\,V_{\zeta}^n(0)=D^{\Lop}_{\zeta}\,\partial_{\Theta}\sum_{\mu=1}^{\nu-1}\binom{\nu}{\mu}\partial_t^{\mu-1}b^1_{\zeta}(0)\,\partial_t^{\nu-\mu-1}c^n_{\zeta}(0).
	\end{equation*}
	Therefore, for $ s\in(0,1) $, using \eqref{eq ana est dt dtheta Ys} with $ 0<s<s_{\nu}\leq 1 $ and estimates \eqref{eq ana hyp toy model v D int} and \eqref{eq ana est norme a}, we have
	\begin{align*}
		\norme{\partial_t^{\nu}V_{\zeta}^n(0)}_{Y_s}&\leq \frac{C\,\gamma_0}{s_{\nu}-s}\sum_{\mu=1}^{\nu-1}\binom{\nu}{\mu}\norme{\partial_t^{\mu-1}b^1_{\zeta}(0)}_{Y_{s_{\nu}}}\norme{\partial_t^{\nu-\mu-1}c^n_{\zeta}(0)}_{Y_{s_{\nu}}}\\
		&\leq \frac{C\,\gamma_0}{s_{\nu}-s} \frac{\normetriple{b^1_{\zeta}}_{E_{\rho}}\normetriple{c^n_{\zeta}}_{E_{\rho}}}{(4\rho)^{\nu-2}}\sum_{\mu=1}^{\nu-1}\binom{\nu}{\mu}\frac{\mu!\,\mathfrak{C}_{\mu-1}}{(1-s_{\nu})^{\mu}}\frac{(\nu-\mu)!\,\mathfrak{C}_{\nu-\mu-1}}{(1-s_{\nu})^{\nu-\mu}}.
	\end{align*}
	Again with relations \eqref{eq ana relations sn} we get
	\begin{align*}\nonumber
		\norme{\partial_t^{\nu}V_{\zeta}^n(0)}_{Y_s}&\leq 	\frac{C\,\gamma_0\,(\nu+1)}{1-s}\frac{\normetriple{b^1_{\zeta}}_{E_{\rho}}\normetriple{c^n_{\zeta}}_{E_{\rho}}}{(4\rho)^{\nu-2}}\frac{\nu!}{(1-s)^\nu}\sum_{\mu=1}^{\nu-1}\mathfrak{C}_{\mu-1}\,\mathfrak{C}_{\nu-\mu-1}\left(1+\inv{\nu}\right)^{\nu}\\
		&\leq C\,\rho^2\,\gamma_0\,e\,\normetriple{b^1_{\zeta}}_{E_{\rho}}\normetriple{c^n_{\zeta}}_{E_{\rho}}\frac{(\nu+1)!}{(1-s)^{\nu+1}\,(4\rho)^{\nu}}\sum_{\mu=1}^{\nu-1}\mathfrak{C}_{\mu-1}\,\mathfrak{C}_{\nu-\mu-1}.\label{eq ana preuve est LBK E rho:inter1}
	\end{align*}
	We can conclude using \eqref{eq ana relation catalan}:
	\begin{equation*}
		\sum_{\mu=1}^{\nu-1}\mathfrak{C}_{\mu-1}\,\mathfrak{C}_{\nu-\mu-1}=\sum_{\mu=0}^{\nu-2}\mathfrak{C}_{\mu}\,\mathfrak{C}_{\nu-2-\mu}=\mathfrak{C}_{\nu-1}\leq \mathfrak{C}_{\nu}.
	\end{equation*}
Last inequality follows from relation \eqref{eq ana relation catalan} and $ \mathfrak{C}_0=1 $.
%	and the conclusion follows.
	
	The proof of estimate \eqref{eq ana est LBK E rho:F} for $ \mathbf{V}:=\mathbb{F}^{\per}\Big(\partial_{\Theta}\int_0^t\b(\sigma)\,d\sigma,\partial_{\Theta}\int_0^t\mathbf{c}(\sigma)\,d\sigma\Big) $ follows the same argument as the previous one, but it is simpler since, according to \eqref{eq ana hyp toy model F per}, $ \mathbb{F}^{\per}_{\zeta}(\partial_{\Theta}\,\a,\partial_{\Theta}\,\a) $ acts as a semilinear term in $ \a $. It is therefore omitted.

	Finally, we take interest into $ \mathbf{V}:=\mathbf{K}^{\Lop}\Big(\int_0^t\b(\sigma)\,d\sigma,\int_0^t\mathbf{c}(\sigma)\,d\sigma\Big) $ which we decompose in $ \big(\tilde{E}_{\rho}^{\mathbb{N}^*}\big)^2 $ as  $ \mathbf{V}=:(\mathbf{V}_{\phi},\mathbf{V}_{\psi})=:\big(V_{\phi}^n,V_{\psi}^n\big)_{n\geq 1} $, with, for $ n\geq 1 $ and $ \zeta=\phi,\psi $,
	\begin{equation*}
		V_{\zeta}^n:=K^{\Lop}_{\zeta}\sum_{k=1}^{n-1}\partial_{\Theta}\big(b^k_{\phi},c^{n-k}_{\psi}\big)
	\end{equation*}
	For $ n\geq 1 $ and $ \zeta=\phi,\psi $, we compute $ V_{\zeta}^n(0)=\partial_tV_{\zeta}^n(0)=0 $, and for $ \nu\geq 2 $,
	\begin{equation*}
		\partial_t^{\nu}V_{\zeta}^n(0)=K^{\Lop}_{\zeta}\sum_{k=1}^{n-1}\partial_{\Theta}\sum_{\mu=1}^{\nu-1}\binom{\nu}{\mu}\partial_t^{\mu-1}b^k_{\phi}(0)\,\partial_t^{\nu-\mu-1}c^{n-k}_{\psi}(0),
	\end{equation*}
	so, for $ s\in(0,1) $, using \eqref{eq ana hyp toy model v D bord} and \eqref{eq ana est dt dtheta Ys} with $ 0<s<s_{\nu}\leq 1 $ and then \eqref{eq ana est norme a},
	\begin{align*}
		\norme{\partial_t^{\nu}V_{\zeta}^n(0)}_{Y_s}&\leq \frac{C\,\gamma_0}{s_{\nu}-s}\sum_{k=1}^{n-1}\sum_{\mu=1}^{\nu-1}\binom{\nu}{\mu}\norme{\partial_t^{\mu-1}b^k_{\phi}(0)}_{Y_{s_{\nu}}}\norme{\partial_t^{\nu-\mu-1}c^{n-k}_{\psi}(0)}_{Y_{s_{\nu}}}\\
		&\leq \frac{C\,\gamma_0}{s_{\nu}-s} \inv{(4\rho)^{\nu-2}}\sum_{k=1}^{n-1}\normetriple{b^k_{\phi}}_{E_{\rho}}\normetriple{c^{n-k}_{\psi}}_{E_{\rho}}\\
		&\quad\sum_{\mu=1}^{\nu-1}\binom{\nu}{\mu}\frac{\mu!\,\mathfrak{C}_{\mu-1}}{(1-s_{\nu})^{\mu}}\frac{(\nu-\mu)!\,\mathfrak{C}_{\nu-\mu-1}}{(1-s_{\nu})^{\nu-\mu}}\\
		&\leq C\,\rho^2\,\gamma_0\,\frac{(\nu+1)!\,\mathfrak{C}_{\nu}}{(1-s)^{\nu+1}(4\rho)^{\nu}}\sum_{k=1}^{n-1}\normetriple{b^k_{\phi}}_{E_{\rho}}\normetriple{c^{n-k}_{\psi}}_{E_{\rho}},
	\end{align*}
using \eqref{eq ana relation catalan} as before. Therefore we get
\begin{equation*}
	\normetriple{V^n_{\zeta}}_{E_{\rho}}\leq C\,\rho^2\,\gamma_0\sum_{k=1}^{n-1}\normetriple{b^k_{\phi}}_{E_{\rho}}\normetriple{c^{n-k}_{\psi}}_{E_{\rho}},
\end{equation*}
so that, multiplying by $ e^{2\,\rho\, n}\prodscalbis{n}^{2d^*} $ and summing over $ n\geq 1 $, we obtain
\begin{align*}
	\sum_{\zeta=\phi,\psi}\sum_{n\geq 1}&e^{2\,\rho\, n}\prodscalbis{n}^{2d^*}\normetriple{V^n_{\zeta}}_{E_{\rho}}^2\\[-10pt]
		&\leq C^2\,\rho^4\,\gamma_0^2\sum_{\zeta=\phi,\psi}\sum_{n\geq 1}e^{2\,\rho\, n}\prodscalbis{n}^{2d^*}\left(\sum_{k=1}^{n-1}\normetriple{b^k_{\phi}}_{E_{\rho}}\normetriple{c^{n-k}_{\psi}}_{E_{\rho}}\right)^2\\
		&\leq C^2\,\rho^4\,\gamma_0^2\sum_{n\geq 1}\left(\sum_{k=1}^{n-1}\frac{\prodscalbis{n}^{2d^*}}{\prodscalbis{k}^{2d^*}\prodscalbis{n-k}^{2d^*}}\right)\\
		&\quad\sum_{k=1}^{n-1}e^{2\,\rho\, k}\prodscalbis{k}^{2d^*}\normetriple{b^k_{\phi}}_{E_{\rho}}^2e^{2\,\rho\, (n-k)}\prodscalbis{n-k}^{2d^*}\normetriple{c^{n-k}_{\psi}}_{E_{\rho}}^2,
\end{align*}
using Cauchy-Schwarz inequality. Therefore, since $ \sum_{k=1}^{n-1}\frac{\prodscalbis{n}^{2d^*}}{\prodscalbis{k}^{2d^*}\prodscalbis{n-k}^{2d^*}} $ is bounded uniformly with respect to $ n\geq 1 $, we get
\begin{equation*}
	\normetriple{\mathbf{V}}^2_{\E^2_{\rho}}=\sum_{\zeta=\phi,\psi}\sum_{n\geq 1}e^{2\,\rho\, n}\prodscalbis{n}^{2d^*}\normetriple{V^n_{\zeta}}_{E_{\rho}}^2\leq C^2\,\rho^4\,\gamma_0^2\, \normetriple{\b}^2_{\E^2_{\rho}}\normetriple{\mathbf{c}}^2_{\E^2_{\rho}},
\end{equation*}
which is the sought inequality \eqref{eq ana est LBK E rho:K}.
\end{proof}

We are now in place to prove existence for \eqref{eq ana b = F(b)} (which is equivalent to \eqref{eq ana dt a}) in the space $ \E_{\rho}^2 $, of analytic functions with respect to $ (t,y,\Theta) $. We follow here the method of \cite{Ukai2001CauchyKovalevskaya,Morisse2020Elliptic}.

\begin{proposition}\label{prop existence bord}
	For every $ M_0>0 $, there exists $ \rho\in(0,1) $ such that for every $ \mathbf{H} $ in $ \E^2_{\rho} $ satisfying $ \normetriple{\mathbf{H}}_{\E^2_{\rho}}<M_0 $, equation \eqref{eq ana b = F(b)} admits a unique solution $ \b $ in $ \E_{\rho}^2 $.
	% satisfying $ \normetriple{\b}_{\E_{\rho}^2}\leq 2M_0 $.
\end{proposition}

\begin{proof}
	In all this proof, for $ R>0 $, $ B_{\rho}(0,R) $ denotes the closed ball of $ \E_{\rho}^2 $ centered at 0 and of radius $ R $. For $ \b $ and $ \mathbf{H} $ in $ \E_{\rho}^2 $, we denote
	\begin{multline*}
		F\big(\mathbf{H},\b\big):=L\int_0^t\b(\sigma)\,d\sigma-\partial_{\Theta}\,\mathbf{D}^{\Lop}\Big(\int_0^t\b(\sigma)\,d\sigma,\int_0^t\b(\sigma)\,d\sigma\Big)\\
		\quad-\mathbb{F}^{\per}\Big(\partial_{\Theta}\int_0^t\b(\sigma)\,d\sigma,\partial_{\Theta}\int_0^t\b(\sigma)\,d\sigma\Big)
		+\mathbf{K}^{\Lop}\Big(\int_0^t\b(\sigma)\,d\sigma,\int_0^t\b(\sigma)\,d\sigma\Big)+\mathbf{H},
	\end{multline*}
so that solving \eqref{eq ana b = F(b)} amounts to find a fixed point of $ F\big(\mathbf{H},.\big): \b \mapsto F\big(\mathbf{H},\b\big) $. Therefore, we will prove that there exist $ R>0 $ and $ \rho\in(0,1) $ such that for every $ \mathbf{H} $ in $ \E^2_{\rho} $ satisfying $ \normetriple{\mathbf{H}}_{\E^2_{\rho}}<M_0 $, the map $ F\big(\mathbf{H},.\big) $ is a contraction from the complete space $ B_{\rho}(0,R) $ to itself.

Consider $ M_0>0 $ and $ \mathbf{H} $ in $ \E^2_{\rho} $ such that $ \normetriple{\mathbf{H}}_{\E^2_{\rho}}<M_0 $. First we need to show that there exist $ \rho\in(0,1) $ and $ R>0 $ such that $ F\big(\mathbf{H},.\big) $ maps $ B_{\rho}(0,R) $ to itself. 
Lemma \ref{lemma ana est LBK E rho} asserts that $ F\big(\mathbf{H},.\big) $ is well defined from $ \E_{\rho}^2 $ to itself and that it satisfies, for $ \b $ in $ \E_{\rho}^2 $,
\begin{equation*}
	\normetriple{F\big(\mathbf{H},\b\big)}_{\E_{\rho}^2}\leq C\,\rho\,\gamma_0\, \normetriple{\b}_{\E_{\rho}^2}+C\,\rho^2\,\gamma_0\, \normetriple{\b}_{\E_{\rho}^2}^2+M_0,
\end{equation*}
for a new positive constant $ C>0 $ independent on $ \rho $, $ \mathbf{H} $, $ M_0 $ and $ \b $.
Therefore, setting $ R:=2M_0$, for $ 0<\rho<C(\gamma_0,M_0) $, with
\begin{equation*}
	C(\gamma_0,M_0):=\min\left(\big[4\,C\,\gamma_0\big]^{-1},\big[8\,C\,\gamma_0\,M_0\big]^{-1/2}\right),
\end{equation*}
the application $ F\big(\mathbf{H},.\big) $ maps the ball $ B_{\rho}(0,R) $ to itself. 

Now we need to show that this map is a contraction, for $ \rho<C(M_0,\gamma_0) $ small enough. We compute, for $ 0<\rho<C(M_0,\gamma_0)  $ and for $ \b,\mathbf{c} $ in $ B_{\rho}(0,R) $, 
\begin{align*}%\label{eq ana Fu-Fv}
	F\big(&\mathbf{H},\b\big)-F\big(\mathbf{H},\mathbf{c}\big)=\\\nonumber
	&L\int_0^t\big(\b-\mathbf{c}\big)(\sigma)\,d\sigma-\partial_{\Theta}\,\mathbf{D}^{\Lop}\left(\int_0^t\b(\sigma)\,d\sigma,\int_0^t\big(\b-\mathbf{c}\big)(\sigma)\,d\sigma\right)\\\nonumber
	&-\partial_{\Theta}\,\mathbf{D}^{\Lop}\left(\int_0^t\big(\b-\mathbf{c}\big)(\sigma)\,d\sigma,\int_0^t\mathbf{c}(\sigma)\,d\sigma\right)-\mathbb{F}^{\per}\left(\partial_{\Theta}\int_0^t\b(\sigma)\,d\sigma,\partial_{\Theta}\int_0^t\big(\b-\mathbf{c}\big)(\sigma)\,d\sigma\right)\\\nonumber
	&-\mathbb{F}^{\per}\left(\partial_{\Theta}\int_0^t\big(\b-\mathbf{c}\big)(\sigma)\,d\sigma,\partial_{\Theta}\int_0^t\mathbf{c}(\sigma)\,d\sigma\right)+\mathbf{K}^{\Lop}\left(\int_0^t\b(\sigma)\,d\sigma,\int_0^t\big(\b-\mathbf{c}\big)(\sigma)\,d\sigma\right)\\\nonumber
	&+\mathbf{K}^{\Lop}\left(\int_0^t\big(\b-\mathbf{c}\big)(\sigma)\,d\sigma,\int_0^t\mathbf{c}(\sigma)\,d\sigma\right).
\end{align*}
Therefore using estimates of Lemma \ref{lemma ana est LBK E rho} and the fact that $ \b,\mathbf{c} $ are in $ B_{\rho}(0,R) $, we get,
\begin{align*}	
	\normetriple{F\big(\mathbf{H},\b\big)-F\big(\mathbf{H},\mathbf{c}\big)}_{\E_{\rho}^2}&\leq C\,\rho\,\gamma_0\Big(1+\rho\,\big(\normetriple{\b}_{\E_{\rho}^2}+\normetriple{\mathbf{c}}_{\E_{\rho}^2}\big)\Big)\normetriple{\b-\mathbf{c}}_{\E_{\rho}^2}\\[5pt]
	&\leq C\,\rho\,\gamma_0\,\big(1+\rho\,R\big)\normetriple{\b-\mathbf{c}}_{\E_{\rho}^2},
\end{align*}
up to changing the constant $ C>0 $ in every line. Thus, for $ \rho<\tilde{C}(\gamma_0,M_0) $, with
\begin{equation*}
	\tilde{C}(\gamma_0,M_0):=\min\Big(\big[4\,\gamma_0\,M_0\big]^{-1},\big[8\,\,\gamma_0\,M_0^2\big]^{-1/2},
	\big[2\,C\,\gamma_0\,\big(1+2\,M_0\big)\big]^{-1}\Big),
\end{equation*}
the map $ F\big(\mathbf{H},.\big) $ is a contraction from $ B_{\rho}(0,R) $ to itself.

Since $ B_{\rho}(0,R) $ is a closed subspace of the Banach space $ \E^2_{\rho} $, the Banach fixed-point theorem gives a unique solution to \eqref{eq ana b = F(b)}.
\end{proof}

\subsubsection{A Cauchy-Kovalevskaya theorem for incoming interior equations}

The aim is now to prove existence of solution to \eqref{eq ana sigma}-\eqref{eq ana bord sigma} with the Cauchy-Kovalev\-skaya type Theorem \ref{theorem Cauchy-Kovalevskaya} using the chain of Banach spaces $ \big(\X_r\big)_{r\in(0,1)} $.

We start by writing system \eqref{eq ana sigma}-\eqref{eq ana bord sigma} in a form suited to apply Theorem \ref{theorem Cauchy-Kovalevskaya}.
Up to multiplying \eqref{eq ana sigma} by a nonzero constant (which is the $ x_d $-component of $ -\mathbf{v}_{\zeta,j} $), system \eqref{eq ana sigma}-\eqref{eq ana bord sigma} can be written as
\begin{equation}\label{eq ana Sigma inhom}
	\begin{cases}
		\partial_{x_d}\boldsymbol{\sigma}=L\,\boldsymbol{\sigma}-\partial_{\Theta}\,\mathbf{D}\big(\boldsymbol{\sigma},\boldsymbol{\sigma}\big)-\partial_{\Theta}\,\mathbb{J}\big(\boldsymbol{\sigma},\boldsymbol{\sigma}\big)+\partial_{\Theta}\,\mathbf{K}\big(\boldsymbol{\sigma},\boldsymbol{\sigma}\big)\\[5pt]
		\boldsymbol{\sigma}_{|x_d=0}=\tilde{\a} %+O\big(\partial_{y,\Theta}(\a\star\a)\big)
		+\mathbf{g},
	\end{cases}
\end{equation}
where $ \boldsymbol{\sigma}:=\big(\boldsymbol{\sigma}_{\phi,1},\boldsymbol{\sigma}_{\phi,3},\boldsymbol{\sigma}_{\psi,1},\boldsymbol{\sigma}_{\psi,3},\boldsymbol{\sigma}_{\nu,1},\boldsymbol{\sigma}_{\nu,3}\big):=\big(\sigma_{\phi,1}^n,\sigma_{\phi,3}^n,\sigma_{\psi,1}^n,\sigma_{\psi,3}^n,\sigma_{\nu,1}^n,\sigma_{\nu,3}^n\big)_{n\geq 1} $, function $ \a $ is the solution to \eqref{eq ana a}-\eqref{eq ana a init}, function $ \tilde{\a} $ is defined from $ \a $ by $ \tilde{\a}:=\big((e_{\phi,1}\cdot r_{\phi,1})\,\a_{\phi},(e_{\phi,3}\cdot r_{\phi,3})\,\a_{\phi},(e_{\psi,1}\cdot r_{\psi,1})\,\a_{\psi},(e_{\psi,3}\cdot r_{\psi,3})\,\a_{\psi},0,0\big) $, boundary term $ \mathbf{g} $ is defined as  $ \mathbf{g}:=\big(\mathbf{g}_{\phi,1},\dots,\mathbf{g}_{\nu,3}\big) $, and, if $ \boldsymbol{\tau}:=\big(\tau_{\phi,1}^n,\dots,\tau_{\nu,3}^n\big)_{n\geq 1} $, 
\begin{align*}
	L\,\boldsymbol{\sigma}&:=\big(\mathbf{v}_{\phi,1}\cdot\nabla_{t,y}\,\boldsymbol{\sigma}_{\phi,1},\dots,\mathbf{v}_{\nu,3}\cdot\nabla_{t,y}\,\boldsymbol{\sigma}_{\nu,3}\big),\\[5pt]
	\mathbf{D} \big(\boldsymbol{\sigma},\boldsymbol{\tau}\big)&:=\big(D_{\phi,1}\,\sigma_{\phi,1}^n\,\tau_{\phi,1}^1,\dots,D_{\nu,3}\,\sigma_{\nu,3}^n\,\tau_{\nu,3}^1\big)_{n\geq 1},\\
	\mathbb{J}\big(\boldsymbol{\sigma},\boldsymbol{\tau}\big)&:=\left(\sum_{\zeta_1,\zeta_2\in\ensemble{\phi,\psi,\nu}}\,\sum_{j_1,j_2\in\ensemble{1,3}}\mathbb{J}^{\zeta_2,j_2}_{\zeta_1,j_1}\big[\sigma^1_{\zeta_1,j_1},\tau^n_{\zeta_2,j_2}\big]\right)_{\substack{\zeta=\phi,\psi,\nu\\j=1,3\\n\geq 1}},\\[0pt]
	\mathbf{K}\big(\boldsymbol{\sigma},\boldsymbol{\tau}\big)&:=\left(K_{\zeta,j}\sum_{\zeta_1,\zeta_2\in\ensemble{\phi,\psi,\nu}}\,\sum_{j_1,j_2\in\ensemble{1,3}}\,\sum_{k=1}^{n-1}\sigma^k_{\zeta_1,j_1}\,\sigma^{n-k}_{\zeta_2,j_2}\right)_{\substack{\zeta=\phi,\psi,\nu\\j=1,3\\n\geq 1}},
\end{align*}
with new\footnote{Due to the fact that we multiplied equation \eqref{eq ana sigma} by a nonzero coefficient to obtain a propagation equation in the normal variable.} $ \mathbf{v}_{\zeta,j} $, $ D_{\zeta,j} $, $ K_{\zeta,j} $ and $ \mathbb{J}_{\zeta_1,j_1}^{\zeta_2,j_2} $, satisfying the same assumptions as the old ones \eqref{eq ana hyp toy model v D int} and \eqref{eq ana hyp toy model J}. Note that there exists a constant $ C>0 $ such that for any $ \a $ in $ \X_r^2 $, we have $ \normetriple{\tilde{\a}}_{\X_r^6}\leq C \normetriple{\a}_{\X_r^2} $.

For $ N_0>0 $, denote by $ \tilde{F} $ the function of $ [-N_0,N_0]\times\X_r^6 $ defined by, for $ |x_d|\leq N_0 $ and $ \boldsymbol{\sigma}\in\X_r^6 $,
\begin{equation*}
	\tilde{F}(x_d,\boldsymbol{\sigma}):=L\,\boldsymbol{\sigma}-\partial_{\Theta}\,\mathbf{D}\big(\boldsymbol{\sigma},\boldsymbol{\sigma}\big)-\partial_{\Theta}\,\mathbb{J}\big(\boldsymbol{\sigma},\boldsymbol{\sigma}\big)+\partial_{\Theta}\,\mathbf{K}\big(\boldsymbol{\sigma},\boldsymbol{\sigma}\big),
\end{equation*}
and set $ \boldsymbol{\sigma}^0:=\tilde{\a}+\mathbf{g} $.
Now system \eqref{eq ana Sigma inhom} is equivalent to the following one, with $ \boldsymbol{\tau}:=\boldsymbol{\sigma}-\boldsymbol{\sigma}^0 $,
\begin{equation}\label{eq ana Tau hom}
	\begin{cases}
		\boldsymbol{\tau}'(x_d)=\tilde{F}\big(x_d,\boldsymbol{\tau}(x_d)+\boldsymbol{\sigma}^0\big)\\
		\boldsymbol{\tau}(0)=0,
	\end{cases}
\end{equation}
which is, with $ F\big(x_d,\boldsymbol{\tau}(x_d)\big):=\tilde{F}\big(x_d,\boldsymbol{\tau}(x_d)+\boldsymbol{\sigma}^0\big) $ for $ |x_d|<N_0 $, in the right form to apply Theorem \ref{theorem Cauchy-Kovalevskaya}. Note that the operator $ F $ actually does not depend on $ x_d $, so all suprema in $ x_d $ in assumptions of Theorem \ref{theorem Cauchy-Kovalevskaya} may be removed below when verifying the assumptions of this theorem on our particular problem \eqref{eq ana Tau hom}. We will therefore omit to indicate the dependency in $ x_d $ of $ F $ and simply write $ F(\boldsymbol{\tau}) $. It remains to check the assumption of Theorem \ref{theorem Cauchy-Kovalevskaya} to obtain existence of solutions to \eqref{eq ana Tau hom}. The key estimates to do so are the following ones.

\begin{lemma}\label{lemma ana est LBR Xr}
	There exists $ C>0 $ such that for $ 0\leq r'<r\leq1 $, for $ \boldsymbol{\sigma},\boldsymbol{\tau} $ in $ \X_r^6 $, the following estimates hold
	\begin{subequations}\label{eq ana est LBR Xr}
		\begin{align}\label{eq ana est LBR Xr:L}
		\normetriple{L\,\boldsymbol{\sigma}}_{\X^6_{r'}}&\leq \frac{C\,\gamma_0}{r-r'}\,\normetriple{\boldsymbol{\sigma}}_{\X_r^6}\\[5pt]\label{eq ana est LBR Xr:B}
		\normetriple{\mathbf{D}\big(\boldsymbol{\sigma},\boldsymbol{\tau}\big)}_{\X^6_r}&\leq C\gamma_0 \normetriple{\boldsymbol{\sigma}}_{\X_r^6}\normetriple{\boldsymbol{\tau}}_{\X_r^6}\\[5pt]\label{eq ana est LBR Xr:R}
		\normetriple{\mathbb{J}\big(\boldsymbol{\sigma},\boldsymbol{\tau}\big)}_{\X^6_r}&\leq C\gamma_0 \normetriple{\boldsymbol{\sigma}}_{\X_r^6}\normetriple{\boldsymbol{\tau}}_{\X_r^6}\\[5pt]\label{eq ana est LBR Xr:etoile}
		\normetriple{\mathbf{K}\big(\boldsymbol{\sigma},\boldsymbol{\tau}\big)}_{\X^6_r}&\leq  C\gamma_0\normetriple{\boldsymbol{\sigma}}_{\X_r^6}\normetriple{\boldsymbol{\tau}}_{\X_r^6}.
		\end{align}
	\end{subequations}
\end{lemma}

\begin{proof}
	Estimate \eqref{eq ana est LBR Xr:L} follows directly from Lemma \ref{lemma ana dy dtheta Xr} and assumption \eqref{eq ana hyp toy model v D int} on $ \mathbf{v}_{\zeta,j} $, for $ \zeta=\phi,\psi,\nu $, and $ j=1,3 $. As for them, estimates \eqref{eq ana est LBR Xr:B} and \eqref{eq ana est LBR Xr:R} rely on the algebra property of $ X_r $ and assumptions \eqref{eq ana hyp toy model v D int} and \eqref{eq ana hyp toy model J} on $ D_{\zeta,j} $ and $ \mathbb{J}^{\zeta_2,j_2}_{\zeta_1,j_1} $. Finally, estimate \eqref{eq ana est LBR Xr:etoile} is proven using assumption \eqref{eq ana hyp toy model v D int} on $ K_{\zeta,j} $ for $ \zeta=\phi,\psi,\nu $, and $ j=1,3 $, and the same arguments used to prove algebra property of $ \X_r $.
\end{proof}

The main result of this part is the following one, which, along with Proposition \ref{prop existence bord}, will prove Theorem \ref{theorem existence ana}.

\begin{proposition}\label{prop existence eq int}
	Consider $ 0<r_1<1 $, and $ \a $ a solution to system \eqref{eq ana a}-\eqref{eq ana a init} given in $ \X_{r_1}^2 $. Then, for every $ M_1>0 $, the following existence and uniqueness result holds: there exists $ \delta>0 $ such that for every $ \mathbf{g} $ in $ \X^6_{r_1} $ satisfying $ \normetriple{\mathbf{g}}_{\X^6_{r_1}}<M_1 $, system \eqref{eq ana Tau hom} admits a unique solution in $ \mathcal{C}^1\big((-\delta(r_1-r),\delta(r_1-r)),\X_r^6\big) $ for each $ r\in(0,r_1) $.
\end{proposition}

\begin{proof}
%Since $ \a $ is in $ \E_{\rho_1}^2 $, according to Lemma \ref{lemma Erho dans Xr}, there exists $ r_1\in(0,1) $ such that $ \a $ is in $ \X^2_{r_1} $, so that $ \tilde{\a} $ is in $ \X_{r_1}^9 $. 
The aim is to apply Theorem \ref{theorem Cauchy-Kovalevskaya} with the scale of Banach spaces $ \big(\X^6_r\big)_{0<r\leq r_1} $. Fix now a constant $ M_1>0 $ as well as\footnote{Constant $ R $ takes part only in the proof, and can be chosen arbitrarily large.} $ R>0 $, and consider $ \mathbf{g} $ in $ \X_{r_1}^6 $ such that $ \normetriple{\mathbf{g}}_{\X_{r_1}^6}<M_1 $. We will now verify assumptions \eqref{eq ana hyp F 1} and \eqref{eq ana hyp F 2}. We compute that for $ \boldsymbol{\tau} $ in $ \X_{r_1}^6 $, we have
\begin{align*}
	F\big(\boldsymbol{\tau}\big)=&L\,\boldsymbol{\tau}+L\,\boldsymbol{\sigma^0}-\partial_{\Theta}\,\mathbf{D}\big(\boldsymbol{\tau},\boldsymbol{\tau}\big)-\partial_{\Theta}\,\mathbf{D}\big(\boldsymbol{\sigma^0},\boldsymbol{\tau}\big)-\partial_{\Theta}\,\mathbf{D}\big(\boldsymbol{\tau},\boldsymbol{\sigma^0}\big)-\partial_{\Theta}\,\mathbf{D}\big(\boldsymbol{\sigma^0},\boldsymbol{\sigma^0}\big)\\
	&-\partial_{\Theta}\,\mathbb{J}\big(\boldsymbol{\tau},\boldsymbol{\tau}\big)-\partial_{\Theta}\,\mathbb{J}\big(\boldsymbol{\tau},\boldsymbol{\sigma^0}\big)-\partial_{\Theta}\,\mathbb{J}\big(\boldsymbol{\sigma^0},\boldsymbol{\tau}\big)-\partial_{\Theta}\,\mathbb{J}\big(\boldsymbol{\sigma^0},\boldsymbol{\sigma^0}\big)\\
	&+\partial_{\Theta}\,\mathbf{K}\big(\boldsymbol{\tau},\boldsymbol{\tau}\big)+\partial_{\Theta}\,\mathbf{K}\big(\boldsymbol{\tau},\boldsymbol{\sigma^0}\big)+\partial_{\Theta}\,\mathbf{K}\big(\boldsymbol{\sigma^0},\boldsymbol{\tau}\big)+\partial_{\Theta}\,\mathbf{K}\big(\boldsymbol{\sigma^0},\boldsymbol{\sigma^0}\big).
\end{align*}
Therefore,
\begin{equation*}
	F(0)=L\,\boldsymbol{\sigma^0}-\partial_{\Theta}\,\mathbf{D}\big(\boldsymbol{\sigma^0},\boldsymbol{\sigma^0}\big)-\partial_{\Theta}\,\mathbb{J}\big(\boldsymbol{\sigma^0},\boldsymbol{\sigma^0}\big)+\partial_{\Theta}\,\mathbf{K}\big(\boldsymbol{\sigma^0},\boldsymbol{\sigma^0}\big),
\end{equation*}
so, using Lemmas \ref{lemma ana dy dtheta Xr} and \ref{lemma ana est LBR Xr}, we get that there exists a constant $ C>0 $ such that for all $ 0<r'<r<r_1 $,
\begin{equation*}
	\normetriple{F(0)}_{\X_{r'}^6}\leq \frac{C\,\gamma_0}{r-r'} \Big(\normetriple{\boldsymbol{\sigma^0}}_{\X_r^6}+\normetriple{\boldsymbol{\sigma^0}}_{\X_r^6}^2\Big)\leq \frac{C\,\gamma_0}{r-r'} \Big(\normetriple{\boldsymbol{\sigma^0}}_{\X_r^6}+\normetriple{\boldsymbol{\sigma^0}}_{\X_r^6}^2\Big),
\end{equation*}
and then, using the fact that $ \normetriple{\tilde{\a}}_{\X_r^6}\leq C \normetriple{\a}_{\X_r^2} $ and $ \normetriple{\mathbf{g}}_{\X_{r_1}^6}<M_1 $, 
\begin{equation*}
	\normetriple{F(0)}_{\X_{r'}^6}\leq \frac{C\,\gamma_0}{r-r'}\Big(\normetriple{\a}_{\X_r^2}+\normetriple{\a}^2_{\X_r^2}+M_1+M_1^2\Big),
\end{equation*}
so assumption \eqref{eq ana hyp F 2} is satisfied with $ M:=C\,\gamma_0\big(\normetriple{a}_{\X_r^2}+\normetriple{a}^2_{\X_r^2}+M_1+M_1^2\big) $.

On the other hand, we have, for $ \boldsymbol{\tau},\boldsymbol{\omega} $ in $ \X_r^6 $,
\begin{align*}
	F\big(\boldsymbol{\tau}\big)-F\big(\boldsymbol{\omega}\big)&=L\,\big(\boldsymbol{\tau}-\boldsymbol{\omega}\big)-\partial_{\Theta}\,\mathbf{D}\big(\boldsymbol{\tau}-\boldsymbol{\omega},\boldsymbol{\tau}\big)-\partial_{\Theta}\,\mathbf{D}\big(\boldsymbol{\omega},\boldsymbol{\tau}-\boldsymbol{\omega}\big)-\partial_{\Theta}\,\mathbf{D}\big(\boldsymbol{\sigma^0},\boldsymbol{\tau}-\boldsymbol{\omega}\big)\\
	&\quad-\partial_{\Theta}\,\mathbf{D}\big(\boldsymbol{\tau}-\boldsymbol{\omega},\boldsymbol{\sigma^0}\big)-\partial_{\Theta}\,\mathbb{J}\big(\boldsymbol{\tau}-\boldsymbol{\omega},\boldsymbol{\tau}\big)-\partial_{\Theta}\,\mathbb{J}\big(\boldsymbol{\omega},\boldsymbol{\tau}-\boldsymbol{\omega}\big)-\partial_{\Theta}\,\mathbb{J}\big(\boldsymbol{\sigma^0},\boldsymbol{\tau}-\boldsymbol{\omega}\big)\\
	&\quad-\partial_{\Theta}\,\mathbb{J}\big(\boldsymbol{\tau}-\boldsymbol{\omega},\boldsymbol{\sigma^0}\big)
	+\partial_{\Theta}\,\mathbf{K}\big(\boldsymbol{\tau}-\boldsymbol{\omega},\boldsymbol{\tau}\big)
	+\partial_{\Theta}\,\mathbf{K}\big(\boldsymbol{\omega},\boldsymbol{\tau}-\boldsymbol{\omega}\big) +\partial_{\Theta}\,\mathbf{K}\big(\boldsymbol{\sigma^0},\boldsymbol{\tau}-\boldsymbol{\omega}\big)\\
	&\quad
	+\partial_{\Theta}\,\mathbf{K}\big(\boldsymbol{\tau}-\boldsymbol{\omega},\boldsymbol{\sigma^0}\big).
\end{align*}
Therefore, using Lemmas \ref{lemma ana dy dtheta Xr} and \ref{lemma ana est LBR Xr}, we get that there exists $ C>0 $ such that for all $ 0<r'<r<r_1 $, for each $ \boldsymbol{\tau},\boldsymbol{\omega} $ in $ \ensemble{\boldsymbol{\sigma}\in \X_r^6\middle| \normetriple{\boldsymbol{\sigma}}_{\X_r^6}<R} $,
\begin{align*}
	\normetriple{F\big(\boldsymbol{\tau}\big)-F\big(\boldsymbol{\omega}\big)}_{\X_{r'}^6}&\leq C\,\gamma_0\Big(1+\normetriple{\boldsymbol{\tau}}_{\X_r^6}+\normetriple{\boldsymbol{\omega}}_{\X_r^6}+\normetriple{\boldsymbol{\sigma}^0}_{\X_r^6}\Big)\frac{\normetriple{\boldsymbol{\tau}-\boldsymbol{\omega}}_{\X_r^6}}{r-r'}\\
	&\leq \frac{C\,\gamma_0\,\big(1+R+\normetriple{\a}_{\X_r^2}+M_1\big)}{r-r'}\,\normetriple{\boldsymbol{\tau}-\boldsymbol{\omega}}_{\X_r^6}.
\end{align*}
This estimate asserts that both the continuity property of $ F $ and assumption \eqref{eq ana hyp F 1} with $ C:=C\,\gamma_0\,\big(1+R+\normetriple{\a}_{\X_r^2}+M_1\big) $ are satisfied. We can therefore apply Theorem \ref{theorem Cauchy-Kovalevskaya} which gives the sought result.
\end{proof}

Proof of Theorem \ref{theorem existence ana} is  now straightforward. Fix two positive constants $ M_0 $ and $ M_1 $.  Proposition \ref{prop existence bord} asserts the existence of $ \rho_1\in(0,1) $ such that for all $ \mathbf{H} $ in $ \E_1^2 $ satisfying $ \normetriple{\mathbf{H}}_{\E_1^2}<M_0 $, there exists a solution $ \a $ in $ \E_{\rho_1}^2 $ to system \eqref{eq ana a}-\eqref{eq ana a init}. Then, Lemma \ref{lemma Erho dans Xr} ensures that there exists $ r_1 $ (depending only on $ \rho_1 $), such that the solution $ \a $ to \eqref{eq ana a}-\eqref{eq ana a init} is in $ \X_{r_1} $. 
%Setting $ r_0:=r_1/2 $, 
Proposition \ref{prop existence eq int} gives the existence of $ \delta>0 $ such that for every $ \mathbf{g} $ in $ \X_{r_1}^6 $ satisfying $ \normetriple{\mathbf{g}}_{\X^6_{r_1}}<M_1 $, there exists a solution $ \boldsymbol{\sigma} $ in $ \mathcal{C}^1\big((-\delta(r_1-r),\delta(r_1-r)),\X_r^6\big) $, for each $ r\in(0,r_1) $, to system \eqref{eq ana sigma}-\eqref{eq ana bord sigma}. This is precisely the statement of Theorem \ref{theorem existence ana}.

\bigskip

To make the simplified model \eqref{eq ana sigma},  \eqref{eq ana bord sigma}, \eqref{eq ana a} and \eqref{eq ana a init} more complicated and closer to the general system  \eqref{eq systeme moyenne}, \eqref{eq systeme sigma hors cas part}, \eqref{eq systeme sigma}, \eqref{eq systeme sigma bord} and \eqref{eq systeme a}, several aspects could be incorporated in the former one. Outgoing equations associated with boundary frequencies $ \phi $, $ \psi $ and $ \nu $ could be integrated in interior equations \eqref{eq ana sigma}. It raises mainly an issue of functional framework, as we solved incoming equations \eqref{eq ana sigma} as propagation equations in the normal variable, which is not a framework suited for outgoing equations. Then it would be possible to incorporate traces of outgoing profiles in boundary conditions \eqref{eq ana bord sigma} and boundary evolution equations \eqref{eq ana a}. For that we would need trace estimates for the chosen functional framework. 
In a more distant perspective, we could integrate profiles associated with boundary frequencies different from $ \phi $, $ \psi $ and $ \nu $ in interior equations \eqref{eq ana sigma} and boundary equations \eqref{eq ana a}, which would require a total change of the functional framework, since we would have to work with almost-periodic functions. We could also consider derivatives of order higher than one in source terms of theses equations \eqref{eq ana sigma} and \eqref{eq ana a}

\section{Instability}

This section is devoted to the proof of instability. More precisely, the aim is to show that the perturbation $ H $ in \eqref{eq systeme 1} interferes at a leading order in the asymptotic expansion \eqref{eq ansatz}. This is not the case in general, where the perturbation $ \epsilon^3\,h^{\epsilon} $ only interferes at order $ \epsilon^2 $ and higher, see \cite{MajdaArtola1988Mixed}. As the perturbation $ \epsilon^3\,h^{\epsilon} $ of $ \epsilon^2\,g^{\epsilon} $ in \eqref{eq systeme 1} is small, we will work with the linearized system of system \eqref{eq systeme moyenne}, \eqref{eq systeme sigma hors cas part}, \eqref{eq systeme sigma}, \eqref{eq systeme sigma bord} and \eqref{eq systeme a},  around the particular solution when the perturbation is zero. To simplify even more the computations we will prove instability on simplified models of the linearized system. The first part of the section focuses on deriving the linearized system for the profiles. 

\subsection{Linearization around a particular solution}

If the perturbation $ H $ is uniformly zero in \eqref{eq systeme 1}, then we are brought back to the case of \cite{CoulombelWilliams2017Mach}, and the solution obtained in the mentioned work is thus a solution to our cascade of equations \eqref{eq systeme moyenne}, \eqref{eq systeme sigma hors cas part}, \eqref{eq systeme sigma}, \eqref{eq systeme sigma bord} and \eqref{eq systeme a} in this particular case. Therefore, according to \cite{CoulombelWilliams2017Mach}, we have the following result.

\begin{proposition}[{\cite[Theorem 1.10]{CoulombelWilliams2017Mach}}]\label{prop exist sol part}
	Let $ T_0>0 $, and consider $ G $ in $ \mathcal{C}^{\infty}\big((-\infty,T_0],\linebreak H^{\infty}(\R^{d-1}\times\T)\big) $, zero for negative times $ t $, and $ H\equiv 0 $. Then there exists $ T\in(0,T_0] $ and unique sequences of functions $ \big(\bar{U}_n^*\big)_{n\geq 0} $, and $ \big(\bar{\sigma}^n_{\zeta,j,\lambda}\big)_{n\geq 0} $ for $ \zeta=\phi,\psi,\nu $ and $ j=1,2,3 $ in  $ \mathcal{C}^{\infty}\big((-\infty,T_0], H^{\infty}(\R^{d-1}\times\R_+\times\T)\big) $ and sequences $ \big(\bar{a}_{\zeta,\lambda}^n \big)_{n\geq 1} $ for $ \zeta=\phi,\psi $ in $ \mathcal{C}^{\infty}\big((-\infty,T_0], \linebreak H^{\infty}(\R^{d-1}\times\T)\big) $, solution of the cascade of equations \eqref{eq systeme moyenne}, \eqref{eq systeme sigma hors cas part}, \eqref{eq systeme sigma}, \eqref{eq systeme sigma bord} and \eqref{eq systeme a}.
\end{proposition}

Note that Theorem \ref{theorem existence ana} constitute a version of Proposition \ref{prop exist sol part} in the case where $ H $ is possibly nonzero, but only on a simplified model of system \eqref{eq systeme moyenne}, \eqref{eq systeme sigma hors cas part}, \eqref{eq systeme sigma}, \eqref{eq systeme sigma bord} and \eqref{eq systeme a}, and with a different functional framework.
Note also that since $ H $ is zero, we have, for $ n\geq 1 $, $ \lambda\in\Z^* $, for the solution of Proposition \ref{prop exist sol part},
\begin{equation}\label{eq insta sigma a bar zero}
	\bar{\sigma}^n_{\zeta,j,\lambda}=0 \quad \mbox{for $ \zeta\neq \phi $ and $ j\in\mathcal{C}(\zeta) $,}\qquad \mbox{and}\qquad
	\bar{a}^n_{\psi,\lambda}=0.
\end{equation}

The aim of this part is to derive the linearization of system \eqref{eq systeme moyenne}, \eqref{eq systeme sigma hors cas part}, \eqref{eq systeme sigma}, \eqref{eq systeme sigma bord} and \eqref{eq systeme a} around the particular solution of Proposition \ref{prop exist sol part}. Schematically, in order to study the general problem of the form $ \mathcal{F}(u)=(G,H,0,\dots) $, we linearize this problem around the particular solution $ \bar{u} $ of $ \mathcal{F}(\bar{u})=(G,0,0,\dots) $ to obtain the linearized problem $ d\mathcal{F}(\bar{u})\cdot u =(0,H,0,\dots) $. We will also simplify the linearized system during its derivation, since for some profiles it is easy to show that they are zero. 

We only detail the linearized equations for the order we are interested in, which are first and second orders, and only for profiles of interest, that is, $ \sigma^n_{\psi,1,\lambda} $, $ \sigma^n_{\psi,2,\lambda} $, $ \sigma^n_{\nu,2,\lambda} $ and $ a^n_{\psi,\lambda} $, for $ \lambda\in\Z^* $ and $ n=1,2 $. Here, opposite to the formulation of \eqref{eq systeme sigma},  we write down each equation separately, as they are now different since $ \bar{\sigma}_{\zeta,j,\lambda}^n $ is zero for $ \zeta\neq \phi $. We also adopt a new color code for these equations. 

For the leading profile, starting from equations \eqref{eq evolution sigma 1}, we get, for the phases $ \psi_1 $, $ \psi_2 $ and $ \nu_2 $, and for $ \lambda\in\Z^* $,
\begin{subequations}\label{eq linearise eq prof princ}
	\begin{align}\label{eq linearise eq prof princ psi1}
		\color{altblue}X_{\psi,1}\,\sigma^1_{\psi,1,\lambda}+\indicatrice_{\lambda=k\lambda_{\psi}}\,J^{\phi,1}_{\nu,2}\,ik\,\bar{\sigma}^1_{\phi,1,-k\lambda_{\phi}}\,\sigma^1_{\nu,2,-k}&\color{altblue}=0,\\[5pt]\label{eq linearise eq prof princ psi2}
		\color{altpink}X_{\psi,2}\,\sigma^1_{\psi,2,\lambda}+\indicatrice_{\lambda=k\lambda_{\psi}}\,J^{\phi,3}_{\nu,2}\,ik\,\bar{\sigma}^1_{\phi,3,-k\lambda_{\phi}}\,\sigma^1_{\nu,2,-k}&\color{altpink}=0,\\[5pt]\label{eq linearise eq prof princ nu2}
		\color{altpurple}X_{\nu,2}\,\sigma^1_{\nu,2,\lambda}
		+J^{\phi,1}_{\psi,1}\,i\lambda\,\bar{\sigma}^1_{\phi,1,-\lambda\lambda_{\phi}}\,\sigma^1_{\psi,1,-\lambda\lambda_{\psi}}+J^{\phi,3}_{\psi,2}\,i\lambda\,\bar{\sigma}^1_{\phi,3,-\lambda\lambda_{\phi}}\,\sigma^1_{\psi,2,-\lambda\lambda_{\psi}}&\color{altpurple}=0.
	\end{align}
\end{subequations}
In equations \eqref{eq linearise eq prof princ psi1} and \eqref{eq linearise eq prof princ psi2}, if $ \lambda\notin \lambda_{\psi}\Z $, no resonance happens, but if $ \lambda=k\lambda_{\psi} $ for some $ k\in\Z^* $, then, for example for the phase $ \psi_1 $, the resonance $ k\lambda_{\phi}\,\phi_1+k\lambda_{\psi}\,\psi_1+k\,\nu_2=0 $ occurs. This explains the presence of factors $ \indicatrice_{\lambda=k\lambda_{\psi}} $ in equations \eqref{eq linearise eq prof princ psi1} and \eqref{eq linearise eq prof princ psi2}.
We also have, for $ j=1,3 $ and $ \lambda\in\Z^* $, the transport equation
\begin{equation}\label{eq linearise eq prof princ phi}
		X_{\phi,j}\,\sigma^1_{\phi,j,\lambda}+D_{\phi,j}\sum_{\lambda_1+\lambda_2=\lambda}i\lambda_2\,\bar{\sigma}^1_{\phi,j,\lambda_1}\,\sigma^1_{\phi,j,\lambda_2}=0.
\end{equation}
%and, for $ \zeta\in\F_{b}\privede{0,\phi,\psi,\nu} $ for $ \zeta=\psi $ and $ j=3 $ and for $ \zeta=\nu $ and $ j=1,3 $,
%\begin{align}\label{eq linearise eq prof princ hors cas part}
%	X_{\zeta,j}\,\sigma^1_{\zeta,j,\lambda}=0.
%\end{align}

As for them, the linearized equations for boundary terms $ a^1_{\phi} $ and $ a^1_{\psi} $ read
\begin{equation}\label{eq linearise eq bord prof princ phi}
X^{\Lop}_{\phi}\,a^1_{\phi,\lambda}
	+v_{\phi}\,\sum_{\lambda_1+\lambda_2=\lambda}i\lambda_2\,\bar{a}^1_{\phi,\lambda_1}\,a^1_{\phi,\lambda_2}+\lambda\sum_{\lambda_1+\lambda_2=\lambda}\gamma_{\phi}(\lambda_1,\lambda_2)\,\big(\bar{a}^1_{\phi,\lambda_1}\,a^1_{\phi,\lambda_2}+a^1_{\phi,\lambda_1}\,\bar{a}^1_{\phi,\lambda_2}\big)=0,
\end{equation}
and
\begin{equation}\label{eq linearise eq bord prof princ psi}
\color{altorange2}X^{\Lop}_{\psi}\,a^1_{\psi,\lambda}
+\indicatrice_{\lambda=k\lambda_{\psi}}\,\Gamma^{\psi}\,ik\,\big(\sigma^1_{\nu,2,k}\big)_{|x_d=0}\,\bar{a}^1_{\phi,-\lambda_{\phi}k}=i\lambda\,b_{\psi}\cdot B\,\big(U^2_{\psi,2,\lambda}\big)_{|x_d,\chi_d=0}.
\end{equation}
There is no term in $ G $ in equation \eqref{eq linearise eq bord prof princ phi} since we linearized around the solution given by Proposition \ref{prop exist sol part} corresponding to the source term $ H=0 $ and we study the influence of a small source term $ H $ on the leading amplitudes $ \sigma^1 $.

Equations for the boundary phase $ \phi $ are decoupled from the others, they can therefore be solved, using for example \cite[Theorem 1.10]{CoulombelWilliams2017Mach}. From \eqref{eq linearise eq bord prof princ phi}, along with the initial condition $ \big(a^1_{\phi,\lambda}\big)_{|t\leq 0}=0 $, we obtain $ a^1_{\phi,\lambda}=0 $ for $ \lambda\in\Z^* $. Using boundary condition \eqref{eq prof princ bord phi} as well as initial condition $ \big(\sigma^1_{\phi,j,\lambda}\big)_{|t\leq 0}=0 $, we also get $ \sigma^1_{\phi,j,\lambda}=0 $ for $ j=1,3 $ and $ \lambda\in\Z^* $. Summing up, we have
\begin{equation}\label{eq linearise sigma a phi zero}
	\sigma_{\phi,\lambda,j}^1=0,\quad a^1_{\phi,\lambda}=0,\quad \forall\lambda\in\Z^*,\forall j=1,3.
\end{equation}
The other frequencies $ \psi_1 $, $ \psi_2 $ and $ \nu_2 $ are totally coupled through equations \eqref{eq linearise eq prof princ} and \eqref{eq linearise eq bord prof princ psi}, and we need to determine the function $ a^1_{\psi} $ on the boundary and thus the outgoing amplitude $ \sigma_{\psi,2,\lambda}^2 $.

\bigskip

In the same way as for the leading profile, for the first corrector, starting from \eqref{eq 1cor evol sigma phi psi nu}, we get, for the phases $ \psi_1 $, $ \psi_2 $ and $ \nu_2 $,
\begin{subequations}\label{eq linearise evol psi nu}
	\begin{align}
		\color{altblue} X_{\psi,1}\,\sigma^2_{\psi,1,\lambda}+\indicatrice_{\lambda=k\lambda_{\psi}}\,J^{\phi,1}_{\nu,2}\,ik\,\big(\bar{\sigma}^1_{\phi,1,-k\lambda_{\phi}}\,\sigma^2_{\nu,2,-k}+\bar{\sigma}^2_{\phi,1,-k\lambda_{\phi}}\,\sigma^1_{\nu,2,-k}\big)&\color{altblue}\\\nonumber&\color{altblue}\hspace{-40pt}=\mbox{\emph{terms in}}\,\big(U_1,(I-P)\,U_2\big),\\
		\color{altpink} X_{\psi,2}\,\sigma^2_{\psi,2,\lambda}+\indicatrice_{\lambda=k\lambda_{\psi}}J^{\phi,3}_{\nu,2}\,ik\,\big(\bar{\sigma}^1_{\phi,3,-k\lambda_{\phi}}\,\sigma^2_{\nu,2,-k}+\bar{\sigma}^2_{\phi,3,-k\lambda_{\phi}}\,\sigma^1_{\nu,2,-k}\big)&
		\\\nonumber&\color{altpink}\hspace{-40pt}=\mbox{\emph{terms in}}\,\big(U_1,(I-P)\,U_2\big),\\
		\color{altpurple}X_{\nu,2}\,\sigma^2_{\nu,2,\lambda}+J^{\phi,1}_{\psi,1}\,i\lambda\,\big(\bar{\sigma}^1_{\phi,1,-\lambda\lambda_{\phi}}\,\sigma^2_{\psi,1,-\lambda\lambda_{\psi}}+\bar{\sigma}^2_{\phi,1,-\lambda\lambda_{\phi}}\,\sigma^1_{\psi,1,-\lambda\lambda_{\psi}}\big)&
		\\\nonumber\color{altpurple}&\color{altpurple}\hspace{-289pt}+J^{\phi,3}_{\psi,2}\,i\lambda\,\big(\bar{\sigma}^1_{\phi,3,-\lambda\lambda_{\phi}}\,\sigma^2_{\psi,2,-\lambda\lambda_{\psi}}+\bar{\sigma}^2_{\phi,3,-\lambda\lambda_{\phi}}\,\sigma^1_{\psi,2,-\lambda\lambda_{\psi}}\big)=\mbox{\emph{terms in}}\,\big(U_1,(I-P)\,U_2\big),
	\end{align}
\end{subequations}
where $ \mbox{\emph{terms in}}\,\big(U_1,(I-P)\,U_2\big) $ refer to quadratic terms in $ U_1 $ or the nonpolarized parts of $ U_2 $, both of them for frequencies $ \zeta_j $, with $ \zeta=\phi,\psi,\nu $ and $ j=1,3 $, terms which therefore will be zero if the corresponding profiles are zero. Equations on other profiles $ \sigma^2_{\zeta,j,\lambda} $, for $ \zeta\in\F_{b}\privede{\psi,\nu} $ and $ (\zeta,j)=(\psi,3),(\nu,1),(\nu,3) $, are not of interest so we do not write them.

For the boundary term $ a^2_{\psi} $, we have, according to \eqref{eq 2cor eq evol a psi},
\begin{align}\label{eq linearise eq bord 1cor}
	&\quad\color{altorange2}X^{\Lop}_{\psi}\,a^2_{\psi,\lambda}
	+\tilde{X}^{\Lop}_{\psi}\,\big(\sigma_{\psi,2,\lambda}^2\big)_{|x_d=0}\\\nonumber
	&\quad\color{altorange2}+\indicatrice_{\lambda=k\lambda_{\psi}}\,\Gamma^{\psi}\,ik\,\big\lbrace\big(\sigma_{\nu,2,k}^1\big)_{|x_d=0}\,\bar{a}^2_{\phi,-\lambda_{\phi}k}+\big(\sigma_{\nu,2,k}^2\big)_{|x_d=0}\,\bar{a}^1_{\phi,-\lambda_{\phi}k}\big\rbrace\\\nonumber
	&\color{altorange2}=i\lambda\,b_{\psi}\cdot B\,\big(U^{3,\osc}_{\psi,2,\lambda}\big)_{|x_d,\chi_d=0}-i\lambda\,b_{\psi}\cdot H_{\lambda}\\\nonumber
	&\quad\color{altorange2}+\partial_{z,\theta}\,\mbox{\emph{terms in}}\,\big(U_1,(I-P)\,U_2,(P\,U_2)_{\zeta\neq \phi,\psi,\nu},U_2^*\big)_{|x_d,\chi_d=0}.
\end{align}
Again, here, equations for $ \psi_1 $, $ \psi_2 $ and $ \nu_2 $ are coupled. As the coupling is difficult to handle, especially with the term $ i\lambda\,b_{\psi}\cdot B\,\big(U^{3,\osc}_{\psi,2,\lambda}\big)_{|x_d,\chi_d=0} $, we will simplify equations \eqref{eq linearise evol psi nu} and \eqref{eq linearise eq bord 1cor} to reduce the coupling, in order to study instability. 

We have obtained the system \eqref{eq linearise eq prof princ}, \eqref{eq linearise eq bord prof princ psi}, \eqref{eq linearise evol psi nu} and \eqref{eq linearise eq bord 1cor}, which is the linearization of system of equations \eqref{eq systeme sigma} and \eqref{eq systeme a}, around the particular solution of \eqref{eq systeme moyenne}, \eqref{eq systeme sigma hors cas part}, \eqref{eq systeme sigma}, \eqref{eq systeme sigma bord} and \eqref{eq systeme a} of Proposition \ref{prop exist sol part} for which the boundary term $ H $ is zero. 

\subsection{Instability on simplified models}

The aim of this section is to show that the system \eqref{eq systeme 1} considered in this article is unstable, namely that a small perturbation $ H $ in the boundary term may interfere up to the leading order. More precisely we prove that there exists a boundary term $ H $ such that, for simplified models of the linearized system \eqref{eq linearise eq prof princ}, \eqref{eq linearise eq bord prof princ psi}, \eqref{eq linearise evol psi nu} and \eqref{eq linearise eq bord 1cor}, the leading perturbations $ \sigma^1_{\psi,j,\lambda} $ and $ \sigma^1_{\nu,j,\lambda} $ are not all zero. 
For this purpose, we argue by contradiction and assume that for every boundary term $ H $, all amplitudes $ \sigma^1_{\psi,j,\lambda} $ and $ \sigma^1_{\nu,j,\lambda} $, for $ j=1,3 $ and $ \lambda\in\Z^* $ are zero. Then we seek for a contradiction. In particular, according to \eqref{eq prof princ bord psi}, it implies that $ a^1_{\psi,\lambda}=0 $ for all $ \lambda\in\Z^* $.

Recall that we have shown above that for the linearized system, profiles $ \sigma^1_{\phi,j,\lambda} $ for $ j=1,3 $ and $ \lambda\in\Z^* $ are zero. Therefore, all leading profiles of frequencies $ \zeta_j $, $ \zeta=\phi,\psi,\nu $ and $ j=1,3 $ are zero. Furthermore, according to formula \eqref{eq 1cor nonpola part U 2 zeta} giving the nonpolarized parts of the first corrector, the nonpolarized parts of $ U_2 $ for frequencies $ \zeta_j $, $ \zeta=\phi,\psi,\nu $ and $ j=1,3 $ are consequently also zero. 
We can also show in a similar manner, that the mean value $ U_2^* $ and the polarized parts of frequencies different from $ \zeta_j $, $ \zeta=\phi,\psi,\nu $ and $ j=1,3 $, are also zero. 

Therefore, equation \eqref{eq linearise eq bord prof princ psi} now reads
\begin{equation}\label{eq insta eq trace psi 2 = 0}
	\big(\sigma^2_{\psi,2,\lambda}\big)_{|x_d=0}=0,
\end{equation}
since $ U^2_{\psi,2,\lambda} $ is polarized, $ a^1_{\psi,\lambda} $ is zero for $ \lambda\in\Z^* $ and the scalar $ b_{\psi}\cdot B\,r_{\psi,2} $ is nonzero. Equation \eqref{eq insta eq trace psi 2 = 0} is the condition which we wish to contradict.

The general linearized equations \eqref{eq linearise eq prof princ}, \eqref{eq linearise eq bord prof princ psi}, \eqref{eq linearise evol psi nu} and \eqref{eq linearise eq bord 1cor} being too difficult to handle at this stage of comprehension, two simplified models are investigated. 

\subsubsection{First simplified model}

We focus first on a very simple simplified model, for which computations can be easily followed through the end, and reflect the general idea of the instability mechanism. In equations \eqref{eq linearise evol psi nu}, most of the resonant terms, which couple the equations, are removed. We also use that both leading profiles and nonpolarized parts of $ U_2 $ of frequencies $ \zeta_j $ for $ \zeta=\phi,\psi,\nu $ and $ j=1,3 $ are zero. We retain at the end, for phases $ \psi_1 $, $ \psi_2 $ and $ \nu_2 $, the incoming evolution equation
\begin{subequations}\label{eq insta eq 1cor BB1}
	\begin{align}
		\color{altblue}\label{eq insta eq 1cor BB1 psi 1} X_{\psi,1}\,\sigma^2_{\psi,1,\lambda}&\color{altblue}=0,
		\intertext{and the two outgoing evolution equations with resonance terms}
		\color{altpink}\label{eq insta eq 1cor BB1 psi2} X_{\psi,2}\,\sigma^2_{\psi,2,\lambda}+\indicatrice_{\lambda=k\lambda_{\psi}}\,J^{\phi,3}_{\nu,2}\,ik\,\bar{\sigma}^1_{\phi,3,-k\lambda_{\phi}}\,\sigma^2_{\nu,2,-k}&\color{altpink}=0,\\[5pt]
		\color{altpurple}\label{eq insta eq 1cor BB1 nu2} X_{\nu,2}\,\sigma^2_{\nu,2,\lambda}+J^{\phi,1}_{\psi,1}\,i\lambda\,\bar{\sigma}^1_{\phi,1,-\lambda\lambda_{\phi}}\,\sigma^2_{\psi,1,-\lambda\lambda_{\psi}}&\color{altpurple}=0.
	\end{align}
\end{subequations}
As for the boundary amplitudes $ a^2_{\psi,\lambda} $ for $ \lambda\in\Z^* $, we remove all traces of first or second profile, and, as usual, $ i\lambda\,b_{\psi}\cdot B\,\big(U^{3,\osc}_{\psi,2,\lambda}\big)_{|x_d,\chi_d=0} $, and we use that terms of $ \partial_{z,\theta}\,\mbox{\emph{terms in}}\,\big(U_1,(I-P)\,U_2,(P\,U_2)_{\zeta\neq \phi,\psi,\nu},U_2^*\big)_{|x_d,\chi_d=0} $ are zero to retain the simple forced transport equation
\begin{equation}\label{eq insta eq 1cor apsi evol BB1}
	\color{altorange2}X^{\Lop}_{\psi}\,a^2_{\psi,\lambda}
	=-i\lambda\,b_{\psi}\cdot H_{\lambda}.
\end{equation}
According to above remarks,  boundary condition \eqref{eq 1cor trace psi} for the incoming amplitude $ \sigma_{\psi,1,\lambda}^2 $ now reads
\begin{equation}\label{eq insta eq 1cor apsi bord BB1}
	\color{altorange2}\big(\sigma_{\psi,1,\lambda}^2\big)_{|x_d=0}\,r_{\psi,1}=a^2_{\psi,\lambda}\,e_{\psi,1}.
\end{equation}
Although system \eqref{eq insta eq 1cor BB1}, \eqref{eq insta eq 1cor apsi evol BB1} and \eqref{eq insta eq 1cor apsi bord BB1} is coupled, it is in an upper  triangular form, so it can be solved using explicit formulas, since we are in presence of transport equations with constant coefficients. This is made precise now, with the proof of the following result.

%\begin{theorem}
%	For a boundary term $ H $ in $ L^2\big((-\infty,T]_t\times\R^{d-1}_y\times\T_{\theta_2}\big) $, there exists a unique sequence of 3-tuples $ (\sigma^2_{\psi,1,\lambda},\sigma^2_{\psi,1,\lambda},\sigma^2_{\psi,1,\lambda})_{\lambda\in\Z^*} $ in $ \mathcal{C}\big(\R^+_{x_d},L^2\big((-\infty,T]_t\times\R^{d-1}_y\big)\big) $ and a unique sequence $ (a^2_{\psi,\lambda})_{\lambda\in\Z^*} $ in $ L^2\big((-\infty,T]_t\times\R^{d-1}_y\big) $ solution to the system \eqref{eq insta eq 1cor BB1}, \eqref{eq insta eq 1cor apsi evol BB1} and \eqref{eq insta eq 1cor apsi bord BB1}.
%	
%	Furthermore, there exists a boundary term $ H $ such that the trace $ \big(\sigma^2_{\psi,2,\lambda}\big)_{|x_d=0} $ of the associated solution to  \eqref{eq insta eq 1cor BB1}, \eqref{eq insta eq 1cor apsi evol BB1} and \eqref{eq insta eq 1cor apsi bord BB1} is nonzero. \todocorentinlater{Pour quel $ \lambda $ ?}
%\end{theorem}

\begin{theorem}\label{th il existe H BBmod 1}
	There exists a boundary term $ H $ in $ L^2\big((-\infty,T]_t\times\R^{d-1}_y\times\T_{\theta_2}\big) $ such that, if the sequence $ (\sigma^2_{\psi,1,\lambda},\sigma^2_{\psi,2,\lambda},\sigma^2_{\nu,2,\lambda})_{\lambda\in\Z^*} $ of tuples of $ \mathcal{C}\big(\R^+_{x_d},L^2\big((-\infty,T]_t\times\R^{d-1}_y\big)\big) $ and the sequence $ (a^2_{\psi,\lambda})_{\lambda\in\Z^*} $ of $ L^2\big((-\infty,T]_t\times\R^{d-1}_y\big) $ are solutions to the system \eqref{eq insta eq 1cor BB1}, \eqref{eq insta eq 1cor apsi evol BB1} and \eqref{eq insta eq 1cor apsi bord BB1}, then the trace $ \big(\sigma^2_{\psi,2,\lambda_{\psi}}\big)_{|x_d=0} $ is nonzero.
\end{theorem}

\begin{proof}
	We consider any boundary term  $ H $ in $ L^2\big((-\infty,T]_t\times\R^{d-1}_y\times\T_{\theta_2}\big) $, and we look for an expression of the trace $ \big(\sigma^2_{\psi,2,\lambda}\big)_{|x_d=0} $ of the associated solution of \eqref{eq insta eq 1cor BB1}, \eqref{eq insta eq 1cor apsi evol BB1} and \eqref{eq insta eq 1cor apsi bord BB1}.
	
First of all, the transport equation \eqref{eq insta eq 1cor apsi evol BB1} on the boundary $ \ensemble{x_d=0} $ can be solved to find
\begin{equation*}
	a^2_{\psi,\lambda}(t,y)=\int_0^t -i\lambda\,b_{\psi}\cdot H_{\lambda}\big(s,y-\mathbf{v}^{\Lop}_{\psi}\,(t-s)\big)\,ds,
\end{equation*}
recalling the notation\footnote{Without lost of generality, we have set $ \beta_{\psi}=1 $ in Lemma \ref{lemme Lax bord}, to simplify the equations.}
\begin{equation*}
	X^{\Lop}_{\psi}=\partial_t+\mathbf{v}^{\Lop}_{\psi}\cdot\nabla_y.
\end{equation*}

According to boundary condition \eqref{eq insta eq 1cor apsi bord BB1},
it follows, with notation\footnote{\label{note1}We assumed here that, with notation of Definition \ref{def sortant rentrant alpha X alpha}, $ -1/\partial_{\xi}\tau_k(\eta,\xi) $ is equal to 1.}
\begin{equation*}
	X_{\psi,1}=\partial_t-\mathbf{v}_{\psi,1}\cdot\nabla_x=:\partial_t-\mathbf{v}'_{\psi,1}\cdot\nabla_y+\partial_{x_d},
\end{equation*}
using the incoming transport equation \eqref{eq insta eq 1cor BB1 psi 1},
\begin{equation*}
	\sigma_{\psi,1,\lambda}^2(t,y,x_d)=
	-\indicatrice_{x_d\leq t}\int_0^{t-x_d} i\lambda\,p_{\psi,1}\,b_{\psi}\cdot H_{\lambda}\Big(s,y+x_d\,\big(\mathbf{v}^{\Lop}_{\psi}+\mathbf{v}_{\psi,1}'\big)-\mathbf{v}^{\Lop}_{\psi}\,(t-s)\Big)\,ds,
\end{equation*}
with a coefficient $ p_{\psi,1}\in\R $ such that $ e_{\psi,1}=p_{\psi,1}\,r_{\psi,1} $. To simplify notation, the coefficient $ p_{\psi,1} $ will be omitted in the following. 
Thus, with notation%\textsuperscript{\ref{note1}}
\footnote{See footnote \ref{note1}.} 
\begin{equation*}
	X_{\nu,2}=\partial_t-\mathbf{v}_{\nu,2}\cdot\nabla_x=:\partial_t-\mathbf{v}'_{\nu,2}\cdot\nabla_y-\partial_{x_d},
\end{equation*}
according to the outgoing transport equation \eqref{eq insta eq 1cor BB1 nu2}, we have
\begin{align*}
	\sigma^2_{\nu,2,\lambda}(t,y,x_d)&=-\int_0^t J^{\phi,1}_{\psi,1}\,i\lambda\,\bar{\sigma}^1_{\phi,1,-\lambda\lambda_{\phi}}(s,x+\mathbf{v}_{\nu,2}(t-s))\,\indicatrice_{x_d\leq2s- t}\\
	&\quad\times\int_0^{2s-x_d-t}	i\lambda\lambda_{\psi}\,b_{\psi}\cdot H_{-\lambda\lambda_{\psi}}\Big(\tau, y+\mathbf{v}_{\nu,2}'\,(t-s)\\
	&\quad+(x_d+t-s)\,\big(\mathbf{v}^{\Lop}_{\psi}+\mathbf{v}_{\psi,1}'\big)-\mathbf{v}^{\Lop}_{\psi}\,(s-\tau)\Big)\,d\tau\,ds.
\end{align*}
In the same way, with similar notation, the outgoing transport equation \eqref{eq insta eq 1cor BB1 psi2} leads to
\begin{align*}
	\sigma^2_{\psi,2,\lambda}(t,y,x_d)&=\int_0^t \indicatrice_{\lambda=k\lambda_{\psi}}\,J^{\phi,3}_{\nu,2}\,ik\,\bar{\sigma}^1_{\phi,3,-k\lambda_{\phi}}\big(s,x+\mathbf{v}_{\psi,2}(t-s)\big)\\
	&\quad\times\int_0^s J^{\phi,1}_{\psi,1}\,ik\,\bar{\sigma}^1_{\phi,1,k\lambda_{\phi}}\big(\tau,x+\mathbf{v}_{\psi,2}(t-s)+\mathbf{v}_{\nu,2}(s-\tau)\big)\,\indicatrice_{x_d\leq 2\tau-t}\\
	&\quad\times\int_0^{2\tau-x_d-t}
	ik\lambda_{\psi}\,b_{\psi}\cdot H_{k\lambda_{\psi}}\Big(\sigma,y+\mathbf{v}_{\nu,2}'\,(s-\tau)+\mathbf{v}_{\psi,2}'\,(t-s)
	\\
	&\qquad+(x_d+t-\tau)\,\big(\mathbf{v}^{\Lop}_{\psi}+\mathbf{v}_{\psi,1}'\big)-\mathbf{v}^{\Lop}_{\psi}\,(\tau-\sigma)\Big)\,d\sigma\,d\tau\,ds.
\end{align*}
The trace of $ \sigma^2_{\psi,2,\lambda} $ on the boundary $ \ensemble{x_d=0} $ is therefore given by
\begin{align}\label{eq insta trace sigma 2 lambda}
	\sigma^2_{\psi,2,\lambda}(t,y,0)&=-i\int_0^t\int_0^s\int_0^{2\tau-t} \indicatrice_{\lambda=k\lambda_{\psi}}\,J^{\phi,3}_{\nu,2}\,J^{\phi,1}_{\psi,1}\,k^3\,\lambda_{\psi}\,\bar{\sigma}^1_{\phi,3,-k\lambda_{\phi}}\big(s,y+\mathbf{v}_{\psi,2}(t-s)\big)\,\\\nonumber
	&\quad \times\bar{\sigma}^1_{\phi,1,k\lambda_{\phi}}\big(\tau,y+\mathbf{v}_{\psi,2}(t-s)+\mathbf{v}_{\nu,2}(s-\tau)\big)\\\nonumber
	&\quad
	\times b_{\psi}\cdot H_{k\lambda_{\psi}}\Big(\sigma,y+\mathbf{v}_{\nu,2}'\,(s-\tau)+\mathbf{v}_{\psi,2}'(t-s)
	\\\nonumber
	&\qquad+(t-\tau)\,\big(\mathbf{v}^{\Lop}_{\psi}+\mathbf{v}_{\psi,1}'\big)-\mathbf{v}^{\Lop}_{\psi}\,(\tau-\sigma)\Big)\,d\sigma\,d\tau\,ds.
\end{align}
%since, for $ j=1,3 $,
%\begin{equation*}
%	\big(\bar{\sigma}^1_{\phi,j,-k\lambda_{\phi}}\big)_{|x_d=0}=p_{\phi,j}\,\bar{a}^1_{\phi,-k\lambda_{\phi}},
%\end{equation*}
%with some coefficient $ p_{\phi,j} $.

We justify now why there is a choice of a boundary term $ H $ such that this trace is nonzero. We take interest into the trace $ \big(\sigma^2_{\psi,2,\lambda_{\psi}}\big)_{|x_d=0} $, which is given by formula \eqref{eq insta trace sigma 2 lambda} with $ \lambda=\lambda_{\psi} $ and therefore $ k=1 $, namely,
\begin{align}\label{eq insta trace sigma 2 cas part}
	\sigma^2_{\psi,2,\lambda_{\psi}}(t,y,0)&=-i\int_0^t\int_0^s\int_0^{2\tau-t} J^{\phi,3}_{\nu,2}\,J^{\phi,1}_{\psi,1}\,\lambda_{\psi}\,\bar{\sigma}^1_{\phi,3,-\lambda_{\phi}}\big(s,y+\mathbf{v}_{\psi,2}(t-s)\big)\,\\\nonumber
	&\quad \times\bar{\sigma}^1_{\phi,1,\lambda_{\phi}}\big(\tau,y+\mathbf{v}_{\psi,2}(t-s)+\mathbf{v}_{\nu,2}(s-\tau)\big)\\\nonumber
	&\quad
	\times b_{\psi}\cdot H_{\lambda_{\psi}}\Big(\sigma,y+\mathbf{v}_{\nu,2}'\,(s-\tau)+\mathbf{v}_{\psi,2}'(t-s)
	\\\nonumber
	&\qquad+(t-\tau)\,\big(\mathbf{v}^{\Lop}_{\psi}+\mathbf{v}_{\psi,1}'\big)-\mathbf{v}^{\Lop}_{\psi}\,(\tau-\sigma)\Big)\,d\sigma\,d\tau\,ds.
\end{align}
We start by constructing $ \bar{\sigma}^1_{\phi,1,\lambda_{\phi}} $ and $ \bar{\sigma}^1_{\phi,3,-\lambda_{\phi}} $ suited for our purpose. It is proven in \cite[section 2.2]{CoulombelWilliams2017Mach} that $ \bar{a}^1_{\phi,\lambda} $, solution to equation \eqref{eq systeme a phi} in the particular case where $ H $ is zero, is the solution to the following equation, for $ \lambda\in\Z^* $
\begin{equation}\label{eq insta eq transport a bar phi lambda}
	X^{\Lop}_{\phi}\,\bar{a}^1_{\phi,\lambda}
	+D^{\Lop}_{\phi}\sum_{\lambda_1+\lambda_2=\lambda}i\lambda_2\,\bar{a}^1_{\phi,\lambda_1}\,\bar{a}^1_{\phi,\lambda_2}+i\lambda\sum_{\lambda_1+\lambda_3=\lambda}\gamma_{\phi}(\lambda_1,\lambda_3)\,\bar{a}^1_{\phi,\lambda_1}\,\bar{a}^1_{\phi,\lambda_3}=-i\lambda\,b_{\phi}\cdot G_{\lambda}.
\end{equation}
We set $ G_{\lambda}=0 $ for $ \lambda\in\Z\privede{\lambda_{\phi},-\lambda_{\phi}} $ and $ G_{\lambda_{\phi}} $, $ G_{-\lambda_{\phi}} $ real, non-negative, and equal to one on the set $ [1/2,2]_{t}\times [-(h+2)\mathbf{V},(h+2)\mathbf{V}]^{d-1}_y $, where we have denoted $ \mathbf{V}:=\big|\mathbf{v}^{\Lop}_{\phi}\big| $ and with $ h\geq 1 $. Solving the transport equation \eqref{eq insta eq transport a bar phi lambda}, we get $ \bar{a}^1_{\phi,\lambda}=0 $  for $ \lambda\in\Z\privede{\lambda_{\phi},-\lambda_{\phi}} $ and $ \bar{a}^1_{\phi,\lambda_{\phi}} $, $ \bar{a}^1_{\phi,-\lambda_{\phi}} $ real, non-negative, and greater than $ 1/2 $ on the set $ [1,2]_{t}\times [-(h+1)\mathbf{V},(h+1)\mathbf{V}]^{d-1}_y $. Now we know that, according to the condition \eqref{eq insta sigma a bar zero} on profiles $ \bar{\sigma}^1_{\zeta,j,\lambda} $ for $ \zeta=\psi,\nu $, $ j=1,2,3 $ and $ \lambda\in\Z^* $, there are no resonance terms in the evolution equation \eqref{eq evolution sigma 1 phi psi nu} for $ \bar{\sigma}^1_{\phi,j,\lambda} $ for $ j=1,3 $ and $ \lambda\in\Z^* $, so theses profiles $ \bar{\sigma}^1_{\phi,j,\lambda} $ for $ j=1,3 $ and $ \lambda\in\Z^* $ satisfy the following incoming transport equation
\begin{subequations}\label{eq insta eq transport sigma bar phi lambda}
	\begin{equation}	X_{\phi,j}\,\bar{\sigma}^1_{\phi,j,\lambda}+D_{\phi,j}\sum_{\lambda_1+\lambda_2=\lambda}i\lambda_2\,\bar{\sigma}^1_{\phi,j,\lambda_1}\,\bar{\sigma}^1_{\phi,j,\lambda_2}=0,
\end{equation}
with the following boundary condition \eqref{eq prof princ bord phi}
\begin{equation}
	\big(\bar{\sigma}_{\phi,j,\lambda}^1\big)_{|x_d=0}\,r_{\phi,j}=\bar{a}_{\phi,\lambda}^1\,e_{\phi,j}.
\end{equation}
\end{subequations}
Solving system \eqref{eq insta eq transport sigma bar phi lambda} seen as a transport propagation equation in the normal direction, with notation\footnote{See footnote \ref{note1}.} 
\begin{equation*}
	X_{\phi,j}=:\partial_t-\mathbf{v}'_{\phi,j}\cdot\nabla_y+\partial_{x_d},
\end{equation*}
we get $ \bar{\sigma}^1_{\phi,j,\lambda}=0 $ for $ j=1,3 $ and $ \lambda\in\Z\privede{\lambda_{\phi},-\lambda_{\phi}} $, and that $ \bar{\sigma}^1_{\phi,1,\lambda_{\phi}} $, $ \bar{\sigma}^1_{\phi,3,-\lambda_{\phi}} $ are real, non-negative, and larger than $ A/2 $ on the set $ [1+\mathbf{V}/(2\mathbf{w}),2+\mathbf{V}/(2\mathbf{W})] _t\times[-h\mathbf{V},h\mathbf{V}]_y\times[0,\mathbf{V}/\mathbf{W}]_{x_d} $, where $ A:=\min\big(|e_{\phi,1}\cdot r_{\phi,1}|,|e_{\phi,3}\cdot r_{\phi,3}|\big) $, $ \mathbf{w}:=\min\big(|\mathbf{v}'_{\phi,1}|,|\mathbf{v}'_{\phi,3}|\big) $ and $ \mathbf{W}:=\max\big(|\mathbf{v}'_{\phi,1}|,|\mathbf{v}'_{\phi,3}|\big) $.

Now that $ \bar{\sigma}^1_{\phi,1,\lambda_{\phi}} $ and $ \bar{\sigma}^1_{\phi,3,-\lambda_{\phi}} $ have been constructed appropriately, we make precise a suitable choice of boundary term $ H $. We denote $ \mathbf{t}:=1+\mathbf{V}/(2\mathbf{w}) $ and $ \mathbf{T}:=2+\mathbf{V}/(2\mathbf{W}) $, and we set the integer $ h $ such that
\begin{equation*}
	h\geq \frac{8(\mathbf{T}-\mathbf{t})}{\mathbf{V}}\,\max \Big(\big|\mathbf{v}'_{\psi,2}\big|,\big|\mathbf{v}'_{\nu,2}\big|,\big|\mathbf{v}'_{\psi,1}\big|,\big|\mathbf{v}^{\Lop}_{\psi}\big|\Big).
\end{equation*}
We also take the boundary term $ H $ such that $ b_{\psi}\cdot H_{\lambda_{\psi}} $ is pure imaginary, and such that its imaginary part is of the sign of $ J^{\phi,3}_{\nu,2}\,J^{\phi,1}_{\psi,1}\,\lambda_{\psi} $ and of modulus one on $ [\mathbf{t},\mathbf{T}]_t\times[-h\mathbf{V},h\mathbf{V}]_y $, namely,
\begin{equation*}
	b_{\psi}\cdot H_{\lambda_{\psi}}=i\,\sign\big(J^{\phi,3}_{\nu,2}\,J^{\phi,1}_{\psi,1}\,\lambda_{\psi}\big)\quad \text{on}\quad [\mathbf{t},\mathbf{T}]_t\times[-h\mathbf{V},h\mathbf{V}]_y.
\end{equation*}
Then we note that, according to \eqref{eq insta trace sigma 2 cas part}, the trace $ \big(\sigma^2_{\psi,2,\lambda_{\psi}}\big)_{|x_d=0} $ satisfies, for $ t\in[\mathbf{t},\mathbf{T}] $ and $ y\in[-h\mathbf{V}/2,h\mathbf{V}/2] $,
\begin{align*}
	\sigma^2_{\psi,2,\lambda_{\psi}}(t,y,0)&\geq -i\int_{\mathbf{t}}^t\int_{\mathbf{t}}^s\int_{\mathbf{t}}^{2\tau-t} J^{\phi,3}_{\nu,2}\,J^{\phi,1}_{\psi,1}\,\lambda_{\psi}\,\bar{\sigma}^1_{\phi,3,-\lambda_{\phi}}\big(s,y+\mathbf{v}_{\psi,2}(t-s)\big)\,\\\nonumber
	&\quad \times\bar{\sigma}^1_{\phi,1,\lambda_{\phi}}\big(\tau,y+\mathbf{v}_{\psi,2}(t-s)+\mathbf{v}_{\nu,2}(s-\tau)\big)\\\nonumber
	&\quad
	\times b_{\psi}\cdot H_{\lambda_{\psi}}\Big(\sigma,y+\mathbf{v}_{\nu,2}'\,(s-\tau)+\mathbf{v}_{\psi,2}'(t-s)
	\\\nonumber
	&\qquad+(t-\tau)\,\big(\mathbf{v}^{\Lop}_{\psi}+\mathbf{v}_{\psi,1}'\big)-\mathbf{v}^{\Lop}_{\psi}\,(\tau-\sigma)\Big)\,d\sigma\,d\tau\,ds\\
	&\geq \big|J^{\phi,3}_{\nu,2}\,J^{\phi,1}_{\psi,1}\,\lambda_{\psi}\big|\frac{A^2}{4},
\end{align*}
since, for $ \mathbf{t}\leq \sigma\leq 2\tau-t  $, $ \mathbf{t}\leq \tau \leq s \leq t \leq \mathbf{T} $ and $ y\in[-h\mathbf{V}/2,h\mathbf{V}/2] $, according to the assumption on $ h $, we have $ \sigma,\tau,s\in[\mathbf{t},\mathbf{T}] $ and 
\begin{align*}
	y+\mathbf{v}_{\psi,2}(t-s)&\in [-h\mathbf{V},h\mathbf{V}]_y\times[0,\mathbf{V}/\mathbf{W}]_{x_d} ,\\
	y+\mathbf{v}_{\psi,2}(t-s)+\mathbf{v}_{\nu,2}(s-\tau)&\in [-h\mathbf{V},h\mathbf{V}]_y\times[0,\mathbf{V}/\mathbf{W}]_{x_d} ,\\
	y+\mathbf{v}_{\nu,2}'\,(s-\tau)+\mathbf{v}_{\psi,2}'(t-s)&\\
	+(t-\tau)\,\big(\mathbf{v}^{\Lop}_{\psi}+\mathbf{v}_{\psi,1}'\big)-\mathbf{v}^{\Lop}_{\psi}\,(\tau-\sigma)&\in [-h\mathbf{V},h\mathbf{V}]_y,
\end{align*}
because $ \mathbf{T}-\mathbf{t}\leq \mathbf{V}/2\mathbf{W} $, up to shrinking $ \mathbf{w} $.

It concludes the proof of Theorem \ref{th il existe H BBmod 1} since we proved that there exists a choice of $ H $ such that the trace $ \big(\sigma^2_{\psi,2,\lambda_{\psi}}\big)_{|x_d=0} $ is nonzero.
\end{proof}

Theorem \ref{th il existe H BBmod 1}, stating that the trace $ \big(\sigma^2_{\psi,2,\lambda_{\psi}}\big)_{|x_d=0} $ is nonzero, contradicts the condition \eqref{eq insta eq trace psi 2 = 0}, so instability is proven. Indeed, we have assumed that for all boundary terms $ H $, the associated amplitudes $ \sigma^1_{\psi,j,\lambda} $ and $ \sigma^1_{\nu,j,\lambda} $ for $ j=1,3 $ and $ \lambda\in\Z^* $ are zero, and, for the simplified model equations  \eqref{eq insta eq 1cor BB1}, \eqref{eq insta eq 1cor apsi evol BB1} and \eqref{eq insta eq 1cor apsi bord BB1}, we found a contradiction. Therefore, for this simplified model, there exists a boundary term $ H $ such that the leading profiles $ \sigma^1_{\psi,j,\lambda} $ and $ \sigma^1_{\nu,j,\lambda} $  for $ j=1,3 $ and $ \lambda\in\Z^* $ are not all zero. It proves that the boundary term $ H $ may interfere at the leading order, which constitutes an instability. 

\subsubsection{Second simplified model}

The second simplified model that we shall consider features additional resonance coupling terms which add difficulties. This time we keep the formulation of simplified models of section \ref{section existence} and multiply equations \eqref{eq linearise evol psi nu} and \eqref{eq linearise eq bord 1cor} by $ e^{i\lambda\Theta} $, to obtain, for the phases $ \psi_1 $, $ \psi_2 $ and $ \nu_2 $, the incoming evolution equation with a resonance term
\begin{subequations}\label{eq insta eq 1cor BB2}
	\begin{align}
		\color{altblue}\label{eq insta eq 1cor BB2 psi 1} X_{\psi,1}\,\sigma^2_{\psi,1}+\partial_{\Theta}\,\mathbb{J}^{\phi,1}_{\nu,2}\big(\bar{\sigma}^1_{\phi,1},\sigma^2_{\nu,2}\big)&\color{altblue}=0,
		\intertext{and the two outgoing evolution equations with resonance terms}
		\color{altpink}\label{eq insta eq 1cor BB2 psi2} X_{\psi,2}\,\sigma^2_{\psi,2}+\partial_{\Theta}\,\mathbb{J}^{\phi,3}_{\nu,2}\big(\bar{\sigma}^1_{\phi,3},\sigma^2_{\nu,2}\big)&\color{altpink}=0,\\[5pt]
		\color{altpurple}\label{eq insta eq 1cor BB2 nu2} X_{\nu,2}\,\sigma^2_{\nu,2}+\partial_{\Theta}\,\mathbb{J}^{\phi,1}_{\psi,1}\big(\bar{\sigma}^1_{\phi,1},\sigma^2_{\psi,1}\big)+\partial_{\Theta}\,\mathbb{J}^{\phi,3}_{\psi,2}\big(\bar{\sigma}^1_{\phi,3},\sigma^2_{\psi,2}\big)&\color{altpurple}=0,
	\end{align}
\end{subequations}
where operators $ \mathbb{J}_{\zeta_2,j_2}^{\zeta_1,j_1} $ have been defined in \eqref{eq ana def J}, and 
equations on the boundary are kept as in the first simplified model \eqref{eq insta eq 1cor apsi evol BB1} and \eqref{eq insta eq 1cor apsi bord BB1}, namely,
	\begin{align}\label{eq insta eq 1cor apsi evol BB2}
		\color{altorange2}X^{\Lop}_{\psi}\,a^2_{\psi}
		&\color{altorange2}=-b_{\psi}\cdot \partial_{\Theta}\,H,\\[5pt]\label{eq insta eq 1cor apsi bord BB2}
		\color{altorange2}\big(\sigma_{\psi,1}^2\big)_{|x_d=0}\,r_{\psi,1}&\color{altorange2}=a^2_{\psi}\,e_{\psi,1}.
	\end{align}
Note that this time, the simplified model features all resonance terms of the general equations.

The obtained system is no longer triangular since additional resonance terms in \eqref{eq insta eq 1cor BB2} relatively to \eqref{eq insta eq 1cor BB1} couple each equation with the others, so we cannot solve it as a sequence of transport equations as before. We use a perturbation method and solve equations \eqref{eq insta eq 1cor BB2} with a fixed point theorem.
We start by solving \eqref{eq insta eq 1cor apsi evol BB2} as a transport equation, and then we deduce, using the incoming transport equations \eqref{eq insta eq 1cor BB2 psi 1} and boundary condition \eqref{eq insta eq 1cor apsi bord BB2}, an expression of $ \sigma_{\psi,1}^2 $ depending on $ \sigma^2_{\nu,2} $. We use this expression in \eqref{eq insta eq 1cor BB2 nu2} to obtain an equation in $ \sigma^2_{\nu,2} $ with a source term depending only on $ \sigma^2_{\psi,2} $. This equation is solved with a fixed point method, using that the source term depending on $ \sigma^2_{\nu,2} $ is ``small'', in a convenient topology, comparing to the transport term, and we get an expression of $ \sigma^2_{\nu,2} $ depending on $ \sigma^2_{\psi,2} $. This expression is finally used in \eqref{eq insta eq 1cor BB2 psi2} which is solved with the same fixed point method.
The result is the following.

\begin{theorem}\label{th il existe H BBmod 2}
	There exists a boundary term $ H $ in $ \mathcal{C}\big([0,T]_t,H^{\infty}(\R^{d-1}_y\times\T_{\theta_2})\big) $ such that, if $ \sigma^2_{\psi,1}, \sigma^2_{\psi,2},\sigma^2_{\nu,2} $ in $ \mathcal{C}^1\big([0,T]_t,H^{\infty}(\R^{d-1}_y\times\R^+_{x_d}\times\T_{\Theta})\big) $ and $ a^2_{\psi} $ in $ \mathcal{C}^1\big([0,T]_t,H^{\infty}(\R^{d-1}_y\times\T_{\Theta})\big) $ are solutions to the system \eqref{eq insta eq 1cor BB2}, \eqref{eq insta eq 1cor apsi evol BB2} and \eqref{eq insta eq 1cor apsi bord BB2}, then the trace $ \big(\sigma^2_{\psi,2}\big)_{|x_d=0} $ is nonzero.
\end{theorem}

\begin{proof}
Similarly as for the first simplified model, from equation \eqref{eq insta eq 1cor apsi evol BB2}, reusing previous notation, we get
\begin{equation*}
	a^2_{\psi}(t,y)=-\int_0^t b_{\psi}\cdot \partial_{\Theta}\,H_{\lambda}\big(s,y-\mathbf{v}^{\Lop}_{\psi}\,(t-s)\big)\,ds.
\end{equation*}
Then system \eqref{eq insta eq 1cor BB2 psi 1}, \eqref{eq insta eq 1cor apsi bord BB2} seen as a transport equation with a source term depending on $ \sigma^2_{\nu,2} $, leads to
\begin{align*}
	\sigma_{\psi,1}^2(t,y,x_d)&=
	-\indicatrice_{x_d\leq t}\int_0^{t-x_d}b_{\psi}\cdot\partial_{\Theta}\, H\Big(s,y+x_d\,\big(\mathbf{v}^{\Lop}_{\psi}+\mathbf{v}_{\psi,1}'\big)-\mathbf{v}^{\Lop}_{\psi}\,(t-s)\Big)\,ds\\
	&\quad+\int_{\max(0,t-x_d)}^t\partial_{\Theta}\,\mathbb{J}^{\phi,1}_{\nu,2}\big(\bar{\sigma}^1_{\phi,1},\sigma^2_{\nu,2}\big)\big(s,x+\mathbf{v}_{\psi,1}\,(t-s)\big)\,ds.
\end{align*}
Therefore, equation \eqref{eq insta eq 1cor BB2 nu2} now reads
\begin{subequations}\label{eq insta eq 1cor nu2bis BB2}
\begin{align} \label{eq insta eq 1cor nu2bis BB2 1}
	&\color{altpurple}\quad X_{\nu,2}\,\sigma^2_{\nu,2}+\partial_{\Theta}\,\mathbb{J}^{\phi,3}_{\psi,2}\big(\bar{\sigma}^1_{\phi,3},\sigma^2_{\psi,2}\big)\\\label{eq insta eq 1cor nu2bis BB2 1 2}
	&\color{altpurple}-\partial_{\Theta}\,\mathbb{J}^{\phi,1}_{\psi,1}\left[\bar{\sigma}^1_{\phi,1},\indicatrice_{x_d\leq t}\int_0^{t-x_d}b_{\psi}\cdot\partial_{\Theta}\, H\Big(s,y+x_d\,\big(\mathbf{v}^{\Lop}_{\psi}+\mathbf{v}_{\psi,1}'\big)-\mathbf{v}^{\Lop}_{\psi}\,(t-s)\Big)\,ds\right]\\\label{eq insta eq 1cor nu2bis BB2 1 3}
	&\color{altpurple}+\partial_{\Theta}\,\mathbb{J}^{\phi,1}_{\psi,1}\left[\bar{\sigma}^1_{\phi,1},\int_{\max(0,t-x_d)}^t\partial_{\Theta}\,\mathbb{J}^{\phi,1}_{\nu,2}\big(\bar{\sigma}^1_{\phi,1},\sigma^2_{\nu,2}\big)\big(s,x+\mathbf{v}_{\psi,1}\,(t-s)\big)\,ds\right]=0.
\end{align}
\end{subequations}
This is a transport equation but with a perturbation term \eqref{eq insta eq 1cor nu2bis BB2 1 3} depending on the unknowns $ \sigma^2_{\nu,2,\lambda} $. It is solved using the following result. For $ s\in[0,+\infty) $, we denote by $ H^{s} $ the Sobolev space $ H^{s}(\R^{d-1}_y\times\R^+_{x_d}\times\T_{\Theta}) $ of regularity $ s $, and $ H^{\infty}:=\bigcap_{s\geq 0}H^s $.
%, and $ \mathcal{C}_b\big([-T,T],H^{\infty}\big) $ the space of continuous bounded functions with values $ H^{\infty} $.
% We introduce at this purpose a new space of time analytic functions. 
%
%\begin{definition}
%	For $ \rho\in(0,1) $, the space $ Z_{\rho} $ is defined as the set of functions $ u $ of $ (-\rho,\rho)\times\R^{d-1}\times\R_+ $ such that there exists $ M>0 $, satisfying, for $ n\geq 0 $, 
%	\begin{equation*}
%		\norme{\partial^n_{t} u(0)}_{H^{\infty}(\R^{d-1}\times\R_+)}\leq \frac{M\,n!}{\rho^n}.
%	\end{equation*}
%The infimum of all $ M>0 $ satisfying this property is denoted by $ \norme{u}_{\rho} $.
%\end{definition}

\begin{lemma}\label{lemma solu Xu+ int u =0}
 	There exists $ T>0 $ such that, for any function $ f $ in $ \mathcal{C}([0,T],H^{\infty}) $, the transport equation
 	%\vspace{-0.4cm}
	\begin{subequations}\label{eq insta eq type A + epsilon B}
	\begin{align}\nonumber
	X_{\nu,2}\,\sigma^2_{\nu,2}&\\\label{eq insta eq type A + epsilon B evol}
	+\partial_{\Theta}\,\mathbb{J}^{\phi,1}_{\psi,1}\Big[\bar{\sigma}^1_{\phi,1},\int_{\max(0,t-x_d)}^t\partial_{\Theta}\,\mathbb{J}^{\phi,1}_{\nu,2}\big(\bar{\sigma}^1_{\phi,1},\sigma^2_{\nu,2}\big)\big(s,x+\mathbf{v}_{\psi,1}\,(t-s)\big)\,ds\Big]&=f(t,x),\\\label{eq insta eq type A + epsilon B init}
		\big(\sigma^2_{\nu,2}\big)_{|t=0}&=0,
	\end{align}
	\end{subequations}
admits a unique solution $ \sigma^2_{\nu,2} $ in $ \mathcal{C}^1([0,T],H^{\infty}) $. 

If, for $ f $ in $ \mathcal{C}([0,T],H^{\infty}) $ we denote by $ \Psi f $ the solution $ \sigma^2_{\nu,2} $ of \eqref{eq insta eq type A + epsilon B} in $ \mathcal{C}^1([0,T],H^{\infty}) $, then, for any $ s\geq 0 $, 
%we have $ \partial_t \big(\Psi f\big)(0)=f(0) $. Furthermore, for $ s\geq 0 $, 
there exists $ C_s>0 $ such that for $ f $ in $ \mathcal{C}([0,T],H^{s}) $, we have
\begin{equation}\label{eq insta est Psi f}
	\norme{\Psi f}_{\mathcal{C}([0,T_s],H^s)}\leq  C_s\,T\norme{f}_{\mathcal{C}([0,T_s],H^s)}.
\end{equation}
\end{lemma}

Before proving Lemma \ref{lemma solu Xu+ int u =0}, we prove the following preliminary result, asserting that the operators $ u\mapsto \partial_{\Theta}\,\mathbb{J}^{\phi,j}_{\zeta,k}\big[\bar{\sigma}^1_{\phi,j},u\big] $  for $ (j,\zeta,k)\in\ensemble{(1,\psi,1),(1,\nu,2),(3,\nu,2),(3,\psi,2)} $ are bounded from $ \mathcal{C}([0,T],H^s) $ to itself, for $ s\in[0,+\infty) $.

\begin{lemma}\label{lemma insta J per bilin }
	For $ s\in[0,+\infty) $ and $ \bar{\sigma}^1_{\phi,1}\in\mathcal{C}([0,T],H^{\infty}) $, there exists $ C_s>0 $ such that for $ u $ in $ \mathcal{C}([0,T],H^s) $, functions $ \partial_{\Theta}\,\mathbb{J}^{\phi,j}_{\zeta,k}\big[\bar{\sigma}^1_{\phi,j},u\big] $  for $ (j,\zeta,k)\in\ensemble{(1,\psi,1),(1,\nu,2),(3,\nu,2),(3,\psi,2)} $ belong to $ \mathcal{C}([0,T],H^s) $ and satisfy 
		\begin{equation}\label{eq insta est J per bilin}
		\norme{\partial_{\Theta}\,\mathbb{J}^{\phi,j}_{\zeta,k}\big[\bar{\sigma}^1_{\phi,1},u\big]}_{\mathcal{C}([0,T],H^s)}\leq C_s \norme{u}_{\mathcal{C}([0,T],H^s)}.
	\end{equation}
\end{lemma}

\begin{proof}
	We make the proof for the operator $ u\mapsto \partial_{\Theta}\,\mathbb{J}^{\phi,1}_{\psi,1}\big[\bar{\sigma}^1_{\phi,1},u\big] $, namely, $ (j,\zeta,k)=(1,\psi,1) $. According to the expression \eqref{eq ana def J} of the operator $ \mathbb{J}^{\phi,1}_{\psi,1} $, we want to estimate, for $ t\in[0,T] $, the following quantity:
	\begin{equation*}
		\norme{\partial_{\Theta}\,\mathbb{J}^{\phi,1}_{\psi,1}\big[\bar{\sigma}^1_{\phi,1},u\big](t)}_{H^s}=\norme{(x,\Theta)\mapsto J_{\zeta_1,j_1}^{\zeta_2,j_2}\sum_{\lambda\in\Z^*}\lambda\,\bar{\sigma}^1_{\phi,1,\lambda}(t,x)\,u_{\lambda}(t,x)\,e^{i\lambda\Theta}}_{H^s}.
	\end{equation*}
Using the same proof as the one of estimate \eqref{eq ana inter 5}, we get
\begin{align*}
	&\norme{\partial_{\Theta}\,\mathbb{J}^{\phi,1}_{\psi,1}\big[\bar{\sigma}^1_{\phi,1},u\big](t)}_{H^s}\\
	&\qquad\leq\norme{(x,\Theta)\mapsto J_{\zeta_1,j_1}^{\zeta_2,j_2}\sum_{\lambda\in\Z^*}\lambda\,\bar{\sigma}^1_{\phi,1,\lambda}(t,x)\,e^{i\lambda\Theta}}_{H^s}\norme{(x,\Theta)\mapsto \sum_{\lambda\in\Z^*}u_{\lambda}(t,x)\,e^{i\lambda\Theta}}_{H^s}\\
	&\qquad\leq C \norme{\bar{\sigma}^1_{\phi,1}(t)}_{H^{s+1}}\norme{u(t)}_{H^s},
\end{align*}
which, taking the supremum in $ t\in[0,T] $, leads to the sought estimate \eqref{eq insta est J per bilin}, since $ \bar{\sigma}^1_{\phi,1} $ is in $ \mathcal{C}([0,T],H^{s+1}) $ according to Proposition \ref{prop exist sol part}. This is sufficient since we do not seek here for a tame estimate. 
\end{proof}

%We now proceed with the proof of Lemma \ref{lemma solu Xu+ int u =0}.

\begin{proof}[Proof \emph{(Lemma \ref{lemma solu Xu+ int u =0})}]
	From now on we fix an integer $ s\geq 0 $, and the aim is to use the Banach fixed point theorem in the Banach space $ \mathcal{C}([0,T],H^s) $. For $ v $ in $ \mathcal{C}([0,T],H^s) $, we denote by $ \Phi\,v $ the solution in $ \mathcal{C}([0,T],H^s) $ of 
	\begin{equation*}
		\begin{cases}
			X_{\nu,2}\,u+v=0,\\
			u_{|t=0}=0,
		\end{cases}
	\end{equation*}
which is therefore given by
\begin{equation*}\label{eq insta expr Phi}
	\Phi\,v(t,x,\Theta)=-\int_0^tv\big(s,x+\mathbf{v}_{\nu,2}\,(t-s),\Theta\big)\,ds.
\end{equation*}
Note that $ \Phi $ is continuous from $ \mathcal{C}([0,T],H^s) $ to itself, and satisfies  \begin{equation}\label{eq insta est Phi}
	\norme{\Phi\,v}_{\mathcal{C}([0,T],H^s)}\leq T \norme{v}_{\mathcal{C}([0,T],H^s)}.
\end{equation} 
Now, if $ \sigma^2_{\nu,2} $ is a solution to \eqref{eq insta eq type A + epsilon B}, it is in this notation a fixed point of the map $ \mathbf{F} $, where, for $ w $ in $ \mathcal{C}([0,T],H^{s}) $, $ \mathbf{F}(w) $ is defined by 
\begin{align*}
		\mathbf{F} (w) : (t,x,\Theta) 
		\mapsto
		\Phi\,\partial_{\Theta}\,\mathbb{J}^{\phi,1}_{\psi,1}\Big[\bar{\sigma}^1_{\phi,1},
		\int_{\max(0,t-x_d)}^t\partial_{\Theta}\,\mathbb{J}^{\phi,1}_{\nu,2}\big(\bar{\sigma}^1_{\phi,1},w\big)\big(s,x+\mathbf{v}_{\psi,1}\,(t-s),\Theta\big)\,ds\Big](t,x,\Theta)\\
		-\Phi\,f(t,x,\Theta).
\end{align*}
We derive now an estimate on the difference $ \mathbf{F}(w) - \mathbf{F}(w') $ for $ w,w' $ in $ \mathcal{C}([0,T],H^s) $. By linearity of the operators, we have
\begin{equation*}
	\mathbf{F}(w) - \mathbf{F}(w') =\Phi\,\partial_{\Theta}\,\mathbb{J}^{\phi,1}_{\psi,1}\Big[\bar{\sigma}^1_{\phi,1},\int_{\max(0,t-x_d)}^t\partial_{\Theta}\,\mathbb{J}^{\phi,1}_{\nu,2}\big(\bar{\sigma}^1_{\phi,1},w-w'\big)\big(s,x+\mathbf{v}_{\psi,1}\,(t-s),\Theta\big)\,ds\Big].
\end{equation*}
Therefore, according to estimates \eqref{eq insta est Phi} and \eqref{eq insta est J per bilin}, we have
\begin{align*}
	&\norme{\mathbf{F}(w) - \mathbf{F}(w')}_{\mathcal{C}([0,T],H^s)}\\
	&\qquad\leq C_s\,T\,\norme{\int_{\max(0,t-x_d)}^t\partial_{\Theta}\,\mathbb{J}^{\phi,1}_{\nu,2}\big(\bar{\sigma}^1_{\phi,1},w-w'\big)\big(s,x+\mathbf{v}_{\psi,1}\,(t-s),\Theta\big)\,ds}_{\mathcal{C}([0,T],H^s)}\\
	&\qquad\leq C_s\,T^2\,\norme{\partial_{\Theta}\,\mathbb{J}^{\phi,1}_{\nu,2}\big(\bar{\sigma}^1_{\phi,1},w-w'\big)}_{\mathcal{C}([0,T],H^s)}\\
	&\qquad \leq C_s^2\,T^2\norme{w-w'}_{\mathcal{C}([0,T],H^s)}.
\end{align*}
Therefore, for $ T>0 $ small enough, $ \mathbf{F} $ is a contraction of $ \mathcal{C}([0,T],H^s) $, and the Banach fixed point theorem therefore gives a unique solution $ \Psi f $ in $ \mathcal{C}([0,T],H^s) $. By linearity  of system \eqref{eq insta eq type A + epsilon B}, the solution $ \Psi f $ may be extended to any time interval, so the time of existence $ T $ does not depend on the regularity $ s\in[0,+\infty) $. Finally, using equation \eqref{eq insta eq type A + epsilon B evol}, we obtain
\begin{multline}\label{eq insta expr dt Psi}
	\partial_t\,\Psi f(t,x,\Theta)=f(t,x,\Theta)+\mathbf{v}_{\nu,2}\cdot\nabla_x\,\Psi f(t,x,\Theta)\\
	-\partial_{\Theta}\,\mathbb{J}^{\phi,1}_{\psi,1}\Big[\bar{\sigma}^1_{\phi,1},\int_{\max(0,t-x_d)}^t\partial_{\Theta}\,\mathbb{J}^{\phi,1}_{\nu,2}\big(\bar{\sigma}^1_{\phi,1},\Psi f\big)\big(s,x+\mathbf{v}_{\psi,1}\,(t-s),\Theta\big)\,ds\Big](t,x,\Theta)
\end{multline}
so $ \partial_t \Psi f $ belongs to $ \mathcal{C}([0,T],H^{\infty}) $ and therefore $ \Psi f $ is actually in $ \mathcal{C}^1([0,T],H^{\infty}) $.
We have proven the first part of Lemma \ref{lemma solu Xu+ int u =0}.

The interest is now made on the boundedness of $ \Psi $.
%properties of $ \Psi $. Recall that $ \Psi f(0)=0 $, so $ \nabla_x\Psi f(0)=0 $. Therefore, according to the expression \eqref{eq insta expr dt Psi} of $ \partial_t\,\Psi f $, we have
%\begin{equation*}
%	\partial_t\Psi f(0)=f(0).
%\end{equation*}
%On the other hand, w
We have, since $ \Psi f = \mathbf{F}(\Psi f) $,
\begin{align*}
	\Psi f(t,x,\Theta)
	=\Phi\,\partial_{\Theta}\,\mathbb{J}^{\phi,1}_{\psi,1}\Big[\bar{\sigma}^1_{\phi,1},\int_{\max(0,t-x_d)}^t\partial_{\Theta}\,\mathbb{J}^{\phi,1}_{\nu,2}\big(\bar{\sigma}^1_{\phi,1},\Psi f\big)\big(s,x+\mathbf{v}_{\psi,1}\,(t-s),\Theta\big)\,ds\Big](t,x,\Theta)\\
	-\Phi\,f(t,x,\Theta),
\end{align*}
and therefore, using estimates \eqref{eq insta est Phi} and \eqref{eq insta est J per bilin}, we have, for $ s\in[0,+\infty) $,
\begin{align*}
\norme{\Psi f}_{\mathcal{C}([0,T],H^s)}&\leq C_s\,T\,\norme{\int_{\max(0,t-x_d)}^t\partial_{\Theta}\,\mathbb{J}^{\phi,1}_{\nu,2}\big(\bar{\sigma}^1_{\phi,1},\Psi f\big)\big(s,x+\mathbf{v}_{\psi,1}\,(t-s),\Theta\big)\,ds}_{\mathcal{C}([0,T],H^s)}\\
&\quad+T\norme{f}_{\mathcal{C}([0,T],H^s)}\\
&\leq C_s^2\, T^2\norme{\Psi f}_{\mathcal{C}([0,T],H^s)}+T\norme{f}_{\mathcal{C}([0,T],H^s)}.
\end{align*}
Thus, for $ T $ small enough (depending on $ s\in[0,+\infty) $), we have
\begin{equation}\label{eq insta inter 1}
	\norme{\Psi f}_{\mathcal{C}([0,T],H^s)}\leq C\,T\norme{f}_{\mathcal{C}([0,T],H^s)}.
\end{equation}
Once again, by linearity of system \eqref{eq insta eq type A + epsilon B}, the estimate \eqref{eq insta inter 1} is propagated to the whole interval $ [0,T] $, which concludes the proof.
\end{proof}

Returning to \eqref{eq insta eq 1cor nu2bis BB2} and using Lemma \ref{lemma solu Xu+ int u =0}, by linearity, the solution $ \sigma^2_{\nu,2} $ of equation \eqref{eq insta eq 1cor nu2bis BB2} reads
 \begin{multline}\label{eq insta expr sigma nu 2}
 	\sigma^2_{\nu,2}=\Psi\,\partial_{\Theta}\,\mathbb{J}^{\phi,1}_{\psi,1}\left(\bar{\sigma}^1_{\phi,1},\indicatrice_{x_d\leq t}\int_0^{t-x_d}b_{\psi}\cdot\partial_{\Theta}\, H\Big(s,y+x_d\,\big(\mathbf{v}^{\Lop}_{\psi}+\mathbf{v}_{\psi,1}'\big)-\mathbf{v}^{\Lop}_{\psi}\,(t-s),\Theta\Big)\,ds\right)\\[5pt]
 	\quad-\Psi\,\partial_{\Theta}\,\mathbb{J}^{\phi,3}_{\psi,2}\big(\bar{\sigma}^1_{\phi,3},\sigma^2_{\psi,2}\big).
 \end{multline}

\bigskip

We proceed now with equation \eqref{eq insta eq 1cor BB2 psi2} which now reads, according to the expression \eqref{eq insta expr sigma nu 2} of $ \sigma^2_{\nu,2} $,
\begin{align}\label{eq insta eq type A + epsilon B 2}
		&\color{altpink} \quad X_{\psi,2}\,\sigma^2_{\psi,2}-\partial_{\Theta}\,\mathbb{J}^{\phi,3}_{\nu,2}\Big[\bar{\sigma}^1_{\phi,3},\Psi\,\partial_{\Theta}\,\mathbb{J}^{\phi,3}_{\psi,2}\big(\bar{\sigma}^1_{\phi,3},\sigma^2_{\psi,2}\big)\Big]\\\nonumber
	&\color{altpink} 	=-\partial_{\Theta}\,\mathbb{J}^{\phi,3}_{\nu,2}\left[\bar{\sigma}^1_{\phi,3},\Psi\,\partial_{\Theta}\,\mathbb{J}^{\phi,1}_{\psi,1}\left(\bar{\sigma}^1_{\phi,1},\vphantom{\int_0^{t-x_d} }\right.\right.\\\nonumber
	&\color{altpink} \hspace{50pt}\left.\left.\indicatrice_{x_d\leq t}\int_0^{t-x_d}b_{\psi}\cdot\partial_{\Theta}\, H\Big(s,y+x_d\,\big(\mathbf{v}^{\Lop}_{\psi}+\mathbf{v}_{\psi,1}'\big)-\mathbf{v}^{\Lop}_{\psi}\,(t-s)\Big)\,ds\right)\right].
\end{align}
This equation is solved using the same method as the one of Lemma \ref{lemma solu Xu+ int u =0}.
%\begin{lemma}\label{lemma solu Xu+ f u =0}
%	Consider $ X:=\partial_t-\mathbf{v}_1\cdot\nabla_x $ a vector field, with $ \mathbf{v}_1=(\mathbf{v}_1',1)\in\R^{d-1}\times\R $, a function $ \bar{\sigma}^1_{\phi,1} $ in $ \mathcal{C}\big([0,T],H^{\infty}\big) $.
%	Then, up to reducing $ T $, system 
%	\begin{subequations}\label{eq insta eq type A + epsilon B 2}
%		\begin{align}\label{eq insta eq type A + epsilon B 2 evol}
%			Xu+\partial_{\Theta}\,\mathbb{J}^{\phi,3}_{\nu,2}\big(\bar{\sigma}^1_{\phi,1},\Psi\,u\big)&=0,\\\label{eq insta eq type A + epsilon B 2 init}
%			u_{|t=0}&=0,
%		\end{align}
%	\end{subequations}
%	admits a unique solution $ u $ in $ \mathcal{C}\big([0,T],H^{\infty}\big) $, that satisfies
%	\begin{equation}\label{eq insta est u f0 2}
%		\norme{u}_{\mathcal{C}([0,T],H^{\infty})}\leq  2\big(T+T^2\big)\norme{f}_{\mathcal{C}([0,T],H^{\infty})},
%	\end{equation}
%	where the size of $ T $ depends on the coefficients $ \bar{\sigma}^1_{\phi,1} $ and $ \bar{\sigma}^1_{\phi,1} $.
%\end{lemma}
For $ v $ in $ \mathcal{C}([0,T],H^{\infty}) $, we still denote by $ \Phi\,v $ the solution in $ \mathcal{C}([0,T],H^{\infty}) $ of 
\begin{equation*}
	\begin{cases}
		X_{\psi,2}\,u+v=0,\\
		u_{|t=0}=0,
	\end{cases}
\end{equation*} 
and we recall that it satisfies, for $ s\in[0,+\infty) $, 
\begin{equation}\label{eq insta est Phi 2}
	\norme{\Phi\,v}_{\mathcal{C}([0,T],H^s)}\leq T \norme{v}_{\mathcal{C}([0,T],H^s)}.
\end{equation}
Now, $ \sigma^2_{\psi,2} $ is a solution to \eqref{eq insta eq type A + epsilon B 2} if and only if it is  a fixed point of the map $ \mathbf{F} $ of $ \mathcal{C}([0,T],H^{\infty}) $ given by 
\begin{multline*}
	\mathbf{F} : w \mapsto
	-\Phi\,\partial_{\Theta}\,\mathbb{J}^{\phi,3}_{\nu,2}\Big[\bar{\sigma}^1_{\phi,3},\Psi\,\partial_{\Theta}\,\mathbb{J}^{\phi,3}_{\psi,2}\big(\bar{\sigma}^1_{\phi,3},w\big)\Big]
		+\Phi\,\partial_{\Theta}\,\mathbb{J}^{\phi,3}_{\nu,2}\left[\bar{\sigma}^1_{\phi,3},\Psi\,\partial_{\Theta}\,\mathbb{J}^{\phi,1}_{\psi,1}\left(\bar{\sigma}^1_{\phi,1},\vphantom{\int_0^{t-x_d} }\right.\right.\\\nonumber
	\left.\left.\indicatrice_{x_d\leq t}\int_0^{t-x_d}b_{\psi}\cdot\partial_{\Theta}\, H\Big(s,y+x_d\,\big(\mathbf{v}^{\Lop}_{\psi}+\mathbf{v}_{\psi,1}'\big)-\mathbf{v}^{\Lop}_{\psi}\,(t-s),\Theta\Big)\,ds\right)\right].
\end{multline*}
For $ w,w' $ in $ \mathcal{C}([0,T],H^{\infty}) $, the difference $ \mathbf{F}(w)-\mathbf{F}(w') $ is given by
\begin{equation*}
	\mathbf{F}(w)-\mathbf{F}(w')=\Phi\,\partial_{\Theta}\,\mathbb{J}^{\phi,3}_{\nu,2}\Big[\bar{\sigma}^1_{\phi,3},\Psi\,\partial_{\Theta}\,\mathbb{J}^{\phi,3}_{\psi,2}\big(\bar{\sigma}^1_{\phi,3},w'-w\big)\Big].
\end{equation*}
Therefore, for $ s\in[0,+\infty) $, according to estimates \eqref{eq insta est J per bilin} and \eqref{eq insta est Psi f}, we have
\begin{align*}
	\norme{\mathbf{F}(w)-\mathbf{F}(w')}_{\mathcal{C}([0,T],H^s)}&\leq C_s\,T\norme{\Psi\,\partial_{\Theta}\,\mathbb{J}^{\phi,3}_{\psi,2}\big(\bar{\sigma}^1_{\phi,3},w'-w\big)}_{\mathcal{C}([0,T],H^s)}\\
	&\leq C_s \, T^2 \norme{w-w'}_{\mathcal{C}([0,T],H^s)},
\end{align*}
so, for $ T>0 $ small enough, $ \mathbf{F} $ is a contraction of $ \mathcal{C}([0,T],H^{s}) $.
The Banach fixed point theorem therefore gives a unique solution $ \sigma^2_{\psi,2} $ to \eqref{eq insta eq type A + epsilon B 2} in $ \mathcal{C}([0,T],H^{s}) $ for $ s\in[0,+\infty) $, that reads, 
\begin{align*}
	\sigma^2_{\psi,2}(t,x)&=
	\Phi\,\mathbb{J}^{\phi,3}_{\nu,2}\Big[\bar{\sigma}^1_{\phi,3},\Psi\,\partial_{\Theta}\,\mathbb{J}^{\phi,3}_{\psi,2}\big(\bar{\sigma}^1_{\phi,3},\sigma^2_{\psi,2}\big)\Big]\\
	&-\Phi\,\mathbb{J}^{\phi,3}_{\nu,2}\left[\bar{\sigma}^1_{\phi,3},\Psi\,\partial_{\Theta}\,\mathbb{J}^{\phi,1}_{\psi,1}\left(\bar{\sigma}^1_{\phi,1},\vphantom{\int_0^{t-x_d} }\right.\right.\\\nonumber
	&\hspace{50pt}\left.\left.\indicatrice_{x_d\leq t}\int_0^{t-x_d}b_{\psi}\cdot\partial_{\Theta}\, H\Big(s,y+x_d\,\big(\mathbf{v}^{\Lop}_{\psi}+\mathbf{v}_{\psi,1}'\big)-\mathbf{v}^{\Lop}_{\psi}\,(t-s),\Theta\Big)\,ds\right)\right].
\end{align*}
Therefore, according to the expression of $ \Phi $ and taking the trace in $ x_d=0 $, we have
\begin{align}\label{eq insta inter 2}
	&\quad\big(\sigma^2_{\psi,2}\big)_{|x_d=0}(t,y)-
	\Phi\,\mathbb{J}^{\phi,3}_{\nu,2}\Big[\bar{\sigma}^1_{\phi,3},\Psi\,\partial_{\Theta}\,\mathbb{J}^{\phi,3}_{\psi,2}\big(\bar{\sigma}^1_{\phi,3},\sigma^2_{\psi,2}\big)\Big]_{|x_d=0}(t,y)\\\nonumber
	&=\int_0^t\partial_{\Theta}\,\mathbb{J}^{\phi,3}_{\nu,2}\left[\bar{\sigma}^1_{\phi,3},\Psi\,\partial_{\Theta}\,\mathbb{J}^{\phi,1}_{\psi,1}\left(\bar{\sigma}^1_{\phi,1},\vphantom{\int_0^{t-x_d} }\int_0^{t}b_{\psi}\cdot\partial_{\Theta}\, H\Big(s,y-\mathbf{v}^{\Lop}_{\psi}\,(t-s)\Big)\,ds\right)\right]\\\nonumber
	&\quad\big(s,y+\mathbf{v}_{\psi,2}\,(t-s)\big)\,ds.
\end{align}
Using similar arguments as the one of Theorem \ref{th il existe H BBmod 1}, we can construct profiles $ \bar{\sigma}^1_{\phi,1} $ and $ \bar{\sigma}^1_{\phi,3} $ and choose a boundary term $ H $ such that the right-hand side of equation \eqref{eq insta inter 2} is nonzero. Therefore, we obtain
\begin{equation}
	\mathbf{C}:=\norme{\big(\sigma^2_{\psi,2}\big)_{|x_d=0}-
		\Phi\,\mathbb{J}^{\phi,3}_{\nu,2}\Big[\bar{\sigma}^1_{\phi,3},\Psi\,\partial_{\Theta}\,\mathbb{J}^{\phi,3}_{\psi,2}\big(\bar{\sigma}^1_{\phi,3},\sigma^2_{\psi,2}\big)\Big]_{|x_d=0}}_{\mathcal{C}([0,T],L^2(\R^{d-1}\times\T))}>0.
\end{equation}
Now note that, using the exact same arguments as the one used to prove estimates \eqref{eq insta est Psi f}, \eqref{eq insta est J per bilin} and \eqref{eq insta est Phi 2}, we can prove that theses estimates \eqref{eq insta est Psi f}, \eqref{eq insta est J per bilin} and \eqref{eq insta est Phi 2} still hold for the traces, in $ \mathcal{C}([0,T],L^2(\R^{d-1}\times\T)) $. Therefore, we get
\begin{align*}
	\mathbf{C}&\leq \norme{\big(\sigma^2_{\psi,2}\big)_{|x_d=0}}_{\mathcal{C}([0,T],L^2)}+\norme{\Phi\,\mathbb{J}^{\phi,3}_{\nu,2}\Big[\bar{\sigma}^1_{\phi,3},\Psi\,\partial_{\Theta}\,\mathbb{J}^{\phi,3}_{\psi,2}\big(\bar{\sigma}^1_{\phi,3},\sigma^2_{\psi,2}\big)\Big]_{|x_d=0}}_{\mathcal{C}([0,T],L^2)}\\
	&\leq \norme{\big(\sigma^2_{\psi,2}\big)_{|x_d=0}}_{\mathcal{C}([0,T],L^2)}+C\,T\norme{\Psi\,\partial_{\Theta}\,\mathbb{J}^{\phi,3}_{\psi,2}\big(\bar{\sigma}^1_{\phi,3},\sigma^2_{\psi,2}\big)_{|x_d=0}}_{\mathcal{C}([0,T],L^2)}\\
	&\leq \norme{\big(\sigma^2_{\psi,2}\big)_{|x_d=0}}_{\mathcal{C}([0,T],L^2)}+C\,T^2\norme{\big(\sigma^2_{\psi,2}\big)_{|x_d=0}}_{\mathcal{C}([0,T],L^2)}\\
	&= \big(1+C\,T^2\big)\norme{\big(\sigma^2_{\psi,2}\big)_{|x_d=0}}_{\mathcal{C}([0,T],L^2)}.
\end{align*}
In conclusion, we obtain that the norm $ \norme{\big(\sigma^2_{\psi,2}\big)_{|x_d=0}}_{\mathcal{C}([0,T],L^2(\R^{d-1)\times\T}))} $ is positive, so, for $ T $ sufficiently small, the trace $ \big(\sigma^2_{\psi,2}\big)_{|x_d=0} $ is nonzero, which is the sought result, concluding the proof.
\end{proof}

In the same manner as for the first simplified model, Theorem \ref{th il existe H BBmod 2} proves that an instability is created for the simplified model \eqref{eq insta eq 1cor BB2}, \eqref{eq insta eq 1cor apsi evol BB2} and \eqref{eq insta eq 1cor apsi bord BB2}.

\bigskip

Once again, as in section \ref{section existence}, it is conceivable to consider a more complex simplified model than  \eqref{eq insta eq 1cor BB2}, \eqref{eq insta eq 1cor apsi evol BB2} and \eqref{eq insta eq 1cor apsi bord BB2}, by integrating, in the equations \eqref{eq insta eq 1cor BB2}, coupling terms with profiles $ \sigma_{\zeta,j} $ with $ \zeta\neq \phi,\psi,\nu $. What seems to be a further step is to add, in equation \eqref{eq insta eq 1cor apsi evol BB2}, terms involving the traces of interior profiles. Among these terms, the one that seems to raise the most difficult issue is $ i\lambda\,b_{\psi}\cdot B\,\big(U^{3,\osc}_{\psi,2,\lambda}\big)_{|x_d,\chi_d=0} $, since it couples equation \eqref{eq insta eq 1cor apsi evol BB2} on the first corrector with the second corrector $ U^3 $.

\section{The example of gas dynamics}\label{section Euler}

We study here the example of three dimensional compressible isentropic Euler equations. The aim is to determine whether or not the configuration of frequencies considered in this work can happen for this system. For $ \mathcal{C}^1 $ solutions, away from vacuum, the equations read
\begin{equation}\label{eq Euler systeme}
	\left\lbrace \begin{array}{lr}
		\partial_tV^{\epsilon}+A_1(V^{\epsilon})\,\partial_1V^{\epsilon}+A_2(V^{\epsilon})\,\partial_2V^{\epsilon}+A_3(V^{\epsilon})\,\partial_3V^{\epsilon}=0&\qquad \mbox{in } \Omega_T, \\[5pt]
		B\,V^{\epsilon}_{|x_3=0}=\epsilon^2\, g^{\epsilon}+\epsilon^M\,h^{\epsilon}&\qquad \mbox{on } \omega_T,  \\[10pt]
		V^{\epsilon}_{|t=0}=0,&
	\end{array}
	\right.
\end{equation}
with $ V^{\epsilon}=(v^{\epsilon},\mathbf{u}^{\epsilon})\in\R^*_+\times\R^3 $, where $ v^{\epsilon}\in\R^*_+ $ represents the fluid volume, and $ \mathbf{u}^{\epsilon}\in\R^3 $ its velocity, and where the functions $ A_j $, $ j=1,2,3 $ are defined on $ \R^*_+\times\R^3 $ as
\begin{equation}
	A_j(V):=\left(\begin{array}{cc}
		\mathbf{u}_j & -v \,^te_j\\[5pt]
		-c(v)^2/v\,e_j  & \mathbf{u}_j\,I_3 
	\end{array}\right)\in\mathcal{M}_4(\R),
\end{equation}
where $ e_j $ is the $ j $-th vector of the canonical basis,
%\begin{align}\label{eq Euler def A1 et A2}
%	& A_1(V):=\left(\begin{array}{cccc}
%		\mathbf{u}_1 & -v & 0 & 0\\[5pt]
%		-c(v)^2/v & \mathbf{u}_1 & 0 & 0\\[5pt]
%		0 & 0 &\mathbf{u}_1 & 0\\ [5pt]
%		 0 &  0 & 0 & \mathbf{u}_1
%	\end{array}\right),\qquad
%	A_2(V):=\left(\begin{array}{cccc}
%		\mathbf{u}_2 & 0 &-v  & 0 \\[5pt]
%		0& \mathbf{u}_2 & 0 & 0  \\[5pt]
%		-c(v)^2/v  & 0 &\mathbf{u}_2 & 0  \\[5pt]
%		0 & 0 & 0&\mathbf{u}_2
%	\end{array}\right),\\
%	& A_3(V):=\left(\begin{array}{cccc}
%	\mathbf{u}_3 & 0 & 0 & -v \\[5pt]
%	0 & \mathbf{u}_3 & 0 & 0 \\[5pt]
%	0 & 0 &\mathbf{u}_3  &  0\\[5pt]
%	-c(v)^2/v  & 0 & 0 &\mathbf{u}_3  \\[5pt]
%\end{array}\right),
%\end{align}
with $ c(v)>0 $ representing the sound velocity in the fluid, depending on its volume $ v $. We study here a perturbation of this system around the equilibrium $ V_0:=(v_0,0,0,u_0) $ with $ v_0>0 $ a fixed volume and $ (0,0,u_0) $ an incoming subsonic velocity, that is, such that $ 0<u_0<c(v_0) $. We denote by $ c_0:=c(v_0) $ the sound velocity in a fluid of the fixed volume $ v_0 $.

In order to study the possibility of existence of a configuration of frequencies satisfying Assumption \ref{hypothese ensemble frequences}, we need to determine a matrix $ B $ satisfying Assumption \ref{hypothese weak lopatinskii}. For which we need to know the dimension of the stable subspace $ E_-(\zeta) $, and construct a basis of it.

\bigskip

Although it will not be used in this part, we derive the expression of various quantities related to hyperbolicity of the Euler system. For $ (\eta,\xi)\in\R^2\times\R $, the matrix $ A(\eta,\xi):=\eta_1\,A_1(V_0)+\eta_2\,A_2(V_0)+\xi\,A_3(V_0) $ associated with the system \eqref{eq Euler systeme} is given by
\begin{equation*}
	A(\eta,\xi)=\left(\begin{array}{ccc}
		u_0 \,\xi & -v_0 \,^t\eta & -v_0\,\xi\\[5pt]
		-c_0^2/v_0\,\eta   & u_0 \,\xi\,I_2  &0 \\[5pt]
		-c_0^2/v_0\,\xi  &  0 & u_0 \,\xi
	\end{array}\right),
%	\left(\begin{array}{cccc}
%		u_0\,\xi & -v_0\, \eta_1 & -v_0\,\eta_2 & -v_0\, \xi\\[5pt]
%		-c_0^2\,\eta_1/v_0 &u_0\,\xi  &0 & 0 \\[5pt]
%		-c_0^2\,\eta_2/v_0 & 0 &u_0\,\xi  &0 \\[5pt]
%		-c_0^2\,\xi/v_0 &0 & 0 &u_0\,\xi
%	\end{array}\right),
\end{equation*}
so the polynomial $ p $ defined as $ 
	p(\tau,\eta,\xi):=\det\big(\tau\,I_4+A(\eta,\xi)\big) $ reads
\begin{equation*}
	p(\tau,\eta,\xi)=(\tau+\xi\,u_0)^2\big((\tau+\xi\,u_0)^2-c_0^2(|\eta|^2+\xi^2)\big).
\end{equation*}
Thus the matrix $ A(\eta,\xi) $ admits a double eigenvalue $ -\tau_2(\eta,\xi) $ and two simple eigenvalues $ -\tau_1(\eta,\xi) $ and $ -\tau_3(\eta,\xi) $ given by
\begin{equation*}
	\tau_1(\eta,\xi)=-\xi\,u_0-c_0\sqrt{|\eta|^2+\xi^2},\quad \tau_2(\eta,\xi)=-\xi\,u_0,\quad \tau_3(\eta,\xi)=-\xi\,u_0+c_0\sqrt{|\eta|^2+\xi^2}.
\end{equation*}
Note that since  $ -\tau_2(\eta,\xi) $ is a double eigenvalue, the Euler system is not strictly hyperbolic, but hyperbolic with constant multiplicity. Despite this difference with Assumption \ref{hypothese stricte hyp}, we study this system since Assumption \ref{hypothese stricte hyp} seems to be a technical assumption.

Now to determine the expression of the stable subspace $ E_-(\zeta) $, we need to study the eigenvalues of the matrix 
\begin{equation*}
	\mathcal{A}(\tau,\eta):=-i\,A_3(V_0)^{-1}\Big(\tau I+\eta_1\, A_1(V_0)+\eta_2\,A_2(V_0)\Big).
\end{equation*} 
We determine that in this case, the hyperbolic region\footnote{That is, the region where $ \mathcal{A}(\tau,\eta) $ has only pure imaginary eigenvalues and is diagonalizable.} is given by 
\begin{equation*}
	 \mathcal{H}:=\ensemble{\big(\tau,\eta\big)\,\middle|\,|\tau|>\sqrt{c_0^2-u_0^2}\,|\eta|}.
\end{equation*}
Then, for $ (\tau,\eta) $ in $ \mathcal{H} $, the eigenvalues of $ \mathcal{A}(\tau,\eta) $ are given by
\begin{subequations}\label{eq ex Euler v.p. hyp}
\begin{align}
	i\,\xi_1(\tau,\eta)&:=i\,\frac{\tau\, u_0-\sign(\tau)\,c_0\,\sqrt{\tau^2-|\eta|^2\,(c_0^2-u_0^2)}}{c_0^2-u_0^2},\\
	i\,\xi_2(\tau,\eta)&:=i\,\frac{\tau \,u_0+\sign(\tau)\,c_0\,\sqrt{\tau^2-|\eta|^2\,(c_0^2-u_0^2)}}{c_0^2-u_0^2},\\\label{eq freq def xi 3}
	i\,\xi_3(\tau,\eta)&:=i\,\frac{-\tau}{u_0},
\end{align}
\end{subequations}
where $ \sign(x):=x/|x| $ for $ x\neq0 $. The eigenvalue $ i\,\xi_3 $ is double, when the two others are simple. We determine that, if we denote $ \alpha_j(\tau,\eta):=\big(\tau,\eta,\xi_j(\tau,\eta)\big) $, the frequency $ \alpha_2(\tau,\eta) $ is outgoing when frequencies $ \alpha_1(\tau,\eta) $ and $ \alpha_3(\tau,\eta) $ are incoming. Since $ i\,\xi_3(\tau,\eta) $ is a double eigenvalue, the dimension $ p $ of the stable subspace $ E_-(\zeta) $ is therefore equal to $ 3 $. This could also have been determined by the number of positive eigenvalues of $ A_3(V_0) $.

The interest is now made on a basis for $ E_-(\zeta) $, which, according to \eqref{eq decomp E_-(zeta)}, can be constructed with eigenvectors of $ \mathcal{A}(\zeta) $ associated with incoming frequencies. We determine that the eigenvectors associated with eigenvalues $ i\,\xi_1(\zeta) $ and $ i\,\xi_3(\zeta) $ are respectively given by
\begin{equation*}
	\lambda\begin{pmatrix}
		|(\eta,\xi_2(\tau,\eta))|\,v_0 \\ 
		c_0 \,\eta \\
		c_0\,\xi_2(\tau,\eta)
	\end{pmatrix},\quad \begin{pmatrix}
		0 \\ \mathbf{a}
	\end{pmatrix},
\end{equation*}
with $ \lambda\in\R $ and where $ \mathbf{a} $ is any vector satisfying $ \mathbf{a}\cdot\big(\eta,\xi_3(\tau,\eta)\big)=0 $. For $ \mathbf{a} $ we can choose for example the two linearly independent vectors $ (\tau\,\eta,|\eta|^2\,u_0) $ and $ (c_0\,\eta_2,-c_0\,\eta_1,0) $ to obtain the following basis of the stable subspace $ E_-(\zeta) $:
\begin{equation*}
r_1(\zeta):=	\begin{pmatrix}
	|(\eta,\xi_2(\tau,\eta))|\,v_0 \\ 
	c_0 \,\eta \\
	c_0\,\xi_2(\tau,\eta)
\end{pmatrix},\quad r^1_3(\zeta):=\begin{pmatrix}
0 \\ \tau\,\eta \\ |\eta|^2\,u_0
\end{pmatrix},\quad r^2_3(\zeta):=\begin{pmatrix}
0 \\ c_0\,\eta_2 \\ -c_0\,\eta_1 \\ 0
\end{pmatrix},
\end{equation*}

We look now for a matrix $ B $, of size $ 3\times 4 $, satisfying the weak Kreiss-Lopatinskii condition \ref{hypothese weak lopatinskii}. More precisely, we want a matrix $ B $ such that $ \ker B \cap E_-(\zeta) $ is nonzero on the specific frequency $ \tau=c_0\,|\eta| $. Note that here we make a restrictive choice, about the locus where $ \ker B \cap E_-(\zeta) $ should be nontrivial. This choice is made since it makes the following computations easier. Since every quantity is homogeneous of degree 1, we can make the computations for $ |\eta|=1 $. For $ \tau=c_0\,|\eta| $ we have $ \xi_2(\tau,\eta)=0 $,
so, denoting $ \eta=(\cos\theta,\sin\theta) $, basis $ \ensemble{r_1(\zeta),r_3^1(\zeta),r_3^2(\zeta)} $ reads
\begin{equation*}
	r_1(\zeta)=	\begin{pmatrix}
		v_0 \\ 
		c_0 \,\cos\theta\\
		c_0 \,\sin\theta\\
		0
	\end{pmatrix},\quad r^1_3(\zeta)=\begin{pmatrix}
		0 \\  c_0 \,\cos\theta\\
		c_0 \,\sin\theta
		\\ u_0
	\end{pmatrix},\quad r^2_3(\zeta)=\begin{pmatrix}
	0 \\  c_0 \,\sin\theta\\
	-c_0 \,\cos\theta\\0
\end{pmatrix}.
\end{equation*}
The condition that $ \ker B \cap E_-(\zeta) $ is trivial is equivalent to the three vectors $ B\,r_1(\zeta), B\,r_3^1(\zeta),\linebreak B\,r_3^2(\zeta) $ being linearly dependent. To study this condition, we write $ B $ in column as
\begin{equation*}
	B=\begin{pmatrix}
	b_1 & b_2 & b_3 & b_4
	\end{pmatrix},
\end{equation*}
and, since $ B $ has to be of rank 3, we can assume that column $ b_4 $ is a linear combination of the three linearly independent vectors $ b_1,b_2,b_3 $ which we chose to be the canonical basis of $ \R^3 $. We write $ 	b_4=\mu_1\,b_1+\mu_2\,b_2+\mu_3\,b_3 $, with $ \mu_1,\mu_2,\mu_3\in\R $. In this notation, the linear dependence of $ B\,r_1(\zeta), B\,r_3^1(\zeta), B\,r_3^2(\zeta) $ is equivalent to 
\begin{equation*}
	v_0\,c_0^2=\mu_1\,u_0\,c_0^2\qquad \mbox{and}\qquad \mu_2\,v_0\,c_0\,\cos\theta=\mu_3\,v_0\,c_0\,\sin\theta,\quad \forall\theta\in\T,
\end{equation*}
so $ \mu_1=v_0/u_0 $ and $ \mu_2=\mu_3=0 $. Multiplying $ B $ by a nonzero constant we obtain 
\begin{equation*}
	B=\begin{pmatrix}
		u_0& 0 & 0 & v_0 \\
		0 & u_0 & 0 & 0\\
		0 & 0 & u_0 & 0
	\end{pmatrix},
\end{equation*}
which gives an example of a matrix $ B $ for which $ \ker B \cap E_-(\zeta) $ is nonzero, and actually of dimension $ 1 $, on $ \tau=c_0\,|\eta| $.

We investigate now if $ \ker B \cap E_-(\zeta) $ is nontrivial only on $ \tau=c_0\,|\eta| $. At this purpose we introduce a practical tool, the Lopatinskii determinant (see \cite[section 4.2.2]{BenzoniSerre2007Multi}), denoted by $ \Delta(\sigma,\eta) $ for $ (\sigma,\eta)\in\Xi $. It is a scalar function such that its zeros are exactly the frequencies for which $ \ker B \cap E_-(\zeta) $ is nontrivial. Its construction can be found in \cite[section 4.6.1]{BenzoniSerre2007Multi}. If we write $ E_-(\sigma,\eta) $ as\footnote{Notation $ \,^{\perp} $ refers to the orthogonal complement relatively to the complex scalar product.}
\begin{equation*}
	E_-(\sigma,\eta)=\ell(\sigma,\eta)^{\perp},
\end{equation*} 
the Lopatinskii determinant is given by the following block determinant:
\begin{equation*}
	\Delta(\sigma,\eta)=\begin{vmatrix}
		B \\ \ell (\sigma,\eta)
	\end{vmatrix}.
\end{equation*}
Calculations made in \cite[section 14.3.1]{BenzoniSerre2007Multi} show that, for $ (\sigma,\eta)\in\Xi $, we can choose 
\begin{equation*}
	\ell(\sigma,\eta):=\big(a,-v_0\,u_0\,^{t}\eta,v_0\,\sigma\big)
\end{equation*}
with
\begin{equation*}
	a:=u_0\,\sigma-\xi_-\,(c_0^2-u_0^2),
\end{equation*}
$ \xi_- $ being the root of negative real part of the following dispersion relation 
\begin{equation*}
	(\sigma+u_0\,\xi)^2-c_0^2\,(\xi^2+|\eta|^2)=0.
\end{equation*} 
Thus the Lopatinskii determinant is given by
\begin{equation*}
	\Delta(\sigma,\eta):=\begin{vmatrix}
		u_0& 0 & 0 & v_0 \\
		0 & v_0 & 0 & 0\\
		0 & 0 & v_0 & 0\\
		a & -v_0\,u_0\,\eta_1 & -v_0\,u_0\,\eta_2 & v_0\,\sigma
	\end{vmatrix}=v_0^3\big[u_0\,\sigma-a\big]=v_0^3\,\xi_-(c_0^2-u_0^2).
\end{equation*}
It is zero if and only if $ \xi_- $ is zero, and this is the case only when $ \sigma $ is real (i.e. for $ (\sigma,\eta)=(\tau,\eta)\in\Xi_0 $) and $ \tau=c_0\,|\eta| $. Therefore $ \ker B \cap E_-(\zeta) $ is nontrivial only on $ \tau=c_0\,|\eta| $, and thus matrix $ B $ satisfies Assumption \ref{hypothese weak lopatinskii}, with 
\begin{equation*}
	\Upsilon:=\ensemble{(\tau,\eta)\,\middle|\,\tau=c_0\,|\eta|}.
\end{equation*}

\bigskip

Now that we have determined a boundary condition $ B $ suited for our problem, we take interest into oscillations.
Thus we consider two hyperbolic frequencies $ \phi $ and $ \psi $ on the boundary which will satisfy our assumptions. First, according to Assumption \ref{hypothese ensemble frequences}, frequencies $ \phi $ and $ \psi $ must be zeros of the Lopatinskii determinant, thus satisfy $ \tau=c_0|\eta| $. If we still take $ |\eta|=1 $, it leads to consider
\begin{equation*}
	\phi:=(c_0,\cos\theta_{\phi},\sin\theta_{\phi})\quad \text{and}\quad \psi:=(c_0,\cos\theta_{\psi},\sin\theta_{\psi}),
\end{equation*}
with $ \theta_{\phi},\theta_{\psi}\in[0,2\pi) $. An immediate computation then gives
\begin{equation*}
	\xi_1(\phi)=\xi_1(\psi)=0,\quad \xi_2(\phi)=\xi_2(\psi)=\frac{2\,M}{1-M^2},\quad \xi_3(\phi)=\xi_3(\psi)=-\inv{M},
\end{equation*}
with $ M:=u_0/c_0\in(0,1) $ being the Mach number.
Therefore, in order to have no resonances between frequencies lifted from $ \phi $ and no resonances between frequencies lifted from $ \psi $, it is sufficient to assume $ M^2 $ irrational.

We now look for a boundary frequency $ \nu:=-\lambda_{\phi}\,\phi-\lambda_{\psi}\,\psi $ with $ \lambda_{\phi},\lambda_{\psi}\in
\Z^* $, which satisfies Assumption \ref{hypothese ensemble frequences}. Frequency $ \nu $ reads
\begin{equation*}
	\nu=\big(-c_0(\lambda_{\phi}+\lambda_{\psi}),-\lambda_{\phi}\cos\theta_{\phi}-\lambda_{\psi}\cos\theta_{\psi},-\lambda_{\phi}\sin\theta_{\phi}-\lambda_{\psi}\sin\theta_{\psi}\big).
\end{equation*}
First we determine in which case $ \nu $ is not in $ \Upsilon $. If we denote $ \nu=(\tau,\eta) $, we have
\begin{equation*}
	\tau^2=c_0^2(\lambda_{\phi}+\lambda_{\psi})^2\quad \text{and}\quad c_0^2\,|\eta|^2=c_0^2(\lambda_{\phi}^2+\lambda_{\psi}^2)+2c_0^2\lambda_{\phi}\lambda_{\psi}\cos(\theta_{\phi}-\theta_{\psi}),
\end{equation*}
so, according to the description of $ \Upsilon $, frequency $ \nu $ is not in $ \Upsilon $ if and only if $ \theta_{\phi}\neq\theta_{\psi} $. Generalizing this to any frequency $ \zeta=\lambda_1\,\phi+\lambda_2\,\psi\in\F_{b}\privede{0} $ asserts that $ \F_b\cap\Upsilon=\ensemble{\phi,-\phi,\psi,-\psi} $ as required by Assumption \ref{hypothese ensemble frequences}. This assumption also demands $ \nu $ to be in the hyperbolic region. We have, if we still denote $ \nu=(\tau,\eta) $,
\begin{equation*}
	\tau^2-|\eta|^2(c_0^2-u_0^2)=u_0^2(\lambda_{\phi}+\lambda_{\psi})^2+2\lambda_{\phi}\lambda_{\psi}(c_0^2-u_0^2)\big[1-\cos(\theta_{\phi}-\theta_{\psi}))\big],
\end{equation*}
so the hyperbolicity condition $ \tau^2-|\eta|^2(c_0^2-u_0^2)>0 $ reads
\begin{equation}\label{eq Euler hyp cond}
	M^2(\lambda_{\phi}+\lambda_{\psi})^2+2\lambda_{\phi}\lambda_{\psi}(1-M^2)\big[1-\cos(\theta_{\phi}-\theta_{\psi}))\big]>0,
\end{equation}
which is satisfied for example when $ \lambda_{\phi} $ and $ \lambda_{\psi} $ are positive.

We take interest now in resonance assumptions \eqref{eq hyp res phi psi nu 1} and \eqref{eq hyp res phi psi nu 2}. We compute
\begin{align*}
	\xi_1(\nu)&=\frac{-M(\lambda_{\phi}+\lambda_{\psi})}{1-M^2}\\
	&\quad+\frac{\sign(\lambda_{\phi}+\lambda_{\psi})\sqrt{2\lambda_{\phi}\lambda_{\psi}\big[1-\cos(\theta_{\phi}-\theta_{\psi})\big]+M^2\big[\lambda_{\phi}^2+\lambda_{\psi}^2+2\lambda_{\phi}\lambda_{\psi}\cos(\theta_{\phi}-\theta_{\psi})\big]}}{1-M^2}\\[5pt]
	\xi_2(\nu)&=\frac{-M(\lambda_{\phi}+\lambda_{\psi})}{1-M^2}\\
	&\quad-\frac{\sign(\lambda_{\phi}+\lambda_{\psi})\sqrt{2\lambda_{\phi}\lambda_{\psi}\big[1-\cos(\theta_{\phi}-\theta_{\psi})\big]+M^2\big[\lambda_{\phi}^2+\lambda_{\psi}^2+2\lambda_{\phi}\lambda_{\psi}\cos(\theta_{\phi}-\theta_{\psi})\big]}}{1-M^2}\\[5pt]
	\xi_3(\nu)&=\frac{\lambda_{\phi}+\lambda_{\psi}}{M}.
\end{align*}
Recalling Remark \ref{remark numbering phi psi} about the numbering of frequencies, we need to check the four possibilities for the couple of resonance, namely, 
\begin{subequations}
	\begin{align}\label{eq Euler res 1}
	\lambda_{\phi}\,\phi_1+\lambda_{\psi}\,\psi_1+\nu_2=0 & \quad \text{and} \quad \lambda_{\phi}\,\phi_3+\lambda_{\psi}\,\psi_2+\nu_2=0,\\\label{eq Euler res 2}
	\lambda_{\phi}\,\phi_1+\lambda_{\psi}\,\psi_3+\nu_2=0 & \quad \text{and} \quad \lambda_{\phi}\,\phi_3+\lambda_{\psi}\,\psi_2+\nu_2=0,\\\label{eq Euler res 3}
	\lambda_{\phi}\,\phi_3+\lambda_{\psi}\,\psi_1+\nu_2=0 & \quad \text{and} \quad \lambda_{\phi}\,\phi_1+\lambda_{\psi}\,\psi_2+\nu_2=0,\\\label{eq Euler res 4}
	\lambda_{\phi}\,\phi_3+\lambda_{\psi}\,\psi_1+\nu_2=0 & \quad \text{and} \quad \lambda_{\phi}\,\phi_1+\lambda_{\psi}\,\psi_2+\nu_2=0.
\end{align}
\end{subequations}
\begin{itemize}[leftmargin=15pt]
\item Since $ \xi_1(\phi)=\xi_1(\psi)=0 $, relation $	\lambda_{\phi}\,\phi_1+\lambda_{\psi}\,\psi_1+\nu_2=0 $ implies that $ \xi_2(\nu)=0 $, and therefore $ \lambda_{\phi}+\lambda_{\psi}=0 $ which is impossible, since it contradicts condition \eqref{eq Euler hyp cond}.
\item We determine that $	\lambda_{\phi}\,\phi_1+\lambda_{\psi}\,\psi_3+\nu_2=0 $ is equivalent to 
\begin{equation*}
	2M^2\lambda_{\phi}\lambda_{\psi}\big[2-\cos(\theta_{\phi}-\theta_{\psi})\big]+M^4\big[\lambda_{\psi}^2+2\lambda_{\phi}\lambda_{\psi}\cos(\theta_{\phi}-\theta_{\psi})\big]-\lambda_{\psi}^2=0
\end{equation*}
and
\begin{equation}\label{eq Euler inter 1}
	(\lambda_{\phi}+\lambda_{\psi})(\lambda_{\phi}M^2-\lambda_{\psi})\geq 0.
\end{equation}
The corresponding second resonance is $ \lambda_{\phi}\,\phi_3+\lambda_{\psi}\,\psi_2+\nu_2=0 $, which is equivalent to 
\begin{equation*}
	2M^2\lambda_{\phi}\lambda_{\psi}\big[2-\cos(\theta_{\phi}-\theta_{\psi})\big]+M^4\big[\lambda_{\phi}^2+2\lambda_{\phi}\lambda_{\psi}\cos(\theta_{\phi}-\theta_{\psi})\big]-\lambda_{\phi}^2=0
\end{equation*}
and
\begin{equation}\label{eq Euler inter 2}
	(\lambda_{\phi}+\lambda_{\psi})(\lambda_{\psi} M^2-\lambda_{\phi})\geq 0.
\end{equation}
Now conditions \eqref{eq Euler inter 1} and \eqref{eq Euler inter 2} are incompatible, so the configuration of resonances \eqref{eq Euler res 1} is impossible.
\item The case of \eqref{eq Euler res 3} is analogous, and is not detailed here.
\item Finally, if for the first resonance we have $	\lambda_{\phi}\,\phi_3+\lambda_{\psi}\,\psi_3+\nu_2=0 $, then the second one must be $	\lambda_{\phi}\,\phi_1+\lambda_{\psi}\,\psi_2+\nu_2=0 $, which is equivalent to 
\begin{equation*}
	\big[1-\cos(\theta_{\phi}-\theta_{\psi})\big]+M^2\big[1+\cos(\theta_{\phi}-\theta_{\psi})\big]=0\quad \mbox{and}\quad (\lambda_{\phi}+\lambda_{\psi})(\lambda_{\psi}-\lambda_{\phi})\geq 0.
\end{equation*}
First equation rewrites
\begin{equation*}
	\cos(\theta_{\phi}-\theta_{\psi})=\frac{1+M^2}{1-M^2},
\end{equation*}
which cannot be satisfied by $ M^2\in(0,1) $. Thus the fourth possibility \eqref{eq Euler res 4} is also impossible. 
\end{itemize}
Therefore, in this case, a situation like the one described in Assumption \ref{hypothese ensemble frequences} \textbf{cannot happen}.

To conclude as for the Euler system, we need to have a discussion about where we have made a choice which puts us in a particular case. The above analysis about frequencies $ \phi $, $ \psi $ and $ \nu $ does not depend on $ B $, but only on the location of cancellation of the Lopatinskii determinant. Thus the only restrictive choice we made is to choose this location as $ \tau=c_0\,|\eta| $. Therefore, for the compressible isentropic Euler equations in space dimension 3, in this particular case, the configuration of frequencies considered in this work which leads to an instability cannot happen.

We have considered here the Euler system in space dimension 3, since, in space dimension 2, the condition $ \tau=c\,|\eta| $ leads to $ \tau=\pm c\,\eta $, preventing to obtain a transverse oscillation. We could also consider the shock problem for the Euler equations, which is the original problem of Majda and Rosales in \cite{MajdaRosales1983Machstem,MajdaRosales1984Machstem}.

\bibliography{Bibliographie.bib}
\bibliographystyle{alpha}
\end{document}